\documentclass{amsart}

\usepackage[left=1.2in, right=1.2in, top=1in, bottom=1in]{geometry}
\usepackage{enumerate}
\usepackage{enumitem}

\usepackage{amsthm}
\usepackage{amsfonts}
\usepackage{fancyhdr}
\usepackage{amssymb}
\usepackage{xy}
\newcommand{\mdd}{\md^{\mathrm{disc}}}
\newcommand{\stm}{\mathrm{St}}
 
\newcommand{\shot}{\clg( \prl_{\mathrm{st}})}
 \newcommand{\gal}{\mathrm{Gal}}
\newcommand{\fung}{\fun^{\mathrm{Gal}}}
\newcommand{\finset}{\mathrm{FinSet}}
\newcommand{\op}{\mathrm{op}}
\newcommand{\cov}{\mathrm{Cov}}
\newcommand{\loc}{\mathrm{Loc}}

\newcommand{\pro}{\mathrm{Pro}}
\newcommand{\tring}{\mathrm{2}\text{-}\mathrm{Ring}}
\usepackage{sseq}
\input xy
\xyoption{all}
\usepackage{mathrsfs}
\usepackage{amsmath}
\usepackage{url}
\usepackage{verbatim}
\usepackage{hyperref}
\usepackage{xy}
\input xy
\xyoption{all}
\newcommand{\e}[1]{\mathbf{E}_{#1}}
\newcommand{\spec}{\mathrm{Spec}}
\newcommand{\prl}{\mathrm{Pr}^L}
 \renewcommand{\part}[1]{}
\newcommand{\holim}{\mathrm{holim}}

\newcommand{\einf}{\mathrm{CAlg}}

\renewcommand{\mod}{\mathrm{Mod}}
\newcommand{\cati}{\mathrm{Cat}_\infty}

 \newcommand{\G}{\mathscr{G}}

\newcommand{\fun}{\mathrm{Fun}}
\newcommand{\fgp}{\mathrm{Gpd}_{\mathrm{fin}}}

\newcommand{\idem}{\mathrm{Idem}}
\usepackage{tikz}
\usepackage{cleveref}

\renewcommand{\sp}{\mathrm{Sp}}
\newcommand{\otop}{\mathcal{O}^{\mathrm{top}}}

\newcommand{\mell}{M_{{ell}}}

\renewcommand{\hom}{\mathrm{Hom}}
\newtheorem{lemma}{Lemma}[section]
\newtheorem*{sublemma}{Sublemma}
\newtheorem{corollary}[lemma]{Corollary}

\newtheorem{conj}[lemma]{Conjecture}
\newtheorem{theorem}[lemma]{Theorem}

\newtheorem{proposition}[lemma]{Proposition}

 \newcommand{\clg}{\mathrm{CAlg}}

\begin{document}

\pagestyle{fancy}
\lhead{}
\rhead{}
\lfoot{}
\rfoot{}
\renewcommand{\headrulewidth}{0pt}

\setlength{\parskip}{0.8mm}
\renewcommand{\rightrightarrows}{\begin{smallmatrix} \to \\
\to \end{smallmatrix} }
\newcommand{\triplearrows}{\begin{smallmatrix} \to \\ \to \\ 
\to \end{smallmatrix} }

\newcommand{\sh}{\mathbf{Sh}}
 
\renewcommand{\ltimes}{\stackrel{\mathbb{L}}{\otimes}}
\newcommand{\psh}{\mathbf{PSh}}
\theoremstyle{definition}
\newtheorem{definition}[lemma]{Definition}
\newcommand{\cd}{\mathrm{cd}}
\newcommand{\OO}{\widetilde{\mathcal{O}}}
\newcommand{\A}{\mathbb{A}}
\newcommand{\F}{\mathcal{F}}
\renewcommand{\P}{\mathbb{P}}
\newcommand{\bl}{\bullet}
\newcommand{\Tmf}{\mathrm{Tmf}}
\newcommand{\TMF}{\mathrm{TMF}}
\renewcommand{\ell}{\mathrm{Ell}}
\newcommand{\tmf}{\mathrm{tmf}}

 \theoremstyle{definition}
\newtheorem{remark}[lemma]{Remark}
\newcommand{\rng}{\mathrm{Ring}}
\newcommand{\gpd}{\mathbf{Gpd}}
\newcommand{\ei}{\mathbb{E}_1}
\renewcommand{\A}{\mathcal{A}_*}

\newcommand{\qcoha}{\qcoh^{\mathrm{ab}}}
\newtheorem{example}[lemma]{Example}
\newtheorem*{exm}{Example}
\newcommand{\md}{\mathrm{Mod}}
\newcommand{\qcoh}{\mathrm{QCoh}}
\newcommand{\pic}{\mathrm{Pic}}

%\frontmatter
\title{The Galois group of a stable homotopy theory}
\author{Akhil Mathew}
\date{\today}
\email{amathew@math.harvard.edu}
\address{Department of Mathematics, Harvard University, Cambridge, MA 02138}

\begin{abstract}
To a ``stable homotopy theory'' (a presentable, symmetric monoidal stable
$\infty$-category), we naturally associate a 
category of finite \'etale algebra objects and, using Grothendieck's 
categorical machine, a profinite group that we call the Galois group.  We then  calculate the Galois groups in several
examples. For instance, we show that the Galois group of the periodic
$\e{\infty}$-algebra of topological modular forms is trivial 
and that the
Galois group of $K(n)$-local stable homotopy theory is an extended version of
the Morava stabilizer group. 
We also describe the Galois group of the stable module
category of a finite group. 
A fundamental idea throughout 
is the purely categorical notion of a ``descendable'' algebra object and an associated analog of
faithfully flat descent in this context.
\end{abstract}
\maketitle

\tableofcontents

\section{Introduction}

Let $X$ be a connected scheme. One of the basic arithmetic invariants that one
can extract from $X$ is the \emph{\'etale fundamental group} $\pi_1(X,
\overline{x})$ relative to a ``basepoint'' $\overline{x} \to X$ (where
$\overline{x}$ is the spectrum of a separably closed field). 
The fundamental group was defined by Grothendieck \cite{sga1} in terms of the category of finite, \'etale
covers of $X$. It provides an analog of the usual 
fundamental group of a topological space (or rather, its profinite
completion), and plays an important role in algebraic geometry and number
theory, as a precursor to the theory of \'etale cohomology. 
From a categorical point of view, it unifies the classical Galois theory of
fields and covering space theory via a single framework.

In this paper, we will define an analog of the \'etale fundamental group, and
construct a form of the Galois correspondence, in 
stable homotopy theory. 
For example, while the classical theory of \cite{sga1} enables one to
define the fundamental (or Galois) group of a commutative ring, we will define
the fundamental group of the homotopy-theoretic analog: an $\e{\infty}$-ring
spectrum. 

The idea of a type of Galois theory applicable to structured ring spectra
begins with Rognes's work in \cite{rognes}, where, for a finite group $G$,
the notion of a \emph{$G$-Galois extension} of $\e{\infty}$-ring spectra $A \to
B$ was introduced (and more generally, $E$-local $G$-Galois extensions for a
spectrum $E$). 
Rognes's definition is an analog of the notion of a finite $G$-torsor of commutative rings
in the setting of ``brave new'' algebra, and it includes many highly
non-algebraic 
examples in stable homotopy theory. For instance, the ``complexification'' map
$KO \to KU$ from real to complex 
$K$-theory is a fundamental example of a $\mathbb{Z}/2$-Galois extension.
Rognes has also explored  the more 
general theory of \emph{Hopf-Galois} extensions, intended as a topological
version of the idea of a torsor over a group \emph{scheme} in algebraic
geometry, as has Hess in \cite{hess}. 
More recently, the PhD thesis of Pauwels  \cite{pauwels} has studied Galois theory in
tensor-triangulated categories. 

In this paper, we
will take the setup of an \emph{axiomatic stable homotopy theory.}  
For us, this will mean: 

\renewcommand{\1}{\mathbf{1}}
\begin{definition} 
An \textbf{axiomatic stable homotopy theory} is a presentable, symmetric monoidal stable
$\infty$-category $(\mathcal{C}, \otimes, \1)$ where the tensor product
commutes with all colimits. 
\end{definition} 

An axiomatic stable homotopy theory defines, at the level of homotopy
categories, a \emph{tensor-triangulated category}. Such axiomatic stable homotopy theories arise not only from stable
homotopy theory itself, but also from representation theory and algebra, and we
will discuss many examples below. 
We will associate, to every
axiomatic stable homotopy theory $\mathcal{C}$, a profinite group (or, in general,
groupoid) which we call the
\emph{Galois group} $\pi_1( \mathcal{C})$. In order to do this, we will give a definition of a
\emph{finite cover} generalizing the notion of a Galois extension,
and, using heavily ideas from descent theory, show that these can naturally be arranged into a Galois category in the
sense of Grothendieck. 
We will actually define two flavors of the fundamental group, one of which
depends only on the structure of the dualizable objects in $\mathcal{C}$ and is
appropriate to the study of ``small'' symmetric monoidal $\infty$-categories. 

Our thesis is that the Galois group of a stable homotopy theory is a natural
invariant that one can attach to it; some of the (better studied) others include the 
algebraic $K$-theory (of the compact objects, say), the lattice of thick
subcategories, and the Picard group. We will discuss several
examples. The classical fundamental group in algebraic geometry can be
recovered as the Galois group of the derived category of quasi-coherent
sheaves. Rognes's Galois theory (or rather, \emph{faithful} Galois theory) is
the case of $\mathcal{C}  = \mod(R)$ 
for $R$ an $\e{\infty}$-algebra.  

Given a stable homotopy theory $(\mathcal{C}, \otimes, \mathbf{1})$, the
collection of all homotopy classes of maps $\mathbf{1} \to \mathbf{1}$ is
naturally a commutative ring $R_{\mathcal{C}}$. 
In general, there is always a surjection of profinite groups
\begin{equation} \label{keymap} \pi_1 \mathcal{C}
\twoheadrightarrow \pi_1^{\mathrm{et}}
\spec  R_{\mathcal{C}}.  \end{equation}
The \'etale fundamental group of $\spec R_{\mathcal{C}}$ represents the ``algebraic''
part of the Galois theory of $\mathcal{C}$. 
For example, if $\mathcal{C}  = \md(R)$ for $R$ an $\e{\infty}$-algebra, then
the ``algebraic'' part of the Galois theory of $\mathcal{C}$ corresponds to
those $\e{\infty}$-algebras under $R$ which are finite \'etale at the level of
homotopy groups. It is an insight of Rognes that, in general, the Galois group
contains a topological component as well: the map 
\eqref{keymap} is generally not an isomorphism. The remaining Galois extensions
(which behave much differently on the level of homotopy groups) can be quite
useful computationally.

In the rest of the paper, we will describe several computations of these Galois
groups in various settings. 
Our basic tool is the following result, which is a refinement of (a natural
generalization of) the main result of \cite{BR2}. 

\begin{theorem} If $R$ is an even periodic $\e{\infty}$-ring with $\pi_0 R$
regular noetherian, then the Galois group of $R$ is that of the
discrete ring
$\pi_0 R$: that is, \eqref{keymap} is an isomorphism. 
\end{theorem} 

Using various techniques of descent theory, and a version of van Kampen's
theorem, we are able to compute Galois groups in
several other examples of stable homotopy theories ``built'' from $\mod(R)$
where $R$ is an even periodic $\e{\infty}$-ring; these include in particular many arising from 
both chromatic stable homotopy theory and modular representation theory. 
In particular, we prove the following three theorems. 

\begin{theorem} 
The Galois group of the $\infty$-category $L_{K(n)} \sp$ of $K(n)$-local
spectra  is the extended Morava stabilizer group. 
\end{theorem} 

\begin{theorem} 
The Galois group of the $\e{\infty}$-algebra $\TMF$ of (periodic) topological
modular forms is trivial. 
\end{theorem} 

\begin{theorem} 
Given a finite group $G$ and a separably closed field $k$ of
characteristic $p$, the Galois group of the stable module $\infty$-category
of $k[G]$ is the profinite completion of the nerve of the category
of $G$-sets of the form $\left\{G/A\right\}$ where $ A \subset G$
is a nontrivial elementary abelian $p$-subgroup. \end{theorem} 

These results suggest a number of other settings in which the computation of
Galois groups may be feasible, for example, 
in stable module $\infty$-categories for finite group \emph{schemes}. We hope that these
results and ideas will, in addition, shed light on some of the other invariants
of $\e{\infty}$-ring spectra and stable
homotopy theories. 

\subsection*{Acknowledgments} 

I would like to thank heartily Mike Hopkins for his advice and
support over the past few years, with this project and others. In addition, I
would like to thank Bhargav Bhatt, Brian Conrad, Gijs Heuts,  
Tyler Lawson, Lennart Meier, Niko Naumann, Justin Noel,  
Oriol Ravent{\'o}s,
Vesna Stojanoska, and in particular Jacob Lurie, for numerous helpful discussions. 
Finally, I would like to thank the referee for several corrections. 
The author was supported by the NSF Graduate Fellowship under grant DGE-1144152.
\part{Descent theory}

\section{Axiomatic stable homotopy theory}
As mentioned earlier, the goal of this paper is to  extract a Galois group(oid) from a \emph{stable
homotopy theory.} Once again, we restate the definition. 

\begin{definition}\label{shot}
A \textbf{stable homotopy theory} is a presentable, symmetric monoidal stable $\infty$-category
$(\mathcal{C}, \otimes, \mathbf{1})$ where the tensor product commutes with all
colimits. 
\end{definition} 

In this section, intended mostly as background, we will describe several general features of the setting
of stable homotopy
theories. We will discuss a number of examples, and then 
construct a basic class of commmutative algebra objects in any such $\mathcal{C}$ (the
so-called ``\'etale algebras'') whose associated corepresentable functors can
be described very easily. 
The homotopy categories of stable homotopy theories, which acquire both a
tensor structure and a compatible triangulated structure, have been described at
length in the memoir \cite{axiomatic}. In addition, their invariants have been
studied in detail in  
the program of tensor triangular geometry of Balmer (cf. \cite{Balmer} for a
survey). 

\subsection{Stable $\infty$-categories}

Let $\mathcal{C}$ be a stable $\infty$-category in the sense of
\cite[Ch. 1]{higheralg}. Recall that stability is a \emph{condition} on an $\infty$-category,
rather than extra data, in the same manner that, in ordinary category theory, being an
abelian category is a property. 
The homotopy category of a stable $\infty$-category is canonically
\emph{triangulated,} so that stable $\infty$-categories may be viewed as 
enhancements of triangulated categories; however, as opposed to traditional
DG-enhancements, stable $\infty$-categories can be used to model phenomena in
stable homotopy theory (such as the $\infty$-category of spectra, or the
$\infty$-category of modules over a structured ring spectrum). 

Here we will describe some general features of stable $\infty$-categories, and
in particular the constructions one can perform with them. 
Most of this is folklore (in the setting of triangulated or DG-categories) or
in \cite{higheralg}.

\begin{definition} 
Let $\cati$ be the $\infty$-category of (small) $\infty$-categories. Given
$\infty$-categories $\mathcal{C}, \mathcal{D}$, the mapping space
$\hom_{\cati}(\mathcal{C}, \mathcal{D})$ is the maximal $\infty$-groupoid
contained in the $\infty$-category $\mathrm{Fun}(\mathcal{C},\mathcal{D})$ of
functors $\mathcal{C} \to \mathcal{D}$. 
\end{definition} 

\newcommand{\scat}{\mathrm{Cat}_\infty^{\mathrm{st}}}
\begin{definition} 
We define an $\infty$-category $\scat$ of (small) stable  $\infty$-categories where: 
\begin{enumerate}
\item The objects of $\scat$ are the stable $\infty$-categories which are
idempotent complete.\footnote{This can be removed, but will be assumed for
convenience.}  
\item Given $\mathcal{C}, \mathcal{D} \in \scat$, the mapping space
$\hom_{\scat}(\mathcal{C}, \mathcal{D})$ is the union of connected components in
$\hom_{\cati}(\mathcal{C}, \mathcal{D})$ spanned by those functors which
preserve finite limits (or, equivalently, colimits). Such functors are called \emph{exact.}
\end{enumerate}
\end{definition} 

The $\infty$-category $\scat$ has all limits, and the forgetful functor
$\scat \to \cati$ commutes with limits. 
For example, given a diagram in $\scat$
\[ \xymatrix{
& \mathcal{C} \ar[d]^F \\
\mathcal{D} \ar[r]^G &  \mathcal{E}
},\]
we can form a pullback $\mathcal{C} \times_{\mathcal{E}} \mathcal{D}$
consisting of triples $(X, Y, f)$ where $X \in \mathcal{C}, Y \in \mathcal{D}$,
and $f\colon F(X) \simeq G(Y)$ is an equivalence. This pullback is automatically
stable. 

Although the construction is more complicated, $\scat$ is also cocomplete.
For example, 
the colimit (in $\cati$) of a \emph{filtered} diagram of stable
$\infty$-categories and exact functors is automatically stable, so that the
inclusion $\scat \subset \cati$ preserves filtered colimits. 
In general, one has:

\begin{proposition} 
\label{scatispresentable}
$\scat$ is a presentable $\infty$-category. 
\end{proposition} 

\newcommand{\catib}{\widehat{\mathrm{Cat}_\infty}}
\newcommand{\prr}{\mathrm{Pr}^R}

To understand this, it is convenient to work with  the (big) $\infty$-category
$\prl$.

\begin{definition}[{\cite[5.5.3]{HTT}}] $\prl$
is the $\infty$-category 
of 
presentable $\infty$-categories and colimit-preserving (or left adjoint)
functors. 
\end{definition} 
The $\infty$-category $\prl$ is known to have all 
colimits (cf. \cite[5.5.3]{HTT}). We briefly review this here. 
Given a diagram $F\colon I \to \prl$, we can form the dual
$I^{\op}$-indexed diagram in
the $\infty$-category $\prr$ of presentable $\infty$-categories and
\emph{right} adjoints between them. Now we can form a \emph{limit} in $\prr$ at
the level of underlying $\infty$-categories; by duality between $\prl,
\prr$ in the form $\prl \simeq (\prr)^{\op}$, this can be identified with the
colimit $\varinjlim_I F$ in $\prl$. 

In other words, for each map $f\colon  i \to i'$ in
$I$, consider the induced adjunction of $\infty$-categories $L_f, R_f\colon  F(i) \rightleftarrows F(i')$. 
Then an object $x$ in $\varinjlim_I F$ is the data of: 
\begin{enumerate}
\item For each $ i \in I$, an object $x_i \in F(i)$.  
\item For each $f \colon  i \to i'$, an isomorphism $x_i \simeq R_f(x_{i'})$. 
\item Higher homotopies and coherences. 
\end{enumerate}

For each $i$, we get a natural functor in $\prl$, $F(i) \to \varinjlim_I F$.
We have a tautological description of the \emph{right adjoint}, which to an
object $x$ in $\varinjlim_I F$ as above returns $x_i \in F(i)$. 

\begin{example} 
Let $\mathcal{S}_\ast$ be the $\infty$-category of pointed spaces and pointed maps
between them. We have an endofunctor $\Sigma\colon  \mathcal{S}_* \to \mathcal{S}_*$
given by suspension, whose right adjoint is the loop functor $\Omega\colon 
\mathcal{S}_* \to \mathcal{S}_*$. The filtered colimit
in $\prl$
of the diagram
\[ \mathcal{S}_* \stackrel{\Sigma}{\to} \mathcal{S}_* \stackrel{\Sigma}{\to}
\dots,  \]
can be identified, by this description, as the $\infty$-category of 
sequences of pointed spaces $(X_0, X_1, X_2, \dots, )$ together with
equivalences $X_n \simeq \Omega X_{n+1}$ for $n \geq 0$: in other words, one
recovers the $\infty$-category of spectra. 
\end{example}

\begin{proposition} \label{procompact}
Suppose $F\colon  I \to \prl$ is a diagram where, for each $i \in I$, the
$\infty$-category $F(i)$ is compactly generated; and where, for each $i \to
i'$, the left adjoint $F(i) \to F(i')$ preserves compact objects.\footnote{This
is equivalent to the condition that the \emph{right adjoints} preserve filtered
colimits.} Then each $F(i) \to \varinjlim_I F$ preserves compact objects, and
$\varinjlim_I F$ is compactly generated. 
\end{proposition} 
\begin{proof} 
It follows from the explicit description of $\varinjlim_I F$, in fact, that
the right adjoints to $F(i) \to \varinjlim_I F$ preserve filtered colimits;
this is dual to the statement that the left adjoints preserve compact objects. 
Moreover, the images of each compact object in each $F(i)$ in $\varinjlim_I F$
can be taken as compact generators, since they are seen to detect equivalences. 
\end{proof} 

\newcommand{\prow}{\mathrm{Pr}^{L, \omega}}

\begin{definition} 
$\prow$ is the $\infty$-category of compactly generated, presentable
$\infty$-categories and colimit-preserving functors which preserve compact
objects. 
\end{definition}

It is fundamental that $\prow$ is equivalent to the $\infty$-category of
idempotent complete, finitely cocomplete $\infty$-categories and finitely
cocontinuous functors, under the construction $\mathcal{C} \to
\mathrm{Ind}(\mathcal{C})$ starting from the latter and ending with the
former (and the dual construction that takes an object in $\prow$ to its
subcategory of compact objects). \Cref{procompact} implies that colimits exist
in $\prow$ and the inclusion $\prow \to \prl$ preserves them. 

\begin{corollary} \label{stpr}
$\prow$ is a presentable $\infty$-category. 
\end{corollary} 
\begin{proof} 
It suffices to show that any idempotent complete, finitely cocomplete
$\infty$-category is a filtered colimit of such of bounded cardinality (when
modeled via quasi-categories, for instance). For simplicity, we will sketch the
argument for finitely cocomplete quasi-categories. The idempotent complete
case can be handled similarly by replacing filtered colimits with
$\aleph_1$-filtered  colimits. 

To see this, let $\mathcal{C}$ be such a quasi-category. 
Consider any countable simplicial subset $\mathcal{D}$ of $\mathcal{C}$ 
 which is a quasi-category.
 We will show that $\mathcal{D}$ is contained in a bigger countable simplicial
 subset $\overline{\mathcal{D}}$ of $\mathcal{C}$ which is a finitely cocomplete
 quasi-category such that $\overline{\mathcal{D}} \to \mathcal{C}$ preserves
 finite colimits. 
 This will show that $\mathcal{C}$ is the filtered union of such
 subsets
 $\overline{\mathcal{D}}$ (ordered by set-theoretic inclusion) and will thus complete the proof. 

Thus, fix $\mathcal{D} \subset \mathcal{C}$ countable. 
For each finite
simplicial set $K$, and each map $K \to \mathcal{D}$, by definition there is
an extension $K^{\rhd} \to \mathcal{C}$ which is a colimit diagram. 
We can find a countable simplicial set $\mathcal{D}'$ such that $\mathcal{D}
\subset \mathcal{D}' \subset \mathcal{C}$ such that every diagram $K \to
\mathcal{D}$ extends over a diagram $K^{\rhd} \to \mathcal{D}'$ such that the
composite $K^{\rhd} \to\mathcal{D}' \to \mathcal{C}$ is a colimit diagram in
$\mathcal{C}$. Applying the small object argument (countably many times), we
can find a countable quasi-category $\mathcal{D}_1$ with $\mathcal{D} \subset \mathcal{D}_1  \subset
\mathcal{C}$ such that any diagram $K \to \mathcal{D}_1$ extends over a diagram
$K^{\rhd} \to \mathcal{D}_1$ such that the composite $K^{\rhd} \to \mathcal{D}_1 \to
\mathcal{C}$ is a colimit diagram. It follows thus that any countable
simplicial subset $\mathcal{D}$ of $\mathcal{C}$ containing all the vertices is 
contained in such a (countable) $\mathcal{D}_1$. (At each stage in the small
object argument, we also have to add in fillers to all inner horns.)

Thus, consider any countable simplicial subset $\mathcal{D} \subset
\mathcal{C}$ which is a quasi-category containing all the vertices of
$\mathcal{C}$, and such that any diagram $K \to \mathcal{D}$ (for $K$ finite)
extends over a diagram $K^{\rhd} \to \mathcal{D}$ such that the composite
$K^{\rhd} \to \mathcal{C}$ is a colimit diagram. We have just shown that
$\mathcal{C}$ is a (filtered) union of such. Of course, $\mathcal{D}$ may not
have all the colimits we want. Consider the (countable) collection $S_{\mathcal{D}}$ of all diagrams
$f\colon  K^{\rhd} \to \mathcal{D}$ whose composite $K^{\rhd} \stackrel{f}{\to}
\mathcal{D} \to \mathcal{C}$ is a colimit. We want to enlarge $\mathcal{D}$ so
that each of these becomes a  colimit, but not too much; we want
$\mathcal{D}$ to remain countable. 

For each $f \in S_{\mathcal{D}}$, consider $\mathcal{D}_{K/} \subset
\mathcal{C}_{K/}$. By construction, we have an object in $\mathcal{D}_{K/}$
which is initial in $\mathcal{C}_{K/}$. By adding a countable number of
simplices to $\mathcal{D}$, though, we can make this initial in
$\mathcal{D}_{K/}$ too; that is, there exists a $\mathcal{D}' \subset
\mathcal{D}$ with the same properties such that the object
defined is initial in $\mathcal{D}'_{K/}$. 
Iterating this process (via the small object argument), we can construct a
countable simplicial subset $\overline{\mathcal{D}} \subset \mathcal{C}$,
containing $\mathcal{D}$, which is a quasi-category and such that any diagram
$K \to \overline{\mathcal{D}}$ extends over a diagram $K^{\rhd} \to
\overline{\mathcal{D}}$ which is a colimit preserved under
$\overline{\mathcal{D}} \to \mathcal{C}$. This completes the proof. 
\end{proof}

We can use this to describe $\scat$. We have a \emph{fully
faithful} functor
\[ \scat \to \prow,  \]
which sends a stable $\infty$-category $\mathcal{C}$ to
the \emph{compactly generated}, presentable stable
$\infty$-category $\mathrm{Ind}(\mathcal{C})$. 
In fact, $\scat$ can be identified with the $\infty$-category of stable,
presentable, and compactly generated $\infty$-categories, and
colimit-preserving functors between them that also preserve compact objects, so that $\scat \subset \prow$ as a
full subcategory.

\begin{proof}[Proof of \Cref{scatispresentable}]
We need to show that $\scat$ has all colimits. Using the explicit construction
of a colimit of presentable $\infty$-categories, however, it follows that a
colimit of presentable, \emph{stable} $\infty$-categories is stable. In
particular, $\scat$ has colimits and they are computed in $\prow$. 

Finally, we need to show that any object in $\scat$ is a filtered union of
objects in $\scat$ of bounded cardinality. This can be argued similarly as
above (we just need to add stability into the mix). 
\end{proof} 

Compare also the treatment of stable $\infty$-categories in \cite{BGT}, which
shows (cf. \cite[Th. 4.22]{BGT}) that $\scat$ can be obtained as an
accessible localization of the $\infty$-category associated to a
combinatorial model category and indeed shows that $\scat$ is compactly
generated \cite[Cor. 4.25]{BGT}.

We will need some examples of limits and colimits in $\scat$. 
Compare \cite[sec. 5]{BGT} for a detailed treatment. 

\begin{definition} 
\label{vq}
Let $\mathcal{C} \in \scat$ and let $\mathcal{D} \subset \mathcal{C}$ be a
full, stable idempotent complete subcategory. 
We define the \textbf{Verdier quotient} $\mathcal{C}/\mathcal{D}$ to be the pushout in $\scat$
\[ \xymatrix{
\mathcal{D} \ar[d] \ar[r] &  \mathcal{C} \ar[d] \\
0 \ar[r] &  \mathcal{C}/\mathcal{D}
}.\]
\end{definition} 

Fix $\mathcal{E} \in \scat$. 
By definition, to give an exact functor $\mathcal{C}/\mathcal{D} \to
\mathcal{E}$ is equivalent to giving an exact functor $\mathcal{C} \to
\mathcal{E}$ which sends every object in $\mathcal{D}$ to a zero object; note
that this is a \emph{condition} rather than extra data. 
The Verdier quotient can be described very explicitly. Namely, consider the
inclusion $\mathrm{Ind}(\mathcal{D}) \subset \mathrm{Ind}(\mathcal{C})$ of
stable $\infty$-categories. For any $X \in \mathrm{Ind}(\mathcal{C})$, there is
a 
natural cofiber sequence
\[ M_{\mathcal{D}} X \to X \to L_{\mathcal{D}} X,  \]
where: 
\begin{enumerate}
\item $M_{\mathcal{D}} X$ is in the full stable subcategory of
$\mathrm{Ind}(\mathcal{C})$ generated under colimits by $\mathcal{D}$ (i.e.,
$\mathrm{Ind}(\mathcal{D})$). 
\item For any $D \in \mathcal{D}$, $\hom_{\mathrm{Ind}(\mathcal{C})}(D,
L_{\mathcal{D}} X)$ is contractible. 
\end{enumerate}
One can construct this sequence by taking $M_{\mathcal{D}}$ to be the right
adjoint to the inclusion functor $\mathrm{Ind}(\mathcal{D}) \subset
\mathrm{Ind}(\mathcal{C})$. 

We say that an object $X \in \mathrm{Ind}(\mathcal{C})$ is \emph{$\mathcal{D}^{\perp}$-local}
if $M_{\mathcal{D}}X$ is contractible. The full subcategory
$\mathcal{D}^{\perp} \subset \mathrm{Ind}(\mathcal{C})$ of
$\mathcal{D}^{\perp}$-local objects is a localization of
$\mathrm{Ind}(\mathcal{C})$, with localization functor given by
$L_{\mathcal{D}}$. 
We have an adjunction
\[ \mathrm{Ind}(\mathcal{C}) \rightleftarrows \mathcal{D}^{\perp},  \]
where the right adjoint, the inclusion $\mathcal{D}^{\perp} \subset
\mathcal{C}$, is fully faithful. 
The inclusion $\mathcal{D}^{\perp} \subset \mathrm{Ind}(\mathcal{C})$ preserves filtered
colimits since $\mathcal{D} \subset \mathrm{Ind}(\mathcal{C})$ consists of
compact objects, so that the localization $L_{\mathcal{D}}$ preserves compact objects. 
Now, the Verdier quotient can be described as  the
subcategory of $\mathcal{D}^{\perp}$ spanned by compact objects (in
$\mathcal{D}^{\perp}$); it is generated under finite colimits and retracts by the image of
objects in $\mathcal{C}$. 
Moreover, $\mathrm{Ind}(\mathcal{C}/\mathcal{D})$ is precisely
$\mathcal{D}^{\perp} \subset \mathrm{Ind}(\mathcal{C})$. 

\begin{remark} 
The pushout diagram defining the Verdier quotient is also a pullback. 
\end{remark} 
\begin{remark} 
A version of this construction makes sense in the world of presentable, stable
$\infty$-categories (which need not be compactly generated). 
\end{remark} 

These Verdier quotients have been considered, for example, in \cite{miller} under the name
\emph{finite localizations.}

\subsection{Stable homotopy theories and 2-rings}
In this paper, our goal is to describe an invariant of \emph{symmetric
monoidal} stable $\infty$-categories. For our purposes, we can think of them as
\emph{commutative algebra objects} with respect to a certain \emph{tensor
product} on $\scat$. We begin by reviewing this and some basic properties of
stable homotopy theories, which are the ``big'' versions of these. 

\begin{definition}[{\cite[4.8]{higheralg}, \cite{BFN}}] 
Given $\mathcal{C}, \mathcal{D} \in \scat$, we define the
\emph{tensor product} $\mathcal{C} \boxtimes
\mathcal{D} \in \scat$ via the universal property
\begin{equation} \label{tensor} \hom_{\scat}( \mathcal{C} \boxtimes \mathcal{D}, \mathcal{E}) \simeq
\mathrm{Fun}'( \mathcal{C} \times \mathcal{D}, \mathcal{E}),  \end{equation}
where $\mathrm{Fun}'(\mathcal{C} \times \mathcal{D}, \mathcal{E})$ consists of
those functors $\mathcal{C} \times \mathcal{D} \to \mathcal{E}$ which preserve
finite colimits in each variable separately. 
\end{definition} 

It is known (see \cite[4.8]{higheralg}) that this defines a symmetric
monoidal structure on $\scat$. The commutative algebra objects are
\emph{precisely} the symmetric monoidal, stable $\infty$-categories
$(\mathcal{C}, \otimes, \mathbf{1})$ such that the tensor product preserves
finite colimits in each variable.

\begin{definition} 
We let $\tring = \clg(\scat)$ be the $\infty$-category of  commutative algebra
objects in $\scat$. 
We will also write 	$\shot$ for the $\infty$-category of stable homotopy
theories (i.e., presentable stable symmetric monoidal $\infty$-categories with
bicocontinuous tensor product); this is the ``big'' version of $\tring$.
\end{definition} 

The tensor product $\boxtimes\colon  \scat \times \scat \to \scat$ preserves filtered
colimits in each variable; this follows from \eqref{tensor}. In particular, since $\scat$ is
a presentable $\infty$-category, it follows that $\tring $ is a presentable
$\infty$-category. 

In this paper, we will define a functor
\[ \pi_{\leq 1} \colon  \tring \to \pro( \fgp)^{\op}, \]
where we will specify what the latter means below, called the Galois
groupoid. 
The Galois groupoid will parametrize certain very special commutative algebra
objects in a given 2-ring. 
Given a stable homotopy theory $(\mathcal{C}, \otimes,
\mathbf{1})$ (in the sense of \Cref{shot}), the invariant we will define will
depend only on the small subcategory $\mathcal{C}^{\mathrm{dual}}$ of
\emph{dualizable} objects in $\mathcal{C}$. 

We will also define a slightly larger version of the Galois groupoid
that will see more of the ``infinitary'' structure of the stable homotopy theory, which will
make a difference in settings where the unit is not compact (such as
$K(n)$-local stable homotopy theory).  
In this case, it will not be sufficient to work with $\tring$. However, the
interplay between $\tring$ and the theory of (large) stable homotopy theories
will be crucial in the following. 

\begin{definition}[{Cf. \cite[4.6.1]{higheralg}}] 
In a symmetric monoidal $\infty$-category $(\mathcal{C}, \otimes, \mathbf{1})$,
an object $X$ is \textbf{dualizable} if there exists an object $Y$ and maps
\[ \mathbf{1} \xrightarrow{\mathrm{coev}} Y \otimes X, \quad X \otimes Y
\xrightarrow{\mathrm{ev}} \mathbf{1},  \]
such that the composites
\[ X \simeq  X \otimes \mathbf{1} \xrightarrow{1_{X} \otimes
\mathrm{coev}} X \otimes Y \otimes X  \xrightarrow{\mathrm{ev} \otimes
1_X} X, 
\quad Y  \simeq   \mathbf{1} \otimes Y \xrightarrow{\mathrm{coev} \otimes 1_Y}
Y \otimes X \otimes Y \xrightarrow{1_Y \otimes \mathrm{ev}} Y
\]
are homotopic to the respective identities. 
In other words, $X$ is dualizable if and only if it is dualizable in the
homotopy category with its induced symmetric monoidal structure. 
\end{definition} 

These definitions force natural homotopy equivalences
\begin{equation} \label{duality} \hom_{\mathcal{C}}(Z, Z' \otimes X ) 
\simeq \hom_{\mathcal{C}}(Z \otimes Y, Z'), \quad Z, Z' \in \mathcal{C}.
\end{equation}
Now let $(\mathcal{C}, \otimes, \mathbf{1})$  be a stable homotopy theory. The
collection of all dualizable objects in $\mathcal{C}$ 
(cf. also \cite[sec. 2.1]{axiomatic})
is a \emph{stable}
and idempotent complete subcategory, which is closed under the monoidal product. Moreover, suppose that $\mathbf{1}$ is $\kappa$-compact for some
regular cardinal $\kappa$. Then 
\eqref{duality} with $Z = \mathbf{1}$ forces any dualizable object $Y$ to be
$\kappa$-compact as well. In particular, it follows that the subcategory of
$\mathcal{C}$ spanned by the dualizable objects is (essentially) small and
belongs to $\tring$. (By contrast, no amount of compactness is sufficient to
imply dualizability). 

We thus have the two constructions: 
\begin{enumerate}
\item Given a stable homotopy theory, take the symmetric monoidal, stable
$\infty$-category of dualizable objects, which is a 2-ring.  
\item Given an object $\mathcal{C} \in \tring$, $\mathrm{Ind}(\mathcal{C})$ is
a stable homotopy theory. 
\end{enumerate}
These two constructions are generally not inverse to one another. However, the
``finitary'' version of the Galois group we will define will be unable to see
the difference.

Next, we will describe some basic constructions in $\tring$. 
The $\infty$-category $\tring$ has all limits, and these may be computed at the level of the
underlying $\infty$-categories. 
As such, these homotopy limit constructions can be used to build new examples
of 2-rings from old ones. 
These constructions will also apply to stable homotopy theories. 
To start with, we discuss Verdier quotients. 

\begin{definition} 
Let $(\mathcal{C}, \otimes, \mathbf{1}) \in \tring$ and let $\mathcal{I}
\subset \mathcal{C}$ be a full stable, idempotent complete subcategory. We say that $\mathcal{I}$ is an
\textbf{ideal} or \textbf{$\otimes$-ideal} if whenever $X \in \mathcal{C}, Y \in \mathcal{I}$, the tensor
product $X \otimes Y \in \mathcal{C}$ actually belongs to $\mathcal{I}$. 
\end{definition} 

If $\mathcal{I} 
\subset \mathcal{C}$ is an ideal, then the Verdier quotient
$\mathcal{C}/\mathcal{I}$ naturally inherits the structure of an object in
$\tring$. This follows naturally from \cite[Proposition 2.2.1.9]{higheralg} and the explicit
construction of the Verdier quotient. By definition,
$\mathrm{Ind}(\mathcal{C}/\mathcal{I})$
consists of the objects $X \in \mathrm{Ind}(\mathcal{C})$ which have the
property that $\hom_{\mathrm{Ind}(\mathcal{C})}(I, X) $ is contractible when $I
\in \mathcal{I}$. We can describe this as the localization of
$\mathrm{Ind}(\mathcal{C})$ at the collection of maps $f\colon  X \to Y$ whose
cofiber belongs to $\mathrm{Ind}(\mathcal{I})$. These maps, however, form an
ideal since $\mathcal{I}$ is an ideal. 
As before, given $\mathcal{D} \in \tring$, we have a natural fully faithful inclusion
\[  \hom_{\tring}(\mathcal{C}/\mathcal{I}, \mathcal{D}) \subset
\hom_{\tring}(\mathcal{C}, \mathcal{D}),  \]
where the image of the map consists of all symmetric monoidal functors
$\mathcal{C} \to \mathcal{D}$ which take every object in $\mathcal{I}$ to a
zero object. 

Finally, we describe some \emph{free} constructions. 
Let $\sp$ be the $\infty$-category of spectra, and let $\mathcal{C}$ be a
 small symmetric monoidal $\infty$-category. Then the $\infty$-category
$\mathrm{Fun}(\mathcal{C}^{\op}, \sp)$ is a stable homotopy theory under the
\emph{Day convolution product} \cite[4.8.1]{higheralg}. Consider  the collection of compact objects in
here, which we will write as the ``monoid algebra'' $\sp^\omega[\mathcal{C}]$. 
One has the universal property 
\[ \hom_{\tring}( \sp^\omega[\mathcal{C}], \mathcal{D}) \simeq
\mathrm{Fun}_{\otimes}( \mathcal{C}, \mathcal{D}),   \]
i.e., 
an equivalence
between functors of 2-rings $\sp[\mathcal{C}] \to \mathcal{D}$ and 
symmetric monoidal functors $\mathcal{C} \to \mathcal{D}$. We can also define the free
stable homotopy theory on $\mathcal{C}$ as the $\mathrm{Ind}$-completion of this
2-ring, or equivalently as $\mathrm{Fun}(\mathcal{C}^{\op}, \sp)$. 

\begin{example} 
The free symmetric monoidal $\infty$-category on a single object is the
disjoint union $\bigsqcup_{n \geq 0} B \Sigma_n$, or the groupoid of finite
sets and isomorphisms between them, with $\sqcup$ as the symmetric monoidal
product. Using this, we can describe the ``free stable homotopy theory'' on a
single object. As above, an object in this stable homotopy theory consists of giving a spectrum $X_n$ with a
$\Sigma_n$-action for each $n$; the tensor structure comes from a convolution
product. If we consider the compact objects in here, we obtain the free 2-ring
on a given object. 
\end{example} 

Finally, we will need to discuss a bit of algebra internal to $\mathcal{C}$. 

\begin{definition} 
There is a natural $\infty$-category of \emph{commutative
algebra objects} in $\mathcal{C}$ (cf. \cite[Ch. 2]{higheralg}) which we will denote by $\clg(\mathcal{C})$. 
When $\mathcal{C} = \sp$ is the $\infty$-category, we will just write $\clg$
for the $\infty$-category of $\e{\infty}$-ring spectra.
\end{definition}

Recall that a commutative algebra object in $\mathcal{C}$ consists of an object
$X \in \mathcal{C}$ together with a multiplication map $m\colon  X \otimes X \to X$
and a unit map $\mathbf{1}\to X$, which satisfy the classical axioms of a
commutative algebra object up to coherent homotopy; for instance, when
$\mathcal{C} = \sp$, one obtains the classical notion of an $\e{\infty}$-ring. 
The amount of homotopy coherence is sufficient to produce the following:
\begin{definition}[{\cite[Sec. 4.5]{higheralg}}] 
Let $\mathcal{C}$ be a stable homotopy theory. 
Given $A \in \clg(\mathcal{C})$, there is a natural $\infty$-category
$\md_{\mathcal{C}}(A)$ of $A$-module objects internal to $\mathcal{C}$.
The $\infty$-category
$\md_{\mathcal{C}}(A)$ acquires the structure of a stable homotopy theory with
the relative $A$-linear tensor product. 
\end{definition} 

The relative $A$-linear tensor product requires the formation of geometric
realizations, so we need infinite colimits to exist in $\mathcal{C}$ for the
above construction to make sense in general.

\subsection{Examples}
Stable homotopy theories and 2-rings occur widely in ``nature,'' and in this section, we
describe a few basic classes of such widely occurring examples. 
We begin with two of the most fundamental ones. 
\begin{example}[Derived categories]
The derived $\infty$-category $D(R)$ of a commutative ring $R$ (cf. \cite[Sec.
1.3]{higheralg})
with the derived tensor product is a stable homotopy theory. 
\end{example}

\begin{example}[Modules over an $\e{\infty}$-ring]
As a more general example, the $\infty$-category $\md(R)$ of modules over an
$\e{\infty}$-ring spectrum $R$ with the relative smash product is a stable
homotopy theory. 
For instance, taking $R = S^0$, we get the $\infty$-category $\sp$ of spectra. 
This is the primary example (together with $E$-localized versions)
considered in \cite{rognes}. 
\end{example}

\begin{example}[Quasi-coherent sheaves]
\label{qcohintro}
Let $X$ be a scheme (or algebraic stack, or even prestack). To $X$, one can
associate a stable homotopy theory $\qcoh(X)$ of \emph{quasi-coherent
complexes} on $X$. By definition, $\qcoh(X)$ is the homotopy limit of the
derived $\infty$-categories $D(R)$ where $\spec R \to X$ ranges over all maps from
affine schemes to $X$. For more discussion, see \cite{BFN}. 
\end{example}

\begin{example} 
\label{modclosed}
Consider a cartesian diagram of $\e{\infty}$-rings 
\[ \xymatrix{
A \times_{A''} A' \ar[d] \ar[r] &  A \ar[d]  \\
A' \ar[r] &  A''
}.\]
We obtain a diagram of stable homotopy theories
\[ \xymatrix{
\md(A \times_{A''} A') \ar[d] \ar[r] & \md(A) \ar[d]  \\
\md(A') \ar[r] & \md( A'')
},\]
and in particular a symmetric monoidal functor
\[ \md(A \times_{A''} A') \to \md(A) \times_{\md(A'')} \md(A').\]
This functor is generally not an equivalence in $\tring$. 

This functor is \emph{always} fully faithful. 
However, if $A ,A',
A''$ are \emph{connective} and $A \to A'', A' \to A''$ induce surjections on
$\pi_0$, then it is proved in \cite[Theorem 7.2]{DAGIX} that 
the functor induces an equivalence on the \emph{connective} objects or, more
generally, on the $k$-connective objects for any $k \in \mathbb{Z}$. 
In particular, if we let $\md^\omega$ denote perfect modules, we have an equivalence of 2-rings
\[ \md^\omega( A \times_{A''} A') \simeq \md^{\omega}(A) \times_{\md^{\omega}(A'')}
\md^{\omega}(A') , \]
since an $A \times_{A''} A'$-module is perfect if and only if its base-changes
to $A, A'$ are. 
However, the $\mathrm{Ind}$-construction generally does not commute even with finite
limits. 
\end{example}

\begin{example}[Functor categories]
As another example of a (weak) 2-limit, we consider any $\infty$-category
$K$ and a stable homotopy theory $\mathcal{C}$; then $\mathrm{Fun}(K,
\mathcal{C})$ is naturally  a stable homotopy theory under the ``pointwise''
tensor product. If $K = BG$ for a group $G$, then this example endows the
$\infty$-category of objects in $\mathcal{C}$ with  a $G$-action with the
structure of a stable homotopy theory. 
\end{example} 

Finally, we list several other miscellaneous examples of stable homotopy
theories.

\begin{example}[Hopf algebras]
\label{Hopfalg}
Let $A$ be a finite-dimensional cocommutative Hopf algebra over the field $k$. In this case, the
(ordinary) category $\mathcal{A}$ of discrete $A$-modules has a natural symmetric monoidal structure
via the $k$-linear tensor product. In particular, its \emph{derived}
$\infty$-category $D(\mathcal{A})$
is naturally symmetric monoidal, and is thus a stable homotopy theory. 
Stated more algebro-geometrically, $\spec A^{\vee}$ is a group scheme $G$ over the
field $k$, and $D(\mathcal{A})$ is the
$\infty$-category of quasi-coherent sheaves of complexes on the classifying
stack $BG$. 
\end{example} 

\begin{example}[Stable module $\infty$-categories]
\label{stmodcat}
Let $A$ be a finite-dimensional cocommutative Hopf algebra over the field $k$. Consider the 
subcategory $D(\mathcal{A})^{\omega} \subset D(\mathcal{A})$ (where $\mathcal{A}$ is the abelian category
of $A$-modules, as in \Cref{Hopfalg}) of $A$-module spectra which are perfect as $k$-module spectra.
Inside $D(\mathcal{A})^{\omega}$ is the subcategory $\mathcal{I}$ of those
objects which
are perfect as $A$-module spectra. This subcategory is stable, and is an
\emph{ideal} by the observation (a projection formula of sorts) that the
$k$-linear tensor product of $A$
with any $A$-module
is free as an $A$-module. 

\begin{definition} 
The \textbf{stable module $\infty$-category} $\stm_A = \mathrm{Ind}(
D(\mathcal{A})^{\omega}/\mathcal{I})$ is the $\mathrm{Ind}$-completion of the  Verdier quotient
$D(\mathcal{A})^{\omega}/\mathcal{I}$. 
If $A = k[G]$ is the group algebra of a finite group $G$, we write $\stm_G(k)$
for $\stm_{k[G]}$. 
\end{definition} 

The stable module $\infty$-categories of finite-dimensional Hopf algebras
(especially group algebras) and their
various invariants (such as the Picard groups and the thick
subcategories) have been studied
extensively in the modular representation theory literature. 
For a recent survey, see \cite{BIK}. 
\end{example}

\begin{example}[Bousfield localizations]
Let $\mathcal{C}$ be a stable homotopy theory, and let $E \in \mathcal{C}$. In
this case, there is a naturally associated stable homotopy theory $L_E
\mathcal{C}$ of \emph{$E$-local objects}. By definition, $L_E \mathcal{C}$ is
a full subcategory of $\mathcal{C}$; an object $X \in
\mathcal{C}$ belongs to $L_E \mathcal{C}$ if and only if whenever $Y  \in
\mathcal{C}$ satisfies $Y \otimes E \simeq 0$, the spectrum
$\hom_{\mathcal{C}}(Y, X)$ is contractible. 
The $\infty$-category $L_E \mathcal{C}$ is symmetric monoidal under the
\emph{$E$-localized tensor product}: since the tensor product of two $E$-local
objects need not be $E$-local, one needs to localize further. 
For example, the unit object in $L_E \mathcal{C}$ is $L_E \mathbf{1}$. 

There is a natural adjunction
\[ \mathcal{C} \rightleftarrows L_E \mathcal{C},  \]
where the (symmetric monoidal) left adjoint sends an object to its
$E$-localization, and where the (lax symmetric monoidal) right adjoint is the
inclusion. 
\end{example}

\subsection{Morita theory}

Let $(\mathcal{C}, \otimes, \mathbf{1})$ be a stable homotopy theory. In
general, there is a very useful criterion for recognizing when $\mathcal{C}$ is
equivalent (as a stable homotopy theory) to the $\infty$-category of modules
over an $\e{\infty}$-ring. 

Note first that if $R$ is an $\e{\infty}$-ring, then the unit object of $\md(R)$
is a compact generator. The following result,
which for stable $\infty$-categories (without the symmetric monoidal structure)
is due to Schwede-Shipley \cite{schwedeshipley} (preceded by ideas of Rickard
and others on tilting theory), 
asserts the converse. 

\begin{theorem}[{\cite[Proposition 8.1.2.7]{higheralg}}] 
\label{luriess}
Let $\mathcal{C}$ be a stable homotopy theory where $\mathbf{1}$ is a compact
generator. Then there is a natural symmetric monoidal equivalence
\[ \md(R) \simeq \mathcal{C} , \]
where $R \simeq \mathrm{End}_{\mathcal{C}}(\mathbf{1})$ is naturally an
$\e{\infty}$-ring. 
\end{theorem}

In general, given a symmetric monoidal stable $\infty$-category $\mathcal{C}$,
the endomorphism ring $R = \mathrm{End}_{\mathcal{C}}( \mathbf{1})$ is
\emph{always} naturally an $\e{\infty}$-ring, and 
one has a natural adjunction
\[   \md(R) \rightleftarrows \mathcal{C}, \]
where the left adjoint ``tensors up'' an $R$-module with $\mathbf{1} \in
\mathcal{C}$, and the right adjoint sends $X \in \mathcal{C}$ to the mapping
spectrum
$\hom_{\mathcal{C}}(\mathbf{1}, X)$, which naturally acquires the structure of
an $R$-module. The left adjoint is symmetric monoidal, and the right adjoint is
\emph{lax} symmetric monoidal. In general, one does not expect the right
adjoint to preserve filtered colimits: it does so if and only if $\mathbf{1}$
is compact. In this case, if $\mathbf{1}$ is compact, we get a fully faithful
inclusion
\[ \md(R) \subset \mathcal{C},  \]
which exhibits $\md(R)$ as a \emph{colocalization} of $\mathcal{C}$. If
$\mathbf{1}$ is not compact, we at least get a fully faithful inclusion of the
\emph{perfect} $R$-modules into $\mathcal{C}$.

For example, let $G$ be a finite $p$-group and $k$ be a field of characteristic
$p$. In this case, every finite-dimensional $G$-representation on a $k$-vector space
is unipotent: any such has a finite filtration whose subquotients are isomorphic to
the trivial representation. 
From this, one might suspect that one has an equivalence
of stable homotopy theories
\(  \fun( BG, \md(k))  \simeq \md ( k^{hG}), \)
where $k^{hG}$ is the $\e{\infty}$-ring of endomorphisms of the unit object
$k$, but this fails because the unit object of $\md( k[G])$ fails to be
compact: taking $G$-homotopy fixed points does not commute with homotopy
colimits. However, by fixing this reasoning, one obtains an equivalence
\begin{equation} \label{reppgroup} \fun( BG, \md^\omega(k))  \simeq \md^{\omega}( k^{hG}), \end{equation}
between perfect $k$-module spectra with a $G$-action and perfect
$k^{hG}$-modules. 
If one works with stable module $\infty$-categories, then the unit object \emph{is}
compact (more or less by fiat) and one has: 

\begin{theorem}[Keller \cite{keller}] \label{keller} Let $G$ be a finite $p$-group and $k$ a
field of characteristic $p$. Then we
have an equivalence of symmetric monoidal $\infty$-categories
\[ \md( k^{tG}) \simeq \stm_G(k),  \]
between the $\infty$-category of modules over the Tate $\e{\infty}$-ring $k^{tG}$
and the stable module $\infty$-category of $G$-representations over $k$. 
\end{theorem} 

The \emph{Tate construction} $k^{tG}$, for our purposes, can be \emph{defined} as the endomorphism
$\e{\infty}$-ring of the unit object in the stable module $\infty$-category
$\stm_G(k)$. As a $k$-module spectrum, it can also be obtained as the cofiber of
the \emph{norm map} $k_{hG} \to k^{hG}$.  
We also refer to \cite[sec. 2]{toruspic} for 
further discussion on this point.

\subsection{\'Etale algebras}

Let $R$ be an $\e{\infty}$-ring spectrum. 
Given an $\e{\infty}$-$R$-algebra $R'$, recall that the homotopy groups 
$\pi_* R'$ form a graded-commutative $\pi_*R$-algebra. 
In general, there is no reason for a given graded-commutative $\pi_* R$-algebra
to be realizable as the homotopy groups in this way, although one often has
various obstruction theories (see for instance \cite{robinsonobstruct, rezkHM,
goersshopkins}
for examples of obstruction theories in different contexts) to attack such questions. There is, however,
always one case in which the obstruction theories degenerate completely. 

\begin{definition} 
An $\e{\infty}$-$R$-algebra $R'$ is \textbf{\'etale} if: 
\begin{enumerate}
\item The map $\pi_0 R \to \pi_0 R'$ is \'etale (in the sense of ordinary
commutative algebra).  
\item The natural map $\pi_0 R' \otimes_{\pi_0 R} \pi_* R \to \pi_* R'$ is an
isomorphism. 
\end{enumerate}
\end{definition} 

The basic result in this setting is that the theory of \'etale algebras 
is entirely algebraic: the obstructions to existence and uniqueness all  vanish. 
\begin{theorem}[{\cite[Theorem 7.5.4.2]{higheralg}}]
\label{etaletopinv}
Let $R$ be an $\e{\infty}$-ring. Then the $\infty$-category of \'etale
$R$-algebras is equivalent (under $\pi_0$) to the ordinary category of \'etale
$\pi_0 R$-algebras. 
\end{theorem}

One can show more, in fact: given an \'etale $R$-algebra $R'$, then for any
$\e{\infty}$ $R$-algebra $R''$, the natural map
\[ \hom_{R/}(R', R'') \to \hom_{\pi_0 R/}( \pi_0 R', \pi_0 R'') \]
is a homotopy equivalence. Using an adjoint functor theorem approach (and the
infinitesimal criterion for \'etaleness), one may
even \emph{define} $R'$ in terms of $\pi_0 R'$ in this manner, although
checking that it has the desired homotopy groups takes additional work. In
particular, note that \'etale $R$-algebras are \emph{0-cotruncated} objects of
the $\infty$-category $\clg_{R/}$: that is, the space of maps out of any such
is always homotopy discrete. 
The finite covers that we shall consider in this paper will also
have this property.

\begin{example} 
This implies that one can adjoin $n$th roots of unity to the sphere spectrum
$S^0$ once $n$ is inverted. An argument of Hopkins implies that the inversion
of $n$ is necessary: for $p > 2$, one cannot adjoin a $p$th root of unity to $p$-adic
$K$-theory, as one sees by considering the $\theta$-operator on $K(1)$-local
$\e{\infty}$-rings under $K$-theory (cf. \cite{kone})  which satisfies $x^p = \psi(x) - p \theta(x)$ where $\psi$
is a homomorphism on $\pi_0$. 
If one could adjoin $\zeta_p$ to $p$-adic $K$-theory, then 
one would have $-p\theta( \zeta_p) = 1 - \zeta_p^a$ for some unit $a \in
(\mathbb{Z}/p\mathbb{Z})^{\times}$, but $p$ does not divide $1 -
\zeta_p^a$ in $\mathbb{Z}_p[\zeta_p]$.
\end{example}

Let $(\mathcal{C}, \otimes, \mathbf{1})$ be a stable homotopy theory. 
We will now attempt to do the above in $\mathcal{C}$ itself. 
We will obtain some of the simplest classes of objects in $\clg(\mathcal{C})$. 
The following notation will be convenient. 
\begin{definition} 
Given a stable homotopy theory $(\mathcal{C}, \otimes, \mathbf{1})$, we will
write
\begin{equation} 
\pi_* X \simeq \pi_* \hom_{\mathcal{C}}( \mathbf{1}, X).
\end{equation} 
\end{definition} 

In particular, $\pi_* \mathbf{1} \simeq \pi_*
\mathrm{End}_{\mathcal{C}}(\mathbf{1}, \mathbf{1})$ is a graded-commutative
ring, and 
for any $X \in \mathcal{C}$, $\pi_* X$ is naturally a $\pi_* \mathbf{1}$-module. 

\begin{remark} 
Of course, $\pi_*$ does not commute with infinite direct sums unless
$\mathbf{1}$ is compact. For example, $\pi_*$ fails to commute with direct sums
in $L_{K(n)} \sp$ (which is actually compactly generated, albeit not by the
unit object). 
\end{remark}

Let $(\mathcal{C}, \otimes, \mathbf{1})$ be a stable homotopy theory. 
As in the previous section, we have an adjunction
of symmetric monoidal $\infty$-categories
\[ \left(\cdot \otimes_R \mathbf{1}, \hom_{\mathcal{C}}(\mathbf{1}, \cdot)
\right) \colon \md(R) \rightleftarrows \mathcal{C},  \]
where $R = \mathrm{End}_{\mathcal{C}}(\mathbf{1})$ is an $\e{\infty}$-ring. 
Given an \'etale $\pi_0 R \simeq \pi_0 \mathbf{1}$-algebra $R_0'$, we can thus
construct an \'etale $R$-algebra $R'$ and an associated object $R' \otimes_R
\mathbf{1} \in \clg(\mathcal{C})$. The object $R' \otimes_R \mathbf{1}$ naturally acquires the
structure of a commutative algebra, and, by playing again with
adjunctions, we find that
\[ \hom_{\clg(\mathcal{C})}(R' \otimes_{R} \mathbf{1}, T) \simeq \hom_{\pi_0
\mathbf{1}}(R'_0, \pi_0 T), \quad T \in \clg(\mathcal{C}).  \]

\begin{definition} 
\label{classicaletale}
The objects of $\clg(\mathcal{C})$ obtained in this manner are called
\textbf{classically \'etale.}
\end{definition} 

The classically \'etale objects in $\clg(\mathcal{C})$ span a 
subcategory of $\clg(\mathcal{C})$. In general, this is not equivalent to the
category of \'etale $\pi_0 R$-algebras if $\mathbf{1}$ is not compact (for example, $\md(R) \to \mathcal{C}$
need not be conservative; take $\mathcal{C} = L_{K(n)} \sp$ and $L_{K(n)} S^0
\otimes \mathbb{Q}$). 
However, note that the functor
\[ \md^\omega(R) \to \mathcal{C},  \]
from the $\infty$-category $\md^\omega(R)$ of perfect $R$-modules into
$\mathcal{C}$, 
is \emph{always} fully faithful. It follows that there is a full subcategory of
$\clg( \mathcal{C})$ \emph{equivalent} to the category of \emph{finite} \'etale
$\pi_0 R$-algebras. This subcategory will give us the ``algebraic'' part of the
Galois group of $\mathcal{C}$. 

We now specialize to the case of \emph{idempotents.}
Let $(\mathcal{C}, \otimes, \mathbf{1})$ be a stable homotopy theory, and $A
\in \clg(\mathcal{C})$ a commutative algebra object, so that $\pi_0 A$ is a
commutative ring. 

\begin{definition} 
An \textbf{idempotent} of $A$ is an idempotent of the commutative ring $\pi_0 A$. 
We will denote the set of idempotents of $A$ by $\idem(A)$. 
\end{definition} 

The set $\idem(A)$ acquires some additional structure; as the set of
idempotents in a commutative ring, it is naturally a \emph{Boolean algebra}
under the multiplication in $\pi_0 A$ and the addition that takes idempotents
$e, e'$ and forms $e  + e' - ee'$. 
For future reference, recall the following: 

\newcommand{\bool}{\mathrm{Bool}}
\begin{definition} 
A \textbf{Boolean algebra} is a commutative ring $B$ such that $x^2 = x$ for
every $x \in B$. The collection of all Boolean algebras forms a full  subcategory
$\bool$ of the category of commutative rings. 
\end{definition}

Suppose given an idempotent $e$ of $A$, so that $1-e$ is also an idempotent. 
In this case, we can obtain a \emph{splitting}
\[ A \simeq A[e^{-1}] \times A[(1-e)^{-1}]  \]
as a product of two objects in $\clg(\mathcal{C})$, as observed in
\cite{mayidem}. 
To see this, we may reduce to the case when $A = \mathbf{1}$, by replacing
$\mathcal{C}$ by $\md_{\mathcal{C}}(A)$. In this case, we obtain the splitting
from the discussion above in \Cref{classicaletale}: $A[e^{-1}]$ and $A[(1 -
e)^{-1}]$ are both classically \'etale and in the thick subcategory
generated by $A$. 
Conversely, given such a
splitting, we obtain corresponding idempotents, e.g., reducing to the case of an
$\e{\infty}$-ring. 

Suppose the unit object $\mathbf{1} \in \mathcal{C}$ decomposes as a product
$\mathbf{1}_1 \times \mathbf{1}_2 \in \clg(\mathcal{C})$. In this case, we have
a decomposition at the level of stable homotopy theories
\[ \mathcal{C} \simeq \md_{\mathcal{C}}( \mathbf{1}_1) \times
\md_{\mathcal{C}}( \mathbf{1}_2),  \]
so in practice, most stable homotopy theories that in practice we will be
interested in will have no such nontrivial idempotents. However, the theory of
idempotents will be very important for us in this paper.

For example, using the theory of idempotents, we can describe maps \emph{out of} a product
of commutative algebras. 

\begin{proposition} \label{prode}
Let $A, B \in \clg( \mathcal{C})$. Then if $C \in \clg( \mathcal{C})$, then we
have a homotopy equivalence
\[ \hom_{\clg(\mathcal{C})}(A \times B, C) \simeq \bigsqcup_{C \simeq C_1
\times C_2} \hom_{\clg(\mathcal{C})}(A, C_1) \times
\hom_{\clg(\mathcal{C})}(B, C_2),   \]
where the disjoint union is taken over all decompositions $C \simeq C_1 \times
C_2$ in $\clg(\mathcal{C})$ (i.e., over idempotents in $C$). 
\end{proposition} 
\begin{proof} 
Starting with a map $A \times B \to C$, we get a decomposition of $C$ into two
factors coming from the two natural idempotents in $A \times B$, whose images
in $C$ give two orthogonal idempotents summing to $1$. Conversely, starting with something in the
right-hand-side, given via maps $A \to C_1$ and $B \to C_2$ and an equivalence
$C \simeq C_1 \times C_2$, we can take the product 
of the two maps to get $A \times B \to C$. 
The equivalence follows from the universal property of localization. 
\end{proof}

For example, consider the case of $A, B = \mathbf{1}$. In this case, we find
that, if $C \in \clg(\mathcal{C})$, then 
\[ \hom_{\clg(\mathcal{C})}(\mathbf{1} \times \mathbf{1}, C)   \]
is homotopy discrete, and 
consists of the \emph{set} of idempotents in $C$. We could have obtained this
from the theory of ``classically \'etale''  objects earlier. 
Using this description as a corepresentable functor, we find: 

\begin{corollary} 
\label{idemlimit}
The functor $A \mapsto \idem(A)$, $\clg(\mathcal{C}) \to \bool$, commutes with limits. 
\end{corollary} 

\begin{remark} 
\label{sq0rem}
\Cref{idemlimit} can also be proved directly. 
Since $\pi_* $ commutes with arbitrary products in $\mathcal{C}$, it follows
that $A \mapsto \idem(A)$ commutes with arbitrary products. It thus suffices to
show that if we have a pullback diagram
\[ \xymatrix{
A \ar[d] \ar[r] & B \ar[d] \\
C \ar[r] & D 
},\]
in $\clg(\mathcal{C})$, 
then the induced diagram of Boolean algebras
\[ \xymatrix{
\idem(A) \ar[d] \ar[r] & \idem(B) \ar[d] \\
\idem(C) \ar[r] & \idem(D )
}\]
is also cartesian. 
In fact, we have a surjective map of commutative rings $\pi_0(A) \to \pi_0(B)
\times_{\pi_0(D)} \pi_0(C)$ whose kernel is the image of the connecting
homomorphism $\pi_1(D) \to \pi_0(A)$. It thus suffices to show that the
product of two elements in the  image of this connecting homomorphism vanishes, 
since square-zero ideals do not affect idempotents. 

Equivalently, we claim that if $x,y \in \pi_0(A)$ map to zero in $\pi_0(B)$
and $\pi_0(C)$, then $xy = 0$.
In fact, $x$  and $y$ define maps  $A \to A$ and, in fact, 
endomorphisms of the exact triangle 
\[ A \to B \oplus C \to D,  \]
and each is nullhomotopic on $B \oplus C$ and on $D$. 
A diagram chase with exact triangles now shows that $xy$ defines the
\emph{zero map} $A \to A$, as desired. 
\end{remark}

\section{Descent theory}

\label{sec:descent}
Let $A \to B$ be a faithfully flat map of discrete commutative rings.
Grothendieck's theory of \emph{faithfully flat descent} (cf. \cite[Exp.
VIII]{sga1}) can be used to describe
the category $\mdd(A)$ of (discrete, or classical) $A$-modules in terms of the
three categories $\mdd(B), \mdd(B \otimes_A B), \mdd(B \otimes_A B \otimes_A B)$. 
Namely, it identifies the category $\mdd(A)$ with the category of $B$-modules
with \emph{descent data}, or states that the diagram
\[ \mdd(A) \to \mdd(B) \rightrightarrows \mdd(B \otimes_A B) \triplearrows
\mdd(B \otimes_A B \otimes_A B),  \]
is a limit diagram in the 2-category of categories. This diagram of
categories comes from the
\emph{cobar construction} on $A \to B$, which is the augmented cosimplicial
commutative ring
\[ A \to B \rightrightarrows B \otimes_A B  \triplearrows \dots .  \]

Grothendieck's theorem can be proved via the \emph{Barr-Beck theorem,} by
showing that if $A \to B$ is faithfully flat, the natural
tensor-forgetful adjunction $\mdd(A)
\rightleftarrows \mdd(B)$ is 
comonadic. 
Such results are extremely useful in practice, for instance because the
category of $B$-modules may be much easier to study. From another point of
view, these results imply that any $A$-module $M$ can be expressed as an equalizer
of $B$-modules (and maps of $A$-modules), via
\[ M \to M \otimes_A B \rightrightarrows M \otimes_A B \otimes_A B,   \]
where the two maps are $m \otimes b \mapsto m \otimes b \otimes 1$ and $m
\otimes b \mapsto m \otimes 1 \otimes b$.

In the setting of ``brave new'' algebra, descent theory for maps of $\e{\infty}$
(or weaker) algebras has been extensively considered in the papers
\cite{DAGdesc, DAGss}. 
In this setting, one has a map of $\e{\infty}$-rings $A \to B$, and one wishes to
describe the stable $\infty$-category $\md(A)$ in terms of the stable
$\infty$-categories $\md(B), \md(B \otimes_A B), \dots $.
A sample result would run along the following lines. 

\begin{theorem}[{\cite[Theorem 6.1]{DAGss}}]\label{luried}
Let $A \to B$ be a map of $\e{\infty}$-rings such that $\pi_0 (A) \to \pi_0(B)$
is faithfully flat and the map $\pi_*(A) \otimes_{\pi_0(A)} \pi_0( B) \to \pi_*
(B)$ is an isomorphism. Then the adjunction $\md(A) \rightleftarrows \md(B)$ is
comonadic, so that $\md(A)$ can be recovered as the totalization of the
cosimplicial $\infty$-category 
\[ \md(B) \rightrightarrows \md(B \otimes_A B) \triplearrows \dots.  \]
\end{theorem} 

In practice, the condition of faithful flatness on $\pi_*(A) \to \pi_*(B)$ can
be weakened significantly; there are numerous examples of morphisms of
$\e{\infty}$-rings which do not behave well on the level of $\pi_0$ but under
which one does have a good theory of descent (e.g., the conclusion of
\Cref{luried} holds). 
For instance, there is a good theory of descent along $KO \to KU$, and this can be
used to describe features of the $\infty$-category $\md(KO)$ in terms of the
$\infty$-category $\md(KU)$. 
One advantage of considering descent in this more
general setting is that $KU$ is  \emph{much simpler} algebraically: its
homotopy groups are given by $\pi_*(KU) \simeq \mathbb{Z}[\beta^{\pm }]$, which
is a regular ring, even one-dimensional (if one pays attention to the
grading), while $\pi_*(KO)$ is of infinite homological dimension. 
There are many additional tricks one has when working with modules over a more
tractable $\e{\infty}$-ring such as $KU$; we shall see a couple of them below in
the proof of \Cref{etalegalois}. 

\begin{remark} 
For some applications of these ideas to computations, see the paper \cite{thick} (for
descriptions of thick subcategories) and 
\cite{GL, MS, HMS} (for calculations of certain Picard groups). 
\end{remark}

In this section, we will describe a class of maps of $\e{\infty}$-rings $A \to
B$ that have an \emph{especially good} theory of descent. We will actually work
in more generality, and fix a stable homotopy theory $(\mathcal{C}, \otimes,
\mathbf{1})$, and isolate a class of commutative algebra objects for which
the analogous theory of descent (internal to $\mathcal{C}$) works especially
well (so well, in fact, that it will be tautologically preserved by any
morphism of stable homotopy theories). Namely, we will define $A \in \clg(\mathcal{C})$ to be
\emph{descendable} if the thick $\otimes$-ideal that $A$ generates contains the
unit object $\mathbf{1} \in \mathcal{C}$. This definition, which is motivated by the
\emph{nilpotence technology} of Devinatz, Hopkins, Smith, and Ravenel
\cite{HS, DHS} (one part of which  states that the map $L_n S^0 \to
E_n$ from the $E_n$-local sphere to Morava $E$-theory $E_n$ satisfies this
property), is enough to imply that the conclusion of \Cref{luried}
holds, and has the virtue of being purely diagrammatic. The definition  has
also been  recently
and independently considered by Balmer \cite{balmersep} (under the name ``nil-faithfulness'') 
in the setting of tensor-triangulated categories. 

In the rest of the section, we will give several
examples of descendable morphisms, and describe in
\Cref{lindesc}
an application to descent for \emph{2-modules} (or linear
$\infty$-categories), which has applications to the study of the \emph{Brauer
group}. This provides a slight strengthening of the descent results in
\cite{DAGQC, DAGdesc}. 

\subsection{Comonads and descent}
The language of $\infty$-categories gives very powerful tools for proving
descent theorems such as \Cref{luried} as well as its generalizations;
specifically, the Barr-Beck-Lurie theorem of \cite{higheralg} gives a criterion
to check when an adjunction is comonadic (in the $\infty$-categorical sense),
although the result is usually stated in its equivalent form for monadic
adjunctions.
This result has recently been reproved from the point of view of weighted
(co)limits by Riehl-Verity \cite{RV}. 

\begin{theorem}[{Barr-Beck-Lurie \cite[Section 4.7]{higheralg}}]
Let $F, G\colon  \mathcal{C} \rightleftarrows \mathcal{D}$ be an adjunction between
$\infty$-categories. Then 
the adjunction is comonadic if and only if: 
\begin{enumerate}
\item $F$ is conservative.  
\item Given a cosimplicial object $X^\bullet$ in $\mathcal{C}$ such that
$F(X^\bullet)$ admits a splitting, then $\mathrm{Tot}(X^\bullet)$ exists in
$\mathcal{C}$ and the map $F( \mathrm{Tot}(X^\bullet)) \to \mathrm{Tot}
F(X^\bullet)$ is an equivalence. 
\end{enumerate}
\end{theorem}

In practice, we will be working with presentable $\infty$-categories, so the
existence of totalizations will be assured. The conditions of the
Barr-Beck-Lurie theorem are thus automatically satisfied if $F$ preserves
\emph{all} totalizations (as sometimes happens) and is conservative. 

\begin{example} 
Let $A \to B$ be a morphism of $\e{\infty}$-rings. The forgetful functor $\md(B) \to
\md(A)$ is conservative and preserves all 
limits and \emph{colimits}. By the adjoint functor theorem, it is a 
left adjoint. (The right adjoint to this functor sends an $A$-module $M$ to the $B$-module
$\hom_A(B, M)$.) By the Barr-Beck-Lurie theorem, this adjunction is comonadic. 
\end{example} 

However,
we will need to consider the more general case. 
Given a comonadic adjunction as above, one can recover any object $C \in \mathcal{C}$ as the homotopy
limit of the \emph{cobar construction}
\begin{equation} \label{cobar} C \to  \left(TC \rightrightarrows T^2 C \triplearrows
\dots \right),  \end{equation}
where $T = GF$ is the induced comonad on $\mathcal{C}$. 
The cobar construction is a cosimplicial diagram in $\mathcal{C}$ consisting of
objects which are in the image of $G$. 

Here a fundamental distinction between $\infty$-category theory and 1-category
theory appears. In 1-category theory, the limit of a cosimplicial diagram can
be computed as a (reflexive) \emph{equalizer}; only the first zeroth and first
stage of the cosimplicial diagram are relevant. In $n$-category theory (i.e.,
$(n, 1)$-category theory), one only needs to work with the $n$-truncation of a
cosimplicial object. But in an $\infty$-category $\mathcal{C}$, given a
cosimplicial diagram $X^\bullet \colon \Delta \to \mathcal{C}$, one obtains a
\emph{tower}  of partial totalizations
\[ \dots \to \mathrm{Tot}_n(X^\bullet) \to \mathrm{Tot}_{n-1}(X^\bullet) \to
\dots \to \mathrm{Tot}_1(X^\bullet) \to \mathrm{Tot}_0(X^\bullet),  \]
whose homotopy inverse limit is the totalization or inverse limit
$\mathrm{Tot}(X^\bullet)$. By definition, $\mathrm{Tot}_n(X^\bullet)$ is the
inverse limit of the $n$-truncation of $X^\bullet$. 

In an $n$-category, the above tower stabilizes at a finite stage: that is,
the successive maps $\mathrm{Tot}_m(X^\bullet) \to
\mathrm{Tot}_{m-1}(X^\bullet)$ become
equivalences for $m$ large (in fact, $m > n $). In
$\infty$-category theory, this is almost never expected. For example, it will
never hold for the cobar constructions that we obtain from descent
along maps of $\e{\infty}$-rings except in trivial cases. In particular, \eqref{cobar} is an infinite homotopy limit rather than
a finite one. 

Nonetheless, there are certain types of towers that exhibit a weaker form
of stabilization, and behave close to finite
homotopy limits if one is willing to include retracts. Even with 
$\infty$-categories, there are several instances where this weaker form of
stabilization occurs, and it is the purpose of this section to discuss that.

\subsection{Pro-objects}

Consider the following two towers of abelian groups:
\[ \xymatrix{
\vdots \ar[d]  \\
\mathbb{Z} \ar[d]^2 \\
\mathbb{Z} \ar[d]^2 \\
\mathbb{Z}
} \quad \quad \quad \quad\quad\quad\quad\xymatrix{
\vdots \ar[d]  \\
\mathbb{Z} \ar[d]^0 \\
\mathbb{Z} \ar[d]^0 \\
\mathbb{Z}
}.\]
Both of these have inverse limit zero. However, there is an essential
difference between the two. The second inverse system has inverse limit zero
for essentially ``diagrammatic'' reasons. In particular, the inverse limit
would remain zero if we applied any additive functor whatsoever. The first
inverse system has inverse limit zero for a more ``accidental'' reason: that
there are no integers infinitely divisible by two. If we tensored this inverse
system with $\mathbb{Z}[1/2]$, the inverse limit would be $\mathbb{Z}[1/2]$.

The essential difference can be described efficiently using the theory of
\emph{pro-objects}: the second inverse system is actually \emph{pro-zero},
while the first inverse system is a more complicated pro-object. The theory of
pro-objects (and, in particular, constant pro-objects) in
$\infty$-categories will be integral to our
discussion of descent, so we spend the present subsection reviewing it. 

We begin by describing the construction that associates to a given $\infty$-category an
$\infty$-category of pro-objects. Although we have already used freely the
(dual) $\mathrm{Ind}$-construction, we review it formally for convenience. 
\begin{definition}[{\cite[Section 5.3]{HTT}}]
Let $\mathcal{C}$ be an $\infty$-category with finite limits. Then the
$\infty$-category $\pro(\mathcal{C})$ is an $\infty$-category with \emph{all}
limits, receiving a map $\mathcal{C} \to \pro(\mathcal{C})$ with the following
properties: 
\begin{enumerate}
\item $\mathcal{C}  \to \pro(\mathcal{C})$ respects finite limits. 
\item Given an $\infty$-category $\mathcal{D}$ with all limits, restriction
induces an equivalence of $\infty$-categories
\[ \mathrm{Fun}^R(\pro(\mathcal{C}), \mathcal{D}) \simeq
\mathrm{Fun}^\omega(\mathcal{C}, \mathcal{D})  \]
between the $\infty$-category 
$\mathrm{Fun}^R(\pro(\mathcal{C}), \mathcal{D}) $ of limit-preserving functors
$\pro(\mathcal{C}) \to \mathcal{D}$ and the $\infty$-category 
$\mathrm{Fun}^\omega(\mathcal{C}, \mathcal{D})$
of functors
$\mathcal{C} \to \mathcal{D}$ which preserve finite limits. 
\end{enumerate}
\end{definition} 

There are several situations in which the $\infty$-categories of pro-objects
can be explicitly described. 
We refer to \cite[Sec. 3.2]{BHH} for a detailed discussion. 
\begin{example}[{Cf. \cite[7.1.6]{HTT}}] 
The $\infty$-category $\pro(\mathcal{S})$
(where $\mathcal{S}$, as
usual, is the $\infty$-category of spaces)
can be described via
\[ \pro(\mathcal{S}) \simeq
\mathrm{Fun}_{\mathrm{acc}}^{\omega-\mathrm{ct}}(\mathcal{S},
\mathcal{S})^{\op};  \]
that is, $\pro(\mathcal{S})$ is anti-equivalent to the $\infty$-category of
accessible\footnote{In other words, commuting with sufficiently filtered
colimits.} functors $\mathcal{S} \to \mathcal{S}$ which respect finite limits.
This association sends a given space $X$ to the functor $\mathrm{Hom}(X,
\cdot)$ and sends formal cofiltered limits to filtered colimits of functors. \end{example} 

\begin{example} 
Similarly, one can describe the $\infty$-category 
$\pro(\sp)$ of \emph{pro-spectra} as the opposite to the $\infty$-category of
accessible, exact functors $\sp \to \sp$ (a spectrum $X$ is sent to
$\hom_{\sp}(X, \cdot)$ via the co-Yoneda embedding). 
\end{example} 

By construction, any object in $\pro(\mathcal{C})$ can be written as a ``formal'' filtered
inverse limit of objects in $\mathcal{C}$: that is, $\mathcal{C}$ generates
$\pro(\mathcal{C})$ under cofiltered limits. Moreover, $\mathcal{C}
\subset \pro(\mathcal{C})$ as a full subcategory. If $\mathcal{C}$ is
idempotent complete, then $\mathcal{C} \subset \pro(\mathcal{C})$ consists of
the cocompact objects.

\begin{remark} 
If $\mathcal{C}$ is  an ordinary category, then $\pro(\mathcal{C})$ is a
discrete category (the usual pro-category) too. 
\end{remark}

We now discuss the inclusion $\mathcal{C} \subset \pro(\mathcal{C})$, where
$\mathcal{C}$ is an $\infty$-category with finite limits.
\begin{definition} 
An object in $\pro(\mathcal{C})$ is \textbf{constant} if it is equivalent to an
object in the image of $\mathcal{C} \to \pro(\mathcal{C})$. 
\end{definition}

\begin{proposition} 
\label{whenisproobjconst}
Let $\mathcal{C}$ have finite limits. 
A cofiltered diagram $F\colon  I \to \mathcal{C}$ defines a constant pro-object if and
only if the following two conditions are satisfied: 
\begin{enumerate}
\item $F$ admits a limit in $\mathcal{C}$.
\item Given any functor $G\colon  \mathcal{C} \to \mathcal{D}$  preserving finite
limits, the inverse limit of $F$ is preserved 
under $G$. 
\end{enumerate}
\end{proposition} 
In other words, the inverse limit of $F$ is required to exist for essentially
``diagrammatic reasons.''
\begin{proof} 

One direction of this is easy to see (take $\mathcal{D} = \pro(\mathcal{C})$). 
Conversely, if $F$ defines a constant pro-object, then given $\mathcal{C} \to
\mathcal{D}$, we consider the commutative diagram
\[ \xymatrix{
\mathcal{C} \ar[d] \ar[r]^G & \mathcal{D} \ar[d]  \\
\pro(\mathcal{C}) \ar[r]^{\widetilde{G}} &  \pro(\mathcal{D})
}.\]
The functor $F\colon  I \to \mathcal{C} \to \pro(\mathcal{C})$ has an inverse limit,
which actually lands inside the full subcategory 
$\mathcal{C} \subset \pro(\mathcal{C})$. Since $\widetilde{G}\colon  \pro(\mathcal{C}) \to
\pro(\mathcal{D})$ preserves all limits, it follows formally that $\widetilde{G} \circ F$
has an inverse limit lying inside $\mathcal{D} \subset \pro(\mathcal{D})$ and
that $G$ preserves the inverse limit. 

\end{proof} 

\begin{example}[Split cosimplicial objects] \label{splitconst}
Let $\mathcal{C}$ be an $\infty$-category  with finite limits. 
Let $X^\bullet$ be a cosimplicial object of $\mathcal{C} $. 
Suppose $X^\bullet$ extends to a \emph{split, augmented cosimplicial object}. 
In this case, the pro-object associated to the
$\mathrm{Tot}$ tower of $X^\bullet$ (i.e., the tower
$\left\{\mathrm{Tot}_n X^\bullet\right\}$) is constant.

In fact, 
let $\mathcal{D}$ be any $\infty$-category, and let $F\colon 
\mathcal{C} \to \mathcal{D}$ be a functor. Let
$\overline{X}\colon  \Delta^+ \to \mathcal{C}$ be the augmented cosimplicial object 
extending $X^\bullet$ that can be split. 
Then, by \cite[Section
4.7.3]{higheralg}, the composite
diagram
\[ \Delta_+ \stackrel{\overline{X}}{\to} \mathcal{C}
\stackrel{F}{\to} \mathcal{D},  \]
is a limit diagram: that is, $F( \overline{X}^{-1}) \simeq \mathrm{Tot} F(
X^\bullet)$, and in particular $\mathrm{Tot} F( X^\bullet)$ exists. 

Suppose $\mathcal{D}$ admits finite limits and $F$ preserves finite limits. 
Then $F( \mathrm{Tot}_n X^\bullet) \simeq \mathrm{Tot}_n F( X^\bullet)$, since
$F$ preserves finite limits,  so that 
\[ F(\overline{X}^{-1}) \simeq \holim_n \mathrm{Tot}_n F(X^\bullet) \simeq \holim_n F(
\mathrm{Tot}_n X^\bullet),  \]
in $\mathcal{D}$. In particular, the tower $F( \mathrm{Tot}_n X^\bullet)$
converges to $F(\overline{X}^{-1})$. 
By \Cref{whenisproobjconst}, this proves constancy as desired. 

\end{example} 

\begin{example}[Idempotent towers] 
Let $X \in \mathcal{C}$ and let $e\colon  X \to X$ be an \emph{idempotent} self-map;
this means not only that $e^2 \simeq e$, but a choice of coherent homotopies,
which can be expressed by the condition that one has an \emph{action} of the
monoid $\left\{1, x\right\}$ with two elements (where $x^2 =x$) on $X$. In this
case, the tower
\[
\dots \to X \stackrel{e}{\to} X \stackrel{e}{\to} X 
,\]
is pro-constant if it admits a homotopy limit (e.g., if $\mathcal{C}$ is
idempotent complete). This holds for the same reasons: the image of an
idempotent is always a \emph{universal} limit (see \cite[Section 4.4.5]{HTT}). 
\end{example}

Conversely, the fact that a pro-object indexed by a cofiltered diagram $F\colon  I
\to\mathcal{C}$ is constant has many useful implications coming from the fact
that the inverse limit of $F$ is ``universal.''
\begin{example} 
Let $(\mathcal{C}, \otimes, \mathbf{1})$ be a stable homotopy theory. Given 
 a cofiltered diagram $F\colon  I \to \mathcal{C}$, it follows that if the induced
 pro-object is constant, then for any $X \in \mathcal{C}$, the natural map
 \[ (\varprojlim_I F(i)) \otimes X \to \varprojlim_I (F(i) \otimes X),\]
 is an equivalence. See \Cref{dualthing}
 below for a partial converse. 
\end{example}

Next, we show that in a finite diagram of $\infty$-categories, a pro-object is
constant if and only if it is constant at each stage. 

Let $K$ be a finite simplicial set, and let $F\colon  K \to \cati$ be a functor
into the $\infty$-category $\cati$ of $\infty$-categories. Suppose that each
$F(k)$ has finite limits and each edge in $K$ is taken to a functor which
respects finite limits. 
In this case, we obtain a natural functor
\begin{equation}\label{rkan} \pro \left( \varprojlim_K F(k) \right)
\to \varprojlim_K \pro( F(k)),  \end{equation}
which respects all limits. 

\begin{proposition} 
The functor
$ \pro \left( \varprojlim_K F(k) \right)
\to \varprojlim_K \pro( F(k))$ is fully
faithful. 
\end{proposition} 
\begin{proof} 
In fact, the functors $F(k) \to \pro(F(k))$ are fully faithful for each $k \in
K$, so that 
\[ \varprojlim_K F(k) \to  \varprojlim_K \pro( F(k)) \]
is fully faithful and respects finite limits. 
In order for the right Kan extension
\eqref{rkan} to be fully faithful, it follows by \cite[Section 5.3]{HTT} that it suffices
for the embedding 
$\varprojlim_K F(k) \to  \varprojlim_K \pro( F(k))$ to land in the
\emph{cocompact} objects. However, over a finite diagram of $\infty$-categories, an
object is cocompact if and only if it is cocompact pointwise, because finite
limits commute with filtered colimits in spaces. 
\end{proof} 

\begin{corollary} 
\label{constantfinpro}
Let $K$ be a finite simplicial set and let $F\colon  K \to \cati$ be a functor as
above. Then a pro-object in
$\varprojlim_K F(k)$ is constant if and only if its evaluation in $\pro( F(k))$
is constant for each vertex $k \in K$. 
\end{corollary} 
\begin{proof} 
We have a commutative diagram
\[ \xymatrix{
\varprojlim_K F(k) \ar[d] \ar[r]^{\simeq} & \varprojlim_K F(k)  \ar[d] \\
\pro( \varprojlim_K F(k)) \ar[r] &  \varprojlim_K \pro(F(k))
},\]
where the bottom arrow is fully faithful. Given an object in $\pro(
\varprojlim_K F(k))$, it is constant if and only if the image in $\varprojlim_K
\pro(F(k))$ belongs to $\varprojlim_K F(k)$. Since each $F(k) \to \pro(F(k))$
is fully faithful, this can be checked pointwise. 
\end{proof} 

\begin{remark} 
The functor \eqref{rkan} is usually not essentially surjective; consider (with
$\mathrm{Ind}$-objects) for
instance the
failure of essential surjectivity in \Cref{modclosed}. 
\end{remark} 

\subsection{Descendable algebra objects}
Let $(\mathcal{C}, \otimes, \mathbf{1})$ be a 2-ring or a  stable homotopy theory. 
In this subsection, we will describe a definition of a commutative algebra object
in $\mathcal{C}$ which ``admits descent'' in a very strong sense, and prove
some basic properties.

We start by recalling a basic definition. 
\begin{definition} 
If $\mathcal{E}$ is a stable $\infty$-category, we will say that a full subcategory
$\mathcal{D} \subset \mathcal{E}$ is \textbf{thick} if $\mathcal{D}$ is closed
under finite limits and colimits and under retracts. In particular,
$\mathcal{D}$ is stable. Further, if $\mathcal{E}$ is given a symmetric
monoidal structure, then  $\mathcal{D}$ is a \textbf{thick 
$\otimes$-ideal} if in addition it is a $\otimes$-ideal. 

Given a collection of objects in $\mathcal{E}$, the thick subcategory (resp.
thick $\otimes$-ideal) that they \textbf{generate} is defined to be the smallest
thick subcategory (resp. thick $\otimes$-ideal) containing that collection. 

\end{definition} 

The theory of thick subcategories, introduced in \cite{DHS, HS}, has played an important role in making
``descent'' arguments in proving the basic structural results of chromatic
homotopy theory. Thus, it is not too surprising that the following definition
might be useful. 
This notion has been independently studied under the name \emph{nil-faithfulness} by Balmer
\cite{balmersep}. 

\begin{definition} \label{admitd}
Given $A \in \clg(\mathcal{C})$, we will say that $A$ \textbf{admits descent}
or is \textbf{descendable}
if the thick $\otimes$-ideal generated by $A$ is all of $\mathcal{C}$. 

More generally, in a stable homotopy theory $(\mathcal{C}, \otimes,
\mathbf{1})$, we will say that a morphism $A \to B$ in $\clg(\mathcal{C})$
\textbf{admits descent} if $B$, considered as a commutative algebra object in
$\md_{\mathcal{C}}(A)$, admits descent in the above sense.

\end{definition} 

\newcommand{\cb}{\mathrm{CB}^\bullet}
\newcommand{\cbaug}{\mathrm{CB}_{\mathrm{aug}}^\bullet}

We now prove a few basic properties of the property of ``admitting descent,''
for instance the (evidently desirable) claim that an analog of \Cref{luried}
goes through. Here is the first observation. 

\begin{proposition} 
\label{faithful}
If $A \in \clg(\mathcal{C})$ admits descent, then $A$ is faithful: if $M \in
\mathcal{C}$, and $M \otimes A \simeq 0$, then $M$ is contractible. 
\end{proposition} 
\begin{proof} 
Consider the collection of all objects $N \in \mathcal{C}$ such that $M \otimes
N \simeq 0$. This is clearly a thick $\otimes$-ideal. Since it contains $A$, it
must contain $\mathbf{1}$, so that $M$ is contractible. 
\end{proof} 

Given $A \in \clg(\mathcal{C})$, one can form the \emph{cobar resolution}
\[ A \rightrightarrows A \otimes A \triplearrows \dots,  \]
which is a cosimplicial object in $\clg(\mathcal{C})$, receiving an
augmentation from $\mathbf{1}$. 
Call this cosimplicial object $\cb(A)$ and the augmented version
$\cbaug(A)$. 

\begin{proposition} 
\label{constpro}
Given $A \in \clg(\mathcal{C})$, $A$ admits descent if and only if 
the cosimplicial diagram $\cb(A)$ defines a constant pro-object on
the level of towers $\left\{\mathrm{Tot}_n \cb(A)\right\}_{n \geq 0}$ which
converges to $\mathbf{1}$ (i.e., $\cbaug(A)$ is a limit diagram). 
\end{proposition} 
\begin{proof}
Suppose $A$ admits descent. 
Consider the collection $\mathcal{C}_{\mathrm{good}}$ of $M \in \mathcal{C}$ such that the augmented
cosimplicial diagram $\cbaug(A) \otimes M$ is a limit diagram, and such that
the induced $\mathrm{Tot}$ tower converging to $M$ defines a constant
pro-object. Our goal is to show that $\mathbf{1} \in
\mathcal{C}_{\mathrm{good}}$. 

Note first that $A \in \mathcal{C}_{\mathrm{good}}$: in fact, the augmented
cosimplicial diagram $\cbaug(A) \otimes A$ is \emph{split} and so is a limit
diagram and defines a constant pro-object (\Cref{splitconst}). Moreover, $\mathcal{C}_{\mathrm{good}}$ is a
thick $\otimes$-ideal. 
The collection of pro-objects which are constant is thick, and the tensor
product of a constant pro-object with any object of $\mathcal{C}$ is constant (and the
limit commutes with the tensor product). Since $A \in \mathcal{C}_{\mathrm{good}}$, it
follows that $\mathbf{1} \in \mathcal{C}_{\mathrm{good}}$, which completes the
proof in one direction. 

Conversely, if $\cbaug(A)$ is a limit diagram, and $\cb(A)$  defines a constant
pro-object, 
it follows that $\mathbf{1}$ is a retract of $\mathrm{Tot}_n \cb(A)$, for $n
\gg 0$. However, $\mathrm{Tot}_n \cb(A)$ clearly lives in the thick 
$\otimes$-ideal generated by $A$, which shows that $A$ admits descent. 
\end{proof}

In other words, thanks to \Cref{constpro}, $A$ admits descent if and only if the unit object $\mathbf{1}$ can be
obtained as a retract of a  finite colimit of a diagram in $\mathcal{C}$ consisting of
objects, each of which admits the structure of a module over $A$.

One advantage of the purely categorical (and finitistic) definition of
admitting descent is that it is preserved under base change. The next result
follows from \Cref{constpro}. 
\begin{corollary} 
Let $F\colon  \mathcal{C} \to \mathcal{C}'$ be a symmetric monoidal functor between
symmetric monoidal, stable $\infty$-categories. Given $A \in
\clg(\mathcal{C})$, if $A$ admits descent, then $F(A)$ does as well. 
\end{corollary} 

\begin{proposition} \label{easydesc}
Let $\mathcal{C} $ be a stable homotopy theory. 
Let $A \in \clg(\mathcal{C})$ admit descent. Then the adjunction
\[ \mathcal{C} \rightleftarrows \md_{\mathcal{C}}(A),  \]
given by tensoring with $A$ and forgetting, is comonadic. In particular, the
natural functor from $\mathcal{C}$ to the totalization 
\[ \mathcal{C} \to \mathrm{Tot}\left( \md_{\mathcal{C}}(A) \rightrightarrows
\md_{\mathcal{C}}(A \otimes A) \triplearrows \dots \right)  \]
is an equivalence. 
\end{proposition} 
\begin{proof} 
We need to check that the hypotheses of the Barr-Beck-Lurie theorem go through. 
We refer to \cite[Th. 4.7.6.2]{higheralg} for the connection between
comonadicity and the totalization of $\infty$-categories considered above,
which is an $\infty$-categorical generalization of the classical
Beck-B{\'e}nabou-Roubaud theorem \cite{BRB}. 

By \Cref{faithful}, tensoring with $A$ is conservative. 
Now, fix a cosimplicial object $X^\bullet\colon  \Delta \to \mathcal{C}$ such that $A
\otimes X^\bullet$ is split. We need to show 
that the map
\[ A \otimes \mathrm{Tot}(X^\bullet) \to \mathrm{Tot}(A \otimes X^\bullet)  \]
is an equivalence. 
This will follow if the pro-object defined by $X^\bullet$ (i.e., by the
$\mathrm{Tot}$ tower) is constant. To see that, consider the collection of
objects $M \in \mathcal{C}$ such that $M \otimes X^\bullet$ defines a constant
pro-object.  By assumption (and \Cref{splitconst}), this collection contains
$A$, and it is a thick $\otimes$-ideal. It follows that $X^\bullet$ itself defines
a constant pro-object, so we are done. 
\end{proof} 

\begin{remark} 
We have used the fact that we have a symmetric monoidal functor $\mathcal{C}
\to \pro( \mathcal{C})$, which embeds $\mathcal{C}$ as a full subcategory of
$\pro( \mathcal{C})$: in particular, the tensor product of two constant
pro-objects in $\pro(\mathcal{C})$ is constant. 
\end{remark}

Finally, we prove a few basic permanence properties for admitting descent. 
\begin{proposition} \label{permanence}
Suppose $\mathcal{C}$ is a stable homotopy theory. 
Let $A \to B \to C$ be maps in $\clg(\mathcal{C})$. 
\begin{enumerate}
\item If $A \to B$ and $B \to C$ admit descent, so does $A \to C$.  
\item If $A \to C$ admits descent, so does $A \to B$. 
\end{enumerate}
\end{proposition} 
\begin{proof} 
Consider the first claim. 
Suppose $A \to B$ and $B \to C$ admit descent.
Then, via the cobar construction, we find that $B$ belongs to the thick subcategory of $\md_{\mathcal{C}}(B)$ generated by 
the $C$-modules. It follows that $B$ belongs to the thick subcategory of
$\md_{\mathcal{C}}(A)$ generated by the $C$-modules, and therefore every $B$-module
belongs to the thick $\otimes$-ideal in $\md_{\mathcal{C}}(A)$ generated by $C$. 
Since $A \to B$ admits descent, we find that the thick $\otimes$-ideal that $C$
generates in $\md_{\mathcal{C}}(A)$ contains $A$. 

For the second claim, we note simply that a $C$-module is in particular a
$B$-module: the thick $\otimes$-ideal that $B$ generates contains any $B$-module,
for instance $C$. 
\end{proof}

\begin{proposition} 
\label{descfinloc}
Let $K$ be a finite simplicial set and let $p\colon  K \to \shot$ be a diagram. Then
a commutative algebra object $A \in \clg( \varprojlim_K p)$ admits descent if
and only if its ``evaluations'' in $\clg( p(k))$ admit descent for each $k \in K$. 
\end{proposition} 
\begin{proof} 
Admitting descent is preserved under symmetric monoidal, exact functors, so one
direction is evident. For the other, if $A \in \clg( \varprojlim_K p)$ has the
property that its image in each $\clg( p(k))$ admits descent, then consider the
cobar construction $\cb(A)$. It defines a constant pro-object after evaluating
at each $k \in K$, and therefore, by \Cref{constantfinpro}, it defines a
constant pro-object in $\varprojlim_K p$ too. The inverse limit is necessarily
the unit (since this is true at each vertex), so $A$ admits descent. 
\end{proof} 
\subsection{Nilpotence}
In this subsection, we present a slightly different  formulation of the
definition of admitting descent, which makes clear the connection with nilpotence. 

Let $(\mathcal{C}, \otimes, \mathbf{1})$ be a stable homotopy theory and let
$A \in \mathcal{C}$ be any object. 
Given a map $f\colon  X \to Y$ in $\mathcal{C}$, we say that $f$ is \emph{$A$-zero}
if $A \otimes X \xrightarrow{1_A \otimes f} A \otimes Y$ is nullhomotopic (as
a morphism in $\mathcal{C}$). 

The collection of all $A$-zero maps forms what is classically called a 
\emph{tensor ideal} in the triangulated category $\mathrm{Ho}(\mathcal{C})$.
The main
result of this subsection is that a \emph{commutative algebra} object $A$ admits descent if and only if this ideal
is nilpotent, in a natural sense.

\begin{definition} 
A collection $\mathcal{I}$ of maps in $\mathrm{Ho}(\mathcal{C})$
is a \textbf{tensor ideal} if the following hold: 
\begin{enumerate}
\item 
For each $X, Y$, the collection of homotopy classes of maps $X \to Y$ that
belong 
to $\mathcal{I}$ is a subgroup. 
\item  Given $f\colon  X \to Y, g\colon  Y \to Z, h \colon  Z \to W$, then if $g \in
\mathcal{I}$, we have $h \circ g \circ f \in \mathcal{I}$. 
\item Given $g\colon  Y \to Z$ in $\mathcal{I}$ and any other object $T \in
\mathcal{C}$, the tensor product $g \otimes 1_T\colon  Y \otimes T \to Z \otimes T$
belongs to $\mathcal{I}$. 
\end{enumerate}
\end{definition} 

For any $A \in \mathcal{C}$, the collection of $A$-zero maps is clearly a
tensor ideal $\mathcal{I}_A$. Given two tensor
ideals $\mathcal{I}, \mathcal{J}$, we will define the product
$\mathcal{I}\mathcal{J}$ to be the smallest tensor ideal containing all
composites $g \circ f$ where $ f\in \mathcal{J}$ and $ g \in \mathcal{I}$. 

\begin{proposition} \label{descnilp}
Let $A \in \clg(\mathcal{C})$ be a commutative algebra object.
Then the following are equivalent: 
\begin{enumerate}
\item There exists $s \in \mathbb{N}$ such that the composite of $s$ consecutive $A$-zero
maps is zero.  
\item $\mathcal{I}_A^s = 0$ for some $s \in
\mathbb{Z}_{\geq 0}$. 
\item $A$ admits descent. 
\end{enumerate}
\end{proposition} 
This result is essentially \cite[Proposition 3.15]{balmersep}. 
\begin{proof} 
Suppose first $A$ admits descent. We want to show that $\mathcal{I}_A^s = 0$
for some $s \gg 0$. Now, $\mathcal{I}_\mathbf{1} =0$, so our strategy is to use
a thick subcategory argument.

We make the following three claims:
\begin{enumerate}
\item 
If $M, N \in \mathcal{C}$, then $\mathcal{I}_M \subset \mathcal{I}_{M \otimes
N}$. 
\item If $N$ is a retract of $M$, then $\mathcal{I}_M \subset \mathcal{I}_N$. 
\item Given a cofiber sequence
\[ M' \to M \to M''  \]
in $\mathcal{C}$, we have
\[ \mathcal{I}_{M'}\mathcal{I}_{M''} \subset \mathcal{I}_M.  \]
\end{enumerate}
Of these, the first and second are obvious. For the third, it suffices to show
that the composite of an $M'$-null map and an $M''$-null map is $M$-null. 
Suppose $f\colon  X \to Y$ is $M''$-null and $g\colon  Y \to Z$ is $M'$-null. We want to
show that $g \circ f$ is $M$-null.
We have a diagram
\[ \xymatrix{
X \otimes M' \ar[d] \ar[r] &  Y \otimes M' \ar[d]  \ar[r] &  Z \otimes M' \ar[d]  \\
X \otimes M \ar[d]  \ar[r] &  Y \otimes M \ar[d]  \ar[r] &  Z \otimes M \ar[d]  \\
X \otimes M'' \ar[r] &  Y \otimes M'' \ar[r] &  Z \otimes M'' 
}.\]
Here the vertical arrows are cofiber sequences. Chasing through this diagram,
we find that $X \otimes M \to Y \otimes M$ factors through $X \otimes M \to Y
\otimes M'$, so that the composite $X \otimes M \to Z \otimes M$ factors
through $X \otimes M \to Y \otimes M' \stackrel{0}{\to} Z \otimes M' \to Z
\otimes M$ and is thus nullhomotopic.

It thus follows (from the above three items) that if $M \in \mathcal{C}$ is arbitrary, then for any
$\overline{M} \in
\mathcal{C}$ belonging to the thick $\otimes$-ideal generated by $M$, we have
\[ \mathcal{I}_M^s \subset \mathcal{I}_{\overline{M}},  \]
for some integer $s \gg 0$. If $\mathbf{1} \in \mathcal{C}$ belongs 
to this thick $\otimes$-ideal, that forces $\mathcal{I}_M$ to be nilpotent. 

Conversely, suppose there exists $s \in \mathbb{Z}_{\geq 0}$ such that the composite of $s$ consecutive
$A$-zero maps is zero. We will show that $A$ admits descent. 
Given an object $M \in \mathcal{C}$, we want to show that $M$ belongs to the
thick $\otimes$-ideal generated by $A$. 
For this, consider the functor
\[ F_1(X) = \mathrm{fib}(X \to X \otimes A);  \]
we have a natural map 
\( F_1(X) \to X , \) which is $A$-zero, and whose cofiber belongs to the thick
$\otimes$-ideal generated by $A$. Iteratively define $F_n(X) = F_1(F_{n-1}(X))$ for
$n > 0$. We get a tower
\[ \dots  \to F_n(M) \to F_{n-1}(M) \to \dots \to F_1(M) \to M,  \]
where all the successive cofibers of $F_i(M) \to F_{i-1}(M)$ belong to the
thick $\otimes$-ideal generated by
$A$. By chasing cofiber sequences, this means that 
the cofiber of each $F_i(M) \to M$ belongs to the thick $\otimes$-ideal generated
by $A$. 

Moreover, each of the maps in this tower is $A$-zero. 
It follows that $F_s(M) \to M$ is zero. Thus the cofiber of $F_s(M) \to M$ is
$M \oplus \Sigma F_s(M)$, which belongs to the thick $\otimes$-ideal generated by
$A$. Therefore, $M$ belongs to this thick $\otimes$-ideal, and we are done. 
\end{proof}

\subsection{Local properties of modules}

In classical algebra, many properties of modules are local for the \'etale (or
flat) topology. These statements can be generalized to the setting of
$\e{\infty}$-ring
spectra, where 
one considers morphisms $R \to R'$ of $\e{\infty}$-rings that are \'etale (or flat, etc.) on the
level of $\pi_0$ and such that the natural map $\pi_0 R' \otimes_{\pi_0 R}
\pi_* R \to \pi_* R'$ is an isomorphism. 

Our next goal is to prove a couple of basic results in our setting 
for 
descendable morphisms.

\begin{proposition} 
\label{compactdesc}
Let $A \to B$ be a descendable morphism of $\e{\infty}$-rings. 
Let $M$ be an $A$-module such that $B \otimes_A M$ is a perfect $B$-module.
Then $M$ is a perfect $A$-module. 
\end{proposition} 
\begin{proof} 
Consider a filtered category $\mathcal{I}$ and a functor $\iota\colon  \mathcal{I}
\to \md(A)$. We then need to show that
\[ \varinjlim \hom_{A}(M, M_\iota) \to \hom_{A}(M, \varinjlim M_{\iota})  ,  \]
is an equivalence. 
Consider the collection $\mathcal{U}$ of $A$-modules $N$ such that
\[ \varinjlim \hom_{A}(M, M_\iota \otimes_{A} N) \to \hom_{A}(M, \varinjlim M_{\iota}
\otimes_{A} N)  ,  \]
is a weak equivalence; we would like to show that it contains
$A$ itself. The collection $\mathcal{U}$ clearly forms a thick subcategory. 
Observe that it contains $N = B$ using the adjunction relation
\[ \hom_A(P, P' \otimes_A B) \simeq \hom_{B}(P \otimes_A B , P' \otimes_A B),   \]
valid for $P, P' \in \md(A)$,
and the assumption that $M \otimes_A B$ is compact in $\md(B)$. More generally,
this implies that 
every tensor product $B \otimes_A \dots \otimes_A B$ belongs to $\mathcal{U}$. Since $A$
is a retract of a finite limit of copies of such $A$-modules, via the cobar
construction, it follows that
$A \in \mathcal{C}$ and that $M$ is compact or perfect in $\md(A)$. 
\end{proof}

\begin{remark}
More generally, the argument of \Cref{compactdesc} shows that if $\mathcal{C}$ is an $A$-linear
$\infty$-category, and $M \in \mathcal{C}$ is an object that becomes compact after
tensoring with $B$ (as an object of $\md_{\mathcal{C}}(B)$), then $M$ was compact
to begin with. \Cref{compactdesc} itself could have also been proved by observing
that $\md(A) $ is a totalization $\mathrm{Tot}\left(\md(B) \rightrightarrows \md(B \otimes_A B)
\triplearrows\right)$ 
by \Cref{easydesc} 
and an $A$-module is thus dualizable (equivalently, compact) if
and only if its base-change to $\md(B)$ is, as dualizability in an inverse limit
of symmetric monoidal $\infty$-categories can be checked vertexwise (cf.
\cite[Prop. 4.6.1.11]{higheralg}). 
\end{remark}

\begin{proposition} 
Let $A \to B$ be a descendable morphism of $\e{\infty}$-rings. Let $M$ be an
$A$-module. Then $M$ is invertible if and only if $M \otimes_A B$ is invertible. 
\end{proposition} 
\begin{proof} 
Observe first that $M \otimes_A B$ is perfect (since it is invertible), so 
$M$ is also perfect via \Cref{compactdesc}. The evaluation map $M
\otimes_A M^{\vee} \to A$ has the property
that it becomes an equivalence after tensoring up to $B$, since the formation
of $M \mapsto M^{\vee}$ commutes with base extension for $M$ perfect. It
follows that $M \otimes_A M^{\vee} \to A$ is itself an equivalence, so that $M$
is invertible. 
\end{proof}

Let $M$ be an $A$-module. If $A \to B$ is a descendable morphism of
$\e{\infty}$-rings such that $M \otimes_A B$ is a finite direct sum of copies
of $B$, the $A$-module $M$ itself need
not
look anything like a free module. (The finite covers explored in this paper are examples.)
However, such ``locally free'' $A$-modules seem to have interesting
and quite restricted properties.

\subsection{First examples}

In the following section, we will discuss more difficult examples of the 
phenomenon of admitting descent, and try to give a better feel for it. Here, we describe
some relatively ``formal'' examples of maps which admit descent. 

We start by considering the evident faithfully flat case. 
In general, we \emph{do not know} if a faithfully flat map $A \to B$ of
$\e{\infty}$-ring spectra (i.e., such that $\pi_0(A) \to \pi_0(B)$ is
faithfully flat and such that $\pi_*(A) \otimes_{\pi_0(A)} \pi_0(B) \to
\pi_*(B)$ is an isomorphism) necessarily admits descent, even in the case of
discrete $\e{\infty}$-rings. This would have some implications. For example, if $A$ and $B$ are discrete commutative rings, it would imply that if $M$ is an $A$-module
and $\gamma \in \mathrm{Ext}^n_A(M, M)$ is a class whose image in
$\mathrm{Ext}^n_{B}(M \otimes_A B, M \otimes_A B)$ vanishes, then $\gamma$ is
nilpotent. 
Nonetheless, one has:

\begin{proposition} 
\label{ffdesc}
Suppose $A \to B$ is a faithfully flat map of $\e{\infty}$-rings such that
$\pi_*(A)$ is countable. Then $A \to B$ admits descent. 
\end{proposition}

\begin{proof} 
We can use the criterion of \Cref{descnilp}. We claim that we can take $s = 2$.
That is, given composable maps $M \to M' \to M''$ of $A$-modules each of which
becomes
nullhomotopic after tensoring up to $B$, the \emph{composite} is nullhomotopic. 

To see this, we observe that any $B$-zero map in $\md(A)$ is \emph{phantom}. In
other words, if $M \to M'$ is $B$-zero, then any composite
\[ P \to M \to M',  \]
where $P$ is a perfect $A$-module, is already nullhomotopic. To see this, note
that $P \to M'$ is $B$-zero, but to show that it is already nullhomotopic, we
can dualize and consider
\[ \pi_* ( \mathbb{D}P \otimes_A M') \to \pi_* ( \mathbb{D}P \otimes_A M'
\otimes_A B),  \]
which is injective since $B$ is faithfully flat over $A$ on the level of
homotopy groups. The injectivity of this map forces any $B$-zero map $P \to M'$
to be automatically zero to begin with. 

Finally, we can conclude if we know that the composite of two phantom maps in
$\md(A)$ is
zero. This claim is \cite[Theorem 4.1.8]{axiomatic}; we need countability of
$\pi_*(A)$ to conclude that homology theories on $A$-modules are representable
(by \cite[Theorem 4.1.5]{axiomatic}). 
\end{proof} 

Without the countability hypothesis, the result about phantom maps is known to
be false. 
It is, however, possible to strengthen \Cref{ffdesc} using 
more recent techniques of \emph{transfinite} Adams representabililty \cite{RM}. We are grateful to Oriol Ravent{\'os} for 
explaining the following to us. 

\begin{proposition} 
Let $ A \to B$ be a faithfully flat morphism of $\e{\infty}$-rings such that
$\pi_*(A)$ has cardinality at most $\aleph_k$ for some $k \in \mathbb{N}$. Then
$A \to B$ admits descent.
\end{proposition} 
\begin{proof} 
As above, it suffices to show that the composite of $k + 2$ phantom maps of $A$-modules
is necessarily nullhomotopic. 
Consider the class $\mathcal{C} = \mathrm{Perf}(A)$ of perfect $A$-modules, which 
at most $\aleph_\kappa$ isomorphism classes of objects. 

Consider the 
category $\md(\mathcal{C})$ of functors 
$\mathcal{C}^{op} \to \mathrm{Ab}$ which preserve finite coproducts. 
Given any object $X \in \md(A)$, the Yoneda lemma gives an object $h_X \in
\md(\mathcal{C})$. Note that $h_X$ is a filtered colimit of functors
representable by objects in $\mathcal{C}$. 
Taking $\alpha = \aleph_0$, we apply \cite[Prop. 2.13]{RM}, 
we find that $h_X$ for any $X \in \mathcal{C}$ has projective dimension $\leq
k+1$. By \cite[Cor. 6.3.5]{RM}, we find that $X$ is $(k+2)-\mathcal{C}$-cellular in the
sense of \cite[Def. 6.1.5]{RM}. Since $X$ was arbitrary, we find by \cite[Prop.
6.1.6]{RM} that the composite of $(k+2)$ phantom maps is zero.  
\end{proof} 

Since descendability is preserved under base change, we obtain: 

\begin{corollary} 
Let $A \to B$ be a faithfully flat map of $\e{\infty}$-rings such that
$\pi_0(B)$ has a presentation $\pi_0(A)$-algebra with at most $\aleph_k$
generators and relations for some $k \in \mathbb{N}$. Then $A \to B$
admits descent. 
\end{corollary} 

For example, a finitely presented faithfully flat map of discrete rings is
descendable. For a finitely presented map $A \to B$ of noetherian
rings, Bhatt and Scholze have shown \cite[Th. 5.26]{BhSch} that $A \to B$ is
admits descent if and only if $\mathrm{Spec}(B) \to \mathrm{Spec}(A)$ is an
$h$-cover, which is significantly weaker. 

In addition to faithfully flat maps which are not too large, there are examples
of descendable maps of $\e{\infty}$-rings which look more like (relatively
mild) quotients. 
\begin{proposition} 
Suppose $A$ is an $\e{\infty}$-ring which is connective and such that $\pi_i A
= 0 $ for $i \gg 0$. Then the map $A \to \pi_0 A$ admits descent. 
\end{proposition} 

\begin{proof} 
Given an $A$-module $M$ such that $\pi_*(M)$ is concentrated in one degree, it
admits the structure of a $\pi_0 A$-module (canonically) and thus belongs to
the thick $\otimes$-ideal generated by $\pi_0 A$. However, $A$ admits a
finite resolution by such $A$-modules, since one has a finite Postnikov
decomposition of $A$ in $\md(A)$ whose successive cofibers have a single
homotopy group, and therefore belongs to the thick $\otimes$-ideal generated by
$\pi_0 A$. 
\end{proof}

\begin{proposition} 
Let $R$ be a discrete commutative ring. Let $I \subset R$ be a nilpotent ideal.
Then the map $R \to R/I$ of discrete commutative rings, considered as a map of
$\e{\infty}$-rings, admits descent.  
\end{proposition} 
\begin{proof} 
For $k \gg 0$, we have a finite filtration of $R$ in the world of discrete $R$-modules
\[ 0 = I^k \subset I^{k-1} \subset \dots \subset I \subset R,  \]
whose successive quotients are $R/I$-modules. This implies that $R/I$ generates all of $\md(R)$ as a thick $\otimes$-ideal. 
\end{proof} 

There are also examples of descendable morphisms where the condition on the
thick $\otimes$-ideals follows from a defining limit diagram. 
\begin{proposition} 
Let $R$ be an $\e{\infty}$-ring and let $X$ be a finite connected CW complex. Then the
map $C^*(X; R) \to R$ given by evaluating at a basepoint $\ast \in X$ admits
descent. 
\end{proposition} 
\begin{proof} 
In fact, the $\e{\infty}$-ring $C^*(X; R)$ is a finite limit (indexed by $X$) of copies of $R$ by definition. 
That is, $C^*(X; R) \simeq \varprojlim_X R$. 
\end{proof} 

\begin{proposition} 
Let $R $ be an $\e{\infty}$-ring and let $x \in \pi_0 R$. Then the map $R \to
R[x^{-1}] \times \widehat{R}_x$ (where $\widehat{R}_x$ is the $x$-adic
completion) admits descent. 
\end{proposition}
\begin{proof} 
This follows from the arithmetic square
\[ \xymatrix{
R \ar[d] \ar[r] & R[x^{-1}] \ar[d]  \\
\widehat{R}_x \ar[r] &  \widehat{R}_x[x^{-1}]
}.\]
All three of the terms in the fiber product here are $R[x^{-1}] \times
\widehat{R}_x$-modules, so $R$ belongs to the thick subcategory generated by the 
$R[x^{-1}] \times
\widehat{R}_x$-modules and we are done. 
\end{proof} 

Next we include a deeper result, which will imply (for example) that the faithful Galois
extensions considered by \cite{rognes} admit descent; this will be very
important in the rest of the paper. 
The theory of nilpotence with respect to a dualizable algebra object
has been treated in more detail in \cite{MNNequiv}.

\begin{theorem} 
\label{cptdescent}
Let $\mathcal{C}$ be a stable homotopy theory. 
Suppose $\mathbf{1} \in \mathcal{C}$ is compact, and suppose $A \in
\clg(\mathcal{C})$ is dualizable and faithful (i.e., tensoring with $A$ is
conservative). Then $A$ admits descent. 
\end{theorem} 
\begin{proof} 
Consider the cobar construction $\cb(A)$ on $A$. The first claim is that it
converges to $\mathbf{1}$: that is, the augmented cosimplicial construction
$\cbaug(A)$ is a limit diagram. To see this, we can apply the Barr-Beck-Lurie
theorem to $A$. Since $A$ is dualizable, we have
for $X, Y \in \mathcal{C}$,
\[ \hom_{\mathcal{C}}(Y, A \otimes X) \simeq \hom_{\mathcal{C}}( \mathbb{D}A
\otimes Y, X), \]
and in particular tensoring with $A$ commutes with all limits in $\mathcal{C}$. 
Since tensoring with $A$ is conservative, we find that the hypotheses of the
Barr-Beck-Lurie theorem go into effect (cf. also \cite[2.6]{banerjee}). In particular, $\cb(A)$ converges to
$\mathbf{1}$ and, moreover, for any $M \in \mathcal{C}$, $\cb(A) \otimes M $
converges to $M$. 
We need to show that the induced pro-object is \emph{constant}, though. This
will follow from the next lemma. 
\end{proof}

\begin{lemma} 
\label{dualthing}
Let $(\mathcal{C}, \otimes, \mathbf{1})$ be a stable homotopy
theory where $\mathbf{1}$ is compact. 
Let $I$ be a cofiltered category, and let $F\colon  I \to \mathcal{C}$ be a functor.
Suppose that for each $i \in I$, $F(i) \in \mathcal{C}$ is dualizable. Then 
$F$ defines a constant pro-object (or is \emph{pro-constant}) if and only if the following are satisfied. 
\begin{enumerate} 
\item  $\varprojlim_I F(i)$ is a dualizable object. 
\item For each object $C \in \mathcal{C}$, the natural map
\begin{equation} \label{thismap} ( \varprojlim_I F(i)) \otimes C \to
\varprojlim_I (F(i) \otimes C)    \end{equation}
is an equivalence. 
\end{enumerate}
\end{lemma} 

\begin{proof} 
Let $\mathbb{D}$ be the duality functor (of internal hom into $\mathbf{1}$); it
induces a contravariant auto-equivalence
on the subcategory $\mathcal{C}^{\mathrm{dual}}$ of dualizable objects in $\mathcal{C}$. 
To say that $F$ defines a constant pro-object in $\mathcal{C}$ (or,
equivalently, 
$\mathcal{C}^{\mathrm{dual}}$)
is to say that
$\mathbb{D}F$, which is an \emph{ind}-object of 
$\mathcal{C}^{\mathrm{dual}}$, defines a constant ind-object. In other words, we
have a commutative diagram
of $\infty$-categories,
\[ \xymatrix{
\mathcal{C}^{\mathrm{dual}} \ar[d]^{\subset} \ar[r]_{\simeq}^{\mathbb{D}} &
\mathcal{C}^{\mathrm{dual, \ op}} \ar[d]^{\subset} \\
\pro(\mathcal{C}^{\mathrm{dual}}) \ar[d]^{\subset}\ar[r]_{\simeq}^{\mathbb{D}} &
\mathrm{Ind}(\mathcal{C}^{\mathrm{dual}})^{\op}  \\
\pro(\mathcal{C})
}.\]
Now, since $\mathcal{C}^{\mathrm{dual}} \subset \mathcal{C}$ consists of compact
objects (since $\mathbf{1} \in \mathcal{C}$ is compact), we know that 
there is a fully faithful inclusion $\mathrm{Ind}( \mathcal{C}^{\mathrm{dual}})
\subset \mathcal{C}$, which sends an ind-object to its colimit. If $\mathcal{C}$ is generated by dualizable objects, this
is even an equivalence, but we do not need this. 

As a result, to show that $\mathbb{D}F \in \mathrm{Ind}(
\mathcal{C}^{\mathrm{dual}})$ defines a constant ind-object, it is sufficient to
show that its colimit in $\mathcal{C}$ actually belongs to
$\mathcal{C}^{\mathrm{dual}}$. 
Let $X = \varprojlim_I F(i) \in \mathcal{C}$; by hypothesis, this is a
dualizable object. 
We have a natural map (in $\mathcal{C}$)
\[ \varinjlim_I \mathbb{D} F(i) \to  \mathbb{D} X,  \]
and if we can prove that this is an equivalence, we will have shown that
$\varinjlim_I \mathbb{D} F(i)$ is a dualizable object and thus the ind-system
is constant. In other words, we must show that if $C \in \mathcal{C}$ is
arbitrary, then the natural map
of spectra
\[ \hom_{\mathcal{C}}(\mathbb{D}X, C) \to \varprojlim_I
\hom_{\mathcal{C}}(\mathbb{D} F(i), C)   \]
is an equivalence. But this map  is precisely
$\hom_{\mathcal{C}}(\mathbf{1}, \cdot)$ applied to \eqref{thismap}, so we are done.
\end{proof} 

\begin{remark} 
This result requires $\mathbf{1}$ to  be compact. If $\mathcal{C}$
is the stable homotopy theory of $p$-adically complete chain complexes of
abelian groups (i.e., the localization  of $D(\mathbb{Z})$ at
$\mathbb{Z}/p\mathbb{Z}$), then $\mathbb{Z}/p\mathbb{Z}$ is a
dualizable, faithful commutative
algebra object, but the associated pro-object is not constant, or the
$p$-adic integers $\mathbb{Z}_p$ would be torsion. 
\end{remark} 

\begin{remark} 
One can prove the same results (e.g., \Cref{cptdescent}) if $A  \in \mathcal{C}$ is given an
\emph{associative} (or $\e{1}$) algebra structure, rather than an
$\e{\infty}$-algebra structure. However, the symmetric monoidal structure on
$\mathcal{C}$ itself is crucial throughout. 
\end{remark} 

\subsection{Application: descent for linear $\infty$-categories}
In fact, the definition of descent considered here gives a more
general result than \Cref{easydesc}. Let
$\mathcal{C}$ be an $A$-linear $\infty$-category in the sense of \cite{DAGss}. 
In other words, $\mathcal{C}$ is a presentable, stable $\infty$-category which
is a \emph{module} in the symmetric monoidal $\infty$-category $\prl$ of presentable,
stable $\infty$-categories over $\md(A)$. This means that there is a bifunctor, which preserves colimits in each variable,
\[ \otimes_A\colon  \md(A) \times \mathcal{C} \to \mathcal{C} , \quad (M, C) \mapsto M \otimes_A
{C} \]
together with additional compatibility data: for instance, equivalences $A \otimes_A M
\simeq M$  for each $M \in \mathcal{C}$. 

Given such a $\mathcal{C}$, one can study, for any $A$-algebra $B$, the
$\infty$-category $\md_{\mathcal{C}}(B)$ of $B$-modules \emph{internal to $\mathcal{C}$}: this is the
``relative tensor product'' in $\prl$
\[ \md_{ \mathcal{C}}(B) = \mathcal{C} \otimes_{\md(A)} \md(B).  \]
Useful references for this, and for the tensor product of presentable
$\infty$-categories, are \cite{gaitsnotes} and \cite{BFN}. 

Informally, $\md_{\mathcal{C}}(B)$ is the target of an $A$-bilinear functor
\[  \otimes_A\colon  \mathcal{C} \times \md(B) \to \md_{\mathcal{C}}(B),  \quad (X, M) \mapsto X
\otimes_A M,\]
which is colimit-preserving in each variable, and it is universal for such.
As in the case $\mathcal{C} = \md(A)$, one has an adjunction
\[ \mathcal{C} \rightleftarrows \md_{\mathcal{C}}(B),  \]
given by ``tensoring up'' and forgetting the $B$-module structure.

One can then ask whether descent holds in $\mathcal{C}$, just as we studied
earlier for $A$-modules. In other words, we can ask whether
$\mathcal{C}$ is equivalent to  the $\infty$-category of
$B$-modules in $\mathcal{C}$ equipped with analogous ``descent data'':
equivalently, whether the ``tensoring up'' functor
$\mathcal{C} \to \md_{\mathcal{C}}(B)$ is comonadic. 
Stated another way, 
we are asking whether, for any $\md(A)$-module
\emph{$\infty$-category} $\mathcal{C}$, we have an equivalence
of $A$-linear $\infty$-categories
\begin{equation}\label{Alineardescent}  \mathcal{C} \simeq \mathrm{Tot}\left( \mathcal{C} \otimes_{\md(A)}
\md(B)^{\otimes (\bullet+1)}\right).  \end{equation}

In fact, the proof of \Cref{easydesc} applies and we get:
\begin{corollary} 
 \label{proconstdescentC}
Suppose $A \to B$ is a descendable morphism of $\e{\infty}$-rings. 
Then $A \to B$ satisfies descent for any $A$-linear $\infty$-category
$\mathcal{C}$ in that 
the functor
from
$\mathcal{C} $
to ``descent data''
is an equivalence. 
\end{corollary}
\begin{proof} 
By the Barr-Beck-Lurie theorem, we need to see that tensoring with $B$ defines a conservative functor
$\mathcal{C} \to \md_{\mathcal{C}}(B)$ which respects $B$-split totalizations. 
Conservativity can be proved as in \Cref{faithful}. Given $R \in \mathcal{C}$,
the collection of $A$-modules $M$ such that $M \otimes_A R \simeq 0$ is a thick
$\otimes$-ideal in $\md(A)$. If $B$ belongs to this thick $\otimes$-ideal, so must $A$, and $R$
must be zero. 

Let $X^\bullet\colon  \Delta \to \mathcal{C}$ be a cosimplicial object which becomes
split after tensoring with $B$. As in \Cref{easydesc}, it suffices to show that
the pro-object that $X^\bullet$ defines is constant in $\mathcal{C}$. 
This follows via the same thick subcategory argument: one considers the
collection of $M \in \md(A)$ such that $X^\bullet \otimes_A M$ defines a
constant pro-object, and observes that $M$ is a thick $\otimes$-ideal containing
$B$, thus containing $A$. 
Thus $X^\bullet$ defines a constant pro-object. 
\end{proof} 

We note that the argument via pro-objects yields a mild strengthening of the
results in the DAG series.
In particular, it shows that if $A \to B$ is a morphism 
of $\e{\infty}$-rings which is faithfully flat and presented by at most
$\aleph_k$ generators and relations (for some $k \in \mathbb{N}$), it satisfies descent for any $A$-linear
$\infty$-category. In the DAG series, this is proved assuming
\emph{\'etaleness}
\cite[Th. 5.4]{DAGdesc}
or for faithfully flat morphisms assuming existence of a
$t$-structure \cite[Th. 6.12]{DAGss}. In fact, this idea of descent via thick subcategories  seems to be the
right setting for considering the above questions, in view of the following
result, which was explained to us by Jacob Lurie: 

\begin{proposition} 
Let $A \to B$ be a morphism of $\e{\infty}$-rings such that, for any $A$-linear
$\infty$-category, descent holds, i.e., we have an equivalence \eqref{Alineardescent}. 
Then $A \to B$ admits descent. 
\end{proposition} 
\begin{proof} 
Suppose $A \to B$ does not admit descent. We will look for a counterexample 
to \eqref{Alineardescent}. 
We will exhibit an $A$-linear presentable $\infty$-category $\mathcal{D}$ and an object $X
\in \mathcal{D}$ such that the totalization of the cobar construction $\cb(B)
\otimes_A X$ is not equivalent to $X$. 

The idea is to take $\mathcal{D} = \pro( \md(A))$ and $X = A$. 
Consider the cobar construction $B \rightrightarrows
 B \otimes_A B\triplearrows \dots$. The totalization of the cobar construction in $\pro(\md(A))$ is
\emph{precisely} the cobar construction considered as a pro-object via the
$\mathrm{Tot}$ tower. In particular,  if  $A \to B$ fails to admit 
descent, the cobar construction does not converge to $A$ in $\pro(\md(A))$. 

In order to make this argument precise, we have to address the fact that $\pro(
\md(A))$ is not really an $A$-linear $\infty$-category: it is not, for example,
presentable. 
Choose a regular cardinal $\kappa$ such that $B$ is $\kappa$-compact as an
$A$-module.
Choose a small subcategory $\mathcal{C} \subset \pro( \md(A))$ which contains
the constant object $A$ and is closed under $\kappa$-small colimits and
countable limits.
Then $\mathcal{C}$ is tensored over the $\infty$-category $\md^\kappa(A)$
of $\kappa$-compact $A$-modules, so the presentable
$\infty$-category $\mathcal{D} = \mathrm{Ind}_\kappa(\mathcal{C})$ is tensored over $\md(A)$
in a compatible manner. Moreover, in this $\infty$-category the totalization of
the cobar construction $B \rightrightarrows B \otimes_A B \triplearrows \dots$
does not converge to $A$ as that does not happen in $\mathcal{C}$. 
\end{proof} 

Finally, we note a ``categorified'' version of
descent, which, while likely far from the strongest possible, is already of interest in 
studying the Brauer group of $\e{\infty}$-rings such as $\TMF$. 
This phenomenon has been extensively studied (under the name ``1-affineness'')
in \cite{gaits}. We will only consider a very simple and special case of this question. 

The idea is that instead of considering descent for modules over
a ring spectrum $R$ (possibly internal to a linear $\infty$-category), we will
consider descent for the linear $\infty$-categories themselves, which we
will call \emph{2-modules},
meaning modules \emph{over} the presentable, symmetric monoidal
$\infty$-category $\md( R)$. 

\newcommand{\tmd}{\operatorname{2-Mod}}
\begin{definition} 
Given an $\e{\infty}$-ring $R$, there is a symmetric monoidal $\infty$-category
$\tmd(R)$
of  $R$-linear
$\infty$-categories with the $R$-linear tensor product. In other words, $\tmd(R)$ consists of modules (in the
symmetric monoidal $\infty$-category of presentable, stable
$\infty$-categories) over $\md(R)$. 
\end{definition} 

For a useful reference, see \cite{gaitsnotes, AntieauGepner}. 
We now record:

\begin{proposition} 
\label{2descformal}
Let $A \to B$ be a descendable morphism of $\e{\infty}$-rings. Then $\tmd$ satisfies descent along $A \to B$. 
\end{proposition} 
As noted in \cite{gaits} and \cite{DAGdesc}, this is a formal consequence of
descent in linear $\infty$-categories (that is, \Cref{proconstdescentC}), but we recall the proof for convenience. 
\begin{proof} 
Recall that we have the adjunction
\[ (F, G)  = \left(  \otimes_{\md(A)} \md(B),  \ \mathrm{forget}\right) \colon \quad  \tmd(A) \rightleftarrows \tmd(B) ,  \]
where $G$ is the forgetful functor from $B$-linear $\infty$-categories to
$A$-linear $\infty$-categories, and where $F$ is ``tensoring up.''
The assertion of the proposition is that this adjunction is comonadic. 
By the Barr-Beck-Lurie theorem, it suffices to show now that $F$ is
conservative and preserves certain totalizations. 

But $F$ is conservative
because any $\mathcal{C}$-linear $\infty$-category can be recovered from its
``descent data'' after tensoring up to $B$ (\Cref{proconstdescentC}). 
Moreover, $F$ commutes with all limits. In fact, $F$ sends an $A$-linear
$\infty$-category $\mathcal{C}$ to the collection of $B$-module objects in
$\mathcal{C}$, and this procedure is compatible with limits. 
\end{proof} 

It would be interesting to give conditions under which one could show that a
2-module over $R$ admitted a compact generator if and only if it did so locally
on $R$ in some sense. This would yield a type of descent for the \emph{Brauer
spectrum} of $R$ (see for instance
\cite{AntieauGepner}), 
whose $\pi_0$ consists of equivalence classes of invertible 2-modules that
admit a compact generator. Descent for compactly generated $R$-linear
$\infty$-categories is known to hold in the \emph{usual} \'etale
topology on $\e{\infty}$-rings \cite[Theorem 6.1]{DAGdesc}, although the proof is long and
complex. Descent also holds for the finite covers considered in this paper
which are \emph{faithful}.  It would be interesting to see if it held for $L_n S^0 \to
E_n$, possibly in some $K(n)$-local sense.

\label{lindesc}

\section{Nilpotence and Quillen stratification}

Let $(\mathcal{C}, \otimes, \mathbf{1})$ be a stable homotopy theory. 
Let $A \in \clg(\mathcal{C})$ be a commutative algebra object in $\mathcal{C}$.
In general, we might hope that (for whatever reason) phenomena in
$\md_{\mathcal{C}}(A)$ might be simpler to understand than phenomena in
$\mathcal{C}$. For example, if $\mathcal{C} = \sp$, we do not know the homotopy
groups of the sphere spectrum, but there are many $\e{\infty}$-rings whose
homotopy groups we do know completely: for instance, $H \mathbb{F}_p$ and $MU$. 
We might then try to use our knowledge of $A$ and some sort of descent to
understand phenomena in $\mathcal{C}$. For instance, we might attempt to
compute the homotopy groups of an object $M \in \mathcal{C}$ by constructing
the cobar resolution
\[ M \to \left( M \otimes A \rightrightarrows M \otimes A \otimes A
\triplearrows \dots \right),  \]
and hope that it converges to $M$. This method is essentially the \emph{Adams
spectral sequence}, which, in case $\mathcal{C} = \sp$, is one of the most
important tools one has for calculating
and understanding the stable homotopy groups of spheres. 

In the previous section, we introduced a type of commutative algebra object $A
\in \clg(\mathcal{C})$ such that, roughly, the above descent method converged
very efficiently --- much more efficiently, for instance, than the classical
Adams or Adams-Novikov spectral sequences. One can see this at the level of descent spectral sequences in the
existence of \emph{horizontal vanishing lines} that occur at finite stages. 
In particular, in this situation, one can understand phenomena in $\mathcal{C}$
from phenomena in $\md_{\mathcal{C}}(A)$ and $\md_{\mathcal{C}}(A \otimes A)$
``up to (bounded) nilpotence.'' We began discussing this in \Cref{descnilp}. The purpose
of this section is to continue that discussion and to describe several
fundamental (and highly non-trivial) examples of commutative algebra
objects that admit descent. 
These ideas have also been explored in \cite{balmersep}, and we learned of the
connection with Quillen stratification from there. 

\subsection{Descent spectral sequences}

Let $\mathcal{C}$ be a stable homotopy theory. 
Let $A \in \clg(\mathcal{C})$ and let $M \in \mathcal{C}$. As usual, we can try
to study $M$ via the $A$-module $M \otimes A$ and, more generally, the cobar
construction $M \otimes \cb(A)$. 
In this subsection, we will describe the effect of descendability on
the resulting spectral sequence. 

\begin{definition} 
The $\mathrm{Tot}$ tower of the cobar construction $M \otimes \cb(A)$ is called
the \textbf{Adams tower} $\{T_n(A, M)\}$ of $M$. 
The induced spectral sequence converging to $\pi_* \varprojlim ( M \otimes
\cb(A))$ is called the \textbf{Adams spectral sequence} for $M$ (based on $A$). 
\end{definition} 

The Adams tower has the property 
that it comes equipped with maps
\[ \xymatrix{ 
& \vdots \ar[d]  \\
& T_2(A, M) \ar[d]  \\
 & T_1(A, M) \ar[d]  \\
M \ar[ru] \ar[ruu] \ar[r] &  T_0(A, M) \simeq A \otimes M 
}.\]
In other words, it is equipped with a map from the \emph{constant} tower at
$M$. We let the cofiber of this map of towers be $\left\{U_n(A, M)\right\}_{n
\geq 0}$. 

The tower $\left\{U_n(A, M)\right\}$ has the property that the cofiber of any map $U_n(A, M)
\to U_{n-1}(A, M)$ admits the structure of an $A$-module. Moreover, each map
$U_n(A, M) \to U_{n-1}(A, M)$ is null after tensoring
with $A$. 

Suppose now that $A$ \emph{admits descent.} In this case, the towers we are
considering have particularly good properties.

\newcommand{\tow}{\mathrm{Tow}}
\begin{definition}[\cite{HPS, thick}]
Let $\tow(\mathcal{C}) = \mathrm{Fun}( \mathbb{Z}_{\geq 0}^{\op},
\mathcal{C})$ be the $\infty$-category 
of towers in $\mathcal{C}$. 

We shall say that a tower $\left\{X_n\right\}_{n
\geq 0}$ is \textbf{nilpotent} if
there exists $N$ such that $X_{n+N} \to X_n$ is null for each $n \in
\mathbb{Z}_{\geq 0}$. 
It is shown in \cite{HPS} that the collection of nilpotent towers is a thick
subcategory of $\tow(\mathcal{C})$. We shall say that a tower is
\textbf{strongly constant} if it belongs to the thick subcategory of
$\tow(\mathcal{C})$ generated by the nilpotent towers and the constant towers. 
\end{definition}

Observe that  a nilpotent tower is pro-zero, and a strongly constant tower is
pro-constant. In general, nilpotence of a tower is \emph{much} stronger than
being pro-zero. For example, a tower $\left\{X_n\right\}$ is pro-zero if there is a 
cofinal set of integers $i$ for which the $X_i$ are contractible. This does not
imply nilpotence. 

We now recall the following fact about strongly constant towers:
\begin{proposition}[\cite{HPS}] Let $\{X_n\}_{n \geq 0} \in \tow(\mathcal{C})$ be a strongly
constant tower. Then, for $Y \in \mathcal{C}$, the spectral sequence for $\pi_* \hom(Y, \varprojlim X_n)$
has a horizontal vanishing line at a finite stage. 
\end{proposition} 

In fact, in \cite{HPS}, it is shown that admitting such horizontal vanishing
lines is a \emph{generic} property of objects in $\tow(\mathcal{C})$: that is, 
the collection of objects with that property is a thick subcategory. Moreover,
this property holds for nilpotent towers and for constant towers. 

\begin{corollary} 
Let $A \in \clg(\mathcal{C})$ admit descent. Then the Adams tower
$\left\{T_n(A, M)\right\}$ is strongly constant. 
In particular, the Adams spectral sequence converges with a horizontal
vanishing line at a finite stage (independent of $M$). 
\end{corollary} 
\begin{proof} 
In fact, by \Cref{descnilp}, it follows that the tower $\left\{U_n(A,
M)\right\}$ is nilpotent, since all the successive maps in the tower are
$A$-zero, so the tower $\left\{T_n(A, M)\right\}$ is therefore strongly constant. 
\end{proof} 

It follows from this that we can get a rough global description of the
graded-commutative ring $\pi_*
\mathbf{1}$ if we have a description of $\pi_* A$. This is the description that
leads, for instance, to the description of various group cohomology rings  ``up
to nilpotents.''
\begin{theorem} \label{Fiso}
Let $A \in \clg(\mathcal{C})$ admit descent. 
Let $R_*$ be the equalizer of $\pi_*(A) \rightrightarrows \pi_*(A \otimes A)$. 
There is a map $\pi_* ( \mathbf{1}) \to R_*$ with the following properties: 
\begin{enumerate}
\item The kernel of $\pi_*(\mathbf{1} ) \to R_*$ is a nilpotent ideal. 
\item Given an element $x \in R_*$ with $Nx = 0$, then $x^{N^k}$ belongs to the
image of $\pi_* ( \mathbf{1}) \to R_*$ for $k \gg 0$ (which can be chosen
uniformly in $N$). 
\end{enumerate}
\end{theorem} 

In the examples arising in practice, one already has a complete or
near-complete picture \emph{rationally}, so the torsion information is the
most interesting. For example, if $p $ is nilpotent in $\pi_* (\mathbf{1} )$,
then the map that one gets is classically called a \emph{uniform
$F$-isomorphism}. 

\begin{proof} 
In fact, $R_*$ as written is the zero-line (i.e., $s = 0$) of the $E_2$-page of
the $A$-based Adams spectral sequence converging to the homotopy groups of $\mathbf{1}$. 
The map that we have written down is precisely the edge homomorphism in the
spectral sequence. We know that anything of positive filtration at $E_\infty$
must be nilpotent of bounded order because of the horizontal vanishing line.
That implies the first claim. 

For the second claim, let $x \in E_2^{0, t}$ be
$N$-torsion. We want to show that $x^{N^k}$ survives the spectral sequence for 
some $k$ (which can be chosen independently of $x$). 
In fact, $x^N$ can support no $d_2$ by the Leibnitz rule. Similarly, $x^{N^2}$
can support no $d_3$ and survives until $E_4$. Since the spectral sequence
collapses at a finite stage, we conclude that some $x^{N^k}$ must survive, and
$k$ depends only on the finite stage at which the spectral sequence collapses. 

\end{proof}

\begin{remark} 
One can obtain an analog of \Cref{Fiso} for any commutative algebra object in
$\mathcal{C}$ replacing $\mathbf{1}$. 
\end{remark} 

\subsection{Quillen stratification for finite groups}
Let $G$ be a finite group, and let $R$ be  a (discrete) commutative ring. 
Consider the $\infty$-category $\md_G(R) \simeq \mathrm{Fun}(BG, \md(R))$ of $R$-module spectra with a
$G$-action (equivalently, the $\infty$-category of module spectra over the
\emph{group ring}), which is symmetric monoidal under the $R$-linear tensor product. 
Given a subgroup $H \subset G$, we have a natural symmetric monoidal functor
\[ \md_G(R) \to \md_H(R),  \]
given by restricting the $G$-action to $H$. 
As in ordinary algebra, we can identify this with a form of tensoring
up: we can identify $\md_H(R)$ with the $\infty$-category
of modules over the commutative algebra object $\prod_{G/H} R \in \md_G(R)$,
with $G$ permuting the factors. 
We state this formally as a proposition (compare \cite{balmerstack, BDAS}). 

\begin{proposition} 
Consider the commutative algebra object 
$\prod_{G/H} R \in \clg( \md_G(R))$, with $G$-action permuting the factors. 
Then the forgetful functor identifies $\md_H(R)$ with the symmetric monoidal
$\infty$-category of modules in $\md_G(R)$ over $\prod_{G/H} R$. 
\end{proposition} 

We can interpret this in the following algebro-geometric manner as well. 
The $\infty$-category
$\md_G(R)$ can be described as the $\infty$-category of quasi-coherent
complexes on the classifying stack $BG$ of the discrete group $G$, over the base
ring $R$. Similarly, $\md_H(R)$ can be described as the $\infty$-category of
quasi-coherent sheaves on $BH$. One has an affine map $\phi\colon  BH \to BG$ (in fact, a
finite \'etale cover), so that quasi-coherent complexes on $BH$ can be
identified with quasi-coherent complexes on $BG$ with a module action by
$\pi_*(\mathcal{O}_{BH})$, which corresponds to $\prod_{G/H} R$. 

In particular, we can attempt to perform ``descent'' along the restriction
functor $\md_G(R) \to \md_H(R)$, or descent with the commutative algebra object
$\prod_{G/H} R$, or descent for quasi-coherent sheaves along the cover $BH
\to BG$. If $R$ contains $\mathbb{Q}$ or, more generally, if $|G|/|H|$ is invertible
in $R$, there are never any problems, because the $G$-equivariant
\emph{norm map} $\prod_{G/H}R  \to R$ will exhibit $R$ as a retract of the
object
$\prod_{G/H} R$, so that the commutative algebra object $\prod_{G/H} R$ is descendable. 

The
question is much more subtle in modular characteristic. For example, given a finite group
$G$ and a field $k$ of characteristic $p$ with $p \mid |G|$, the group
cohomology $H^*(G; k)$ is always infinite-dimensional, which prevents the
commutative algebra object $\prod_G k$ from being descendable. Nonetheless, one
has the following result. Recall that a group is called \emph{elementary
abelian} if it is of the form $(\mathbb{Z}/p)^n$ for some prime number $p$.

\begin{theorem}[Carlson \cite{carlson}, Balmer \cite{balmersep}] \label{BC}
Let $G$ be a finite group, and let $\mathcal{A}$ be a collection of elementary
abelian subgroups of $G$ such that every maximal elementary abelian
subgroup of $G$ is conjugate to an element of $\mathcal{A}$. Then the
commutative algebra object $\prod_{H \in \mathcal{A}}\prod_{G/H} R$ admits
descent in $\md_G(R)$. \end{theorem} 

In other words, there is a strong theory of descent
along the map $\bigsqcup_{A \in \mathcal{A}} BA \to BG$ of stacks. 
%One only needs to consider nontrivial elementary abelian $p$-subgroups for $p$
%noninvertible in $R$. 
If $p$ is invertible in $R$ and $H$ is an elementary
abelian $p$-group, then $\prod_{G/H} R \in \md_G(R)$ is a retract of $\prod_{G}
R$. 
To translate to our terminology, we note that \cite[Theorem 2.1]{carlson}
states that there is a finitely generated $\mathbb{Z}[G]$-module $V$ with the
property that there exists a
finite filtration $0 = V_0 \subset \dots \subset V_k = \mathbb{Z} \oplus V$
such that the successive quotients are all \emph{induced}
$\mathbb{Z}[G]$-modules from elementary abelian subgroups of $G$. Given an
object of $\md_G(\mathbb{Z})$ which is induced from $H \subset G$, we observe
that it is naturally a module in $\md_G(\mathbb{Z})$ over $\prod_{G/H}
\mathbb{Z}$.

Note moreover that the map
\begin{equation} \label{ss} \bigsqcup_{A \in \mathcal{A}} BA \to BG,
\end{equation}
which we have identified as having a good theory of descent, is explicit enough
that we can also write down the relative fiber product 
$\left(\bigsqcup_{A \in \mathcal{A}} BA \right) \times_{BG}
\left(\bigsqcup_{A \in \mathcal{A}} BA\right)$ via a double coset decomposition. 
Stated another way, the tensor products of commutative algebra objects of the
form $\prod_{G/H} R$, which appear in the cobar construction, can be described explicitly.

From this, and \Cref{Fiso} (and the immediately following remark), one obtains the following corollary, which is known to modular
representation theorists and is a generalization of the description by
Quillen \cite{equiv} of the cohomology ring of a finite group up to
$F$-isomorphism. 

\begin{corollary} 
 Let $R$ be an $\e{2}$-algebra in $\md(\mathbb{Z})$ with an action of the finite
group $G$. Suppose $p$ is nilpotent in $R$. 
Let $\mathcal{A}$ be the collection of all elementary abelian $p$-subgroups of
$G$.
Then the map 
\[ R^{hG} \to \prod_{A \in \mathcal{A}} R^{hA},  \]
has nilpotent kernel in $\pi_*$. The image, up to uniform $F$-isomorphism, consists of all
families which are compatible under restriction and conjugation. 
\end{corollary} 

A family $(a_A \in \pi_* R^{hA})_{A \in \mathcal{A}}$ is compatible under
retriction and conjugation if, whenever $g \in G$ conjugates $A $ into $A'$,
then the induced map $R^{hA} \simeq R^{hA'}$ carries $a_{A}$ into $a_{A'}$;
and, furthermore, whenever $B \subset A$, then the map $R^{hA} \to R^{hB}$
carries $a_A$ into $a_B$. These compatible families form the $E_2$-page of the
descent spectral sequence for the cover \eqref{ss}. 
When $R = \mathbb{F}_p$ (as was considered by Quillen), the above corollary is extremely useful 
since the cohomology rings of elementary abelian groups are entirely known
and easy to work with. 
Given a connected space $X$ with $\pi_1 X$ finite, one could also apply it to
the $\pi_1$-action on $C^*(\widetilde{X}; \mathbb{F}_p)$ where $\widetilde{X}$
is the universal cover. 

We will use this picture extensively in the future, in particular for the
\emph{stable module $\infty$-categories.}
For now, we note a simple example of one of its consequences. 

\begin{corollary} 
The inclusion $\mathbb{Z}/p \subset \mathbb{Z}/p^k$ induces a descendable map
of $\e{\infty}$-rings
\[ \mathbb{F}_p^{h \mathbb{Z}/p^k} \to \mathbb{F}_p^{h \mathbb{Z}/p},  \]
for each $k > 0$. 
\end{corollary} 
\begin{proof} 
Consider the $\infty$-category $\md_{\mathbb{Z}/p^k}(\mathbb{F}_p)$ of
$\mathbb{F}_p$-module spectra with a $\mathbb{Z}/p^k$-action. Inside here we
have the commutative algebra object $\prod_{ \mathbb{Z}/p^{k-1}} \mathbb{F}_p$
which, by \Cref{BC}, admits descent. 

Note that, 
as in
\eqref{reppgroup}, the subcategory $\md^\omega_{\mathbb{Z}/p^k}(
\mathbb{F}_p)$ of perfect $\mathbb{F}_p$-modules with a $\mathbb{Z}/p^k$-action
is symmetric monoidally equivalent to the $\infty$-category of perfect $\mathbb{F}_p^{h
\mathbb{Z}/p^k}$-modules. Thus, if we show that 
$\prod_{ \mathbb{Z}/p^{k-1}} \mathbb{F}_p$ generates the unit $\mathbb{F}_p$ itself as a
thick $\otimes$-ideal in 
$\md^\omega_{\mathbb{Z}/p^k}(
\mathbb{F}_p)$ (rather than the larger $\infty$-category $\md_{\mathbb{Z}/p^k}(
\mathbb{F}_p)$), we will be done. But this extra claim comes along for free,
since we can use the cobar construction. The cobar construction on
$\prod_{\mathbb{Z}/p^{k-1}} \mathbb{F}_p$ is constant as a pro-object either way, and that means
that $\mathbb{F}_p$ belongs to the thick $\otimes$-ideal generated by 
$\prod_{ \mathbb{Z}/p^{k-1}} \mathbb{F}_p$ in the smaller $\infty$-category. 
\end{proof} 

We refer to \cite{MNNequiv, MNNequiv2} for many further examples of these phenomena in
equivariant homotopy theory (e.g., when $R$ is replaced by a ring spectrum)
and analogs of $F$-isomorphism and induction theorems.

\subsection{Stratification for Hopf algebras}

Let $k$ be a field of characteristic $p$, and let $A$ be a
finite-dimensional commutative Hopf algebra over $k$.
One may attempt to obtain a similar picture in the derived $\infty$-category of
$A$-comodules. 
This has been considered by several authors, for example in \cite{palmieri,
wilkerson, FP}. 
The case of the previous subsection was $A = \prod_G k$ when $G$ is
a finite group, given the coproduct dual to the multiplication in $k[G]$. 
In this subsection, which will not be used in the sequel, we describe the
connection between some of this work and the notion of descent theory
considered here. 
In this subsection, we assume that all Hopf algebras $A$ that occur are \emph{graded connected}, i.e., 
$A = \bigoplus_{i \in \mathbb{Z}_{\geq 0}} A_i$ with $A_0 = k$ and the Hopf
algebra structure respects the grading.

The Hopf algebra $A$ defines a \emph{finite group scheme} $G =
\spec A$ over $k$, and we are interested in the $\infty$-category of quasi-coherent
complexes on the classifying stack $B G$ and in understanding descent in here. 
For every closed subgroup $H \subset G$, we obtain a morphism of stacks
\[ f_H\colon  BH \to BG,  \]
which is \emph{affine}, even finite: in particular, quasi-coherent sheaves on $BH$ can be
identified with modules in $\qcoh(BG)$ over $(f_H)_*(\mathcal{O}_{BH}) \in
\clg( BG)$. One might hope that a certain collection of (proper) subgroup schemes $H
\subset G$ would have the property that the commutative algebra objects $(f_H)_*(\mathcal{O}_{BH})$
jointly generate, as a thick $\otimes$-ideal, all of $\qcoh( BG)$.

When $G$ is constant (although this is not covered by our present graded
connected setup), then the Quillen stratification theory (i.e.,
\Cref{BC}) identifies such a collection of subgroups. The key step is to show
that if $G$ is not elementary abelian, then the collection of 
$(f_H)_*(\mathcal{O}_{BH})$
as $H$ ranges over \emph{all} proper
subgroups of $G$ jointly satisfy descent. 
The picture is somewhat more complicated for finite group schemes. 

\begin{definition}[Palmieri \cite{palmieri}] A group scheme $G$ is \textbf{elementary}
if it is commutative and satisfies the following
condition. Let $\mathcal{O}(G)^{\vee}$ be the ``group algebra,'' i.e., the dual to the
ring $\mathcal{O}(G)$ of functions on $G$.  Then   for every $x$ in the augmentation ideal of
$\mathcal{O}(G)^{\vee}$, we have $x^p = 0$. 
Dualizing, this is equivalent to the condition that the
\emph{Verschiebung} should annihilate $G$. 
\end{definition}

\begin{remark}
The 
``group algebra'' $\mathcal{O}(G)^{\vee}$ 
plays a central role here because $\qcoh(BG)$, if we forget the symmetric
monoidal structure, is simply $\md( \mathcal{O}(G)^{\vee})$; the Hopf algebra
structure on $\mathcal{O}(G)^{\vee}$ gives rise to the symmetric monoidal
structure. 
\end{remark}

In \cite{palmieri}, Palmieri also defines  a weaker notion of
\emph{quasi-elementarity} for finite group schemes $G$, in terms of the
vanishing of certain products of Bocksteins.
It is this more complicated condition that actually ends up intervening. 

Consider first a group scheme $G$ of rank $p$ over $k$ (which is necessarily
commutative), arising as the spectrum of a graded connected Hopf algebra. Then 
the underlying algebra $\mathcal{O}(G)^{\vee}$ is isomorphic to $k[x]/x^p$. 
In particular, there is a basic generating class $\beta \in H^2( BG ) \simeq
\mathrm{Ext}^2_{\mathcal{O}(G)^{\vee}}(k, k)$ called the \emph{Bockstein}
$\beta_G$. The Bockstein, considered as a map $\mathbf{1} \to
\Sigma^2\mathbf{1}$ in $\qcoh(BG)$, has the property
that 
the \emph{cofiber} of $\beta$ belongs to the thick subcategory generated by the
``regular representation'' $\mathcal{O}(G)^{\vee}$, in view of the exact
sequence of $\mathcal{O}(G)^{\vee} \simeq k[x]/x^p$-modules
\[ 0 \to k \to  \mathcal{O}(G)^{\vee} \to \mathcal{O}(G)^{\vee} \to k \to 0,  \]
which exhibits the two-term complex $\mathcal{O}(G)^{\vee} \to
\mathcal{O}(G)^{\vee}$ as the cofiber of $\beta$ (up to a shift). 
Since the map $\mathcal{O}(G)^{\vee} \to \mathcal{O}(G)^{\vee}$ is nilpotent
(it is given by multiplication by $x$), it follows that the thick subcategory
generated by the cofiber of $\beta$ is equal to that generated by the standard
representation. 

\begin{definition} \label{quasi}
A group scheme $G$ arising from a graded connected Hopf algebra is \textbf{quasi-elementary} if the product $\prod_{\phi\colon  G
\twoheadrightarrow G'} \phi^*( \beta_{G'})$ for all surjections $\phi\colon  G
\twoheadrightarrow G'$ for $G'$ a group scheme of rank $p$ (always respecting
the grading), is not nilpotent in
the cohomology of $BG$. \end{definition}

\begin{remark} 
Let $G = \spec A$ be a nontrivial group scheme arising from a graded connected Hopf algebra.
Then there is always a surjective map $G \twoheadrightarrow G'$ with $G'$ of
rank $p$ 
(respecting the grading). To see this, we observe that there is a nontrivial
primitive element $x \in A_i$ for $i  > 0$ and, replacing $x$ with a suitable
power, we may assume that $x^p = 0$. This defines the map to a graded version
of $\alpha_p$. 
\end{remark}

For finite groups, it is a classical theorem of Serre that quasi-elementarity is
equivalent to being elementary abelian: if $G$ is a finite $p$-group which is not
elementary abelian, then there are nonzero classes $\alpha_1, \dots, \alpha_n \in
H^1(G; \mathbb{Z}/p)$ such that the product of the Bocksteins $\prod \beta
(\alpha_i)$ vanishes. Serre's result is, as explained in \cite{carlson,
balmersep}, at the source of the Quillen stratification theory for finite
groups, in particular \Cref{BC}. 

\begin{proposition}[{Cf. \cite[Th. 1.2]{palmieri}}]
Let $G$ be a finite group scheme arising from a graded connected Hopf
algebra over $k$.  Then
$G$ is not quasi-elementary if and only if the 
objects $(f_{H})_*(\mathcal{O}_{BH}) \in \clg( \qcoh(BG))$, for $H \subset G$ a proper
normal subgroup scheme (respecting the grading), generate the unit as
a thick $\otimes$-ideal. 
\end{proposition} 
\begin{proof} 
Suppose $\kappa$ is nilpotent. 
For each rank $p$ quotient $\phi\colon  G \twoheadrightarrow G'$, we have a map 
$\mathbf{1} \to \Sigma^2\mathbf{1}$ in $\qcoh(BG')$ whose cofiber is in the thick subcategory of
$\qcoh(BG')$
generated by the pushforward of the structure sheaf via $\ast \to BG'$. Pulling
back, we get, for each rank $p$ quotient $\phi\colon  G \twoheadrightarrow G'$ with
kernel $H_{\phi}$, a map $\beta_\phi\colon  \mathbf{1} \to \mathbf{1}[2]$ in $\qcoh(BG)$ whose cofiber is in the
thick subcategory generated by $(f_{H_\phi})_*(\mathcal{O}_{BH_\phi})$ where
$f_{H_\phi}\colon  BH_\phi \to BG$ is the natural map. 
It follows in particular that the cofiber of each $\beta_{\phi}$ belongs to the
thick subcategory $\mathcal{C} \subset \qcoh( BG)$ generated by the
$\{ (f_H)_*(\mathcal{O}_{BH}) \}$
for $ H
$ a proper normal subgroup scheme of $G$.  
Therefore, using the octahedral axiom, the cofiber of each \emph{composite} of a finite string of
$\beta_{\phi}$'s (e.g., $\kappa$ and its powers)
belongs to $\mathcal{C}$. 
It follows finally that, by nilpotence of $\kappa$, the unit object itself
belongs to $\mathcal{C}$.

Conversely, suppose that the $\{(f_{H})_* ( \mathcal{O}_{BH})\}$ generate the
unit as a thick $\otimes$-ideal: that is, descent holds.  In this case, we show
that the product of Bocksteins 
$\kappa = \prod_{\phi\colon  G
\twoheadrightarrow G'} \phi^*( \beta_{G'})$ in \Cref{quasi}
is forced to be nilpotent. 
In fact, we observe that  for every
proper normal subgroup $H \subset G$,  there exists a morphism from $G/H $ to a
rank $p$ group scheme $Q$. The pull-back of the Bockstein $\beta_Q$ to $H^2(BG)$ restricts to
zero on $H$; in particular, $\kappa$ restricts to zero on each normal subgroup
scheme of $G$. By descent, it follows that $\kappa$ is nilpotent. 

\end{proof} 

By induction, one gets: 
\begin{corollary} 
Let $G$ be a group scheme over $k$ arising from a graded connected Hopf algebra. Then the commutative algebra
objects $(f_H)_*(\mathcal{O}_{BH}) \in \clg( \qcoh(BG))$, as $H \subset G$
ranges over all the quasi-elementary subgroup schemes (respecting the grading), admits descent. 
\end{corollary}

Unfortunately, it is known that quasi-elementarity and elementarity are not equivalent for
general finite group schemes \cite{wilkerson}. There is, however, one important case when this
is known.

Let $p = 2$. 
Consider the 
dual Steenrod algebra $\mathcal{A} \simeq \mathbb{F}_2[\xi_1, \xi_2, \dots ]$. This is a graded, connected, and
commutative (but not cocommutative) Hopf algebra over $\mathbb{F}_2$. 
The object
$\spec \mathcal{A}$, which is now an (infinite-dimensional) group scheme,
admits an elegant algebro-geometric interpretation as the 
automorphism group scheme of the formal additive group $\widehat{\mathbb{G}_a}$. 
Let $A$ be a finite-dimensional graded Hopf algebra quotient of the dual Steenrod
algebra, so that $G = \spec A$ is a finite group scheme inside the group scheme of
automorphisms of $\widehat{\mathbb{G}_a}$. 

\begin{theorem}[Wilkerson \cite{wilkerson}]  Let $A$ be as above, and let $\mathcal{B}$ range
over all the elementary subgroup schemes $H \subset G$ (respecting the grading). Then the map
\( \bigsqcup_{H \in \mathcal{B}} BH \to B G,  \)
admits descent, in the sense that the commutative algebra object $\prod_{H \in
\mathcal{B}} (f_H)_{*}(\mathcal{O}_{BH}) \in \clg(\qcoh(BG))$ does. 
\end{theorem} 

In particular, it is known that for subgroup schemes of $\spec
\mathcal{A}$ (cut out by homogeneous ideals), elementarity and quasi-elementarity are equivalent. Related ideas
have been used by Palmieri \cite{palmieriquillen} to give a complete
description of the cohomology of such Hopf
algebras up to $F$-isomorphism at the prime 2.

\subsection{Chromatic homotopy theory}

Thick subcategory ideas were originally introduced in chromatic homotopy theory. 
Let $E_n$ denote a Morava $E$-theory of height $n$; thus $\pi_0(E_n) \simeq
W(k)[[v_1, \dots,v_{n-1}]]$ where $W(k)$ denotes the Witt vectors on a perfect
field $k$ of characteristic $p$. 
Moreover, $\pi_*(E_n) \simeq \pi_0(E_n)[t_2^{\pm 1}]$ and $E_n$ is thus
\emph{even periodic}; the associated formal group is the Lubin-Tate universal
deformation of a height $n$ formal group over the field $k$. 
By a deep theorem of Goerss-Hopkins-Miller, $E_n$ has the (canonical) structure
of an $\e{\infty}$-ring. 

Let $L_n $ denote the functor of localization at $E_n$. 
The basic result is the following: 

\begin{theorem}[{Hopkins-Ravenel \cite[Chapter 8]{ravenelorange}}]
\label{HR}
The map $L_n S^0 \to E_n$ admits descent. 
\end{theorem} 

In other words, the $E_n$-based Adams-Novikov spectral sequence degenerates
with a horizontal vanishing line at a finite stage, for any spectrum.
This degeneration does \emph{not} happen at the $E_2$-page (e.g., for the sphere) and
usually implies that a great many differentials are necessary early on. 
\Cref{HR}, which implies that $E_n$-localization is \emph{smashing}, is
fundamental to the global structure of the stable homotopy category and its
localizations. 
As in the finite group case, one of the advantages of results such as \Cref{HR}
is that $E_n$ is much simpler algebraically than is $L_n S^0$. 

The Hopkins-Ravenel result is a basic finiteness property of the $E_n$-local
stable homotopy category. It implies, for instance, that many homotopy limits
that one takes (such as the homotopy fixed points for the $\mathbb{Z}/2$-action
on $KU$) behave much more like finite homotopy limits than infinite ones. 

\begin{example} 
Let $R$ be an $\e{2}$-ring spectrum
which is $L_n$-local. Then it follows that the map from $\pi_*(R)$ to the
zero-line of the 
$E_2$-page of the Adams-Novikov spectral sequence for $R$ is an $F$-isomorphism.
Indeed, we know that the map from $\pi_*(R)$ to the 
zero-line at $E_2$ is a rational isomorphism and, moreover, everything above
the $s =0$ line vanishes at $E_2$. (This comes from the algebraic fact that
rationally, the moduli stack of formal groups is a $B \mathbb{G}_m$ and has no
higher cohomology.)
\end{example} 

\begin{example} 
Let $R$ be an $L_n$-local ring spectrum. Then any class in $\pi_*(R)$ which
maps to zero in $(E_n)_*(R)$ is  nilpotent. This is a very special case of the
general (closely related) nilpotence theorem of \cite{DHS, HS}. 
For an $\e{\infty}$-ring, by playing with power operations, one can actually
prove a stronger result \cite{MNN}: any \emph{torsion} class is
automatically nilpotent. 
\end{example}

\part{The Galois formalism}

\section{Axiomatic Galois theory}

Let $(X, \ast)$ be a pointed, connected topological space. 
A basic and useful invariant of $(X, \ast)$ is the \emph{fundamental group}
$\pi_1(X, \ast)$, defined as the group of homotopy classes of paths $\gamma\colon 
[0, 1] \to X$ with $\gamma(0) = \gamma(1) = \ast$. 
This definition has the disadvantage, at least from the point of view of an
algebraist, of intrinsically using the unit
interval $[0, 1]$ and the topological structure of the real numbers
$\mathbb{R}$. However, the fundamental group also has another incarnation that
makes no reference to the theory of real numbers, via the theory of
\emph{covering spaces}. 

\begin{definition} \label{defcov}
A map $p \colon  Y \to X$ of topological spaces is a \textbf{covering space} if, for
every $x \in X$, there exists a neighborhood $U_x$ of $x$ such that in the
pullback 
\[ \xymatrix{
Y \times_X U_x\ar[d]  \ar[r] & Y \ar[d]  \\
U_x \ar[r] & X
},\]
the map $Y \times_X U_x \to U_x$ has the form $\bigsqcup_S U_x \to U_x$ for a
set $S$. 
\end{definition} 

The theory of covering spaces makes, at least a priori, 
no clear use of $[0, 1]$. Moreover, understanding the theory of covering
spaces  of $X$ is
essentially equivalent to 
understanding the group $\pi_1(X, \ast)$. If $X$ is locally contractible, then one has
the following basic result: 

\begin{theorem} \label{coveringspaces}
Suppose $X$ is path-connected and locally contractible. 
Let $\cov_X$ be the category of maps $Y \to X$ which are covering spaces. 
Then, we have an equivalence of categories $\cov_X \simeq
\mathrm{Set}_{\pi_1(X, \ast)}$, which sends a cover $p\colon  Y \to X$ to the fiber
$p^{-1}(\ast)$ with the monodromy action of $\pi_1(X, \ast)$. 
\end{theorem} 

The fundamental group $\pi_1(X, \ast)$ can, in fact,  be \emph{recovered} from
the structure of the category $\cov_X$. 
This result suggests that the theory of the fundamental group should be more
primordial than its definition might suggest; at least, it might be expected to have
avatars in other areas of mathematics in which the notion
of a covering space makes sense.

Grothendieck realized, in \cite{sga1}, that there is a purely algebraic notion
of a \emph{finite cover} for a scheme (rather than a topological space): that
is, given a scheme $X$, one can define a version of $\cov_X$ that corresponds
to the topological notion of a finite cover. When $X$ is a variety over the
complex numbers $\mathbb{C}$, the algebraic notion turns out to be equivalent
to the topological notion of a finite cover of the complex points
$X(\mathbb{C})$ with the analytic topology. As a result, in \cite{sga1}, it was possible to define a
\emph{profinite group} classifying these finite covers of schemes. Grothendieck
had to prove a version of \Cref{coveringspaces} without an a priori definition
of the fundamental group, and did so by \emph{axiomatizing} the properties that
a category would have to satisfy in order to arise as the category of finite
sets equipped with a continuous action of a profinite group. He could then
\emph{define} the group in terms of the category of finite covers. The main objective
of this paper is to obtain similar categories from stable homotopy theories.

The categories
that appear in this setting are called \emph{Galois categories}, and the
theory of Galois categories will be
reviewed in this section. 
We will, in particular, describe a version of Grothendieck's Galois theory
that does not require a fiber functor, relying primarily on versions
of descent theory.

\subsection{The fundamental group}
To motivate the definitions, we begin by quickly reviewing how the classical
\'etale fundamental group of \cite{sga1} arises.

\begin{definition} 
Let $f\colon  Y \to X$ be a finitely presented map of schemes. We say that $f\colon  Y \to X$ is 
\textbf{\'etale} if $f$ is flat and the sheaf $\Omega_{Y/X}$ of relative
K\"ahler differentials vanishes. 
\end{definition} 

\'Etaleness is the algebro-geometric analog of being a ``local homeomorphism''
in the complex analytic topology. 
Given it, one can define 
the analog of a (finite) covering space. 

\begin{definition} 
A map $f\colon  Y \to X$ is a \textbf{finite cover} (or finite covering space) if $f$
is finite and \'etale. The collection of all finite covering spaces of $X$ forms a
category $\cov_X$, a full subcategory of the category of schemes over $X$. 
\end{definition} 

The basic example of a finite \'etale cover is the map $\bigsqcup_S X \to X$.
If $X$ is connected, then a map $Y \to X$ is a finite cover if and only if it
\emph{locally}  has this form with respect to the flat topology. In other
words, a map $Y \to X$ is a finite cover if and only if there exists a finitely
presented, faithfully flat map $X' \to X$ such that the pull-back
\[ \xymatrix{
X' \times_X Y \ar[d] \ar[r] &  Y \ar[d] \\
X' \ar[r] &  X
},\]
is of the form $\bigsqcup_S X' \to X'$ where $S$ is a finite set; if $X$ is
not connected, the number of sheets may vary over $X$. In other
words, one has an analog of \Cref{defcov}, where ``locally'' is replaced by ``locally
in the flat topology.'' This strongly suggests that the
algebro-geometric definition 
of a finite cover is a good analog of the conventional topological one.

\begin{example} \label{galpt}
Suppose $X = \spec k$ where $k$ is an algebraically closed field. In this case,
there is a canonical equivalence of categories
\[ \cov_X \simeq \finset , \]
where $\finset$ is the category of finite sets,
which sends an \'etale cover $Y \to X$ to its set of connected components. 
\end{example}

Fix a geometric point $\overline{x} \to X$, and assume that $X$ is a
\emph{connected} scheme. 
Grothendieck's idea is to extract the fundamental group $\pi_1(X,
\overline{x})$ directly from the structure of the 
\emph{category} $\cov_X$. 
In particular, as in \Cref{coveringspaces}, the category $\cov_X$ will be
equivalent to the
category of representations (in finite sets) of a certain (profinite) group
$\pi_1(X, \overline{x})$. 

\begin{definition} 
The \textbf{fundamental group} $\pi_1(X, \overline{x})$ of the pair $(X,
\overline{x})$ is given by the
automorphism group of the forgetful functor
\[ \cov_X \to \finset,  \]
which consists of the composite
\[ \cov_X \to \cov_{\overline{x}} \simeq \finset, \]
where the first functor is the pull-back and the second is the equivalence of
\Cref{galpt}. 
\end{definition}

The automorphism group of such a functor naturally acquires the structure of a
\emph{profinite} group, and the forgetful functor above naturally lifts to a
functor
$\cov_X \to \finset_{\pi_1(X, \overline{x})}$, where $\finset_{\pi_1(X,
\overline{x})}$ denotes the category of finite sets equipped with a
continuous action of the profinite group 
$\pi_1(X,
\overline{x})$. 

Then, one has: 

\begin{theorem}[Grothendieck \cite{sga1}] 
\label{etalecorr}
The above functor $\cov_X \to \finset_{\pi_1(X, \overline{x})}$ is an
equivalence of categories. 
\end{theorem}

Grothendieck proved \Cref{etalecorr} by \emph{axiomatizing} the properties
that a category would have to satisfy in order to be of the form $\finset_{G}$
for $G$ a profinite group, and checking that any $\cov_X$ is of this form.
We review the axioms here. 

Recall that, in a category $\mathcal{C}$, a map $X \to Y$ is a \emph{strict
epimorphism} if the natural diagram
\[ X \times_Y X \rightrightarrows X \to Y,  \]
is a coequalizer. 

\begin{definition}[{Grothendieck \cite[Exp. V, sec. 4]{sga1}}] \label{defgalcat}
A \textbf{classical Galois category} is a category $\mathcal{C}$ equipped with a
\textbf{fiber functor} $F\colon  \mathcal{C} \to \finset$ satisfying the following
axioms: 
\begin{enumerate}
\item $\mathcal{C} $ admits finite limits  and $F$ commutes
with finite limits. 
\item  $\mathcal{C}$ admits finite coproducts and $F$ commutes with finite
coproducts. 
\item $\mathcal{C}$ admits quotients by finite group actions, and $F$ commutes
with those. 
\item $F$ is conservative and preserves strict epimorphisms.
\item Every map $ X \to Y$ in $\mathcal{C}$ admits a factorization $X \to Y'
\to Y$ where $X \to Y'$ is a strict epimorphism and where $Y' \to Y$ is a
monomorphism, which is in addition an inclusion of a summand. 
\end{enumerate}
\end{definition} 

Let $\mathcal{C}$ be a classical Galois category with fiber functor $F\colon  \mathcal{C} \to
\finset$. 
Grothendieck's Galois theory shows that $\mathcal{C}$ can be recovered as the
category of finite sets equipped with a continuous action of a certain
profinite group. 

\begin{definition} 
The \textbf{fundamental (or Galois) group} $\pi_1(\mathcal{C})$ of a
classical Galois category
$(\mathcal{C}, F)$ in the sense of Grothendieck is the automorphism group of the functor $F\colon  \mathcal{C} \to
\finset$. 
\end{definition} 

The fundamental group of $\mathcal{C}$ is naturally a profinite group, as a
(non-filtered)
inverse
limit of finite groups. Note that if $\mathcal{C}$ is a classical Galois
category with fiber functor $F$, if $\pi_1(\mathcal{C})$ is the 
Galois group, then the fiber functor $\mathcal{C} \to \finset$ naturally lifts
to 
\[ \mathcal{C} \to \finset_{\pi_1(\mathcal{C})},  \]
just as before. 
\begin{proposition}[{Grothendieck \cite[Exp. V, Theorem 4.1]{sga1}}] \label{galcorconn}
If $(\mathcal{C}, F)$ is a classical Galois category, then the functor $\mathcal{C}
\to \finset_{\pi_1(\mathcal{C})}$ as above is an equivalence of categories. 
\end{proposition} 

Given a connected scheme $X$ with a geometric point $\overline{x} \to X$, then
one can show that the category $\cov_X$ equipped with the above fiber functor
(of taking the preimage over $\overline{x}$ and taking connected components) is
a classical Galois category. 
The resulting fundamental group is a very useful invariant of a scheme, and for
varieties over an algebraically closed fields of
characteristic zero can be computed  by profinitely completing the topological
fundamental group (i.e., that of the $\mathbb{C}$-points). 
In particular, \Cref{etalecorr} is a special case of \Cref{galcorconn}.

\subsection{Definitions}

In this section, we will give an exposition of Galois theory  appropriate to the
nonconnected setting. 
Namely, to a type of category which we will simply call a {Galois
category}, we will attach a \emph{profinite groupoid:} that is, a pro-object
in the $(2, 1)$-category of groupoids with finitely many objects and finite
automorphism groups. The advantage of this approach, which relies heavily on
descent theory,  is that we will not start
by assuming the existence of a fiber functor, since we might not have one a
priori.

Axiomatic Galois theory in many forms has a voluminuous literature. 
The original treatment, of course, is \cite{sga1}, reviewed in the previous
subsection. A textbook reference for some of these ideas is \cite{borceux}.
In \cite{GR}, an approach to Galois theory (in the connected
case) for almost rings is
given that does not assume a priori the existence of a fiber functor.
The use of profinite groupoids in Galois theory is well-known (e.g.,
\cite{cjf, Magid}),  and the main
result below (\Cref{galequiv}) is presumably familiar to experts; we have included
a detailed exposition for lack of a precise reference and because our
$(2, 1)$-categorical approach may be of some interest.
Certain types of infinite Galois theory have been developed in the work of
Bhatt-Scholze \cite{BSpro} on the pro-{\'e}tale topology; we will not touch
on anything related to this here.  
Finally, we note that forthcoming work of Lurie will treat an embedding from
the $\infty$-category of profinite spaces into that of $\infty$-topoi,
which provides a vast generalization of
these ideas. 

We start by reviewing some category theory. 
\begin{definition} 
We say that an object $\emptyset$ in a category $\mathcal{C}$ is \textbf{empty}
if any map $x \to \emptyset$ is an isomorphism, and if $\emptyset$ is initial. 
\end{definition} 

For example, the empty set is an empty object in the category of sets. In the
\emph{opposite} to the category of commutative rings, the zero ring is empty. 

\begin{definition} 
Let $\mathcal{C}$ be a category admitting finite coproducts, such that the
initial object (i.e., the empty coproduct) is empty. We shall say that
$\mathcal{C}$ admits \textbf{disjoint coproducts}
if for any $x, y \in
\mathcal{C}$, the natural square
\[ \xymatrix{
\emptyset \ar[d] \ar[r] & x \ar[d] \\
y \ar[r] &  x \sqcup y
},\]
is cartesian. 
\end{definition} 

The category of sets (or more generally, any topos) admits disjoint coproducts.
The \emph{opposite} of the category of commutative rings also admits disjoint
coproducts. 

\begin{definition} 
Let $\mathcal{C}$ be a category admitting finite coproducts and finite limits. We will say that
\textbf{coproducts are distributive} if for every $y \to x$ in $\mathcal{C}$,
the pullback functor
\( \mathcal{C}_{/x} \to \mathcal{C}_{/y}  \)
commutes with finite coproducts. 
\end{definition} 

Similarly, the category of sets (or any topos) and the opposite to the category
of commutative rings satisfy this property and are basic examples to keep in
mind.

Suppose $\mathcal{C}$ admits disjoint and distributive coproducts. Then $\mathcal{C}$ acquires
the following very useful feature (familiar from \Cref{prode}). Given an object $x \simeq x_1 \sqcup x_2$ in
$\mathcal{C}$, then we have a natural equivalence
of categories,
\[ \mathcal{C}_{/x} \simeq \mathcal{C}_{/x_1} \times \mathcal{C}_{/x_2},  \]
which sends an object $y \to x$ of $\mathcal{C}_{/x}$ to the pair $(y \times_x
x_1 \to x_1, y \times_x x_2 \to x_2)$.

\begin{definition} 
Let $\mathcal{C}$ be a category admitting finite limits. Given a map $y \to x$
in $\mathcal{C}$, we have an adjunction
\begin{equation} \label{adjn} \mathcal{C}_{/y} \rightleftarrows
\mathcal{C}_{/x},  \end{equation}
where the left adjoint sends $y' \to y$ to the composite $y' \to y \to x$, and
the right adjoint takes the pullback along $y \to x$. 
We will say that $y \to x$ is an \textbf{effective descent morphism} if the above
adjunction is 
monadic. 

By the Beck-B\'enabou-Roubaud theorem that establishes the
connection between monads and descent \cite{BRB}, we can reformulate the
notion equivalently as follows. Form the bar construction
in $\mathcal{C}$,
\[ \dots \triplearrows y \times_x y \rightrightarrows y, \]
which is a simplicial object in $\mathcal{C}$ augmented over $x$. 
Applying the pullback functor everywhere, we get a cosimplicial category
\[ \mathcal{C}_{/y} \rightrightarrows \mathcal{C}_{/y \times_x y}
\triplearrows \dots ,  \]
receiving an augmentation from $\mathcal{C}_{/x}$. Then $y \to x$ is an
effective descent morphism if the functor
\[ \mathcal{C}_{/x} \to \mathrm{Tot}\left(  \mathcal{C}_{/y} \rightrightarrows \mathcal{C}_{/y \times_x y}
\triplearrows \dots \right),\]
is an equivalence of categories. 
If $\mathcal{C}$ is an $\infty$-category, we can make the same definition. 
\end{definition} 

We note that whether or not a map $y \to x$ is an effective descent morphism can be
checked using the Barr-Beck theorem applied to the adjunction \eqref{adjn}. 
Namely, the pullback along $y \to x$ needs to preserve reflexive coequalizers
which are split under pullback, and it needs to be conservative. 
%In particular, it follows that the \emph{base change} of an effective
%descent $y \to x$ along any map $x' \to x$ is still an effective descent morphism. 

Finally, we are ready to define a Galois category. 

\newcommand{\galcat}{\mathrm{GalCat}}
\begin{definition} \label{galax}
A \textbf{Galois category} is a category $\mathcal{C}$ such that: 
\begin{enumerate}
\item $\mathcal{C} $ admits finite limits and coproducts, and the initial
object $\emptyset$ is empty. 
\item Coproducts are disjoint and distributive in $\mathcal{C}$.
\item Given an object $x$ in $\mathcal{C}$, there is an effective
descent morphism $x' \twoheadrightarrow \ast$ (where $\ast$ is the terminal object) and a decomposition $x' = x'_1 \sqcup \dots \sqcup x'_n$
such that 
each map $x \times x'_i \to x'_i$ decomposes as the fold map $x \times x'_i \simeq
\bigsqcup_{S_i} x'_i \to x_i'$ for a finite set $S_i$. 
\end{enumerate}

The collection of all Galois categories and functors between them
(which are required to preserve coproducts, effective descent morphisms, and finite
limits) can be organized into a $(2, 1)$-category $\galcat$.
Given $\mathcal{C}, \mathcal{D} \in \galcat$, we will let
$\fung(\mathcal{C}, \mathcal{D})$ denote the groupoid of functors 
$\mathcal{C} \to \mathcal{D}$ in $\galcat$. 
\end{definition} 

In other words, we might say that an object $x \in \mathcal{C}$ is in \emph{elementary form}
if $x \simeq \bigsqcup_S \ast$ for some finite set $S$. More generally, 
if there exists a decomposition $\ast \simeq \ast_1 \sqcup \dots \sqcup \ast_n$, such
that, as an object of $\mathcal{C}_{} \simeq \prod_i \mathcal{C}_{/\ast_i}$,
each $y \times_{\ast} \ast_i \to \ast_i$ is in elementary form, we say that $y$ is in
\emph{mixed elementary form.} Then the \emph{defining} feature of a Galois category
is that, locally, every 
object is in mixed elementary form. 

Our first goal is to develop some of the basic properties of Galois categories. 
First, we need a relative version of the previous paragraph. 

\begin{definition} 
Let $\mathcal{C}$ be a category satisfying the first two conditions of
\Cref{galax} (which we note are preserved by passage to $\mathcal{C}_{/x}$
for any $x \in \mathcal{C}$). 
 We say that a map $f\colon  x \to y$ is
\emph{setlike} if there are finite sets $S, T$ such that $x \simeq \bigsqcup_S
\ast$, $y \simeq \bigsqcup_T \ast$ and the map $x \to y$ comes from a map of
finite sets $S \to T$. 
This implies that $x \in \mathcal{C}_{/y}$ is in mixed elementary form. 
\end{definition} 

For example, if $y = \ast$, then $x \to y$ is setlike if and only if $x$ is in
elementary form. 
Suppose $x, y$ are in elementary form, so that $x \simeq \bigsqcup_S \ast$ and
$y \simeq \bigsqcup_T \ast$. Then a map $x \to y$ is not necessarily setlike.
However, by the disjointness of coproducts, it follows that the map
$\bigsqcup_S \ast \to \bigsqcup_T \ast$ gives, for each $s \in S$, a decomposition of the terminal
object $\ast$ as a disjoint union of objects $\ast_{t}^{(s)}$ for each $t \in T$. It follows
that, refining these decompositions, there exists a decomposition $\ast \simeq \ast_1 \sqcup \dots \sqcup \ast_n$
such that the map $x \to y$ becomes setlike after pulling back along $\ast_i
\to \ast$. In particular, $x \to y$ is locally setlike. The same argument works
if $x, y$ are disjoint unions of summands of the terminal object. 

More generally, we have: 

\begin{proposition} 
\label{locsetlike}
Let $f\colon  x \to y$ be any map in the Galois category $\mathcal{C}$. Then there exists an effective
descent morphism $t \twoheadrightarrow \ast$ and a decomposition $t \simeq
\bigsqcup_{i=1}^n t_i$ such that the map $x \times t_i \to y \times t_i$ in
$\mathcal{C}_{/t_i}$ is setlike. 
More generally, given any finite set of maps $f_j \colon x_j \to y_j$ we can
find such a decomposition such that each $f_j \times t_i$ is setlike.
\end{proposition} 

\begin{proof} 
We can choose $t$ such that $(x \sqcup y) \times t$ is in mixed elementary
form: in particular, we have a decomposition $t \simeq t_1 \sqcup \dots \sqcup
t_n$ such that $(x \sqcup y) \times t_i$ is a disjoint union of copies of
$t_i$ in $\mathcal{C}_{/t_i}$. It follows that $x \times t_i \to t_i$ and $y
\times t_i \to t_i$ are objects in $\mathcal{C}_{/t_i}$ which are disjoint
union of summands of copies of the terminal object $t_i \in
\mathcal{C}_{/t_i}$. 
Using the previous discussion, it follows that we can refine the $t_i$ further
(by splitting into summands) to make $x \to y$ setlike on each summand. 
A similar argument would work for any finite set of morphisms in $\mathcal{C}$. 
\end{proof}

\begin{corollary} 
Let $\mathcal{C}$ be a Galois category and let $x \in \mathcal{C}$. Then
$\mathcal{C}_{/x}$ is a Galois category. 
\end{corollary} 
\begin{proof} 
The first two axioms are evident. For the third, fix a map $y \to x$ in
$\mathcal{C}$ (thus defining an object of $\mathcal{C}_{/x}$). By
\Cref{locsetlike}, we can find an
object $x' \in \mathcal{C}$ together with an effective descent morphism $x'
\twoheadrightarrow \ast$ such that $y \times x' \to x \times x'$ becomes, after
decomposing $x'$ into a disjoint union of summands, setlike in
$\mathcal{C}_{/x'}$. It follows that $y \times x' \to x' \times x$ is in mixed
elementary form as an object of $\mathcal{C}_{/x \times x'}$.  
\end{proof}

The notion of an effective descent morphism is a priori not so well-behaved, which
might be a cause for worry. Our next goal is to show that this is not the case. 

\begin{proposition} \label{colimdist}
A Galois category $\mathcal{C}$ admits finite colimits, which are
distributive over pullbacks. 
\end{proposition} 

\begin{proof} 
Let $K$ be a finite category; choose a map $p\colon  K \to
\mathcal{C}_{/x}$ for some object $x \in \mathcal{C}$. Since $\mathcal{C}_{/x}$ 
is itself a Galois category, we can replace $\mathcal{C}_{/x}$ with
$\mathcal{C}$ and show that if $y \in \mathcal{C}$ is arbitrary, then the
natural map
\begin{equation} \label{climeq} \varinjlim_K (y \times p(k) ) \to y \times \varinjlim_K p(k),  \end{equation}
is an equivalence, and in particular the colimits in question exist. 

There is one case in which the above would be automatic. 
Since $\mathcal{C}$ has finite coproducts,
  we
can define the product of a finite set with any object in $\mathcal{C}$. 
Suppose there exists a
diagram $\overline{p}\colon  K \to \finset$ and an object $u \in \mathcal{C}$ such
that $p = \overline{p} \times u$. For example, suppose that for every
morphism in $K$, the image in $\mathcal{C}$ is setlike; then this would happen. In this case, both sides of \eqref{climeq} are defined and are given by $y \times u \times \varinjlim_K
\overline{p}$, since finite coproducts distribute over pullbacks. 

We will say  that a diagram $p\colon  K \to \mathcal{C}$ is \emph{good} if it
arises from a $\overline{p}\colon  K \to \finset$ and an $u \in \mathcal{C}$; the
good case is thus straightforward. 
If we have a finite decomposition of the terminal object $\ast =
\bigsqcup_{i=1}^n \ast_i$ such that the restriction $p \times_{\ast} \ast_i$ is
good, then we say that $p$ is \emph{weakly good}. In this case, using
$\mathcal{C} \simeq \prod_{i=1}^n \mathcal{C}_{/\ast_i}$, we conclude that
\eqref{climeq} is defined and holds. 

We can reduce to the good (or weakly good) case via descent. 
There exists an effective descent morphism $x \to \ast$ such that $p \times x\colon  K \to
\mathcal{C}_{/x}$ is weakly good by \Cref{locsetlike}. 
Using
the expression $\mathcal{C} \simeq \mathrm{Tot}\left( \mathcal{C}_{/x \times \dots
\times x}\right)$, it follows that \eqref{climeq} must be true at each stage in
the totalization, and the respective colimits are compatible with the coface
and coboundary maps, so that it is (defined and) true in the totalization. 
\end{proof}

\begin{remark} 
In the above argument, we have tacitly used the following fact. Consider a
category $I$ and an $I$-indexed family of categories (or $\infty$-categories) $(\mathcal{C}_i)_{i \in
I}$. Consider a functor $p\colon  K \to \varprojlim_I \mathcal{C}_i$, where $K$ is
a fixed simplicial set. Suppose each
composite $K \stackrel{p}{\to} \varprojlim_I \mathcal{C}_i \to \mathcal{C}_{i_1}$
(for each $i_1 \in I$)  admits a colimit and suppose these colimits are
preserved by the various maps in $I$. Then $p$ admits a colimit compatible with
the colimits in each $\mathcal{C}_i$. 
\end{remark} 

\begin{corollary} 
The composite of two effective descent morphisms in a Galois category $\mathcal{C}$ is an
effective descent morphism.\footnote{Results on this question in more general
categories are contained in \cite{ST, RST}.} If $x \to y$ is any map in $\mathcal{C}$ and $y' \to y$
is an effective descent morphism, then $x \to y$ is an effective descent morphism if and
only if $x \times_y y' \to y'$ is one. 
\end{corollary} 
\begin{proof} 
Since 
(\Cref{colimdist})
a Galois category has finite colimits, which distribute over pull-backs,
it follows by the Barr-Beck theorem a map $x \to y$ is an effective descent morphism if and only if it
is conservative. This is preserved under compositions. The second statement is
proved similarly, since one only has to check conservativity locally. 
\end{proof} 
\begin{proposition} 
Given a map $f\colon  x \to y$ in the Galois category $\mathcal{C}$, the following are
equivalent: 
\begin{enumerate}
\item  $f$ is an effective descent morphism. 
\item $f$ is a strict epimorphism. 
\item For every $y' \to y$ with $y'$ nonempty, the pullback $x \times_y y'$ is
nonempty. 
\end{enumerate}
\end{proposition} 
\begin{proof} 
All three conditions can be checked locally. 
After base-change by an effective descent morphism $t \twoheadrightarrow \ast$ and
a decomposition $t \simeq t_1 \sqcup \dots \sqcup t_n$, we can assume that the
map $x \to y$ is setlike, thanks to \Cref{locsetlike}. In this case, the result
is obvious. 
\end{proof} 

We now discuss a few facts about functors between Galois categories. 
These will be useful when we analyze $\galcat$ as a 2-category in the next
section.

\begin{proposition} 
Let $\mathcal{C}, \mathcal{D}$ be Galois categories. A functor $\mathcal{C} \to
\mathcal{D}$ in $\galcat$ preserves finite colimits. 
\end{proposition} 
\begin{proof} 
This is proved as in \Cref{colimdist}: any functor preserves colimits of
\emph{good} diagrams (in the terminology of the proof of \Cref{colimdist}), and 
after making a base change one may reduce to this case. 
\end{proof} 

Next, we include a result that shows that $\galcat$ (or, rather, its opposite) behaves, to some extent,
like a Galois category itself; at least, it satisfies a version of the first
axiom of \Cref{galax}. 
\begin{definition} 
\label{connectedgalois}
A Galois category $\mathcal{C}$ is \textbf{connected} if there exists no
decomposition $\ast \simeq \ast_1 \sqcup \ast_2$ with $\ast_1, \ast_2$
nonempty and if additionally $\emptyset \not\simeq \ast$. 
\end{definition}

\begin{proposition} \label{prodgc}
Let $\mathcal{C}$ be a connected Galois category and let $\mathcal{C}_1,
\mathcal{C}_2$ be Galois categories. Then $\mathcal{C}_1 \times \mathcal{C}_2
\in \galcat$ and  we have an equivalence of groupoids
\[ \fung(\mathcal{C}_1 \times \mathcal{C}_2, \mathcal{C}) \simeq 
\fung( \mathcal{C}_1, \mathcal{C}) \sqcup \fung(\mathcal{C}_2, \mathcal{C}),
\]
\end{proposition} 
The above equivalence of groupoids is as follows. 
Given a functor $\mathcal{C}_i \to \mathcal{C}$ for $i \in
\left\{1,2\right\}$, we obtain a functor $\mathcal{C}_1 \times \mathcal{C}_2
\to \mathcal{C}$ by composing with the appropriate projection. 

\begin{proof} 
The assertion that $\mathcal{C}_1 \times \mathcal{C}_2 \in \galcat$ is easy to
check. 
Consider a functor $F\colon  \mathcal{C}_1 \times \mathcal{C}_2 \to \mathcal{C}$ in
$\galcat$. 
Note that every object 
$(x,y) \in \mathcal{C}_1 \times \mathcal{C}_2$ decomposes as the disjoint union
$(x, \emptyset) \sqcup (\emptyset, y)$. 
For example, in $\mathcal{C}_1 \times \mathcal{C}_2$, the terminal object $\ast
= (\ast, \ast)$
decomposes as the union $\ast_1 \sqcup \ast_2$ where $\ast_1$ is terminal in
$\mathcal{C}_1$ and empty in $\mathcal{C}_2$, and $\ast_2$ is terminal in
$\mathcal{C}_2$ and empty in $\mathcal{C}_1$. 
It follows that $F(\ast_1) = \emptyset$ or $F(\ast_2)  = \emptyset$ since
$\mathcal{C}$ is connected. If $F( \ast_1) = \emptyset$ and therefore
$F(\ast_2) = \ast$, then we have 
for $x \in \mathcal{C}_1, y \in \mathcal{C}_2$,
\[ F((x,y)) \simeq F((x,y) \times \ast_2) \simeq F( (\emptyset, y)),   \]
so that $F$ (canonically) factors through $\mathcal{C}_2$. 
The other case is analogous. 
\end{proof} 

More generally, let $\mathcal{C}$ be an arbitrary Galois category, and fix
$\mathcal{C}_1, \mathcal{C}_2 \in \galcat$. We find, by the same reasoning,
\begin{equation} \label{prodgcgen} \fung( \mathcal{C}_1 \times \mathcal{C}_2,
\mathcal{C}) \simeq \bigsqcup_{\ast = \ast_1
\sqcup \ast_2} \fung(\mathcal{C}_1, \mathcal{C}_{/\ast_1}) \times \fung(
\mathcal{C}_2, \mathcal{C}_{/\ast_2}), \end{equation}
where the disjoint union is taken over all decompositions of the terminal
object in $\mathcal{C}$. 

This concludes our preliminary discussion of the basic properties of Galois categories. 
In the next subsection, we will give another description of the $(2,
1)$-category of Galois categories. For now, though, we describe a basic method of extracting Galois
categories from other sources. 

\begin{definition} 
\label{galcontextdef}
A \textbf{Galois context} is an $\infty$-category $ \mathcal{C}$
satisfying the first two axioms of \Cref{galax} together with a class
$\mathcal{E} \subset \mathcal{C}$ of morphisms such that: 
\begin{enumerate}
\item $\mathcal{E} $ is closed under composition and base change and contains
every equivalence. 
\item  Every morphism in $\mathcal{E}$ is an effective descent morphism. 
\item Given a cartesian diagram
\[ \xymatrix{
x' \ar[d] \ar[r] & x \ar[d] \\
y' \ar[r] &  y
},\]
where $y' \to y \in \mathcal{E}$, then $x \to y $ belongs to $\mathcal{E}$ if and
only if $x' \to y'$ does. 
\item A map $x \to  y \simeq y_1 \sqcup y_2$ belongs to $\mathcal{E}$ if and only if $x
\times_{y} y_1 \to y_1$ and $x \times_y y_2 \to y_2$ belong to $\mathcal{E}$. 
\item For any object $x \in \mathcal{C}$ and any finite nonempty set $S$, the
fold map $\bigsqcup_S x \to x$ belongs to $\mathcal{E}$. 
\end{enumerate}

Given Galois contexts $(\mathcal{C}, \mathcal{E})$ and $(\mathcal{D},
\mathcal{E}')$, a \textbf{functor of Galois contexts} $F \colon(\mathcal{C},
\mathcal{E}) \to (\mathcal{D}, \mathcal{E}')$ will mean a functor of
$\infty$-categories $\mathcal{C} \to \mathcal{D}$ which respects finite limits
and coproducts and which carries morphisms in $\mathcal{E}$ to morphisms in
$\mathcal{E}'$. 
\end{definition}

\begin{definition}

Given a Galois context $(\mathcal{C}, \mathcal{E})$, we say that an object $x \in \mathcal{C}$ is \textbf{Galoisable} (or
$\mathcal{E}$-Galoisable) if there
exists a map $y \to \ast$ in $\mathcal{E}$ such that the pullback $x \times y
\to y$ is in mixed elementary form in $\mathcal{C}_{/y}$, as in the discussion
after \Cref{galax}.  In
other words, we require that there is a decomposition $y \simeq y_1 \sqcup
\dots \sqcup y_n$ such that each $x \times y_i \to y_i$ decomposes as  a finite
coproduct $\bigsqcup_{S_i} y_i \to y_i$. 
\end{definition} 

Given a category satisfying the first two axioms of \Cref{galax}, 
the following result enables us to extract a Galois category by considering
the Galoisable objects. 

\begin{proposition} \label{galcontext}
Let $(\mathcal{C}, \mathcal{E})$ be a Galois context. Then the collection of Galoisable objects in $\mathcal{C}$
(considered as a full subcategory of $\mathcal{C}$) forms a Galois category. 
\end{proposition} 
\begin{proof} 
Note first that the collection of Galoisable objects actually forms a
\emph{category} rather than an $\infty$-category: that is, the mapping space
between any two Galoisable objects is (homotopy) discrete. More precisely, if
$x \in \mathcal{C}$ is Galoisable and $x' \in \mathcal{C}$ is arbitrary, then 
we claim that $\hom_{\mathcal{C}}(x', x)$ is discrete. To see this, we choose
an effective descent morphism $u_1 \sqcup \dots \sqcup u_n \twoheadrightarrow \ast$
such that each map $u_i \times x \to x$ is in elementary form. Using the
expression $\mathcal{C} \simeq \mathrm{Tot}( \mathcal{C}_{/u_1 \times \dots
\times u_n})$, one reduces to the case where $x$ is a (disjoint) finite
coproduct of copies
of the terminal object $\ast$. In this case, $\hom_{\mathcal{C}}(x',
\bigsqcup_S \ast)$ 
is the \emph{set} of all $S$-labeled decompositions of $x'$ as direct sums of
subobjects, using the expression $\mathcal{C}_{/\bigsqcup_S \ast} \simeq
\prod_S \mathcal{C}_{/\ast} \simeq \prod_S \mathcal{C}$. 

Suppose $y  \in \mathcal{C}$ is a Galoisable object. We need to show that there is
a Galoisable object $t'$ and an $\mathcal{E}$-morphism $t' \twoheadrightarrow
\ast$ such that the
pullback $y \times t' \to t'$ is in mixed elementary form. By assumption, we
know that we can do this with \emph{some} object $t \in \mathcal{C}$ with an
$\mathcal{E}$-morphism $t \twoheadrightarrow \ast$, but we do not have any
control of $t$. 
We will find a \emph{Galoisable} choice of $t'$ by an inductive procedure.

Define the \emph{rank} of a
 Galoisable object $y \in \mathcal{C}$ as follows. If $y$ is mixed elementary, with respect to a 
decomposition $\ast \simeq \bigsqcup_{i=1}^n \ast_i$ (with the $\ast_i$
nonempty) and $y = \bigsqcup_{i=1}^n
\bigsqcup_{S_i} \ast_i$ for finite sets $S_i$, we define the rank to be $\sup_i
|S_i|$. 
In general, we make a base change in $\mathcal{C}$ along some 
$\mathcal{E}$-morphism $t \to \ast$ (by a not necessarily
Galoisable object) 
to reduce to this case. In other words, to define the rank of $y$, we choose an
$\mathcal{E}$-morphism $t \twoheadrightarrow \ast$ such that $y \times t \to t$
is in mixed elementary form in $\mathcal{C}_{/t}$, and then consider the rank
of that.

If the rank is zero, then $y = \emptyset$. 
We now use {induction} on the rank of $y$. 
First, we claim that there is a decomposition $\ast \simeq \ast_1 \sqcup
\ast_2$ such that $y \to \ast$ factors through an $\mathcal{E} $-morphism
$ y \to \ast_1$. (Meanwhile, $y \times_{\ast} \ast_2 = \emptyset$.) To see
this decomposition and claim, we can work locally on $\mathcal{C} \simeq \mathrm{Tot}( \mathcal{C}_{/t
\times \dots \times t})$ to reduce to the case in which $y$ is already in mixed
elementary form, for which the assertion  is evident. 
Thus we can reduce to the case where $y \to \ast$ is an $\mathcal{E}$-morphism. 

 Now consider the pullback $y \times y \to y$. 
This admits a section, so we have $y \times y \simeq y \sqcup c$ where $c$ is
another Galoisable object in $\mathcal{C}_{/y}$; to see that $c$ exists, one works locally using
$t$ to reduce to the mixed elementary case. However, by working locally again, one sees that the rank
of $c $ is one less than the rank of $y $. We can reduce 
the rank one by one, splitting off pieces, to get down to the case where $y = 
\emptyset$. 
\end{proof} 

In fact, the above argument shows that if $x \in \mathcal{C}$ is Galoisable,
there exists a Galoisable $y \in \mathcal{C}$ together with a morphism $y
\twoheadrightarrow \ast$ which belongs to $\mathcal{E}$ such that $x \times y
\to y$ is in mixed elementary form. 

\begin{corollary} 
Let $(\mathcal{C}, \mathcal{E})$ be a Galois context. Then a map $x \to y$
between Galoisable objects in $\mathcal{C}$ is an effective descent morphism in the
category of Galoisable objects if and only if it belongs to $\mathcal{E}$. 
Therefore, a functor of Galois contexts induces a functor of Galois categories. 
\end{corollary} 
\begin{proof} 
Working locally (because of the local nature of belonging to $\mathcal{E}$,
and in view of the preceding remark), we may assume the map $x \to y$ is setlike, in which case it
is evident. 
\end{proof}

\subsection{The Galois correspondence}
The Galois correspondence for groupoids gives an alternate description of the
$(2, 1)$-category $\galcat$. To see this, we describe the building
blocks in $\galcat$. 

\begin{example} \label{gpgal}
Let $G$ be a finite group. Then the category $\finset_G$ of finite $G$-sets is a Galois
category. Only the last axiom requires verification. In fact, given any finite $G$-set $T$, we have an effective
descent morphism $G \to \ast$ such that $T \times G$, as a $G$-set, is a disjoint
union of copies of $G$ (since it is free). 

This Galois category enjoys a convenient universal property, following
\cite{cjf}. 
\begin{definition} 
Let $\mathcal{C}$ be a Galois category and let $G$ be a finite group. A
\textbf{$G$-torsor} in $\mathcal{C}$ consists of an object $x \in \mathcal{C}$
with a $G$-action such that there exists an effective descent morphism $y
\twoheadrightarrow \ast$ such that $y \times x \in \mathcal{C}_{/y}$, as an
object with a $G$-action, is given by 
\[ y \times x \simeq \bigsqcup_G y,  \]
where $G$ acts on the latter by permuting the summands. 
For instance, $x$ could be $\bigsqcup_G \ast$. 
The collection of $G$-torsors forms a full subcategory
$\mathrm{Tors}_G(\mathcal{C})
\subset \mathrm{Fun}(BG, \mathcal{C})$. 
\end{definition} 

The Galois category $\finset_G$ has a natural example of a $G$-torsor: namely,
$G$ itself. The next result states that it is \emph{universal} with respect to
that property. 
\begin{proposition} 
\label{torsors}
If $\mathcal{C}$ is a Galois category, there is a natural equivalence between
$\fung(\finset_G, \mathcal{C})$ and the category
$\mathrm{Tors}_G(\mathcal{C})$ of $G$-torsors in
$\mathcal{C}$. 
\end{proposition} 

\begin{proof} 
Any functor of Galois categories preserves torsors for any finite group. 
In particular, given a functor $F\colon  \finset_G \to \mathcal{C}$ in $\galcat$, one gets a natural
choice of $G$-torsor in $\mathcal{C}$ by considering $F(G)$. Since everything
in $\finset_G$ is a colimit of copies of $G$, the choice of $F(G)$ determines
everything else. Together with the Yoneda lemma, this implies that the functor from $\fung(\finset_G,
\mathcal{C})$ to $G$-torsors is fully faithful. 

It remains to argue that, given a $G$-torsor in $\mathcal{C}$,
one can construct a corresponding functor $\finset_G \to \mathcal{C}$ in
$\galcat$. 
In other words, we want to show that the fully faithful functor
\[ \fung( \finset_G, \mathcal{C}) \to \mathrm{Tors}_G(\mathcal{C}) ,  \]
is essentially surjective. 
However, writing $\mathcal{C}$ as a
totalization of $\mathcal{C}_{/x \times \dots \times x}$, one may assume 
the $G$-torsor is trivial, in which case the claim is evident. 
\end{proof} 

\end{example}

More generally, we can build Galois categories from finite groupoids. 
This will be very important from a 2-categorical point of view. 

\begin{definition} 
We say that a groupoid $\G$ is \textbf{finite} if 
$\G$ has finitely many isomorphism classes of objects and, for each object $x
\in \G$, the automorphism group $\mathrm{Aut}_{\G}(x)$ is finite. 
The collection of all finite groupoids, functors, and natural transformations is naturally organized into a (2,
1)-category $\fgp$. 
\end{definition} 

In other words, a finite groupoid is a 1-truncated homotopy type such that
$\pi_0$ is finite, as is $\pi_1$ with any choice of basepoint. 

Given a finite groupoid $\G$, the category $\fun(\G, \finset)$ of functors from
$\G$ into the category of finite sets forms a Galois category. 
This is a generalization of \Cref{gpgal} and follows from it since the
categories $\fun(\G, \finset)$ are finite products of the Galois categories of
finite $G$-sets as $G$ varies over the automorphism groups. 
If we interpret $\G$ as a 1-truncated homotopy type, then this is
precisely the category of finite \emph{covering spaces} of $\G$, or of local
systems of finite sets on $\G$. 

It follows that we get a functor
of $(2, 1)$-categories
\[ \fgp^{\op} \to \galcat,  \]
sending a finite groupoid $\G$ to the associated functor category $\fun( \G,
\finset)$. Note that, for example,  a
natural transformation between functors of finite groupoids gives a natural
transformation at the level of Galois categories. 

In order to proceed further, we need a basic formal property of $\galcat$:
\begin{proposition}
The $(2, 1)$-category $\galcat$ admits filtered colimits, which are computed
at the level of the underlying categories: the colimit of a diagram of
Galois
categories and functors between them (which respect coproducts, finite limits,
and effective descent morphisms) in the $(2, 1)$-category of categories is again a
Galois category. 
\end{proposition}

\begin{proof} 
Let $F\colon  I \to \galcat$ be a filtered diagram of Galois categories. Our claim is
that the colimit $\varinjlim_I F$ is a Galois category and the natural
functors $F(j) \to \varinjlim_I F$ respect the relevant structure. We first
observe that $\varinjlim_I F$ has all finite limits and colimits, and
the functors $F(j) \to \varinjlim_I F$ respect those. This holds for any
filtered diagram of $\infty$-categories and functors preserving finite limits
(resp. colimits) as a formal consequence of the commutation of finite limits
and filtered colimits in the $\infty$-category of spaces. For example, every
finite diagram in $\varinjlim_I F$ factors through a finite stage. 
From this, the first two axioms of \Cref{galax} follow. 

Next, we want to claim that the functors $F(j) \to \varinjlim_I F$ respect
effective descent morphisms. Once we have shown this, the last axiom of \Cref{galax}
will follow, since we know it at each stage $F(j)$. 
In fact, let $x \to y$ be an effective descent morphism in $F(j)$. 
Then, we need to check that pull-back along $x \to y$ is conservative and
respects finite colimits in $\varinjlim_I F$; however, this follows since it
holds in each $F(j')$, since finite colimits and pullbacks are preserved under
$F(j') \to \varinjlim_I F$.

Finally, it follows from the previous paragraph that since every object in each $F(j)$ is locally in mixed
elementary form, with respect to effective descent morphisms in $F(j)$, the same is
true in $\varinjlim_I F$, since every object in the colimit comes from a
finite stage. 
\end{proof} 

It follows that we get a natural functor
\[ \mathrm{Pro}( \fgp)^{\op} \simeq \mathrm{Ind}( \fgp^{\op}) \to \galcat , \]
i.e., a \emph{contravariant} functor from the $(2, 1)$-category 
$ \mathrm{Pro}( \fgp)^{\op}$ into the $(2,
1)$-category of Galois categories. 
We give this a name. 

\begin{definition} 
A \textbf{profinite groupoid} is an object of $\pro( \fgp^{\op})$. 
\end{definition} 

We will describe some features of the $(2, 1)$-category of profinite groupoids
in the next subsection. 
In the meantime, the main result can now be stated as follows. 

\begin{theorem}[The Galois correspondence] \label{galequiv}
The functor $\mathrm{Pro}( \fgp)^{\op}  \to \galcat$ is an equivalence of
2-categories. 
\end{theorem} 

\begin{proof} 
We first check that the functor is fully faithful. To do this, first fix
\emph{finite} groupoids $\G, \G'$. We want to compare
the categories of functors $\fun( \G, \G')$ and $\fun^{\mathrm{Gal}}( \fun(\G', \finset), \fun(\G, \finset))$. 
In particular, we want to show that
\begin{equation} \label{agalmap}  \fun( \G, \G') \to  \fun^{\mathrm{Gal}}(
\fun(\G', \finset), \fun(\G, \finset)),\end{equation}
is an equivalence of groupoids. 
We can reduce to the case where $\mathscr{G}$ has one isomorphism class of
objects, since both sides of \eqref{agalmap} send coproducts in $\G$ to
products of groupoids. We can also reduce to the case where 
$\G'$ has a single connected component, since if $\G$ is connected, then both sides of
\eqref{agalmap} take coproducts in $\G'$ to coproducts. 
This is clear for the left-hand-side. For the right-hand-side, note that
coproducts in $\G'$ go over to \emph{products} in $\galcat$ for $\fun( \G',
\finset)$. Now use \Cref{prodgc} to describe the corepresented functor for a product
in $\galcat$. 
In order to show that \eqref{agalmap} is an equivalence when $\G, \G'$ are
finite groupoids, it thus suffices to work with \emph{groups}. 
We can do this extremely explicitly. 

In the case of \emph{finite groups}, given any two such $G, G'$, the groupoid
of maps between the associated groupoids has connected components given by the
conjugacy classes of homomorphisms $G \to G'$. Given any $f\colon  G \to G'$, the
automorphism group of $f$ is the centralizer of the image $f(G)$. 
To understand $\fung(\finset_{G'}, \finset_G)$, we can use \Cref{torsors}. We
need to describe the category of $G'$-torsors in $\finset_G$.
Any such gives a $G'$-torsor in $\finset$ by forgetting, so a $G'$-torsor in
$\finset_G$ yields in particular a copy of $G'$ with $G$ acting
$G'$-equivariantly (i.e., $G$ acts by right 	multiplication by various elements of $G'$). It follows that any torsor arises by considering a
homomorphism $\phi\colon  G \to G'$ and using that to equip the $G$-torsor $G' \in \finset_{G'}$ with 
the structure of a $G$-set. 
A natural transformation of functors, or a morphism of torsors, is
given by a conjugacy in $G'$ between two homomorphisms $G \to G'$: an
automorphism of the torsor comes from right multiplication by an element of
$G'$ which centralizes the image of $G \to G'$. 
This verifies full faithfulness for finite groupoids, i.e., that
\eqref{agalmap}  is an equivalence if $\mathscr{G}, \mathscr{G}'$ are finite.

Finally, we need to check that the full faithfulness holds for all
\emph{profinite} groupoids. That is a formal consequence of the fact that
$\fun( \G, \finset)$ is a \emph{compact} object in $\galcat$ for $\G$ a finite
groupoid. If $\G$ is connected, this is a consequence of the universal
property, \Cref{torsors}, since a torsor involves a finite amount of data. In general, the observation follows from the
connected case together with \Cref{prodgc} (and the remarks immediately
following, in particular \eqref{prodgcgen}). 

To complete the proof, we need to show that the functor is essentially surjective: that is, every
Galois category arises from a profinite groupoid. 
For this, we need another lemma on the formal structure of $\galcat$. 

\begin{lemma} 
\label{finlimgal}
$\galcat$ admits finite limits, which are preserved under $\galcat \to \cati$. 
\end{lemma} 
\begin{proof} 
Since $\galcat$ has a terminal object (the terminal category), it suffices to
show that given a diagram
\[ \xymatrix{
& \mathcal{C}' \ar[d] \\
\mathcal{C}'' \ar[r] &  \mathcal{C}
},\]
in $\galcat$, the category-theoretic fiber product is still a Galois category.
Of the axioms in \Cref{galax}, only the third needs checking. 
Note first that a map $x \to y$ in $\mathcal{C}' \times_{\mathcal{C}} \mathcal{C}''$ is
an effective descent morphism if it is one in $\mathcal{C}'$ and
$\mathcal{C}''$.  This follows from
the fact that the formation of overcategories and totalizations are compatible
with fiber products  of categories. 

Let $x$ be an object of the fiber product. We want to show that $x$ is locally
in mixed elementary form. As before, we can perform induction on the
\emph{rank} of $x$ (defined as the maximum of the ranks of the images in
$\mathcal{C}', \mathcal{C}''$). 
The natural map $x \to \ast$ has the property that
$\ast \simeq \ast_1 \sqcup \ast_2$ where $x \to \ast$ factors through an
effective descent morphism $x \twoheadrightarrow \ast_1$. In fact, we can construct
these on $\mathcal{C}', \mathcal{C}''$ and they have to match up on
$\mathcal{C}$. So, we can assume that $x \to \ast$ is an effective
descent morphism. Now after base-change along $x \to \ast$, we can find a section
of $x \times x \to x$ and thus obtain a splitting of $x \times x $ (since we
can in $\mathcal{C}', \mathcal{C}''$). Using induction on the rank, we can
conclude as before. 
\end{proof} 

\begin{remark} 
\label{limgalcat}
The same logic shows that $\galcat$ admits arbitrary limits, although they are
no longer preserved under the forgetful functor $\galcat \to \cati$; one has to
take the subcategory of the categorical limit consisting of objects whose rank
is bounded. 
\end{remark} 

Let $\mathcal{C}$ be any Galois category, which we want to show lies in the
image of the fully faithful functor $\pro( \fgp)^{\mathrm{\op}} \to \galcat$. In
order to do this, we will write $\mathcal{C}$ as a filtered colimit of
subcategories which do belong to the image.

Let $\mathcal{C}$ be a Galois category. Then, if $\mathcal{C}$ is not the
terminal category (i.e., if the map $\emptyset \to \ast$ in $\mathcal{C}$ is
not an isomorphism),
there is a faithful functor $\finset \to \mathcal{C}$ which sends a finite set
$S$ to $\bigsqcup_S \ast$. This is a functor in $\galcat$ and defines, for
every nonempty Galois category $\mathcal{C}$, a (non-full) Galois subcategory
$\mathcal{C}_{\mathrm{triv}}$. 
In other words, we take the objects which are in elementary form and the 
setlike maps between them. 
More generally, if $\ast$ decomposes as $\ast =
\ast_1 \sqcup \dots \sqcup \ast_n$, we can define a subcategory
$\mathcal{C}_{\mathrm{triv}}^{\mathrm{loc}} \subset \mathcal{C}$ by writing
$\mathcal{C} \simeq \prod_{i=1}^n \mathcal{C}_{/\ast_i}$ and taking the
subcategory $\mathcal{C}^{\mathrm{loc}}_{\mathrm{triv}} = \prod_{i=1}^n
({\mathcal{C}_{/\ast_i}})_{\mathrm{triv}}$.

Let $y\twoheadrightarrow \ast$ be an effective
descent morphism and let $y \simeq y_1 \sqcup \dots \sqcup y_n$ be a decomposition
of $y$. We define a map $f\colon  x \to x'$ in $\mathcal{C}$ to be \emph{split} with respect to $y$
and the above decomposition if $f \times y_i \in \mathcal{C}_{/y_i}$ is setlike for each $i = 1, 2,
\dots, n$. 
Via descent theory, we can write this subcategory as 
\[ \mathcal{C}' = \mathrm{Tot}\left( \prod_{i=1}^n \mathcal{C}^{\mathrm{triv}}_{/y_i}
\rightrightarrows \prod_{i, j = 1}^n \mathcal{C}^{\mathrm{triv}}_{y_i \times
y_j} \triplearrows \dots  \right).
\]
In other words, this subcategory of $\mathcal{C}$ arises as an inverse limit
(indexed by a cosimplicial diagram) of products of copies of $\finset$. Any
such is the category of finite covers of a finite CW complex (presented by
3-skeleton of the dual
simplicial set\footnote{We recall that the 3-truncation of the above
totalization is sufficient to compute the totalization.}) and is thus in the image of $\pro( \fgp)^{\op}$. However, $\mathcal{C}$ is the filtered union over all such
subcategories as we consider effective descent morphisms $y_1 \sqcup \dots y_n
\twoheadrightarrow \ast$ with the $\left\{y_i\right\}$ varying. 
It follows that $\mathcal{C}$ is the filtered colimit in $\galcat$ of objects
which belong to the image of $\pro( \fgp)^{\op} \to \galcat$, and is therefore in the
image of $\pro( \fgp)^{\op}$ itself. 
\end{proof} 

\Cref{galequiv} enables us to make the following fundamental definition. 
\begin{definition} 
Given a Galois category $\mathcal{C}$, we define the \textbf{fundamental
groupoid} or \textbf{Galois groupoid} $\pi_{\leq 1}\mathcal{C}$ of $\mathcal{C}$ as the associated
profinite groupoid under the correspondence of \Cref{galequiv}. 
\end{definition} 

We next use the Galois correspondence to obtain a few technical results on
torsors. 

\begin{corollary} 
\label{torsorenough}
The Galois categories $\finset_G$
jointly detect equivalences: given a functor in $\galcat$, $F\colon  \mathcal{C} \to
\mathcal{D}$, if $F$ induces an equivalence on the categories of $G$-torsors
for each finite group $G$, then $F$ is an equivalence. In other words, if 
the map
\begin{equation} \label{tmap} \mathrm{Tors}_G(\mathcal{C})\to
\mathrm{Tors}_G(\mathcal{D}) \end{equation}
is an equivalence of groupoids for each $G$, then $F$ is an equivalence. 
\end{corollary} 
\begin{proof} 
By \eqref{prodgcgen}, it follows that if 
\eqref{tmap} is always an equivalence, then the map
\[ \hom_{\galcat}( \fun( \mathscr{G}, \finset), \mathcal{C}) \to \hom_{\galcat}( \fun(
\mathscr{G}, \finset), \mathcal{D}),  \]
is an equivalence for each finite \emph{groupoid} $\mathscr{G}$. Dualizing,
and using the Galois correspondence, we
find that the map $\pi_{\leq 1} \mathcal{D} \to \pi_{\leq 1} \mathcal{C}$ of
profinite groupoids has the property that 
\[ \hom_{\pro( \fgp)}( \pi_{\leq 1} \mathcal{C}, \mathscr{G}) \to  
\hom_{\pro( \fgp)}( \pi_{\leq 1} \mathcal{D}, \mathscr{G})
\]
is always an equivalence, for every finite groupoid $\mathscr{G}$. 
However, we know that finite groupoids generate 
$\pro( \fgp)$ under cofiltered limits, so we are done. 
\end{proof} 

\begin{corollary} 
\label{splitbytorsor}
Let $\mathcal{C}$ be a Galois category and $x \in \mathcal{C}$ be an object.
Then there exists a $G$-torsor $y$ in $\mathcal{C}$ for some finite group
$G$ such that $x \times y \to y$ is in mixed elementary form in
$\mathcal{C}_{/y}$. 
\end{corollary} 
\begin{proof} 
We can reduce to the case where $\mathcal{C} = \fun( \mathscr{G}, \finset)$ for
$\mathscr{G}$ a finite groupoid, since $\mathcal{C}$ is a filtered colimit of
such. Let $\mathscr{G}$ have objects $x_1, \dots , x_n$ up to isomorphism with
automorphism groups $G_1, \dots, G_n$. Then, there is a natural $G_1 \times
\dots \times G_n$-torsor $y$ on $\mathscr{G} \simeq \bigsqcup_{i=1}^n BG_i$ 
(which on the $i$th summand is the universal cover times the trivial $\prod_{j
\neq i} G_j$-torsor) such that any object $x$ in $\mathcal{C}$ 
has the property that $y \times x$ is in mixed elementary form. 
\end{proof}

\subsection{Profinite groupoids}

Given \Cref{galequiv}, it behooves us to discuss the 2-category $\pro(\fgp)$ of
profinite groupoids in more detail. We begin by studying connected components.

We have a natural functor $\pi_0 \colon  \fgp \to \finset$ sending a groupoid to its
set of isomorphism classes of objects. Therefore, we get a functor
$\pi_0 \colon  \pro( \fgp) \to \pro( \finset)$ which 
is uniquely determined by the properties that it recovers the old $\pi_0$ for
finite groupoids and that it commutes with cofiltered limits. 
Recall that the
category $\pro(\finset)$ is the category of compact, Hausdorff, and totally
disconnected topological spaces, under the 
realization functor which sends a profinite set to its inverse limit (in the
category of sets) with the inverse limit topology. 
It follows that the collection of ``connected
components'' of a  profinite groupoid is one of these.

\begin{remark} 
Note that $\pi_0\colon  \fgp \to \finset$ does not commute with finite inverse
limits, so that its right Kan extension to 
$\pro(\fgp)$ does not. 
While the reader might object that there should be a $\lim^1$ obstruction to
the commutation of $\pi_0$ and cofiltered limits (of towers, say), we
remark that $\lim^1$-terms always vanish for towers of finite groups. 
\end{remark}

In practice, we will mostly be concerned with 
the case where the (profinite) set $\pi_0$ of connected components is a
singleton. 

\begin{definition} 
We say that a profinite groupoid is \textbf{connected} if its $\pi_0$ is a
singleton. 
The collection of connected profinite groupoids spans a full subcategory $\pro(
\fgp)^{\geq 0} \subset \pro( \fgp)$. 
\end{definition}

In general, it will thus be helpful to have an explicit description of this profinite set. 
Recall that there is an algebraic description of $\pro(\finset)$ given by
\emph{Stone duality.}
Given a Boolean algebra $B$, the spectrum $\spec B$ of prime ideals (with its Zariski topology) is
an example of a profinite set, i.e., it is compact, Hausdorff, and totally
disconnected. 
Recall now: 
\begin{theorem}[Stone duality] The functor $B \mapsto \spec B$ establishes an
anti-equivalence $\bool^{\op} \simeq \pro( \finset)$. 
\end{theorem}

For a textbook reference on Stone duality, see \cite{Johnstone}. 
The Galois correspondence in the form of \Cref{galequiv} can be thought of as a
mildly categorified version of Stone duality. 
In particular, we can use Stone duality to describe 
$\pi_0$ of a profinite groupoid. 

\begin{proposition} \label{connectedigal}
Let $\mathcal{C}$ be a Galois category. Then $\pi_0( \pi_{\leq 1} \mathcal{C})$
corresponds, under Stone duality, to the Boolean algebra of subobjects $x
\subset \ast$. 
\end{proposition} 

Let $\mathcal{C}$ be a Galois category. Given two subobjects $x, y \subset
\ast$  of the terminal object, 
we define their product to be the categorical product $x \times y$. Their sum
is the minimal subobject of $\ast$ containing both $x,y$: in other words, the
image of $x \sqcup y \to \ast$. 
By working locally, it follows that this actually defines a Boolean algebra. 

\begin{proof} 
In fact, if $\mathcal{C}$ is a Galois category corresponding to a \emph{finite}
groupoid, the result is evident. Since the construction above sends filtered
colimits of Galois categories to  filtered colimits of Boolean algebras, we can
deduce it for any Galois category in view of \Cref{galequiv}. 
\end{proof}

In practice, the Galois categories that we will be considering will be
connected (in the sense of \Cref{connectedgalois}). By \Cref{connectedigal}, it
follows that a Galois category $\mathcal{C}$ is
connected if and only if $\pi_{\leq 1} \mathcal{C}$ is connected as a
profinite groupoid. 
In our setting, this will amount to the condition that certain commutative
rings are free from idempotents. 
With this in mind, we turn our attention to the \emph{connected} case. 
Here we will be able to obtain a very strong connection with the (somewhat more
concrete) theory of \emph{profinite groups.}

The $2$-category $\pro( \fgp)$ has a terminal object $\ast$, the 
contractible profinite groupoid. Under the Galois correspondence, this
corresponds to the category $\finset$ of finite sets. 

\begin{definition} 
A \textbf{pointed profinite groupoid} is a profinite groupoid
$\mathscr{G}$ together with a 
map $\ast \to \mathscr{G}$ in $\pro( \fgp)$. The collection of pointed profinite groupoids forms
a $(2, 1)$-category, the undercategory $\pro( \fgp)_{\ast/}$. 
\end{definition} 

\newcommand{\fg}{\mathrm{FinGp}}

For example, let $G$ be a profinite group, so that $G$ is canonically a pro-object in finite
groups. Applying the classifying space functor to this system, we obtain a
\emph{pointed} profinite groupoid $BG \in \pro( \fgp)$ as the formal inverse
limit of the finite groupoids $ B(G/U)$ as $U \subset G$ ranges over the open
normal subgroups, since each $B(G/U)$ is pointed. By construction, the associated Galois category is
$\varinjlim_{U \subset G} \finset_{G/U}$, or equivalently, the category of
finite sets equipped with  a \emph{continuous} $G$-action (i.e., an action which
factors through $G/U$ for $U$ an open normal subgroup). 
We thus obtain a functor
\[ B\colon  \pro( \fg) \to \pro( \fgp)_{\ast/},  \]
where $\fg$ is the category of finite groups and $\pro(\fg)$ is the category of
profinite groups. 
Observe that this functor is \emph{fully faithful}, since the analogous functor
$B\colon  \fg \to (\fgp)_{\ast/}$ is fully faithful, and each $BG$ for $G$
finite defines a cocompact object of $\pro( \fgp)_{\ast/}$. 

There is a rough inverse to this construction, given by taking the
``fundamental group.''
In
general, if $\mathcal{C}$ is an $\infty$-category with finite limits, and $C
\in \mathcal{C}$ is an object, then the natural functor
\[ \pro( \mathcal{C}_{C/}) \to \pro(\mathcal{C})_{C/}  \]
is an equivalence of $\infty$-categories. 
In the case of $\mathcal{C}  = \fgp$, we know that there is a functor
\begin{equation} \pi_1\colon  (\fgp)_{\ast/} \to \fg , \label{p1}\end{equation}
to the category $\fg$ of finite groups, given by the usual fundamental group of
a pointed space, or more categorically as the automorphism group of the
distinguished point. 
As above, let $\pro( \fg)$ be the category of profinite groups and continuous
homomorphisms. 

\begin{definition}
We define a functor $\pi_1\colon  \pro( \fgp)_{\ast/} \to \mathrm{Pro}( \fg)$ from
the $2$-category of pointed profinite groupoids to the category of profinite
groups given by right Kan extension of \eqref{p1}, so that $\pi_1$ agrees
with the old $\pi_1$ on pointed finite groupoids and commutes with filtered 
limits. 
\end{definition}

Given a pointed finite groupoid $\mathscr{G}$, we have a natural map
\begin{equation} \label{nattransp} B \pi_1 ( \mathscr{G}) \to \mathscr{G},  \end{equation}
and by general formalism, we have a natural transformation of the form
\eqref{nattransp} on $\pro( \fgp)_{\ast/}$.

\begin{proposition}
\label{gppt}
Given an object $\mathscr{G} \in \pro( \fgp)_{\ast/}$, the following
are equivalent: 
\begin{enumerate}
\item $\mathscr{G}$ is connected, i.e., $\pi_0 \mathscr{G}$ is a singleton.  
\item The map $B \pi_1 \mathscr{G} \to \mathscr{G}$ is an equivalence in
$\pro( \fgp)_{\ast/}$. 
\end{enumerate}
In particular, the functor $B\colon  \pro( \fg) \to \pro( \fgp)_{\ast/}$ is fully
faithful with image consisting of the pointed connected profinite groupoids. 
\end{proposition}
\begin{proof} 
The second statement clearly implies the first: any $B G$ for $G$ a profinite
group is connected, as the inverse limit of connected profinite groupoids. 
We have also seen that the functor $B$ is fully faithful, since it is fully
faithful on finite groups. 
It remains to show that if $\mathscr{G}$ is a pointed, connected profinite
groupoid, then the map $B \pi_1 \mathscr{G} \to \mathscr{G}$ is an equivalence. 

For this, we write $\mathscr{G}$ as a cofiltered limit $\varprojlim_I
\mathscr{G}_i$, where $I$ is a filtered partially ordered set indexing the $\mathscr{G}_i$ and each $\mathscr{G}_i$ is a 
pointed finite groupoid. We know that $\mathscr{G}$ is connected, though each
$\mathscr{G}_i$ need not be. However, we obtain a new inverse system
$\left\{B \pi_1 \mathscr{G}_i\right\}$ equipped with a map to the inverse system
$\{\mathscr{G}_i\}$ and we want to show that the two inverse systems are
pro-isomorphic. 
We need to produce an inverse map of pro-systems $\left\{\mathscr{G}_i\right\} \to
\left\{B \pi_1 \mathscr{G}_i\right\}$. 
For this, we need to produce for each $i \in I$
an element $j \geq i$ and a map
\[ \mathscr{G}_j \to B \pi_1 \mathscr{G}_i.  \]
These should define an element of $\varprojlim_i \varinjlim_j \hom(
\mathscr{G}_j,  B \pi_1 \mathscr{G}_i)$.
In order to do this, we simply note that there exists $j \geq i$ such that the
map $\mathscr{G}_j \to \mathscr{G}_i$ lands inside the connected component $B
\pi_1 \mathscr{G}_i$ of $\mathscr{G}_i$ at the basepoint, because otherwise the
pro-system would not be connected as a filtered inverse limit of nonempty finite
sets is nonempty. 
One checks easily that the two maps of pro-systems define an isomorphism
between $\left\{\mathscr{G}_i\right\}$ and $\left\{B
\pi_1\mathscr{G}_i\right\}$. 
\end{proof}

Let $\mathscr{G} \in \pro( \fgp)$ be a 
\emph{connected} profinite groupoid. 
This means that the space of maps $\ast \to \mathscr{G}$ in $\pro( \fgp)$ is
connected, i.e., there is only one such map up to homotopy. (This is not
entirely immediate, but will be a special case of \Cref{mappingspaces} below.) Once we choose a
map, we point $\mathscr{G}$ and then the data is essentially equivalent to
that of a
profinite group
in view of \Cref{gppt}. If we do not point $\mathscr{G}$, then what we have is
essentially a profinite group ``up to conjugacy.''

\begin{proposition} 
\label{mappingspaces}
Let $G, G'$ be  profinite groups. Then the space $\hom_{\pro( \fgp)}( BG, BG')$
is given as follows: 
\begin{enumerate}
\item The connected components are in one-to-one correspondence with conjugacy
classes of continuous homomorphisms  $f\colon  G \to G'$. 
\item The group of automorphisms of a given continuous homomorphism $f\colon  G \to
G'$ is given by the centralizer in $G'$ of the image of $f$. 
\end{enumerate}
\end{proposition} 
In other words, if we restrict our attention to the subcategory $\pro(
\fgp)^{\geq 0} \subset \pro(
\fgp)$ consisting of \emph{connected} profinite groupoids, then it has a simple
explicit description as a 2-category where the objects are the profinite
groups, maps are continuous homomorphisms, and 2-morphisms are conjugations. 
\begin{proof} 
This assertion is well-known when $G, G'$ are finite groups: maps between $BG$
and $BG'$ in $\fgp$ are as above. 
The general case follows by passage to cofiltered limits. 
Let $G = \varprojlim_U G/U, G' = \varprojlim_V G'/V$ where $U$ (resp. $V$)
ranges over the open normal subgroups of $G$ (resp. $G'$). 
In this case, we have
\[ \hom_{\pro( \fgp)}(BG, BG') \simeq \varprojlim_V \varinjlim_U \hom_{\fgp}(
B(G/U), B(G/V)),  \]
and passing to the limit, we can conclude the result for $G, G'$ profinite, if
we observe that the set of conjugacy classes of continuous homomorphisms $G \to G'$
is the inverse limit of the sets of conjugacy classes of continuous
homomorphisms $G \to G'/V$ as $V \subset G$ ranges over open normal subgroups. 
The assertion about automorphisms, or conjugacies, is easier.

To see this in turn, 
suppose given continuous homomorphisms $\phi_1, \phi_2\colon  G \to G'$ such that,
for every continuous map $\psi\colon  G' \to G''$ where $G''$ is finite, the
composites $\psi \circ \phi_1, \psi \circ \phi_2$ are conjugate. We claim that
$\phi_1, \phi_2$ are conjugate. The collection of all surjections $\psi \colon  G'
\to G''$ with $G''$ finite forms a filtered system, and for each $\psi$, we consider
the (finite) set $F_{\psi} \subset G''$ of $x \in G''$ such that $\psi
\circ \phi_2 = x ( \psi \circ \phi_1 ) x^{-1}$. Since by hypothesis each
$F_{\psi}$ is nonempty, it follows that the inverse limit is nonempty, so that
$\phi_1, \phi_2$ are actually conjugate as homomorphisms $G \to G'$. 
Conversely, suppose given for each $\psi\colon  G' \to G''$ with $G''$ finite a \emph{conjugacy class}
of continuous maps $\phi_{\psi}\colon  G \to G''$, and suppose these are compatible with one
another; we want to claim that there exists a conjugacy class of continuous
homomorphisms $\phi\colon  G \to G'$ that lifts all  the $\phi_{\psi}$. For this, we
again consider the \emph{finite} nonempty sets $G_{\psi}$ of all continuous
homomorphisms $G \to G''$ in the conjugacy class of $\phi_{\psi}$, and observe
the inverse limit of these is nonempty. Any point in the inverse limit gives a
continuous homomorphism $G \to G''$ with the desired property. 
\end{proof} 

\section{The Galois group and first computations}

Let $(\mathcal{C}, \otimes, \mathbf{1})$ a stable homotopy theory. In this
section, we will make the main definition of this paper, and describe two candidates for the \emph{Galois group} (or, in
general, groupoid) of
$\mathcal{C}$. Using the descent theory described in 
\Cref{sec:descent}, we will
define a category of \emph{finite covers} in the
$\infty$-category $\clg(\mathcal{C})$
of commutative algebra objects in $\mathcal{C}$. Finite covers will be
those commutative algebra objects which ``locally'' look like direct factors of
products of copies of the unit. There are two possible definitions of
``locally,'' which lead to slightly different Galois groups. 
We will show that these $\infty$-categories of finite covers are actually
Galois categories in the sense of \Cref{galax}. Applying the Galois correspondence,
we will obtain a profinite groupoid. 

The rest of this paper will be devoted to describing the Galois group in certain
special instances. 
In this section, we will begin that process by showing that the Galois group is
entirely algebraic in two particular instances: connective $\e{\infty}$-rings
and even periodic $\e{\infty}$-rings with regular $\pi_0$. In 
either of these cases, one has various algebraic tricks to study modules via
their homotopy groups. 
The
associated $\infty$-categories of modules turn out to be extremely useful
building blocks for a much wider range of stable homotopy theories. 

\subsection{Two definitions of the Galois group}

Let $(\mathcal{C}, \otimes, \mathbf{1})$ be a stable homotopy theory, as before. 
We will describe two possible analogs of ``finite \'etaleness'' appropriate to
the categorical setting. 

\newcommand{\clgf}{\clg^{\mathrm{cov}}}
\newcommand{\clgw}{\clg^{\mathrm{w.cov}}}

\begin{definition} 
\label{deffinitecover}
An object $A \in \clg(\mathcal{C})$ is a \textbf{finite cover}  if there exists
an $A' \in \clg(\mathcal{C})$ such that: 
\begin{enumerate}
\item  $A'$ admits descent,  in the sense of
\Cref{admitd}. 
\item   $A \otimes A' \in \clg( \md_{\mathcal{C}}(A'))$ is of
the form $\prod_{i=1}^n A'[e_i^{-1}]$, where for each $i$, $e_i $ is an
idempotent in $A'$. 
\end{enumerate}
The finite covers span a subcategory $\clgf(\mathcal{C}) \subset \clg(
\mathcal{C})$. 
\end{definition} 

\begin{definition} 
\label{weakfincover}
An object $A \in \clg(\mathcal{C})$ is a \textbf{weak finite cover}  if there exists
an $A' \in \clg(\mathcal{C})$ such that: 
\begin{enumerate}
\item The functor $\otimes A'\colon  \mathcal{C} \to \mathcal{C}$ commutes with all homotopy
limits. 
\item The functor $\otimes A'$ is conservative. 
\item  
$A \otimes A' \in \clg( \md_{\mathcal{C}}(A'))$ is of
the form $\prod_{i=1}^n A'[e_i^{-1}]$, where for each $i$, $e_i $ is an
idempotent in $A'$. 
\end{enumerate}
The weak finite covers span a subcategory $\clgw(\mathcal{C}) \subset
\clg(\mathcal{C})$. 
\end{definition}

Our goal is to show that both of these definitions give rise to Galois
categories in the sense of the previous section, which we will do using the
general machine of \Cref{galcontext}. Observe first that
$\clg(\mathcal{C})^{\op}$ satisfies the first two conditions of \Cref{galax}.

\begin{lemma} 
Given $\mathcal{C}$ as above, consider the $\infty$-category
$\clg(\mathcal{C})^{\op}$ and the collection of morphisms $\mathcal{E}$ given
by the maps $A \to B$ which admit descent. Then $(\clg(\mathcal{C})^{\op},
\mathcal{E})$ is a Galois context in the sense of \Cref{galcontextdef}. 
\end{lemma} 
\begin{proof} 
The composite of two descendable morphisms is descendable by \Cref{permanence},
descendable morphisms are effective descent morphisms by \Cref{easydesc}, and the locality
of descendability (i.e., the third condition of \Cref{galcontextdef}) follows from the second part of \Cref{permanence}. 
The remaining conditions are straightforward. 
\end{proof} 

\begin{lemma} 
Given $\mathcal{C}$ as above, consider the $\infty$-category
$\clg(\mathcal{C})^{\op}$ and the collection of morphisms $\mathcal{E}$ given
by the maps $A \to B$ such that the functor $\otimes_A B \colon  \md_{\mathcal{C}} (A) \to
\md_{\mathcal{C}}(B)$ commutes with limits and is conservative. 
Then $(\clg(\mathcal{C})^{\op},
\mathcal{E})$ is a Galois context in the sense of \Cref{galcontextdef}. 
\end{lemma} 

\begin{proof} 
It is easy to see that $\mathcal{E}$ satisfies the first axiom of
\Cref{galcontextdef}, and we can apply Barr-Beck-Lurie to see comonadicity of
$\otimes_A B$ (i.e., the second axiom). The fourth and fifth axioms are straightforward. 

Finally, suppose $A \to B$ is a morphism in $\clg(\mathcal{C})$ and $A \to A'$ belongs
to $\mathcal{E}$, i.e., tensoring $\otimes_A A'$ commutes with limits and is
conservative. Suppose $A' \to B' \stackrel{\mathrm{def}}{=}A' \otimes_A B$ has the same property. Then we
want to claim that $A \to B$ belongs to $\mathcal{E}$. 

First, observe that $\otimes_A B$ is conservative. If $M \in
\md_{\mathcal{C}}(A)$ is such that $M \otimes_A B \simeq 0$, then $(M \otimes_A
A') \otimes_{A'} B'$ is zero, so that $M \otimes_A A'$ is zero as
$A' \to B'$ belongs to $\mathcal{E}$, and thus $M = 0$. 
Finally, we need to check the claim about $\otimes_A B $ commuting with
limits. In other words, given $\left\{M_i\right\} \in \md_{\mathcal{C}}(A)$, we need to
show that the natural map
\[ B \otimes_A \prod M_i \to \prod (M_i \otimes_A B)  \]
is an equivalence. We can do this after tensoring with $A'$, so we need to see that
\[  A' \otimes_A B  \otimes_A \prod M_i  \to A' \otimes_A \prod (M_i \otimes_A B) \]
is an equivalence. However, since tensoring with $A'$ commutes with limits, this map is 
\[ B' \otimes_{A'} \prod( M_i \otimes_A A')  \to \prod (M_i \otimes_A A')
\otimes_{A'} B', \]
which is an equivalence since $\otimes_{A'}B'$ commutes with limits by
assumption. 
\end{proof}

The basic result of this section is the following.

\begin{theorem} 
\label{basicgal}
Given $\mathcal{C}$, $\clgf(\mathcal{C})^{\op}$ and $\clgw(\mathcal{C})^{\op}$ are
Galois
categories, with $\clgf(\mathcal{C}) \subset \clgw(\mathcal{C})$. 
If $\mathbf{1} \in \mathcal{C}$ is compact, then the two are the same. 
\end{theorem}

\begin{proof} 
This follows from \Cref{galcontext} if we take $\clg(\mathcal{C})^{\op}$ as our input
$\infty$-category.
As we checked above, we have two candidates for $\mathcal{E}$, both of which
yield Galois contexts. The Galoisable objects yield either the finite covers or
the weak finite covers.

Next, we need to note that a finite cover is actually a weak finite cover.
Note first that either a finite cover or a weak finite cover is dualizable,
since dualizability can be checked locally in a limit diagram of
symmetric monoidal $\infty$-categories. However, the 
argument of \Cref{galcontext} (or the following corollary) shows that, given a finite cover $A \in
\clg(\mathcal{C})$, we can choose the descendable $A' \in \clg(\mathcal{C})$
such that $A \otimes A'$ is in mixed elementary form so that $A'$ itself is a
finite cover: in particular, so that $A'$ is dualizable. Therefore, we can
choose $A'$ so that $\otimes A'$ commutes with arbitrary homotopy limits.

Finally, we need to see that the two notions are equivalent in the case where
$\mathbf{1}$ is compact. For this, we use the reasoning of the previous paragraph to
argue that if $A \in \clgw(\mathcal{C})$, then there exists an object $A' \in
\clgw(\mathcal{C})$ such that the dual to $\mathbf{1} \to A'$ is a distinguished
effective descent morphism (i.e., tensoring with $A'$ is conservative and  commutes with homotopy limits)
and such that $A' \to A \otimes A'$ is in mixed elementary form. However, in
this case, $A'$ is dualizable, as an element of $\clgw(\mathcal{C})$, so it
admits descent in view of \Cref{cptdescent}.
Therefore, $A$ is actually a finite cover. 
\end{proof}

\begin{proposition} 
Let $F\colon  \mathcal{C} \to \mathcal{D}$ be a morphism of stable homotopy theories,
so that $F$ induces a functor $\clg(\mathcal{C}) \to \clg(\mathcal{D})$. Then
$F$ carries $\clgf(\mathcal{C})$ into $\clgf(\mathcal{D})$ and
$\clgw(\mathcal{C})$ into $\clgw(\mathcal{D})$. 
\end{proposition}
\begin{proof} 
Let $A \in \clgw(\mathcal{C})$. Then there exists $A' \in \clgw(\mathcal{C})$,
which is a $G$-torsor for some finite group $G$,
such that $A \otimes A'$ is a finite product of localizations of $A'$ at
idempotent elements, in view of \Cref{splitbytorsor}. 
Therefore, $F(A) \otimes F(A')$ is a finite product of localizations of $F(A')$ at
idempotent elements. 

Now $F(A') \in \clg(\mathcal{D})$ is dualizable since $A'$ is, so tensoring with $F(A')$
commutes with limits in $\mathcal{D}$. If we can show that tensoring with $F(A')$
is \emph{conservative} in $\mathcal{D}$, then it will follow that $F(A)$
satisfies the conditions of \Cref{weakfincover}. 
In fact, we will show that the smallest \emph{ideal} of $\mathcal{D}$ closed
under arbitrary colimits and containing $F(A')$ is all of $\mathcal{D}$. This implies
that any object $Y \in \mathcal{D} $ with $Y \otimes F(A') \simeq 0$ must
actually be contractible. 

To see this, recall that $A'$ has a $G$-action. We have a \emph{norm map} (cf. 
\cite[sec. 2.1]{DAGrat} for a general reference in this context) 
\[ A'_{hG} \to A'^{hG} \simeq \mathbf{1}  , \]
which we claim is an equivalence (\Cref{normmap} below). After applying $F$, we find that $F(A')_{hG}
\simeq \mathbf{1}$, which proves the claim and thus shows that tensoring with $F(A')$ is faithful. 

If $A \in \clgf(\mathcal{C})$, then we could choose the torsor $A'$ so that it
actually belonged to $\clgf(\mathcal{C})$ as well. The image $F(A')$ thus
is a descendable commutative algebra object in $\mathcal{D}$ since
descendability is a ``finitary'' condition that does not pose any convergence
issues with infinite limits. So, by similar (but easier) logic, we find
that $F(A) \in \clgf(\mathcal{D})$. 
\end{proof} 

\begin{lemma} 
\label{normmap}
Let $\mathcal{C}$ be a stable homotopy theory and let $A \in
\clgw(\mathcal{C})^{\op}$ be a $G$-torsor, where $G$ is a finite group. Then
the norm map $A_{hG}\to A^{hG} \simeq \mathbf{1}$ is an equivalence. 
\end{lemma} 
\begin{proof} 
It suffices to prove this after tensoring with $A$; note that tensoring with
$A$ is conservative and commutes with all homotopy limits. 
However, after tensoring with $A$, the $G$-action on $A$ becomes induced, so
the norm map is an equivalence. 
\end{proof} 

\newcommand{\pw}{\pi_1^{\mathrm{weak}}}

Finally, we can make the main definition of this paper. 
\begin{definition} Let $(\mathcal{C}, \otimes, \mathbf{1})$ be a stable
homotopy theory. 
The \textbf{Galois groupoid}  $\pi_{\leq 1}(\mathcal{C})$ of $\mathcal{C}$ is the
Galois groupoid of the Galois category $\clgf(\mathcal{C})^{\op}$. The \textbf{weak
Galois groupoid} $\pi_{\leq 1}^{\mathrm{weak}}(\mathcal{C})$
is the Galois groupoid of $\clgw(\mathcal{C})^{\op}$. 
When $\mathbf{1}$ has no nontrivial idempotents, we will write
$\pi_1(\mathcal{C}), \pw(\mathcal{C})$ for the \textbf{Galois group} (resp.
\textbf{weak
Galois group}) of $\mathcal{C}$ with the understanding that these groups are
defined ``up to conjugacy.''
\end{definition} 

As above, we have an inclusion $\clgw(\mathcal{C}) \subset \clgf(\mathcal{C})$
of Galois categories. 
In particular, we obtain a morphism of profinite groupoids
\begin{equation} \label{weakmap}\pi_{\leq 1}^{\mathrm{weak}} (\mathcal{C}) \to
\pi_{\leq 1}(\mathcal{C}).  \end{equation}
The dual map on Galois categories is fully faithful. In particular, if
$\mathcal{C}$ is \emph{connected}, so that $\pi_1, \pw$ can be represented by
profinite groups, the map \eqref{weakmap} is \emph{surjective.}
Moreover, by \Cref{basicgal}, if $\mathbf{1}$ is compact, \eqref{weakmap} is
an equivalence. 

In the following, we will mostly be concerned with the Galois groupoid, which
 is more useful for
computational applications because of the rapidity of the descent. The weak
Galois groupoid is better behaved as a functor out of the $\infty$-category of
stable homotopy theories. We will
discuss some of the differences further below. The weak Galois groupoid
seems in particular useful for potential applications in $K(n)$-local homotopy theory where
$\mathbf{1}$ is not compact. 
Note, however, that the Galois groupoid depends only on the 2-ring of
\emph{dualizable objects} in a given stable homotopy theory, because the
property of admitting descent (for a commutative algebra object which is
dualizable) is a finitary one. So, the Galois groupoid can be viewed as a
functor $\tring \to \pro( \fgp)^{\op}$. 

\begin{definition} 
We will define the \textbf{Galois group(oid)} of an $\e{\infty}$-ring $R$ to be that
of $\md(R)$. Note that the weak Galois group(oid) and the Galois group(oid) of $\md(R)$
are canonically isomorphic, by \Cref{basicgal}. 
\end{definition}

In any event, both the profinite groupoids of \eqref{weakmap} map to something
purely algebraic. Given a finite \'etale cover of the ordinary commutative ring
$R_0 =
\pi_0 \mathrm{End}_{\mathcal{C}}(\mathbf{1})$, we get a commutative algebra
object in $\mathcal{C}$. 

\begin{proposition} 
Let $R'_0$ be a finite \'etale $R_0$-algebra. The induced classically \'etale
object of $\clg( \mathcal{C})$ is a finite cover, and we have a fully faithful
embedding
\[ \cov_{\spec R_0} \subset \clgf(\mathcal{C})^{\op},  \]
from the category $\cov_{\spec R_0}$ of schemes finite \'etale over
$\spec R_0$ into the opposite to the category $\clgf(\mathcal{C})$. 
\end{proposition} 
This was essentially first observed in \cite{BRrealiz}. 

\begin{proof}
We can assume that $\mathcal{C} = \md(R)$ for $R$ an $\e{\infty}$-ring, because
if $R = \mathrm{End}_{\mathcal{C}}(\mathbf{1})$, we always have an embedding
$\md^{\omega}(R) \subset \mathcal{C}$ and everything here happens inside
$\md^\omega(R)$ anyway. It follows from \Cref{etaletopinv} that we have a fully
faithful embedding $\cov_{\spec R_0} \subset \clg(\mathcal{C})^{\op}$, so it
remains only to show that any classically \'etale algebra object coming from a
finite \'etale $R_0$-algebra $R'_0$ is in fact a finite cover. However, we know that
there exists a finite \'etale $R_0$-algebra $R''_0$ such that: 
\begin{enumerate}
\item $R''_0$ is faithfully flat over $R_0$.  
\item $R'_0 \otimes_{R_0} R''_0$ is the localization of $\prod_S R''_0$ at an
idempotent element, for some finite set $S$. 
\end{enumerate}

We can realize $R'_0, R''_0$ topologically by $\e{\infty}$-rings $R', R''$
under $R$. Now $R''$ admits descent over $R'$, as a finite faithfully flat $R$-module, and
$R' \otimes_R R''$ is the localization of $\prod_S R''$ at an idempotent
element, so that $R' \in \clgf(\md(R))$. 
\end{proof}

The classically \'etale algebras associated to finite \'etale $R_0$-algebras give the ``algebraic'' part of the Galois group
and fit into a sequence
\begin{equation}\label{weakmap2} \pw(\mathcal{C}) \twoheadrightarrow \pi_1(\mathcal{C})
\twoheadrightarrow \pi_1^{\mathrm{et}} \spec R_0.  \end{equation}

\begin{definition} 
We will say that the Galois theory of $\mathcal{C}$ is \textbf{algebraic} if
these maps are isomorphisms. 
\end{definition} 

It is an insight of \cite{rognes} that  
the second map in \eqref{weakmap2}
is generally not an isomorphism: that is, there are examples of finite covers that are
genuinely topological and do not appear so at the level of homotopy groups.
We will review the connection between our definitions and Rognes's work in the next section. 

\subsection{Rognes's Galois theory}

In \cite{rognes}, Rognes introduced the definition of a \emph{$G$-Galois
extension} of an $\e{\infty}$-ring $R$ for $G$ a finite group. 
(Rognes also considered the case of a \emph{stably dualizable group}, which
will be discussed only incidentally in this paper.) 
 Rognes worked in the setting of $E$-local spectra for $E$ a
fixed spectrum. The same definition would work in a general stable homotopy
theory. In this subsection, we will connect Rognes's definition with ours. 

\begin{definition}[Rognes] \label{defgalr}
Let $(\mathcal{C}, \otimes, \mathbf{1})$ be a stable homotopy theory. 
An object $A \in \clg(\mathcal{C})$ with the action of a finite group $G$ (in
$\clg(\mathcal{C})$) is a \textbf{$G$-Galois extension} if: 
\begin{enumerate}
\item The map $\mathbf{1}  \to A^{hG}$ is an equivalence. 
\item The map $A \otimes A \to \prod_G A $ (given informally by $(a_1, a_2)
\mapsto \{a_1 g(a_2)\}_{g \in G}$) is an equivalence.
\end{enumerate}
We will say that $A$ is a \textbf{faithful $G$-Galois extension} if further
tensoring with $A$ is conservative. 
\end{definition} 

General $G$-Galois extensions in this sense are outside the scope of this
paper. In general, there is no reason for a $G$-Galois extension to 
be well-behaved at all with respect to descent theory. By an example of Wieland
(see \cite{Rognes2}), the map $C^*(B \mathbb{Z}/p; \mathbb{F}_p) \to \mathbb{F}_p$ given by
evaluating on a point is a $\mathbb{Z}/p$-Galois extension, but one cannot
expect to carry out descent along it in any manner. 
However, one has: 

\begin{proposition} 
\label{rognesequiv}
A faithful $G$-Galois extension in $\mathcal{C}$ is equivalent to a $G$-torsor
in the Galois category $\clgw(\mathcal{C})$. 
\end{proposition}

This in turn relies on:

\begin{proposition}[{\cite[Proposition 6.2.1]{rognes}} ]
\label{galdual}
Any $G$-Galois extension $A$ of the unit is dualizable. 
\end{proposition} 
 
 The proof in \cite{rognes} is stated for the $E$-localization of $\md(A)$ for
 $A$ an $\e{\infty}$-ring, but it is valid in any such setting. 

\begin{proof}[Proof of \Cref{rognesequiv}]
A $G$-torsor in $\clgw(\mathcal{C})$ is, by definition, a commutative algebra
object $A$ with an action of $G$ such that there exists an $A' \in
\clg(\mathcal{C} )$ such that $\otimes A'$ is conservative and commutes with
limits, with $A' \otimes A \simeq \prod_G A'$ as an $A'$-algebra and
compatibly with the $G$-action. 
This together with descent along $\mathbf{1} \to A'$ implies that the map
$\mathbf{1} \to A^{hG}$ is an equivalence. Similarly, the map $A \otimes A \to
\prod_G A $ is well-defined in $\mathcal{C}$ and becomes an equivalence after
base-change to $A'$ (by checking for the trivial torsor), so that it must have
been an equivalence to begin with. 

Finally, if $\mathbf{1} \to A$ is a faithful $G$-Galois extension in the sense
of \Cref{defgalr}, then $A$ is dualizable by \Cref{galdual}, so that  $\otimes
A$ commutes with limits. Moreover, $\otimes A$ is faithful by assumption. Since 
$A \otimes A$ is in elementary form, it follows that $A \in
\clgw(\mathcal{C})$ and is in fact a $G$-torsor. 
\end{proof} 

The use of $G$-torsors will be very helpful in making arguments.
For example, given a connected Galois category, any nonempty
object is a quotient of a $G$-torsor for some finite group $G$; in
fact, understanding the Galois theory is equivalent to understanding torsors
for finite groups. 

\begin{corollary} 
\label{Gtorsor}
A $G$-torsor in the Galois category $\clgf(\mathcal{C})$ is equivalent to a
$G$-Galois extension $A \in \clg(\mathcal{C})$ such that $A$
admits descent. 
\end{corollary} 
\begin{proof} 
Given a $G$-torsor in $\clgf(\mathcal{C})$, it follows easily  that
it generates all of $\mathcal{C}$ as a thick $\otimes$-ideal, since
descendability can be checked locally and since a trivial torsor is descendable. Conversely, if $A$
is a $G$-Galois extension with this property, then $A$ is a finite cover of the
unit: we can take as our descendable commutative algebra object (required 
by \Cref{deffinitecover})
$A$ itself. 
\end{proof}

\begin{corollary} 
If $|G|$ is invertible in $\pi_0 \mathrm{End}( \mathbf{1})$, then a $G$-torsor in $\clgw(\mathcal{C})$
actually belongs to $\clgf(\mathcal{C})$. In particular, if $\mathbb{Q} \subset
\pi_0 \mathrm{End}( \mathbf{1})$, then the two fundamental groups are the same:
\eqref{weakmap2} is an isomorphism. 
\end{corollary} 
\begin{proof} 
In any stable $\infty$-category $\mathcal{D}$ where $|G|$ is invertible (i.e.,
multiplication by $|G|$ is an isomorphism on each object), then for any object
$X \in \mathrm{Fun}(BG, \mathcal{C})$, $X^{hG}$ is a retract of $X$. 
In fact, the composite
\[ X^{hG} \to X \to X_{hG} \stackrel{N}{\to} X^{hG},  \]
is an equivalence, where $N$  is the norm map. 

In particular, given a $G$-torsor $A \in \clgw(\mathcal{C})$, 
we have $\mathbf{1} \simeq A^{hG}$, so that $\mathbf{1}$ is a retract of $A$:
in particular, the thick $\otimes$-ideal $A$ generates contains all of
$\mathcal{C}$, so that (by \Cref{Gtorsor}) it belongs to $\clgf(\mathcal{C})$. 
This proves the first claim of the corollary. 

Finally, if $\mathbb{Q} \subset \pi_0 \mathrm{End}( \mathbf{1})$, then 
fix a weak finite cover $B \in \clgw(\mathcal{C})$. There is a $G$-torsor $A
\in \clgw(\mathcal{C})$
for some finite group $G$ 
such that $A \otimes B$ is a localization of a product of copies of $A$ at
idempotent elements. Since the thick $\otimes$-ideal that $A$ generates contains
all of $\mathcal{C}$ by the above, it follows that $B$ is actually a finite cover. 
\end{proof} 

\subsection{The connective case}

The rest of this paper will be devoted to computations of Galois groups. These
computations are usually based on descent theory together with results stating
that we can identify the Galois theory in certain settings as entirely 
algebraic. Our first result along these lines shows in particular that we can 
recover the classical \'etale fundamental group of a commutative ring. More
generally, we can describe the Galois group of a  connective $\e{\infty}$-ring
purely algebraically.

\begin{theorem} 
\label{connectivegal}
Let $A$ be a connective $\e{\infty}$-ring. 
Then the map $  \pi_1( \md(A)) \to \pi_1^{\mathrm{et}} \spec
\pi_0 A$ is an equivalence; that is, 
all finite covers or weak finite covers are classically \'etale. 
\end{theorem} 

\begin{remark} 
This result, while not stated explicitly in \cite{rognes}, seems to be folklore. 
One has the following intuition: a connective $\e{\infty}$-ring consists of its
$\pi_0$ (which is  a discrete commutative ring) together with higher homotopy
groups $\pi_i, i > 0$ which can be thought of as ``fuzz,'' a generalized sort
of nilthickening.  Since
nilpotents should not affect the \'etale site, we would expect the Galois theory 
to be invariant under the map $A \to \tau_{\leq 0} A$ in this case. 
\end{remark} 

\begin{proof} 
Let $A$ be a connective $\e{\infty}$-ring. 
The argument was explained for $\pi_0 A$ noetherian  in \cite[Example
5.5]{MM}, and the general case can be reduced to this using the commutation
of Galois theory and filtered colimits (\Cref{galcolim} below). 
In fact,  the $\infty$-category of connective $\e{\infty}$-rings is compactly
generated and any compact object has noetherian $\pi_0$. 
Therefore, the result assuming $\pi_0 A$ noetherian implies it in general
since any connective $\e{\infty}$-ring is a filtered colimit of compact
objects. 
\end{proof} 

The above argument illustrates a basic technique one has: one tries, whenever
possible, to reduce to the case of $\e{\infty}$-rings which satisfy
\emph{K\"unneth isomorphisms}. In this case, one can attempt to study
$G$-Galois extensions using algebra. 

\begin{example}[{Cf. \cite[Theorem 10.3.3]{rognes}}] 
The Galois group of $\sp$ is trivial, since $\sp$ is the $\infty$-category of
modules over the sphere $S^0$, and the \'etale fundamental group of $\pi_0(S^0)
\simeq \mathbb{Z}$ is trivial by Minkowski's theorem that the discriminant of
a number field is always $> 1$ in absolute value. 
\end{example}

\subsection{Galois theory and filtered colimits}

In this subsection, we will prove that Galois theory behaves well with respect
to filtered colimits.   

\begin{theorem} \label{galcolim}
The functor $A \mapsto \clgf( \md(A)), \clg \to \cati$
commutes with filtered colimits. In particular, given a filtered diagram $I \to
\clg$, the map
\[ \pi_{\leq 1}  \md({\varinjlim_I A_i})\to \varprojlim_{I}
\pi_{\leq 1}  \md(A_i), \]
is an equivalence of profinite groupoids. 
\end{theorem}

\Cref{galcolim} will be a
consequence of 
some categorical technology together and is a form of ``noetherian descent.'' 
To prove it, we can work with $G$-torsors in view of \Cref{torsorenough}. 
Given an $\e{\infty}$-ring $A \in \clg$, we let $\gal_G(A)$ be the category of
faithful $G$-Galois extensions of $A$: that is, the category of $G$-torsors in
$\clgf(A)$. 
We need to show that
given a filtered diagram $\left\{A_i\right\}$ of $\e{\infty}$-rings, the functor
\[ \varinjlim \mathrm{Gal}_G(A_i) \to \mathrm{Gal}_G(\varinjlim A_i),   \]
is an equivalence of categories: i.e., that it is fully faithful and
essentially surjective. 
We start by showing that faithful Galois
extensions are compact $\e{\infty}$-algebras. 

\begin{lemma} 
Let $A \to B$ be a faithful $G$-Galois extension. Then $B$ is a compact
object
in the $\infty$-category $ \einf_{A/}$ of $\e{\infty}$-algebras over $A$. 
Moreover, $\hom_{\einf_{A/}}(B, \cdot)$ takes values in homotopy discrete
spaces. 
\end{lemma} 
\begin{proof} 
First, recall that if $A \to B$ is a \emph{classically \'etale} extension, then the result
is true. In fact, if $A \to B$ is classically \'etale, then for any $\e{\infty}$-$A$-algebra
$A'$, the natural map
\[ \hom_{\einf_{A/}}(B, A') \to \hom_{\rng_{\pi_0 A/}}(\pi_0 B, \pi_0 A'),  \]
is an equivalence. 
Moreover, $\pi_0 B$, as an \'etale $\pi_0 A$-algebra, is finitely presented or
equivalently compact in $\rng_{\pi_0 A/}$. The result follows for an \'etale
extension. 

Now, a Galois extension need not be classically \'etale, but it becomes \'etale after an
appropriate base change, so we can use descent theory. 
Recall  that we have an equivalence
of symmetric monoidal $\infty$-categories
\[ \md(A) \simeq \mathrm{Tot} \left( \md(B) \rightrightarrows \md(B \otimes_A
B ) \triplearrows \dots \right).  \]
Upon taking commutative algebra objects, we get an equivalence
of $\infty$-categories
\[ \einf_{A/} \simeq 
\mathrm{Tot}\left( \einf_{B/} \rightrightarrows \einf_{B \otimes_A B /}
\triplearrows \dots \right)
. \]
The object $B \in \einf_{A/}$ 
becomes classically \'etale, thus compact, after base-change along $A \to B$. 
We may
now apply the next sublemma to conclude. 
\end{proof}

\begin{sublemma}
Let $\mathcal{C}^{-1} \in \prl$ be a presentable $\infty$-category and $\mathcal{C}^\bullet$ a
cosimplicial object in $\prl$ with an equivalence
of $\infty$-categories
\[ \mathcal{C}^{-1} \simeq \mathrm{Tot}( \mathcal{C}^\bullet).  \]
Suppose that $x \in \mathcal{C}^{-1}$ is an object such that:
\begin{itemize}
\item  The image $x^i$  of $x$ in $\mathcal{C}^i, i \geq 0$ is compact for each $i$. 
\item  There exists $n$ such that the image $x^i$ of $x$ in each $\mathcal{C}^i$ is $n$-cotruncated
in the sense that 
\[ \hom_{\mathcal{C}^i}(x^i, \cdot)\colon  \mathcal{C}^i \to \mathcal{S}\]
takes values in the subcategory $\tau_{\leq n} \mathcal{S} \subset \mathcal{S}$
of $n$-truncated spaces. (This follows once $x^0$ is $n$-cotruncated.) 
\end{itemize}
Then $x$ is compact (and $n$-cotruncated) in $\mathcal{C}^{-1}$). 
\end{sublemma}
\begin{proof} 
Given objects $w, z \in \mathcal{C}^{-1}$, 
the natural map
\[ \hom_{\mathcal{C}}(w, z) \to \mathrm{Tot} \hom_{\mathcal{C}^\bullet}(
w^\bullet, z^\bullet)  \]
is an equivalence, where for each $i \geq 0$,  $w^i, z^i$ are the objects
in $\mathcal{C}^i$ that are the images of $w, z$. 

Therefore, it follows that $\hom_{\mathcal{C}^{-1}}(x, \cdot)\colon  \mathcal{C}^{-1} \to \mathcal{S}$ is
the totalization of a cosimplicial functor $\mathcal{C}^{-1} \to \mathcal{S}$
given by $\hom_{\mathcal{C}^\bullet}(x^\bullet, \cdot^\bullet)$. Each of the
terms in this cosimplicial functor, by assumption, commutes with filtered
colimits and takes values in $n$-truncated spaces. The sublemma thus follows
because the totalization functor
\[ \mathrm{Tot}\colon  \mathrm{Fun}( \Delta, \tau_{\leq n} \mathcal{S}) \to
\mathcal{S},  \]
lands in $\tau_{\leq n} \mathcal{S}$, and commutes with filtered colimits: a
totalization of $n$-truncated spaces can be computed by a partial
totalization, and finite limits and filtered colimits of spaces commute with
one another. 
\end{proof} 

Next, we prove a couple of general categorical lemmas about compact objects in
undercategories and filtered colimits.

\begin{lemma} 
Let $\mathcal{C}$ be a compactly generated, presentable $\infty$-category and
let $\mathcal{C}^\omega$ denote the collection of compact objects. 
Then, for each $x \in \mathcal{C}$, the undercategory $\mathcal{C}_{x/}$ is
compactly generated.
Moreover, the subcategory $(\mathcal{C}_{x/})^\omega$ is generated under finite
colimits and retracts by the morphisms of the form $x \to x \sqcup y$ for $y
\in \mathcal{C}^\omega$. 
\end{lemma}
\begin{proof} 
To prove this, recall that if $\mathcal{D}$ is any presentable
$\infty$-category and $\mathcal{E} \subset \mathcal{D}$ is a (small)
subcategory of
compact objects, closed under finite colimits, then there is induced a
map in $\prl$
\[ \mathrm{Ind}( \mathcal{E}) \to \mathcal{D},  \]
which is an equivalence of $\infty$-categories precisely when 
$\mathcal{E}$ \emph{detects equivalences}: that is, when a map $x \to y$ in
$\mathcal{D}$ is an equivalence when $\hom_{\mathcal{D}}(e, x) \to
\hom_{\mathcal{D}}(e, y)$ is a homotopy
equivalence for all $e \in \mathcal{E}$. 
Indeed, in this case, it follows that $\mathrm{Ind}	 (\mathcal{E}) \to
\mathcal{D}$ is a fully faithful functor, which embeds
$\mathrm{Ind}(\mathcal{E})$ as a full subcategory of $\mathcal{D}$ closed under
colimits. But any fully faithful left adjoint whose right adjoint is
conservative is an equivalence of $\infty$-categories. 
This argument is a very slight variant of Proposition 5.3.5.11 of \cite{HTT}. 

Now, we apply this to $\mathcal{C}_{x/}$. Clearly, the objects $x \to x
\sqcup y$ in $\mathcal{C}_{x/}$, for $y \in \mathcal{C}^\omega$, are compact. Since
\[ \hom_{\mathcal{C}_{x/}}(x \sqcup y, z) = \hom_{\mathcal{C}}(y, z),  \]
it follows from the above paragraph if $\mathcal{C}$ is compactly generated,
then the $x \to x \sqcup y$ in $\mathcal{C}_{x/}$ \emph{detect equivalences}
and thus generate $\mathcal{C}_{x/}$ under colimits. More precisely, if
$\mathcal{E} \subset \mathcal{C}_{x/}$ is the subcategory generated under
finite colimits by the $x \to x \sqcup y, y \in
\mathcal{C}^\omega$, then the natural functor $\mathrm{Ind}( \mathcal{E}) \to
\mathcal{C}_{x/}$ is an equivalence. Since
$(\mathrm{Ind}(\mathcal{E}))^{\omega}$ is the idempotent completion of
$\mathcal{E}$ (Lemma 5.4.2.4 of \cite{HTT}), the lemma follows. 

\end{proof}

Let $\mathcal{C}$ be a compactly generated, presentable $\infty$-category. 
We observe that the association $x \in \mathcal{C} \mapsto
(\mathcal{C}_{x/})^\omega$ is actually functorial in $x$. 
Given a morphism $x \to y$, we get a functor
\[ \mathcal{C}_{x/} \to \mathcal{C}_{y/}  \]
given by pushout along $x \to y$. Since the right adjoint (sending a map $y \to
z$ to the composite $x \to y \to z$) commutes with filtered colimits, it
follows that $\mathcal{C}_{x/} \to \mathcal{C}_{y/}$ restricts to a functor on
the compact objects. 
We get a functor
\[ \Phi\colon  \mathcal{C} \to \mathrm{Cat}_\infty, \quad x \mapsto
(\mathcal{C}_{x/})^\omega.  \]
Our next goal is to show that $\Phi$ commutes with filtered colimits. 
\begin{lemma} \label{fullfaithfulthing}
The functor $\Phi$ has the property that 
for any filtered diagram $x\colon  I \to \mathcal{C}$, the  natural functor
\begin{equation} \label{phi} \varinjlim_{I}  \Phi( x_i) \to \Phi( \varinjlim_I
x_i),\end{equation}
is an equivalence of $\infty$-categories. \end{lemma} 
\begin{proof} 
Full faithfulness of $\Phi$ is a formal consequence of the definition of a compact object. 
In fact, an element of $\varinjlim_I \Phi(x_i)$ is represented by an object $i
\in I$ and a map $x_i \to y_i$ that belongs to $(\mathcal{C}_{x_i/})^\omega$. 
We will denote this object by $(i, y_i)$. 
This object is the same as that represented by $x_j \to y_i \sqcup_{x_i} x_j$
for any map $i \to j$ in $I$. 

Given two such objects in $\varinjlim_I \Phi(x_i)$, we can represent them both 
by objects $x_i \to y_i, x_i \to z_i$ for some index $i$. Then 
\[ \hom_{\varinjlim_I \Phi(x_i)}( (i, y_i), (i, z_i)) = 
\varinjlim_{j \in I_{i/}}\hom_{\mathcal{C}_{x_j/} }(
y_j, z_j),   \]
where $y_j, z_j$ denotes the pushforwards of $y_i, z_i$ along $x_j \to z_j$. 

Let $x = \varinjlim_I x_i$, and let $y, z$ denote the pushforwards of $y_i,
z_i$ all the way along $x_i \to x$. Then our claim is that the map
\[ \varinjlim_{j \in I_{i/}}\hom_{\mathcal{C}_{x_j/} }(
y_j, z_j) \to  \hom_{\mathcal{C}_{x/}}(y, z) 
\]
is an equivalence. 
Now, we write
\begin{align*}  \hom_{\mathcal{C}_{x/}}(y, z) &  \simeq 
\hom_{\mathcal{C}_{x_i/}}(y_i, z) \\
& \simeq 
\hom_{\mathcal{C}_{x_i/}}(y_i, \varinjlim_{j \in I_{i/}} z_j ) \\
& \simeq
\varinjlim_{j \in I_{i/}}\hom_{\mathcal{C}_{x_i/}}(y_i, z_j ) \\
& \simeq 
\varinjlim_{j \in I_{i/}}\hom_{\mathcal{C}_{x_j/}}(y_j, z_j ), 
\end{align*}
and we get the equivalence as desired. 

Finally, to see that \eqref{phi} establishes the right hand side as the idempotent
completion of the first, we use the description of compact objects in
$\mathcal{C}_{x/}$. 
To complete the prooof, note now that a filtered colimit of idempotent complete $\infty$-categories is
itself idempotent complete \cite[Lemma 7.3.5.16]{higheralg}. 
\end{proof}

\begin{corollary} 
\label{0truncatedcolim}
Hypotheses as above, the functor $\Psi: x \mapsto (\mathcal{C}_{x/})^{\omega, \leq
0}$ sending $x$ to the category of 0-cotruncated, compact objects in
$\mathcal{C}_{x/}$ has the property that the natural functor
\( \varinjlim_I \Psi(x_i) \to  \Psi( \varinjlim x_i)  \)
is an equivalence. 
\end{corollary} 

This follows from the previous lemma, because 0-cotruncatedness of an object
$y$ is equivalent to the claim that the map $S^1 \otimes y \to y$ is an
equivalence.

\begin{proof}[Proof of \Cref{galcolim}]
For $A$ an $\e{\infty}$-ring, let $(\einf_{A/})^{\omega, \leq 0}$ be the
(ordinary) category of $0$-cotruncated, compact $\e{\infty}$-$A$-algebras;
this includes any finite cover of $A$, for example, since finite covers of
$A$ are locally \'etale. 
Then we have a fully faithful inclusion
of $\infty$-categories
\[ \mathrm{Gal}_G(A) \subset \mathrm{Fun}(BG, (\einf_{A/})^{\omega, \leq 0}).  \]
Although $BG$ is not compact in the $\infty$-category of $\infty$-categories,
the truncation to $n$-categories for any $n$ is: $BG$ can be represented as a
simplicial set with finitely many simplices in each dimension. 
Therefore, the right-hand-side has the property that it commutes with filtered
colimits in $A$ by \Cref{0truncatedcolim}. 
Thus, for any filtered diagram $A\colon  I \to \einf$, the functor
\[ \varinjlim_{i \in I}\mathrm{Gal}_G(A_i) \to \mathrm{Gal}_G( \varinjlim_{i
\in I} A_i), \]
is fully faithful.

Moreover, given a $G$-Galois extension $B$ of $A = \varinjlim_I A_i$, there
exists $i \in I$ and a compact, 0-cotruncated $A_i$-algebra $B_i$ with a $G$-action, such that 
$A \to B$ is obtained by base change from $A_i \to B_i$. 
It now suffices to show that $A_i \to B_i$ becomes $G$-Galois after some base
change $A_i \to A_j$. 
For any $j \in I$ receiving a map from $i$, we let $B_j = A_j \otimes_{A_i}
B_i$. 
We are given that $\varinjlim_{j \in I_{i/}} B_j$ is a faithful $G$-Galois extension of
$\varinjlim_{j \in I_{i/}} A_j$ and we want to claim that there exists $j$
such that $B_j$ is a faithful $G$-Galois extension of $A_j$.

Now, the condition for $A_j \to B_j$ to be faithfully $G$-Galois has two parts:
\begin{enumerate}
\item $B_j \otimes_{A_j} B_j \to \prod_G B_j$ should be an equivalence. 
\item  $A_j \to B_j$ should be descendable.
\end{enumerate}

The first condition is detected at a ``finite stage''  since  both the source
and target are compact objects of $\clg_{A_j/}$. 
Unfortunately, we do not know how to use this line of argument alone to argue that
the $A_j \to B_j$'s are faithful $G$-Galois for some $j$, although we suspect that it
is possible. 

Instead, we use some obstruction theory. The map $A \to B$
exhibits $B$ as a perfect $A$-module. For any $\mathbf{E}_{1}$-ring $R$, let
$\mod^\omega(R)$ be the stable $\infty$-category of perfect $R$-modules. Then the
natural functor
\[ \varinjlim_I \mod^\omega( A_i) \to \mod^\omega(A) , \]
is an equivalence of $\infty$-categories.\footnote{One does not even need to worry
about idempotent completeness here because we are in a stable setting, and any
self-map $e\colon  A \to A$ with $e^2 \simeq e$ can be extended to an idempotent.}
It follows that we can ``descend'' the perfect $A$-module $B$ to a perfect
$A_j$-module $B'_j$ for some $j$ (asymptotically unique), and we can descend the 
 multiplication map $B \otimes_A B \to B$ (resp. the unit map $A \to B$) to
 $B'_j \otimes_{A_j} B'_j \to B'_j$ (resp. $A_j \to B'_j$). We can also assume
 that homotopy associativity holds for $j$ ``large.'' 
The $G$-action on $B$ in the \emph{homotopy category} of perfect
 $A$-modules descends to an action on $B'_j$ in the \emph{homotopy category} of
 perfect $A_j$-modules, and the equivalence $B \otimes_A B \simeq \prod_G B$
 descends to an equivalence $B'_j \otimes_{A_j} B'_j \simeq \prod_G B'_j$. 
Finally, the fact that the thick subcategory that $B$ generates contains $A$ 
can also be tested at a finite stage. 

The upshot is that, for $j$ large, we can ``descend'' the $G$-Galois extension
$A \to B$ to a perfect $A_j$-module  $B'_j$ with the portion of the structure of a
$G$-Galois extension that one could see \emph{solely from the homotopy
category.}  However, using obstruction theory one can promote this to a genuine
Galois extension. In \Cref{obstruct} below, 
we show that $B'_j$ can be promoted to an $\mathbf{E}_{\infty}$-algebra (in $A_j$-modules)
for $j \gg 0$ with a $G$-action, which is a  faithful $G$-Galois extension. 

It follows that the $B'_j$ lift $B$ to $A_j$ for $j \gg 0$, and even with the
$G$-action (which is unique in a faithful Galois extension; see Theorem 11.1.1
of \cite{rognes}). 
\end{proof}

\begin{theorem} \label{obstruct}
Let $A'$ be an $\mathbf{E}_{\infty}$-ring, and let $B'$ be  a perfect $A'$-module such
that the thick subcategory generated by $B'$ contains $A'$. Suppose
given:
\begin{enumerate}
\item A homotopy commutative, associative and unital multiplication $B'\otimes_{A'} B' \to B'$. 
\item A $G$-action on $B'$ in the homotopy category, commuting with the
multiplication and unit maps, such that the map $B'
\otimes_{A'} B' \to \prod_G B'$ is an equivalence of $A$-modules.  
\end{enumerate}
Then $B'$ has a unique $\mathbf{E}_{\infty}$-multiplication extending the given 
homotopy commutative one, and $A \to B$ is faithful $G$-Galois (in
particular, the $G$-action in the homotopy category extends to a strict one of $\mathbf{E}_{\infty}$-maps on $B$). 
\end{theorem} 
 Here we use an argument, originally due to Hopkins in a different setting, that will be
 elaborated upon further in joint work with Heuts; as such, we give a sketch of the
 proof. 

\begin{proof} We use the obstruction theory of \cite{robinsonobstruct} (see also \cite[Sec. 3]{angeltveit}) to produce a unique
$\mathbf{E}_{1}$-structure. 
Since $B' \otimes_{A'} B' $ is a finite product of copies of $B'$, it follows
that $B'$ satisfies a perfect universal coefficient formula in the sense of
that paper. 
The obstruction theory
developed there
states that the 
obstructions to producing an $\mathbf{E}_{1}$-structure lie in $\mathrm{Ext}^{n,
3-n}_{\pi_*(B' \otimes_{A'} B')}( B'_*, B'_*)$ for $n\geq 4$, and the
obstructions to uniqueness in the groups
$\mathrm{Ext}^{n,
2-n}_{\pi_*(B' \otimes_{A'} B')}( B'_*, B'_*)$ for $n \geq 3$. 
The hypotheses of the lemma imply that $B'_*$ is a projective $\pi_*(B'
\otimes_{A'} B')$-module, though, so that all the obstructions (both to
uniqueness and existence) vanish. 

Our next goal is to promote this to an $\mathbf{E}_{\infty}$-multiplication extending the 
given $\mathbf{E}_{1}$-structure. 
We claim that the space of $\mathbf{E}_{1}$-maps between 
any tensor power $B'^{\otimes m}$ and any other tensor power $B'^{\otimes n}$ of $B'$ is homotopy
discrete and equivalent to the collection of  maps of $A$-\emph{ring spectra}: that
is, homotopy classes of maps $B'^{\otimes m} \to B'^{\otimes n}$ (in
$A$-modules) that commute
with the multiplication laws \emph{up to homotopy.}
This is a consequence of the analysis in \cite{rezkHM} (in particular,
Theorem 14.5 there), and the fact that the
$B'^{\otimes n}$-homology of $B'^{\otimes m}$ is \'etale, so that the
obstructions of \cite{rezkHM} all vanish. 

It follows that if $\mathcal{C}$ is the smallest symmetric monoidal $\infty$-category 
of $\mathrm{Alg}(\mod(A'))$ (i.e., $\mathbf{E}_{1}$-algebras in $\mod(A')$) containing $B'$, then 
$\mathcal{C}$ is equivalent to an ordinary symmetric monoidal category, which
is equivalent to a full subcategory of the category of $A$-ring spectra. Since
$B'$ is a commutative algebra object in that latter category, it follows that
it is a commutative algebra object of $\mathrm{Alg}(\mod(A'))$, and
thus gives an $\mathbf{E}_{\infty}$-algebra. 
The $G$-action, since it was by maps of $A$-ring spectra, also comes along. 
\end{proof}

\subsection{The even periodic and regular case}
Our first calculation of a Galois group was in \Cref{connectivegal}, where
we
showed that the Galois group of a connective $\e{\infty}$-ring was entirely
algebraic. 
In this section, we will show (\Cref{etalegalois}) that the analogous statement
holds for an even periodic $\e{\infty}$-ring with regular (noetherian) $\pi_0$. 
As in the proof of \Cref{connectivegal}, the strategy is to reduce to
considering ring spectra with K\"unneth isomorphisms. Unfortunately, in the
nonconnective setting, the ``residue field'' ring spectra one wants can be constructed only as
$\e{1}$-algebras (rather than as $\e{\infty}$-algebras), so one has to work somewhat
harder. 

\begin{definition} 
An $\e{\infty}$-ring $A$ 
is \textbf{even periodic} if: 
\begin{enumerate}
\item $\pi_i A = 0$ if $i$ is odd.  
\item The multiplication
map $\pi_2 A \otimes_{\pi_0 A} \pi_{-2} A  \to \pi_0 A$ is an isomorphism. 
\end{enumerate}
In particular, $\pi_2 A$ is an invertible $\pi_0 A$-module; if it is free
of rank one, then $\pi_*(A) \simeq \pi_0(A) [t_{2}^{\pm 1}]$ where $|t_2| = 2$. 
\end{definition}

Even periodic $\e{\infty}$-rings (such as complex $K$-theory $KU$) play a central role in chromatic homotopy
theory because of the connection, beginning with Quillen, with the theory of
\emph{formal groups.} We will also encounter even periodic $\e{\infty}$-rings
in studying stable module $\infty$-categories for finite groups below. 
The $\infty$-categories of modules over them turn out to be fundamental
building blocks for many other stable homotopy theories, so an understanding of
their Galois theory will be critical for us.

We begin with the simplest case. 
\begin{proposition} 
\label{fieldreg}
Suppose $A$ is an even periodic $\e{\infty}$-ring with $\pi_0 A \simeq k[t^{\pm
1}]$ where $|t| = 2$ and $k$ a field. Then the Galois theory of $A$ is
algebraic:  $\pi_{ 1} \md(A) \simeq  \mathrm{Gal}(k^{\mathrm{sep}}/k)$. 
\end{proposition} 
\begin{proof} 
We want to show that any finite cover of $A$ is \'etale at the level of
homotopy groups; flat would suffice. 
Let $B$ be a $G$-Galois extension of $A$. Then $B \otimes_A B \simeq \prod_G
B$. Since $\pi_*(A)$ is a graded field, it follows that
\[ \pi_*(B) \otimes_{\pi_*(A)} \pi_*(B) \simeq \prod_G \pi_*(B).  \]
Moreover, since $B$ is a perfect $A$-module, it follows that $\pi_*(B)$ is a
finite-dimensional $\pi_*(A)$-module. 

Making a base-change $t \mapsto 1$, we can work in $\mathbb{Z}/2$-graded
$k$-vector spaces rather than graded $k[t^{\pm 1}]$-modules. 
So we get a $\mathbb{Z}/2$-graded commutative (in the graded sense) $k$-algebra 
$B'_* = B'_0 \oplus B'_1$ with the property that we have an equivalence of
$\mathbb{Z}/2$-graded $B'_*$-algebras
\begin{equation} \label{split} B'_* \otimes_k B'_* \simeq \prod_G B'_*.
\end{equation}
Observe that this tensor product is the \emph{graded} tensor product.

From this, we want to show  that $B'_1 = 0$, which will automatically force
$B'_0$ to be \'etale over $k$. Suppose
first that the characteristic of $k$ is not 2. By \Cref{gradedloc} below, there exists a map of graded $k$-algebras $B_*' \to
\overline{k}$. We
can thus compose with the map $k \to B'_* \to \overline{k}$ and use 
\eqref{split} to conclude that $B'_* \otimes_k \overline{k} \simeq \prod_G
\overline{k}$ as a graded $k$-algebra. 
This in particular implies that $B'_1 = 0$ and that $B'_0$ is a finite separable
extension of $k$, which proves \Cref{fieldreg} away from the prime 2. 

Finally, at the prime 2, 
we need to show that \eqref{split} still implies that $B'_1 =0 $. 
In this case, $B'_0 \oplus B'_1$ is a \emph{commutative} $k$-algebra and
\eqref{split} implies that it must be \'etale. After extending scalars to
$\overline{k}$, $B'_0 \oplus B'_1$ must, as a 
commutative ring, be isomorphic to $\prod_G \overline{k}$. However, any
idempotents in $B'_0 \oplus B'_1$ are clearly concentrated in degree zero. 
So, we can make the same conclusion at the prime $2$. 
\end{proof} 

\begin{lemma} 
\label{gradedloc}
Let $k$ be an algebraically closed field with $2 \neq 0$, and $A'_*$ a  nonzero
finite-dimensional $\mathbb{Z}/2$-graded
commutative $k$-algebra. Then there exists a map of graded $k$-algebras $A'_*\to k$. 
\end{lemma} 
\begin{proof} 
Induction on $\dim A'_1$. If $A'_1 = 0$, we can use the ordinary theory of
artinian rings over algebraically closed fields. If there exists a nonzero $x \in A'_1$, we can form the two-sided ideal $(x)$: this is equivalently the left
or right ideal generated by $x$. In particular, anything in $(x)$ has square
zero. It follows that $1 \notin (x)$ and we get a map
of $k$-algebras
\[ A'_* \to A'_*/(x),  \]
where $A'_*/(x)$ is a \emph{nontrivial} finite-dimensional $\mathbb{Z}/2$-graded commutative ring
of smaller dimension in degree one. We can thus continue the process.
\end{proof}

We can now prove our main result. 
\begin{theorem} \label{etalegalois}
Let $A$ be an even periodic $\e{\infty}$-ring with $\pi_0 A$
regular noetherian. Then the Galois theory of $A$ is algebraic. \end{theorem} 

Most of this result appears in \cite{BR2}, where the Galois group
of $E_n$ is identified at an odd prime (as the Galois group of its $\pi_0$). 
Our methods contain the modifications needed to handle the prime $2$ as well.
\begin{remark}
This will also show that all Galois extensions of $A$ in the sense of
\cite{rognes} are faithful. 
\end{remark}

\begin{proof} [Proof of \Cref{etalegalois}]
Fix a finite group $G$ and let $B$ be a $G$-Galois extension of $A$, so that
\[ A \simeq B^{h G}, \quad B \otimes_A B \simeq \prod_G B.  \]
By \Cref{galdual}, $B$ is a perfect $A$-module; in particular, the homotopy
groups of $B$ are finitely generated $\pi_0 A$-modules. 

Our goal is to show that $B$ is even periodic and that $\pi_0 B$ is \'etale
over $\pi_0 A$. To do this, we may reduce to the case of $\pi_0 A$
{regular} \emph{local}, by
checking at each localization. 
We are now in the following situation. The $\e{\infty}$-ring $A$ is even
periodic, with $\pi_0 A$ local  with its maximal ideal generated by a
regular sequence $x_1, \dots, x_n \in \pi_0 A$ for $n = \dim \pi_0 A$. 
Let $k$ be the residue field of $\pi_0 A$. 
In this case, then one can define a \emph{multiplicative homology theory}
$P_*$ on
the category of $A$-modules 
via
\[ P_*(M) \stackrel{\mathrm{def}}{=} \pi_* ( M/(x_1, \dots, x_n) M) \simeq
\pi_* (M \otimes_A A/(x_1, \dots, x_n)), \]
where $A/(x_1, \dots, x_n) \simeq A/x_1 \otimes_A \dots \otimes_A A/x_n$. 
More precisely, it is a consequence of the work of Angeltveit
\cite[Sec. 3]{angeltveit} that $A/(x_1, \dots, x_n)$ can
be made (noncanonically) an
$\e{1}$-algebra in $\md(A)$. 
In particular, $A/(x_1, \dots, x_n)$ is, at the very least, a ring object in the
homotopy category of $A$-modules; this weaker assertion, which is all that we
need, is proved directly in \cite[Theorem 2.6]{EKMM}. 
The fact that each $A/x_i$ acquires the structure of a ring object in the
homotopy category of $A$-modules already means that for any $A$-module $M$, the
homotopy groups of $M/x_i M\simeq M \otimes_A A /x_i$ are actually
$\pi_0(A)/(x_i)$-modules. 

In any event, $M \mapsto P_*(M)$ is a multiplicative homology theory taking values in
$k[t^{\pm 1}]$-modules. It satisfies a K\"unneth isomorphism,
\[ P_*(M) \otimes_{k[t^{\pm 1}]} P(N) \simeq P_*(M \otimes_A N) , \]
by a standard argument: with $N$ fixed, both sides define homology theories on $A$-modules;
there is a natural map between the two; moreover, this map is an isomorphism for $M = A$.
This implies that the natural map is an isomorphism by a five-lemma argument. 
Note that  the $\e{1}$-ring $A/(x_1, \dots, x_n)$ is usually not homotopy
commutative if $ p = 2$. 

For convenience, rather than working in the category of graded $k[t^{\pm
1}]$-modules, we will work in the (equivalent) category of $\mathbb{Z}/2$-graded $k$-vector
spaces, and denote the modified functor by $Q_*$ (instead of $P_*$). 
Since $A \to B$ is $G$-Galois, it follows from $B \otimes_A B \simeq \prod_G
B$ that there is an isomorphism of $\mathbb{Z}/2$-graded $k[G]$-modules,
\[ Q_*(B) \otimes_k Q_*(B) \simeq \prod_{G} Q_*(B). \]
In particular, it follows that
\begin{equation} \label{dims} \dim Q_0 (B) + \dim Q_1(B) = |G|.  \end{equation}

We now use a Bockstein spectral sequence argument to bound the rank of
$\pi_0 B$ and $\pi_1 B$. 

\begin{lemma} \label{BSSlem}
Let $M$ be a perfect $A$-module. Suppose that $\dim_k Q_0(M) = a$. Then the
rank of $\pi_0
M$ as a $\pi_0 A$-module (that is, the dimension after tensoring with the
fraction field) is at most $a$. 
\end{lemma} 

\begin{proof}
Choose a system of parameters $x_1, \dots, x_n$ for the maximal ideal of $\pi_0
A$. 
If $M$ is as in the statement of the lemma, then we are given that 
\[ \dim \pi_0(M/(x_1, \dots, x_n) M)  \leq a. \]
We consider the sequence of $A$-modules 
\[ M_i = M/(x_1, \dots, x_i)M = M \otimes_A A/x_1 \otimes_A \dots \otimes_A
A/x_i;  \]
here $\pi_0(M_i)$ is a finitely generated module over the regular local ring $\pi_0(A)/(x_1, \dots, x_i)$. 
For instance, $\pi_0(M_n)$ is a module over the residue field $k$, and our
assumption is that its rank is at most $a$.

\newtheorem*{indstep}{Inductive step}
We make the following inductive step. 
\begin{indstep}
If $\pi_0(M_{i+1})$ has rank $\leq a$ as a module over the regular local ring
$\pi_0(A)/(x_1, \dots, x_{i+1})$, then $\pi_0(M_i)$ has rank $\leq a$ as a
module over the regular local ring $\pi_0(A)/(x_1, \dots, x_i)$. 
\end{indstep}

To see this, consider the cofiber sequence 
\[ M_i \stackrel{x_i}{\to} M_i \to M_{i+1},  \]
and the induced injection in homotopy
groups
\[ 0 \to \pi_0(M_i)/x_{i}\pi_0 M_i \to \pi_0(M_{i+1}). \]
We now apply the following sublemma. 
By descending induction on $i$, this will imply the desired claim.

\newcommand{\rank}{\mathrm{rank}}
\begin{sublemma} 
Let $(R, \mathfrak{m})$ be a regular local ring, $x \in \mathfrak{m} \setminus
\mathfrak{m}^2$. Consider a finitely generated $R$-module $N$. Given an
injection
\[ 0 \to N/xN \to N',  \]
where $N'$ is a finitely generated $R/(x)$-module, we have
\[ \rank_{R} N \leq \rank_{R/(x)} N'.  \]
\end{sublemma} 
\begin{proof} 
When $R$ is a discrete valuation ring (so that $R/(x)$ is a field), this follows from the structure theory of 
finitely generated modules over  a PID. 

To see this in general, we may localize at the prime ideal $(x) \subset R$
(and thus replace the pair $(R, R/(x))$ with $R_{(x)}, R_{(x)}/(x) R_{(x)}$), 
which does not affect the rank
of either side, and which reduces us to the DVR case. 
\end{proof} 

With the sublemma, we can conclude that $\rank_{\pi_0(A)/(x_1, \dots, x_i)} \pi_0(M_i) \leq a$ for all $i$
by descending induction on $i$, which completes the proof of \Cref{BSSlem}. 

\end{proof}

By \Cref{BSSlem}, it now follows that $\pi_0 B$, as a $\pi_0 A$-module, has
rank  at most $a = \dim_k Q_0(B)$, where $a \leq |G|$. However, when we invert everything in $\pi_0 A$
(i.e., form the fraction field $k( \pi_0 A))$, then ordinary Galois theory goes into
effect (\Cref{fieldreg}) and $\pi_0 B \otimes_{\pi_0 A} k( \pi_0 A)$ is 
a finite \'etale $\pi_0 A$-algebra with Galois group $G$. 
In particular, it follows that $a = |G|$. 

As a result, by \eqref{dims}, $Q_1(B) = 0$. It follows, again by the Bockstein spectral
sequence, in the form of \Cref{easyBSS} below, that $B$ is evenly graded and
$\pi_* B$ is free as an $A$-module. 
In particular, $\pi_0(B \otimes_A B) \simeq \pi_0 B \otimes_{\pi_0 A} \pi_0 B$,
which means that we get an isomorphism
\[ \pi_0 B \otimes_{\pi_0 A} \pi_0 B \simeq \prod_G \pi_0 B,  \]
so that $\pi_0 B$ is \'etale over $\pi_0 A$ (more precisely, $\spec \pi_0 B
\to \spec \pi_0 A$ is a $G$-torsor), as desired. 
This completes the proof of \Cref{etalegalois}. 
\end{proof} 
\begin{lemma}[] \label{easyBSS}
Let $A$ be an even periodic $\e{\infty}$-ring such that $\pi_0 A$ is regular local
and $n$-dimensional,
with maximal ideal $\mathfrak{m} = (x_1, \dots, x_n)$. Let $M$ be a perfect
$A$-module such that the $A$-module $M/(x_1, \dots, x_n)M$ satisfies $\pi_1(
M/(x_1, \dots, x_n) M) = 0$. Then $\pi_1(M) = 0$ and $\pi_0(M)$ is a free
$\pi_0(A)$-module. 
\end{lemma} 

\begin{proof}
\Cref{easyBSS} follows from a form of the Bockstein spectral sequence: the
evenness implies that there is no room for differentials; 
Proposition 2.5 of  \cite{HoveyS} treats the case  of $A = E_n$. 
We can give a direct argument as follows. 

Namely, we show that $\pi_1(M/(x_1, \dots, x_i) M) = 0$ for $i = 0,
1, \dots, n$, by descending induction on $i$. By assumption, it holds for $i
= n$. The inductive step is proved as in the proof of \Cref{BSSlem}, except
that Nakayama's lemma is used to replace the sublemma. 
This shows that $\pi_1(M) = 0$. 

Now, inducting in the other direction (i.e., in ascending order in $i$), we find that $x_1, \dots, x_n$ defines a
regular sequence on $\pi_0(M)$ and 
the natural map
\[ \pi_0(M)/(x_1, \dots, x_i) \to \pi_0(M/(x_1, \dots, x_i)), \]
is an isomorphism. This implies that the depth of $\pi_0(M)$ as a
$\pi_0(A)$-module is equal to $n$, so that $\pi_0(M)$ is a free
$\pi_0(A)$-module. 
\end{proof}

\section{Local systems, cochain algebras, and stacks}

The rest of this paper will be focused on the calculations of Galois groups in
certain examples of stable homotopy theories, primarily those arising from
chromatic homotopy theory and modular representation theory. 
The basic ingredient, throughout, is to write a given stable homotopy theory as
an \emph{inverse limit} of simpler stable homotopy theories to which one can apply
known algebraic techniques such as \Cref{etalegalois} or \Cref{connectivegal}. 
Then, one puts together the various Galois groupoids that one has via
techniques from descent theory. 

In the present section, we will introduce these techniques in slightly more elementary 
settings. 

\subsection{Inverse limits and Galois theory}

Our approach can be thought of as an elaborate version of van Kampen's theorem. 
To begin, let us recall the setup of this. 
Let $X$ be a topological space, and 
let $U, V \subset X$ be open subsets which cover $X$. In this case, the diagram
\[ \xymatrix{
U \cap V \ar[d] \ar[r] &  U \ar[d]  \\
V \ar[r] &  X
},\]
is a homotopy pushout. In order to give a covering space $Y \to
X$, it suffices to give a 
covering space $Y_U \to U$, a covering space $Y_V \to V$, and an isomorphism
$Y_U|_{U \cap V} \simeq Y_V|_{U \cap V}$ of covers of $U \cap V$. In other
words, the diagram of categories
\begin{equation} \label{vankampeneq}
 \xymatrix{
\cov_X \ar[d]  \ar[r] &  \cov_U \ar[d]  \\
\cov_{V} \ar[r] &  \cov_{U \cap V}
},\end{equation}
is cartesian, where for a space $Z$, $\cov_Z$ denotes the category of
topological covering spaces of $Z$.  It follows that the dual diagram on fundamental \emph{groupoids}
\[ \xymatrix{
\pi_{\leq 1}(U \cap V) \ar[d] \ar[r] & \pi_{\leq 1}(V) \ar[d] \\
\pi_{\leq 1}(V) \ar[r] &  \pi_{\leq 1}(X)
}\]
is, dually, \emph{cocartesian.}
In particular, van Kampen's theorem is a formal consequence of descent theory
for covers. 

As a result, one can hope to find analogs of van Kampen's theorem in other
setting. For instance, if $X$ is a \emph{scheme} and $U, V \subset X$ are open subschemes,
then descent theory implies that the diagram
\eqref{vankampeneq} (where $\cov$ now refers to \emph{finite} \'etale covers) is
cartesian, so the dual diagram on \'etale fundamental groupoids is cocartesian. 

Our general approach comes essentially from the next result: 

\begin{proposition} 
Let $K$ be a simplicial set and let $p\colon  K \to \shot$ be a functor to
the $\infty$-category $\shot$ of stable homotopy theories. 
Then we have a natural equivalence in $\galcat$,
\begin{equation} \label{clgw} 
\clgw \left( \varprojlim_K p \right) \simeq \varprojlim_{k \in K} \clgw( p(k) )
.  \end{equation}
\label{vkweak}
\end{proposition} 

\newcommand{\tors}{\mathrm{Tors}}
\begin{proof} 
The statement that \eqref{clgw} is an equivalence equates to the statement that for
any finite group $G$, to give a $G$-torsor in the stable homotopy
theory $\varprojlim_K p$ is equivalent to
giving a compatible family of $G$-torsors in $p(k), k \in K$. (Recall, however, from
\Cref{limgalcat} that infinite limits in $\galcat$ exist, but they do not
commute with the restriction $\galcat \to \cati$ in general.) We observe that
we have a natural functor from the left-hand-side of \eqref{clgw} to the
right-hand-side which is fully faithful (as both are subcategories of the
$\infty$-category of commutative algebra objects in $\varprojlim_K p$), so that 
the functor
\[ \tors_G\left(  \clgw \left( \varprojlim_K p \right)\right)
\to \varprojlim_{k \in K} \tors_G(\clgw( p(k) ))
\]
is fully faithful.

We need to show that if $A \in \mathrm{Fun}(BG, \clgw(
\varprojlim_K p))$ has the property that its image in $\mathrm{Fun}(BG,
\clgw(p(k)))$ for each $k \in K$ is a $G$-torsor, then it is a
$G$-torsor to begin with. 
However, $A$ is dualizable, since it is dualizable locally (cf.
\cite[Prop. 4.6.1.11]{higheralg}), and it is faithful,
since it is faithful locally, i.e., at each $k \in K$. The map $A \otimes A \to \prod_G A$ is an
equivalence since it is an equivalence locally, and putting these together, $A$
is a $G$-torsor. 
\end{proof} 

In the case where we work with finite covers, rather than weak finite covers, additional finiteness hypotheses
are necessary. 

\begin{proposition} 
\label{vankampen}
Let $K$ be a simplicial set and let $p\colon  K \to \tring$ be a functor. 
Then we have a natural fully faithful inclusion 
\begin{equation} \label{clgf} \clgf( \varprojlim_K p(k)) \to \varprojlim_K
\clgf( p(k)),  \end{equation}
which is an equivalence if $K$ is finite. 
\end{proposition} 
\begin{proof} 
Since both sides are subcategories of $\clg( \varprojlim_K p(k)) =
\varprojlim_K \clg(p(k))$, the fully faithful inclusion is evident. The
main content
of the result is that if $K$ is 
finite, then the inclusion is an equivalence. In other words, we want to show
that 
given a commutative algebra object in $\varprojlim_K p(k)$ which becomes a
finite cover upon restriction to each $p(k)$, then it is a finite cover in the
inverse limit. 
Since both sides of \eqref{clgf} are Galois categories (thanks to
\Cref{finlimgal}), it suffices to show that $G$-torsors on either side are
equivalent. In other words, given a compatible diagram of $G$-torsors in the
$\clgf(p(k))$, we want the induced diagram in $\clg(\varprojlim_K p(k))$ to
be a finite cover. 

So let $A \in \mathrm{Fun}(BG, \clg( \varprojlim_K p))$ be such that its
evaluation at each vertex $k \in K$ defines a $G$-torsor in $\clgf(
p(k))$. We need to show that $A \in \clgf( \varprojlim_K p)$. For this, in view
of \Cref{Gtorsor}, it suffices to show that $A$ admits descent. But this
follows in view of \Cref{descfinloc} and the fact that the image of $A$ in
each $k \in K$ admits descent in the stable homotopy theory $p(k)$. 
\end{proof} 

Using the Galois correspondence, one finds: 
\begin{corollary} \label{vk2}
In the situation of \Cref{vankampen} or \Cref{vkweak}, we have an equivalence
in $\pro( \fgp)$:
\begin{equation} 
\label{VK}
\varinjlim_K \pi^{\mathrm{weak}}_{\leq 1} p(k) \simeq \pi_{\leq
1}^{\mathrm{weak}}( \varprojlim_K p(k)), \quad 
\varinjlim_K \pi_{\leq 1} p(k) \simeq \pi_{\leq 1}( \varprojlim_K p(k)).
\end{equation} 
\end{corollary} 

For example, let $U, V \subset X$ be open subsets of a scheme $X$. Then we have
an equivalence
\[ \qcoh(X) \simeq \qcoh(U) \times_{\qcoh(U \cap V)} \qcoh(V),  \]
by descent theory. The resulting homotopy pushout diagram that one obtains on
fundamental groupoids (by \eqref{VK}) is the 
van Kampen theorem for open immersions of schemes. 

Using this, one can also obtain a van Kampen theorem for gluing \emph{closed}
immersions of schemes. 
For simplicity, we state the result for commutative rings. Let $A', A, A''$ be
(discrete) commutative rings and consider \emph{surjections} 
$A' \twoheadrightarrow A, A'' \twoheadrightarrow A$. 
In this case, one has a pull-back square (as we recalled in \Cref{modclosed})
\[ \xymatrix{
\md^\omega(A' \times_A A'') \ar[d] \ar[r] & \md^\omega(A') \ar[d] \\
\md^\omega(A'') \ar[r] &  \md^\omega(A)
}.\]
Note that the analog without the compactness, or more generally connectivity, hypothesis
would fail.
Using \eqref{VK}, and the observation that the Galois groupoid depends only on
the dualizable objects, we obtain the following well-known corollary: 
\begin{corollary} 
We have a pushout of profinite groupoids $$\pi_{\leq 1}^{\mathrm{et}}( \spec (A' \times_A A'')) \simeq
\pi_{\leq 1}^{\mathrm{et}}( 
\spec A') \sqcup_{\pi_{\leq 1}^{\mathrm{et}}( \spec A)}
\pi_{\leq 1}^{\mathrm{et}}( \spec
A'').$$
\end{corollary} 

This result is one expression of the intuition that $\spec (A' \times_A A'')$
is obtained by ``gluing together'' the schemes $\spec A', \spec A''$ along the closed
subscheme $\spec A$. This idea in derived algebraic geometry has been studied extensively in
\cite{DAGIX}. 

These ideas are often useful even in cases when one can only \emph{approximately}
resolve a stable homotopy theory as an inverse limit of simpler ones; one can
then obtain \emph{upper bounds} for Galois groups. 
For example, let $K$ be a simplicial set, and consider 
a diagram $f\colon  K \to \einf$. Let $A = \varprojlim_K f(k)$. In this case, one has
always has a functor
\[ \md(A)\to \varprojlim_K \md({f(k)}),  \]
which is \emph{fully faithful} on the perfect $A$-modules since the right
adjoint preserves the unit. If $K$ is finite, it
is fully faithful on all of $\md(A)$. 
It follows that, \emph{regardless} of any finiteness hypotheses on $K$, there are fully faithful inclusions 
\begin{equation} \label{variousinc} \clgf(\md(A)) \subset  \clgf (
\varprojlim_K \md({f(k)})) 
\subset
\varprojlim_K \clgf(
\md({f(k)}))
.  \end{equation}
We will explore the interplay between these different Galois categories in the
next section. 
They can be used to give {upper bounds} on the Galois group
of $A$ since fully faithful inclusions of connected Galois categories are
dual to \emph{surjections}  of profinite groups.

\subsection{$\infty$-categories of local systems}
\label{subseclocsys}
In this subsection, we will introduce the first example of the general van
Kampen approach (\Cref{vankampen}), for the case of a \emph{constant} functor. 

Let $X$ be a connected space, which we consider as an $\infty$-groupoid. 
Let $(\mathcal{C}, \otimes, \mathbf{1})$ be a stable homotopy theory, which
we will assume connected for simplicity. 

\begin{definition}
The functor category $\mathrm{Fun}(X, \mathcal{C})$ 
acquires the structure of a symmetric monoidal $\infty$-category via
the ``pointwise'' tensor product. 
We will call this the $\infty$-category of \textbf{$\mathcal{C}$-valued local
systems} on $X$ and denote it by $\loc_X(\mathcal{C})$. 
\end{definition}

This is a special case of the van Kampen setup of the previous section, when we
are considering a functor from $X$ to $\tring$ or $\shot$ which is
constant with value $\mathcal{C}$. 
This means that, with no conditions whatsoever, we have
\[ \pw( \loc_X( \mathcal{C})) \simeq \widehat{\pi_1 X }
\times\pw(\mathcal{C}),   \]
in view of \Cref{vkweak}, where $\widehat{\pi_1 X}$ denotes the profinite
completion of the fundamental group $\pi_1 X$. Explicitly, 
given a functor $f\colon  X \to \finset$, 
we obtain (by mapping into $\mathbf{1}$) a local system in $\clg(\mathcal{C})$
parametrized by $X$. These are always weak finite covers in
$\loc_X(\mathcal{C})$, and these come from finite covers of $X$ or local
systems of finite sets on $X$. 
Given weak finite covers in $\mathcal{C}$ itself, we can take the constant
local systems at those objects to obtain weak finite covers in
$\loc_X(\mathcal{C})$. 

If, further, $X$ is a {finite} CW complex, it follows that 
\[  \pi_1( \loc_X( \mathcal{C})) \simeq \widehat{\pi_1 X }
\times\pi_1(\mathcal{C}) , \]
in view of \Cref{vankampen}. 
We will use this to begin describing the Galois theory of a basic class of
nonconnective $\e{\infty}$-rings, the cochain algebras on connective ones. 

In particular, let $\mathcal{C} = \mod(E)$ for an $\e{\infty}$-algebra $E$, so
that we can regard $\loc_X( \md(E)) = \mathrm{Fun}(X, \mod(E))$ as parametrizing ``local
systems of $E$-modules on $X$.'' 
The unit object in $\loc_X( \md(E))$ has endomorphism $\e{\infty}$-ring given
by the cochain algebra $C^*(X; E)$. 
Therefore, we have an adjunction
of stable homotopy theories
\[ \mod( C^*(X; E)) \rightleftarrows \loc_X( \md(E)),  \]
between modules over the $E$-valued cochain algebra $C^*(X; E)$ and 
$\loc_X( \md(E))$, where the right adjoint $\Gamma$ takes the global sections
(i.e., inverse limit) over $X$. 
The left adjoint is fully faithful when restricted to the perfect $C^*(X;
E)$-modules and in general if $\mathbf{1}$ is compact in $\loc_X( \md(E))$. 
Therefore, we get surjections of fundamental groups
\begin{equation} \label{somefundgps}
\widehat{\pi_1 X } \times \pi_1( \md(E))
\simeq \pw( \loc_X( \md(E))) \twoheadrightarrow
\pi_1 (\loc_X( \md(E)))  \twoheadrightarrow 
\pi_1( \md(C^*(X; E))).
\end{equation}

In this subsection and the next, we will describe the objects and maps in
\eqref{somefundgps} in some specific instances. 

\begin{example} 
If $X$ is simply connected, then this map is an isomorphism, given the natural
section $\md(E) \to \loc_X(\md(E))$ which sends an $E$-module to the constant
local system with that value, so $E$ and $C^*(X; E)$ have the same
fundamental group.  
\end{example}

 Suppose $X$ has the homotopy type of a \emph{finite} CW complex, so
that the functor $\Gamma$ is obtained via a finite homotopy limit and in
particular commutes with all homotopy colimits. 
In this case, as we mentioned earlier, the unit object in $\loc_X( \md(E))$ is compact, so that the map 
$\pw( \loc_X( \md(E))) \to \pi_1( \loc_X(\md(E)))$ is an isomorphism. 
In this case, the entire problem 
boils down to understanding the image of the fully faithful,
colimit-preserving functor $\md( C^*(X; E)) \to
\loc_X( \md(E))$. 

By definition, $\md( C^*(X; E))$ is generated by the unit object, so its image
in $\loc_X( \md(E))$ consists of the full subcategory of $\loc_X(
\md(E))$ generated by the unit
object, which is the \emph{trivial} constant local system. In particular, we
should think of $\md( C^*(X; E)) \subset \loc_X( \md(E))$ as the
``ind-unipotent'' local systems of $E$-modules parametrized by $X$. 
We can see some of that algebraically. 

\begin{definition} 
Let $A$ be a module over a commutative ring $R$ and let $G$ be a group acting
on $A$ by $R$-endomorphisms. We say
that the action is \textbf{unipotent} if there exists a finite filtration
of $R$-modules
\[ 0 \subset A_1 \subset A_2 \subset \dots \subset A_{n-1} \subset A_n = A,  \]
which is preserved by the action of $G$, such that the $G$-action on each
$A_i/A_{i-1}$ is trivial. We say that the $G$-action is \textbf{ind-unipotent} if
$A$ is a filtered union of $G$-stable submodules $A_\alpha \subset A$ such
that the action of $G$ on each $A_\alpha$ is unipotent. 
\end{definition}

\begin{proposition} 
\label{unipotentthing}
Let $X$ be a connected space. Consider an object $M$ of $\loc_X(\md(E))$ and let
$M_x$ be the underlying $E$-module for some $x \in X$. 
Suppose $M$ belongs to the localizing subcategory of $\loc_X(\md(E))$ generated 
by the unit. 
Then,  the action of $\pi_1(X, x)$ on each
$\pi_0E$-module $\pi_k (M_x)$ is
ind-unipotent. 

Conversely, suppose $E$ is connective. Given $M \in \loc_X( \md(E))$ 
such that the monodromy action of $\pi_1(X, x)$ on each
$\pi_k (M_x)$ is ind-unipotent, 
then if $M$ is additionally $n$-coconnective
for some $n$ and if $X$ is  a finite CW complex, we have $M \in \md( C^*(X; E))
\subset \loc_X( \md(E))$. 
\end{proposition} 
\begin{proof} 
Clearly the unit object of $\loc_X( \md(E))$ has unipotent action of
$\pi_1(X, x)$ on its homotopy groups: the monodromy
action by $ \pi_1(X, x)$ is trivial. 
The collection of objects of $\loc_X(\md(E))$ with ind-unipotent action of
$\pi_1(X, x)$ is seen to be a localizing subcategory using long exact sequences.
The first assertion follows. 

For the final assertion, since $X$ is a finite CW complex, the functor 
$\md( C^*(X; E)) \to \loc_X(\md(E))$ is fully faithful and commutes with
colimits. 
We can write $M$ as a colimit of the local systems
of $E$-modules
\[   0   \simeq \tau_{\geq n} M \to \tau_{\geq n-1} M \to \tau_{\geq n-2} M
\to \dots ,  \]
where each term in the local system has only finitely many homotopy groups. It suffices
to show that each $\tau_{\geq k} M$ belongs to $\md( C^*(X; E)) \subset \loc_X(
\md(E))$. Working inductively, one reduces to the case where $M$ itself has a
single nonvanishing homotopy group (say, a $\pi_0$) with ind-unipotent action
of $\pi_1(X, x)$. Since the subcategory of $\loc_X(\md(E))$ consisting of local
systems $M$ with $\pi_*(M_x) = 0$ for $\ast \neq 0$ is an ordinary category,
equivalent to the category of local systems of $\pi_0 E$-modules on $X$, our
task is one of algebra. One reduces (from the algebraic definition of
ind-unipotence) to showing that if $M_0$ is a $\pi_0 E$-module, then the
induced object in $\loc_X( \md(E))$ with trivial $\pi_1(X, x)$-action belongs
to $\md(C^*(X; E))$. However, this object comes from the $C^*(X; E)$-module
$C^*(X; \tau_{\leq 0} E) \otimes_{\pi_0 E} M_0$.   
\end{proof} 

\begin{remark} 
\label{rems1}
Suppose $X$ is \emph{one-dimensional}, so that $X$ is a wedge of
finitely many circles.  Then, for any $E$, any $M \in \loc_X(\md(E))$ such that
the action of  $\pi_1(X, x)$ is ind-unipotent on $\pi_*(M_x)$
belongs to the image of $\md( C^*(X; E)) \to \loc_X( \md(E))$. In other words,
one needs no further hypotheses on $E$ or $M_x$. 

To see this, 
we need to show (by \Cref{luriess}) that the inverse limit functor
\[ \Gamma = \varprojlim_X\colon  \loc_X( \md(E)) \to \md( C^*(X; E)),  \]
is \emph{conservative} when restricted to those local systems with the above
ind-unipotence property on homotopy groups. 
Recall that one has a spectral sequence
\[  E_2^{s,t} = H^s( X; \pi_{t}M_x) \implies \pi_{t-s} \Gamma(X, M),   \]
for computing 
the homotopy groups of the inverse limit. The $s = 0$ line of the $E_2$-page is
\emph{never} zero if the action is ind-unipotent  unless $M = 0$: there are
always fixed points for the action of $\pi_1(X, x)$ on $\pi_*(M_x)$. 
If $X$ is one-dimensional, the spectral sequence degenerates at $E_2$ for 
dimensional reasons; this forces the inverse limit $\varprojlim_X M$ to be
nonzero unless $M = 0$. 
\end{remark} 

As we saw earlier in this subsection, in order to construct finite covers of the unit object in
$\loc_X( \md(E))$, we can consider a local system of finite sets
$\left\{Y_x\right\}_{x \in X}$ on $X$ (i.e., a finite cover of $X$), and
consider the local system $\{C^*(Y_x; E)\}_{x \in X}$ of $\e{\infty}$-algebras
under $E$. 
The induced object in $\loc_X( \md(E))$ will generally not be unipotent in this
sense. In fact, unless there is considerable torsion, this will almost never be
the case. 

For example, suppose $G$ is a finite group, and let $R$ be a commutative ring.
Consider the $G$-action on $\prod_G R$. The group action is ind-unipotent if
$G$ is a $p$-group (for some prime number $p$) where $p$ is nilpotent in $R$. 
\begin{proof} 
Suppose $q \mid |G|$ and $q$ is not nilpotent in $R$, but the $G$-action on
$\prod_G R$ is ind-unipotent. It follows that we can invert $q$ and, after some
base extension, assume that $R$ is  a \emph{field} with $q \neq 0$. We can
even assume $\zeta_q \in R$. We need to
show that the standard representation is not ind-unipotent when $q \mid |G|$;
this follows from restricting $G$ to $\mathbb{Z}/q \subset G$, and observing that 
various nontrivial one-dimensional characters occur and these must map trivally
into any unipotent representation. 

Conversely, if $G$ is a $p$-group and $p$ is nilpotent in $R$, then by
filtering $R$, we can assume $p = 0 $ in $R$. Now in fact \emph{any}
$R[G]$-module is ind-unipotent, because the augmentation ideal of $R[G]$ is
nilpotent. 
\end{proof} 

\begin{corollary} 
\label{awayfromp}
Suppose $p$ is not nilpotent in the $\e{\infty}$-ring $R$. Then the 
surjection $\widehat{\pi_1 X } \times \pi_1 \md(E) \twoheadrightarrow \pi_1
\md(C^*(X; E)$ factors through $\widehat{\pi_1 X}_{p^{-1}}$ where
$\widehat{\pi_1 X}_{p^{-1}}$ denotes the profinite completion away from $p$. 
\end{corollary}

\begin{corollary} 
If $R$ is a $\e{\infty}$-ring such that $\mathbb{Z} \subset \pi_0 R$, then 
the map $\pi_1 \md(R) \to \pi_1 \md( C^*(X; R))$ is an isomorphism of profinite
groups. 
\end{corollary}

\begin{remark} 
In $K(n)$-local stable homotopy theory, the 
comparison question between modules over the cochain $\e{\infty}$-ring and
local systems has been studied in \cite[sec. 5.4]{ambidexterity}. 
\end{remark}

Putting these various ideas together, it is not too hard to prove the following
result, whose essential ideas are contained in \cite[Proposition
5.6.3]{rognes}. 
Here $\widehat{\pi_1 X}_p$ denotes the pro-$p$-completion of $\pi_1 X$. 
\begin{theorem} 
\label{padicgalois}
Let $X$ be a finite CW complex. Then if $R$ is an $\e{\infty}$-ring with $p$
nilpotent and such that $\pi_i R = 0$ for $i \gg 0$ (e.g., a field of
characteristic $p$), then the natural map
\begin{equation} \label{pmap}
\widehat{\pi_1 X}_p \times \pi_1 \md(R) \to \pi_1 \md( C^*(X; R))
\end{equation} 
is an isomorphism. 
\end{theorem} 

\begin{proof} 
By \Cref{awayfromp}, the natural map $\widehat{\pi_1 X} \times \pi_1 \md(R)
\twoheadrightarrow \pi_1 \md( C^*(X; R))$ does in fact factor through 
the quotient of the source where $\widehat{\pi_1 X }$ is replaced by its
pro-$p$-completion. It suffices to show that the induced map 
\eqref{pmap} is an isomorphism. Equivalently, we need to show that if $Y \to X$
is a finite $G$-torsor for $G$ a $p$-group, then $C^*(X; R) \to C^*(Y; R)$ is a
faithful $G$-Galois extension.
Equivalently, we need to show that if $\left\{Y_x\right\}_{x \in X}$ is the
local system of finite sets defined by the finite cover $Y \to X$, then the
local system of $R$-modules $\left\{C^*(Y_x; R)\right\}_{x \in X}$ 
(which gives a $G$-Galois cover of the unit in $\loc_X( \md(R))$)
actually belongs to
the image of $\md( C^*(X; R))$. However, this is a consequence of
\Cref{unipotentthing} because the monodromy action is by elements of the
$p$-group $G$. Any $G$-module over a ring with $p$ nilpotent is ind-unipotent. 
\end{proof}

\begin{remark} 
Let $Y \to X$ be a map of spaces, and let $R$ be as above. Then there are two natural local systems of
$R$-module spectra on $X$ that one can construct: 
\begin{enumerate}
\item The object of $\mathrm{Loc}_X(\mod(R)))$ obtained from the $C^*(X;
R)$-module $C^*(Y; R)$, i.e., the local system $C^*(Y; R) \otimes_{C^*(X; R)}
C^*(\ast; R)$ which is a local system as $\ast$ ranges over $X$. 
\item Consider the fibration $Y \to X$ as a local system of spaces $\{Y_x\}
$ on $X$, $x \in X$, and apply $C^*(\cdot; R)$ everywhere. 
\end{enumerate}
In general, these local systems are not the same: they are the same only if the
$R$-valued
\emph{Eilenberg-Moore spectral sequence} for the square
\[ \xymatrix{
Y_x \ar[d] \ar[r] &  Y \ar[d]  \\
\left\{x\right\} \ar[r] &  X
},\]
converges, for every choice of basepoint $x \in X$. 
This question can be quite subtle, in general. \Cref{padicgalois} is
essentially equivalent to the convergence of the $R$-valued Eilenberg-Moore
spectral sequence when $Y \to X$ is a $G$-torsor for $G$ a $p$-group. This is
the approach taken by Rognes in \cite{rognes}. 
\end{remark}

Finally, we close with an example suggesting further questions. 

\begin{example} 
The topological part of the Galois group of $C^*(S^1; \mathbb{F}_p)$ is
precisely $\widehat{\mathbb{Z}}_p$. The Galois covers come from the maps
\[ C^*(S^1; \mathbb{F}_p) \to C^*(S^1; \mathbb{F}_p),  \]
dual to the degree $p^n$ maps $S^1 \to S^1$. 
This would not work over the sphere $S^0$ replacing $\mathbb{F}_p$, in view
of \Cref{awayfromp}. However, this \emph{does} work in $p$-adically completed
homotopy theory.  

Let $\sp_p$ be the $\infty$-category of $p$-complete (i.e., $S^0/p$-local) spectra, and let
$\widehat{S}_p$ be the $p$-adic sphere, which is the unit of $\sp_p$. The map
$C^*(S^1; \widehat{S}_p) \to C^*(S^1; \widehat{S}_p)$ which is dual to the
degree $p$ map $S^1 \to S^1$ is a 
$\mathbb{Z}/p$-weak Galois extension in $\sp_p$. 
In particular, it will follow that the weak Galois group of $\sp_p$ is
the product of $\widehat{\mathbb{Z}_p}$  with that of $\sp_p$ itself. 

To see this, 
note that we have a fully faithful embedding
\[ L_{S^0/p} \md( C^*(S^1; \widehat{S}_p)) \simeq \md_{\sp_p}( C^*(S^1;
\widehat{S}_p)) \subset \loc_{S^1}( \sp_p).   \]
In $\loc_{S^1}( \sp_p)$, we need to show that the local system of $p$-complete
spectra obtained from the cover $S^1 \stackrel{p}{\to} S^1$ actually belongs to the
subcategory of $\loc_{S^1}( \sp_p)$ generated under colimits by the unit
(equivalently, by the constant local systems). 

In order to prove this claim, it suffices to prove the analog after quotienting
by $p^n$ for each $p$, since for any $p$-complete spectrum $X$, we have
\[ X \simeq \Sigma^{-1} L_{S^0/p} ( \varinjlim_n (X \otimes S^0/p^n)),   \]
as the colimit $\varinjlim_n S^0/p^n$ (where the successive maps are
multiplication by $p$) 
has $p$-adic completion given by the suspension of the $p$-adic sphere. 
But on the other hand, we can apply \Cref{rems1} to the cofiber of $p^n$ on our
local system, since an order $p$ automorphism on a $p$-torsion abelian group
is always ind-unipotent. 

By contrast, the analogous assertion would fail if we worked in the setting of
\emph{all} $C^*(S^1; \widehat{S}_p)$-modules (not $p$-complete ones): the
(weakly) Galois covers constructed are only Galois after $p$-completion. This
follows because $C^*(S^1; \widehat{S}_p)$ has coconnective rationalization, and 
all the Galois covers of it are \'etale (as we will show in
\Cref{coconnectivecov}). 
\end{example}

\subsection{Stacks and finite groups}

To start with, let $k$ be a separably closed field of characteristic $p$ and let $G$ be a finite group.
Consider the stable homotopy theory $\md_G(k)$ of $k$-module spectra equipped
with an action of $G$, or equivalently the $\infty$-category $\loc_{BG}(
\md(k))$ of local systems of $k$-module spectra on $BG$. We will explore the Galois theory of $\md_G(k)$ and the
various inclusions \eqref{variousinc}.

\begin{theorem} 
\label{Grepgal}
Let $k$ be separably closed of characteristic $p$. 
$\pw( \md_G(k)) \simeq G$ but $\pi_1(\md_G(k))$ is the quotient of $G$ by the
normal subgroup generated by the order $p$  elements. 
\end{theorem} 
\begin{proof} 
The assertion of $\pw( \md_G(k))$ is immediate: the weak ``Galois closure''
(i.e., maximal connected object in the Galois category) of the unit
in $\md_G(k)$ is $\prod_G k$, thanks to \Cref{vkweak}. The more difficult part of the result concerns
the (non-weak) Galois group. 

Any finite cover $A \in \clg( \md_G(k))$ must be given by an action of $G$ on an underlying
$\e{\infty}$-$k$-algebra which must be $\prod_S k$ for $S$ a finite set; $S$
gets a natural $G$-action, which determines everything. 
In particular, we get that $A$ must be a product of copies of $\prod_{G/H} k $. 
We need to determine which of these are actually finite covers. We can always
reduce to the Galois case, so given a surjection $G \twoheadrightarrow G'$, we
need a criterion for when $\prod_{G'} k \in \clg( \md_G(k))$ is a finite cover. 

Fix an order $p$ element $g \in G$. We claim that if $\prod_{G'} k \in
\clg(\md_G(k))$ is a finite cover, then $g$ must map to the identity in $G'$. In fact,
otherwise, we could restrict to $\mathbb{Z}/p \subset G$ to find (after
inverting an idempotent of the restriction) that $\prod_{\mathbb{Z}/p} k$ would be a finite cover
in $\md_{\mathbb{Z}/p}(k)$. This is impossible since 
$( \prod_{\mathbb{Z}/p} k)^{h \mathbb{Z}/p} \simeq k$ while $k^{h
\mathbb{Z}/p}$ has infinitely many homotopy groups; thus the unit cannot be in
the thick $\otimes$-ideal generated by 
$\prod_{\mathbb{Z}/p} k)$. It follows from this that if $\prod_{G'} k$ is a
finite cover in $\md_G(k)$, then every order $p$ element must map to the
identity in
$G'$. 

Conversely, suppose $G \twoheadrightarrow G'$ is a surjection annihilating
every order $p$ element. We claim that $\prod_{G'} k$ is a finite cover in
$\md_G(k)$. Since it is a $G'$-Galois extension of the unit, it suffices to
show that it is descendable by \Cref{Gtorsor}. For this, by the
Quillen stratification theory (in particular, \Cref{BC}), one can check this
after restricting to an elementary abelian $p$-subgroup. But after such a
restriction, our commutative algebra object becomes a finite product of copies
of the unit. 
\end{proof}

\begin{corollary} 
\label{galhG}
Let $k$ be a separably closed field of characteristic $p>0$. The Galois group $C^*(BG; k) \simeq
k^{hG}$ is given by the quotient of the pro-$p$-completion of $G$ by the order
$p$ elements in $G$. 
\end{corollary} 
By the pro-$p$-completion of $G$, we mean the maximal quotient of $G$ which is
a $p$-group. In other words, we take the smallest normal subgroup $N \subset G$
such that $|G|/|N|$ is a power of $p$,  and then take the normal subgroup $N'$ generated
by $N$ and the order $p$ elements in $G$. The Galois group of $C^*(BG; k)$ is
the quotient $G/N'$. 

\begin{proof} 
Observe that the $\infty$-category of perfect $C^*(BG; k)$-modules is a full
subcategory of the $\infty$-category $\loc_{BG}(\md(k)) \simeq \md_G(k)$ of
$k$-module spectra equipped with a $G$-action. We just showed in \Cref{Grepgal}
that the Galois group of the latter was the quotient of $G$ by the normal
subgroup generated by the order $p$ elements. In other words, the
descendable connected Galois
extensions of the unit in $\md_G(k)$ were the products $\prod_{G'} k$ where $G
\twoheadrightarrow G'$ is a surjection of groups annihilating the order $p$
elements. 

It remains to determine which of these Galois covers actually belong to the
thick subcategory generated by the unit $\mathbf{1} \in \md_G(k)$. As we have
seen, that implies that the monodromy action of $\pi_1(BG) \simeq G$ on
homotopy groups is
ind-unipotent; this can only happen (for a permutation module) if $G'$ is a
$p$-group. If $G'$ is a $p$-group, though, then the unipotence
assumption holds and $\prod_{G'} k$ does belong 
to the thick subcategory generated by the unit, so these do come from $\md(
C^*(BG; k))$. 
\end{proof} 

\begin{remark} 
Even if we were interested only in 
$\e{\infty}$-rings and their modules, for which the Galois group and weak
Galois group coincide, the proof of \Cref{galhG} makes clear the importance of
the distinction (and the theory of descent via thick subcategories) in general
stable homotopy theories. We needed thick subcategories and Quillen stratification
theory to run the argument. 
\end{remark}

\begin{example} 
\label{weakinv}
We can thus obtain a weak \emph{invariance result} for Galois groups (which we
will use later). Let $R$ be an $\e{\infty}$-ring under $\mathbb{F}_p$, given
trivial $\mathbb{Z}/p$-action. Then the
Galois theories of $R$ and $R^{h \mathbb{Z}/p}$ are the same, i.e., 
$R \to R^{h \mathbb{Z}/p}$ induces an equivalence on Galois groupoids. 
In fact, we know from $\md^\omega( R^{h \mathbb{Z}/p}) \subset \mathrm{Fun}(B
\mathbb{Z}/p, \md^\omega(R))$ that Galois extensions of $R^{h\mathbb{Z}/p}$
come either from those of $R$ or from the $\mathbb{Z}/p$-action. 
However, $\prod_{\mathbb{Z}/p} R$ is not a $\mathbb{Z}/p$-torsor because the
thick $\otimes$-ideal it generates 
in $\fun( B \mathbb{Z}/p, \md^\omega(R))$
cannot contain the unit: in fact, the
Tate construction on $R$ with $\mathbb{Z}/p$ acting trivially is nonzero, while 
the Tate construction on anything in the thick $\otimes$-ideal generated by
$\prod_{\mathbb{Z}/p} R$ is trivial. 
\end{example}

Consider now, instead of a finite group, an algebraic stack $\mathfrak{X}$.
As discussed in \Cref{qcohintro}, one has  a natural stable homotopy theory $\qcoh(\mathfrak{X})$ of
quasi-coherent complexes on $\mathfrak{X}$, obtained via
\[ \qcoh( \mathfrak{X}) = \varprojlim_{\spec  A\to \mathfrak{X}} D( \md(A)),  \]
where we take the inverse limit over all maps $\spec  A\to X$; we could
restrict to smooth maps. It follows from \Cref{connectivegal} that a 
\emph{weak finite cover}  in $\qcoh( \mathfrak{X})$ is the compatible 
assignment of a finite \'etale $A$-algebra for each map $\spec A \to
\mathfrak{X}$. In other words, the weak Galois group of $\qcoh(\mathfrak{X})$ is the
\'etale fundamental group of the stack $\mathfrak{X}$. 

If the unit object in $\qcoh(\mathfrak{X})$  is
compact, the weak Galois group and the Galois group of $\qcoh(
\mathfrak{X})$ are the same. One can make this conclusion
if $\mathfrak{X}$ is \emph{tame}, which roughly means that (if $\mathfrak{X}$
is Deligne-Mumford) the orders of the stabilizers are invertible (cf.
\cite[Theorems B and C]{HallRydh}). If this
fails, then the weak Galois group and the Galois group need not be the same,
and one gets a \emph{canonical quotient} of the \'etale fundamental group of an
algebraic stack, the Galois group of $\qcoh(\mathfrak{X})$. 

\begin{example} 
Let $G$ be a finite group, and let $\mathfrak{X} = BG$ over a separably closed field of characteristic
$p$. Then $\qcoh( \mathfrak{X})$ is precisely the $\infty$-category $\md_G(k)$
considered in the previous section. The fundamental group of $\mathfrak{X}$ is
$G$, and the main result of the previous subsection (\Cref{Grepgal}) implies that
the difference between the Galois group of $\qcoh( \mathfrak{X})$ and the \'etale fundamental group of
$\mathfrak{X}$ is precisely the order $p$ elements in the latter. 
\end{example}

Thus, we know that for any map of stacks $B \mathbb{Z}/p \to \mathfrak{X}$ where $p$ is not
invertible on $\mathfrak{X}$, the $\mathbb{Z}/p$
must vanish in the fundamental group of $\qcoh( \mathfrak{X} )$ (but not necessarily
in the fundamental group of $\mathfrak{X}$). When $\mathfrak{X} = BG$ for some
finite group, this is the \emph{only} source of the difference between two
groups. We do not know what the difference looks like in general. 

Next, as an application of these ideas, we include an example that shows that
the Galois group is a sensitive invariant of
an $\e{\infty}$-ring: that is, it can vary as the $\e{\infty}$-structure varies
within a fixed $\e{1}$-structure. 
\begin{example} 
Let $k$ be a separably closed field of characteristic $p > 0$. 
Let $\alpha_{p^2}$ be the usual rank $p^2$ group scheme over $k$  and let
$(\alpha_{p^2})^{\vee}$ be its Cartier dual, which is another infinitesimal
commutative
group scheme. 
Let $\mathbb{Z}/p^2$ be the usual constant group scheme. 
Consider the
associated classifying stacks $B \mathbb{Z}/p^2$ and $B (
\alpha_{p^2})^{\vee}$, and the associated cochain $\e{\infty}$-rings
$C^*(B \mathbb{Z}/p^2; k)$ and $C^*( B (\alpha_{p^2})^{\vee}; k)$ defined as
endomorphisms of the unit of quasi-coherent sheaves. 

Since $\alpha_{p^2}^{\vee}$ is infinitesimal, it follows that the fundamental
group of the stack $B ( \alpha_{p^2})^{\vee}$ is trivial and 
in particular that $\pi_1 \md( C^*( B ( \alpha_{p^2})^{\vee}; k))$ is trivial.
In other words, we are using the geometry of the stack to \emph{bound above}
the possible Galois group for the $\e{\infty}$-ring of cochains with values in
the structure sheaf. 
However, by \Cref{galhG}, we have $\pi_1 \md(
C^*(B\mathbb{Z}/p^2; k)) \simeq \mathbb{Z}/p$. 

Finally, we note that there is a canonical equivalence of $\e{1}$-rings between
the two cochain algebras. In fact, the $k$-linear \emph{abelian} category of
(discrete) quasi-coherent sheaves on $B
\mathbb{Z}/p^2$ can be identified with the category of modules over the group
ring $k[\mathbb{Z}/p^2]$, which is noncanonically isomorphic to the algebra
$k[x]/(x^{p^2})$. The $k$-linear \emph{abelian} category of discrete
quasi-coherent sheaves on $B (
\alpha_{p^2})^{\vee}$ is identified with the category of modules over the ring
of functions on $\alpha_{p^2}$, which is $\mathbb{F}_p[x]/x^{p^2}$. In
particular, we get a $k$-linear \emph{equivalence} between either the
abelian or derived categories of sheaves in
either case. Since the cochain $\e{\infty}$-rings we considered are (as
$\e{1}$-algebras) the endomorphism rings of the object $k$ (which is the same
representation either way), we find that they
are equivalent as $\e{1}$-algebras. 
\end{example}

\section{Invariance properties}

Let $R$ be a (discrete) commutative ring and let $I \subset R$ be a nilpotent
ideal. Then it is a classical result in commutative algebra, the
``topological invariance of the \'etale site,'' \cite[Theorem 8.3, Exp.
I]{sga1}, that the
\'etale site of $\spec R$ and the closed subscheme $\spec R/I$ are equivalent. 
In particular, given an \'etale $R/I$-algebra $\overline{R}'$, it can be lifted
\emph{uniquely} to an \'etale $R$-algebra $R'$ such that $R'
\otimes_R R/I \simeq \overline{R}'$. 

In this section, we will consider analogs
of this result for $\e{\infty}$-rings. 
For example, we will prove: 

\begin{theorem} \label{p=0}
Let $R$ be an $\e{\infty}$-algebra under $\mathbb{Z}$ with 
$p$ nilpotent in $\pi_0 R$. Then the map 
\[ R \to R \otimes_{\mathbb{Z}} \mathbb{Z}/p,  \]
induces an isomorphism on fundamental groups. 
\end{theorem}

Results such as \Cref{p=0} will be 
extremely useful for us. For example, it will be integral to our computation
 of
the Galois groups of stable module $\infty$-categories over finite groups. 
\Cref{p=0}, which is immediate in the case of $R$ \emph{connective} (thanks
to \Cref{connectivegal} together with the classical topological invariance
result), seems to
be very non-formal in the general case. 

Throughout this section, we assume that our stable homotopy theories are
\emph{connected.}

\subsection{Surjectivity properties}

We begin with some generalities from \cite{sga1}. 
We have the following easy lemma. 
\begin{lemma} 
\label{surj}
Let $G \to H$ be a morphism of profinite groups. Then the following are
equivalent: 
\begin{enumerate}
\item  $G \to H$ is surjective. 
\item For every finite (continuous) $H$-set $S$, $S$ is connected if and only
if the $G$-set obtained from $S$ by restriction is connected. 
\end{enumerate}
\end{lemma}

Let $(\mathcal{C}, \otimes, \mathbf{1})$ be a (connected) stable homotopy
theory. Given a commutative algebra object $A \in \mathcal{C}$, we 
have  functors $\clgf(\mathcal{C}) \to \clgf( \md_{\mathcal{C}}(A))
,\clgw(\mathcal{C}) \to \clgw( \md_{\mathcal{C}}(A))
$ given by
tensoring with $A$. Using the Galois correspondence, this comes from the map of
profinite groups $\pi_1( \md_{\mathcal{C}}(A)) \to \pi_1( \mathcal{C})$ by
restricting continuous representations in finite sets. 
The following is a consequence of \Cref{surj}.
\begin{proposition} \label{surjthing}
Let $A \in \clg(\mathcal{C})$ be a commutative algebra object with the
following property: given any $A' \in \clg(\mathcal{C} )$ which is a weak
finite cover,
the map 
\begin{equation} \label{idemmap} \idem(A') \to \idem(A \otimes A')
\end{equation}
is an isomorphism. Then the induced maps
\[ \pi_1(\md_{\mathcal{C}}(A)) \to \pi_1( \mathcal{C}) , 
\quad \pw( \md_{\mathcal{C}}(A)) \to \pw(\mathcal{C}),
\]
are surjections of profinite groups. 
\end{proposition} 

Thus, it will be helpful to have some criteria for when 
maps of the form \eqref{idemmap} are isomorphisms. 

\begin{definition} 
Given $A \in \clg(\mathcal{C})$, we will say that $A$ is \textbf{universally
connected} if for every $A' \in \clg(\mathcal{C})$, the map $\idem(A') \to \idem(A' \otimes A)$ in
\eqref{idemmap} is an isomorphism. 
\end{definition} 

It follows by \Cref{surjthing} that if $A$ is universally connected, then
$\pw (\md_{\mathcal{C}}(A)) \to \pw( {\mathcal{C}})$ and 
$\pi_1( \md_{\mathcal{C}}(A)) \to \pi_1 ({\mathcal{C}})$ are  surjections;
moreover, this holds after any base change in $\clg(\mathcal{C})$. That is, if
$A' \in \clg(\mathcal{C})$, then 
the map $\pi_1 ( \md_{\mathcal{C}}(A \otimes A')) \to \pi_1(
\md_{\mathcal{C}}(A'))$ is a surjection, and similarly for the weak Galois
group.

Note first that if $A$ \emph{admits descent}, then \eqref{idemmap} is always an
injection, since for any $A'$, we can recover $A'$ as the totalization of the
cobar construction on $A$ tensored with $A'$ and since 
$\idem$ commutes with limits (\Cref{idemlimit}). In fact, it thus follows that if $A$ admits
descent, then 
$\idem(A')$ is the equalizer of the two maps $\idem(A \otimes A')
\rightrightarrows \idem(A \otimes A 
\otimes A')$. 
More generally, one can obtain a weaker conclusion under weaker hypotheses: 

\begin{proposition} 
\label{faithinj}
If $A \in \clg(\mathcal{C})$ is faithful (i.e., tensoring with $A$ is a
conservative functor $\mathcal{C} \to \mathcal{C}$), then the map \eqref{idemmap} is
always an injection, for any $A' \in \clg(\mathcal{C})$. 
\end{proposition} 
\begin{proof} 
It suffices to show that if $e \in \idem(A')$ is an idempotent which maps to
zero in $\idem(A \otimes A')$, then $e$ was zero to begin with. The hypothesis
is that $A'[e^{-1}]$ becomes contractible after tensoring with $A$, and since
$A$ is faithful, it was contractible to begin with; that is, $e$ is zero. 
\end{proof} 

We thus obtain the following criterion for universal connectedness. 

\begin{proposition} 
\label{invprop}
Let  $(\mathcal{C}, \otimes, \mathbf{1})$ be a connected stable homotopy theory. Suppose
$A \in \clg(\mathcal{C})$ is an object with the properties: 
\begin{enumerate}
\item $A$ is descendable.  
\item The multiplication map $A \otimes A \to A$ is faithful. 
\end{enumerate}
Then $A$ is universally connected. 
\end{proposition} 

\begin{proof} 
We will show that if $B \in \clg(\mathcal{C})$ is {arbitrary},
then the map $\idem(B) \to \idem(A \otimes B)$ is an \emph{isomorphism}. 
Since $A$ is descendable, we know  that there is an equalizer diagram
\[ \idem(B) \to \idem(A \otimes B) \rightrightarrows \idem(A \otimes A
\otimes B).  \]
To prove the proposition, it suffices to show that the two maps $\idem(A \otimes B)
\rightrightarrows \idem(A \otimes A \otimes B)$ are equal. 

However, these
maps become equal after composing with the map $\idem(A \otimes A \otimes
B) \to \idem(A \otimes B)$ induced by the multiplication $A \otimes A \to
A$. 
Since $A \otimes A \to A$ is faithful, the map 
$\idem(A \otimes A \otimes
B) \to \idem(A \otimes B)$ is injective by \Cref{faithinj}, which thus proves the result. 
\end{proof}

\Cref{invprop} is thus almost a tautology, although the basic idea will be quite useful for us. 
Unfortunately, the hypotheses are rather restrictive. If $A$ is a local
artinian ring and $k$ the residue field, then the map $A \to k$ admits descent.
However, the multiplication map $k \otimes_A k \to k$ need not be faithful:
$k \otimes_A k$  has always infinitely many homotopy
groups (unless $A = k$ itself). Nonetheless, we can prove:

\begin{proposition} \label{surjart} Let $k$ be a field. 
Let $A$ be a connective $\e{\infty}$-ring with a map $A \to k$
inducing a surjection on $\pi_0$. Suppose $A \to k$ admits descent. Then
$A \to k$ is universally connected. 
\end{proposition} 

\begin{proof} 
Once again, we show that for any $A' \in \clg_{A/}$, the map $A'
\to A' \otimes_A k$ induces an isomorphism on idempotents. 
Since $A \to k$ is descendable, it suffices to show that the two maps
\[ \idem(A' \otimes_A k) \rightrightarrows \idem(A' \otimes_A k \otimes_A k)  \]
are the same. 
For this, we know that the two maps become the same after composition with the
multiplication map $ A ' \otimes _A (k \otimes_A k) \to A' \otimes_A k$. To
show that the two maps are the same, it will suffice to show that they are
\emph{isomorphisms.} In other words, since we have a commutative diagram
\begin{equation} \label{idemdiag} \idem( A ' \otimes_A k) \rightrightarrows \idem(A' \otimes_A k \otimes_A
k) \to \idem(A' \otimes_A k) , \end{equation}
where the composite arrow is the identity, it suffices to show that
\emph{either one} of the two maps $\idem(A' \otimes_A k) \rightrightarrows
\idem(A' \otimes_A k \otimes_A k)$ is an isomorphism.

More generally, we claim that for any $k$-algebra $R$, the map
\[ R   \to R \otimes_k (k \otimes_A k),  \]
induced by the map of $k$-algebras $k \to k \otimes_A k$, induces
an \emph{isomorphism}
on idempotents. (In \eqref{idemdiag}, this is the map that we get from free,
without using the fact that $A' \otimes_A k$ was the base-change of an
$A$-algebra.) 
Since we have a K\"unneth isomorphism, this follows from the following
purely algebraic lemma.

\begin{lemma} 
Let $R_*$ be a graded-commutative $k$-algebra and let $R'_*$ be a
graded-commutative \emph{connected} $k$-algebra: $R'_0 \simeq k$ and $R'_i = 0$
for $i < 0$. Then the natural map from idempotents in $R_*$ to idempotents in the graded
tensor product $R_*
\otimes_k R'_*$ is an isomorphism. 
\end{lemma} 
\begin{proof} 
We have a map
\[ \idem(R_*) \to \idem(R_* \otimes_k R'_*),  \]
which is injective, since the map $k \to R'_*$ admits a section in the category
of graded-commutative $k$-algebras. But the ``reduction'' map $\idem(R_* \otimes_k R'_*) \to
\idem(R_*)$ is also injective. In fact, since idempotents form a Boolean
algebra, it suffices to show that an idempotent
in $R_* \otimes_k R'_*$ that maps to zero in $R_*$ must have been zero to begin
with. However, such an idempotent would belong to the ideal $R_* \otimes_k R'_{>0}$,
which easily forces it to be zero. 
\end{proof} 

\end{proof}

\begin{example} 
\Cref{surjart} applies in the setting of an artinian ring mapping to its residue
field. However, we also know that the map $A \to A/\mathfrak{m}$ for $A$
artinian and $\mathfrak{m}$ a maximal ideal can be obtained as a finite
composition of square-zero extensions, so we could also appeal to
\Cref{squarezerogeneral} below. 
\end{example}

\subsection{Square-zero extensions}

Given the classical topological invariance of the \'etale site, the
following is not so surprising. 

\begin{proposition} 
\label{trivsqzero}
If $A$ is an $\e{\infty}$-ring and $M$ an $A$-module, then the natural map $A
\to A \oplus M$ (where $A \oplus M$ denotes the trivial ``square zero''
extension of $A$ by $M$), induces an isomorphism on fundamental groups. 
\end{proposition} 

This will follow from the following more general statement. 

\begin{proposition} 
Let $R$ be an $\e{\infty}$-ring with no nontrivial idempotents. Let $X$ be a
two-fold loop object in the $\infty$-category $\clg_{R//R}$ of $\e{\infty}$-$R$-algebras
over $R$. Then the map $R \to X$ induces an isomorphism on fundamental groups. 
\end{proposition} 
Note that a one-fold delooping is insufficient, because of the example of cochains on
$S^1$ (cf. \Cref{padicgalois}). 
\begin{proof} 
In view of \Cref{idemlimit}, we see that $X$ has no nontrivial idempotents. 
Next, observe that we have maps $R \to X \to R$ by assumption, so that, at the
level of fundamental groups, we get a section of the map $\pi_1( \md(X)) \to
\pi_1( \md(R))$. In particular, the map $\pi_1( \md(X)) \to \pi_1( \md(R))$ is
surjective. We thus need to show that the map $\pi_1( \md(R)) \to \pi_1( \md(X))$
(coming from $X \to R$) is also surjective, which we can do via
\Cref{surjthing}. 

To see that, suppose $X \simeq \Omega^2 Y$ where $Y$ is an object in
$\clg_{R//R}$. We want to show that the fundamental group of $\md(X)$ is surjected
onto by that of $\md(R)$. Consider the pull-back diagram
of $\e{\infty}$-algebras,
\[ \xymatrix{
\Omega^2 Y \ar[d] \ar[r] &  R \ar[d]  \\
R \ar[r] &  \Omega Y
}.\]
Using \Cref{idemlimit}  again, 
we find that $\Omega Y$ has no nontrivial idempotents. Therefore,
we have  maps 
\[ \pi_1(\md(R)) \to \pi_1( \md(R) \times_{\md({\Omega Y})}
\md(R)) )
\twoheadrightarrow \pi_1( \md({\Omega^2 Y})) . \]
The second map is a surjection since it comes from a fully faithful inclusion
of stable homotopy theories
$\md( \Omega^2 Y) \subset \md(R) \times_{\md( \Omega Y)} \md(R)$. Since $
\Omega Y$ has no nontrivial idempotents, 
$\pi_1 \md( \Omega Y)$ receives a map from $\pi_1 \md(R)$ and we have
$\pi_1(   \md(R) \times_{\md({\Omega Y})}
\md(R)) ) \simeq \pi_1( \md(R)) \sqcup_{\pi_1 ( \md( \Omega Y))} \pi_1
(\md(R))$. This implies that the first map is a surjection too, as desired. 
\end{proof}

We can also consider the behavior of the Galois group under (not
necessarily trivial) square-zero extensions. 
Recall (see \cite[sec. 7.4.1]{higheralg}) that these are obtained as follows. Given an
$\e{\infty}$-ring $A$ and an $A$-module $M$, for every map $\phi\colon  A \to A \oplus M$
in $\clg_{/A}$, we can form the pull-back
\[ \xymatrix{
A' \ar[d] \ar[r] &  A \ar[d]^{0} \\
A \ar[r]^{\phi} &  A \oplus M
},\]
where 
$0\colon  A \to A \oplus M$ is the standard map (informally, $a \mapsto (a, 0)$). 
The resulting map $A' \to A$ is referred to as a \textbf{square-zero
extension} of $A$, by $\Omega M$. 
\begin{corollary} 
\label{sqzerosurj}
\label{squarezerogeneral}
Notation as above, the map $\pi_{1} \md( A') \to \pi_{ 1} \md(A)$ is
a surjection. 
In fact, $A' \to A$ is universally connected.
\end{corollary} 
\begin{proof} 
It suffices to show that $A' \to A$ is universally connected. 
This follows from the fact that $\idem$ commutes with inverse limits, since one
checks that the two maps $A \rightrightarrows A \oplus M$ are universally
connected. 
\end{proof}

The Galois group is not invariant under arbitrary square-zero extensions.
Let $A = \mathbb{C}[x^{\pm 1}]$ where $|x|  = 0$ be the free 
$\e{\infty}$-algebra under $\mathbb{C}$ on an invertible degree zero generator (so that $A$ is discrete). 
Consider the $\mathbb{C}$-derivation $A \to A$ sending a Laurent polynomial $f(x) $ to its
derivative. 
Then, when we form the pull-back
\[ \xymatrix{
A' \ar[d] \ar[r] & A \ar[d]^0 \\
A \ar[r]^{f \mapsto (f, f')} & A \oplus A
},\]
the pull-back is given by an $\e{\infty}$-algebra $A'$ with $\pi_0 A' \simeq
\mathbb{C}, \pi_{-1} A' \simeq \mathbb{C}$, and $\pi_i A' = 0$ otherwise. The
Galois theory of this $\e{\infty}$-ring is algebraic
because this $\e{\infty}$-ring is necessarily the free $\e{\infty}$-ring on a
degree $-1$ generator, or equivalently the trivial square-zero
extension $\mathbb{C} \oplus \Omega \mathbb{C}$. So its Galois group is
trivial, by \Cref{trivsqzero}. 
However, the map $\mathbb{C} \oplus \Omega \mathbb{C} \to
\mathbb{C}[x^{\pm 1}]$ does not induce an isomorphism on Galois groups: that
of the former is trivial, while that of the latter is $\widehat{\mathbb{Z}}$. 

\subsection{Stronger invariance results}
We will now prove the main invariance results of the present section. 
\begin{theorem} 
\label{invresult}
Let $A$ be a regular local ring with residue field $k$ and maximal ideal
$\mathfrak{m} \subset A$. Let $R$ be an
$\e{\infty}$-ring under $A$ such that $\mathfrak{m}$ is nilpotent in $\pi_0 R$.
Then the natural map 
\[ R \to R \otimes_A k,  \]
induces an isomorphism on fundamental groups. 
\end{theorem} 
\begin{proof} 
We start by showing that $\pi_1(\md(R \otimes_A k)) \to \pi_1(\md(R))$ is always a
surjection; in other words, we must show that for any $\e{\infty}$-algebra $R'$
under $R$, the natural map
\begin{equation} \idem(R') \to \idem( R' \otimes_R (R \otimes_A k))
\simeq \idem(R' \otimes_A k)   \label{idema}\end{equation}
is an isomorphism. 

Since $k$ is a perfect $A$-module, it follows that $R \otimes_A k$ is a perfect 
$R$-module. Moreover, $R \otimes_A k$ is faithful as an $R$-module because
tensoring over $A$ with $k$ is faithful on the subcategory of $\md(A)$
consisting of $A$-modules whose homotopy groups are $\mathfrak{m}$-power
torsion. It follows that $R \to R \otimes_A k$ is descendable in view of
\Cref{cptdescent}. 
Therefore, the map \eqref{idema} is an injection. Since the map 
\[ k \otimes_A k \to k,  \]
is descendable, as $k \otimes_A k$ is connective with bounded homotopy groups
and $\pi_0$ given by $k$, it follows from \Cref{invprop} that (by tensoring
this with $R$) that 
$\pi_1(\md(R \otimes_A k)) \to \pi_1(\md(R))$  is a surjection. 

Consider the cobar construction
\begin{equation} \label{RAcb} R \to R \otimes_A k \rightrightarrows R \otimes_A k \otimes_A k
\triplearrows \dots,  \end{equation}
where all $\e{\infty}$-rings in question have no nontrivial idempotents. 
We will use this and descent theory to complete the proof. 

Note that we can make 
the two maps 
$\pi_{\leq 1}( \md(R \otimes_A k \otimes_A k)) \rightrightarrows 
\pi_{\leq 1}( \md(R \otimes_A k))$ into pointed maps by choosing a basepoint of
$\pi_{\leq 1}\md(R
\otimes_A k)$ and using the multiplication map $R \otimes_A (k \otimes_A k) \to
R \otimes_A k$. 
We conclude that (by descent theory and \eqref{RAcb}) that $\pi_1(\md(R))$ is the coequalizer of the two maps
\[ \pi_1 ( \md({ R \otimes_A k \otimes_A k})) \rightrightarrows \pi_1( \md({R
\otimes_A k})),  \]
choosing basepoints as above.

We claim here that the multiplication map
$R \otimes_A (k \otimes_A k) \to R \otimes_A k$
induces a \emph{surjection} on fundamental groups. 
Given this, we can construct a diagram
\[ \pi_1( \md(R \otimes_A k)) \twoheadrightarrow \pi_1( \md( R \otimes_A k
\otimes_A k)) \rightrightarrows \pi_1( \md(R \otimes_A k),  \]
where the two composites are equal. 
This completes the proof that $\pi_1 (\md(R)) \simeq \pi_1( \md(R \otimes_A
k))$, subject to the proof of surjectivity.

To prove surjectivity, we observe that $R \otimes_A k \otimes_A k \to R \otimes_A k$ 
induces a surjection on fundamental groups, in view of \Cref{surjart}, since $k
\otimes_A k \to k$ satisfies the conditions of that result; since $A$ is
regular, $k \otimes_A k$ is connective and has only finitely many nonzero
homotopy groups, so $k \otimes_A k \to k$ admits descent.

 \end{proof} 

It seems likely that \Cref{invresult} can be strengthened considerably,
although we have not succeeded in doing so. 
For example, one would like to believe that if $R$ is a discrete commutative
ring and $I \subset R$ is an ideal of square zero, then given an
$\e{\infty}$-$R$-algebra $R'$, the map $R' \to R' \otimes_R R/I$ would induce
an isomorphism on fundamental groups. We do \emph{not} know whether this is
true in general. By \Cref{sqzerosurj}, it does induce a surjection at least. 
The worry is that one does not have good control on the homotopy groups of a
relative tensor product of $\e{\infty}$-ring spectra; there is a spectral
sequence, but the filtration is in the opposite direction than what one wants.

However, in the case when the $\e{\infty}$-rings satisfy mild connectivity
hypotheses, one can prove the following much stronger result. 
\begin{theorem} \label{slightlyconn}
Suppose $R$ is a connective $\e{\infty}$-ring with finitely many homotopy
groups and $I \subset \pi_0 R$ a nilpotent ideal. Let $R'$ be an $\e{\infty}$-$R$-algebra which is $(-n)$-connective for $n
\gg 0$. Then the map $R' \to R' \otimes_R \pi_0(R)/I$ induces an isomorphism on
fundamental groups. 
\end{theorem} 
For example, one could take $I = 0$, and the statement is already nontrivial. 
We need first two lemmas:

\begin{lemma} 
\label{adamssurj}
Let $A$ be a connective $\e{\infty}$-ring and let $A'$ be an
$\e{\infty}$-$A$-algebra which is $(-n)$-connective for $n \gg 0$. Then the
natural map
\begin{equation} \label{idemaa} \idem(A') \to \idem(A' \otimes_A \pi_0 A)
\end{equation}
is an isomorphism. In particular, it follows that $\pi_1 \md({A' \otimes_A
\pi_0 A }) \to \pi_1 \md({A'})$ is a surjection. 
\end{lemma} 
\begin{proof} 
In fact, by a connectivity argument (taking an inverse limit over Postnikov
systems), the Adams spectral sequence based on the map $A \to \pi_0 A$ converges
for any $A$-module which is $(-n)$-connective for $n \gg 0$. 
In other words, we have that
\[ A' = \mathrm{Tot}\left( A' \otimes_A \pi_0 A \rightrightarrows A' \otimes_A
\pi_0 A \otimes_A \pi_0 A \triplearrows \dots \right)  ,\]
so that, since $\idem$ commutes with limits, we find that $\idem(A')$ is the
equalizer of the two maps $\idem(A' \otimes_A \pi_0 A) \rightrightarrows \idem(
A' \otimes_A \pi_0 A \otimes_A \pi_0 A)$. In particular, 
\eqref{idemaa} is always injective. Moreover, by the same reasoning, 
the multiplication map $\pi_0 A \otimes_A \pi_0 A \to \pi_0 A$ (which is also a map
from a connective $\e{\infty}$-ring to its zeroth Postnikov section) induces an
injection
\[ \idem(A' \otimes_A \pi_0 A \otimes_A \pi_0 A ) \hookrightarrow \idem( A' \otimes_A
\pi_0 A),  \]
which equalizes the two maps 
$\idem(A' \otimes_A \pi_0 A) \rightrightarrows \idem(
A' \otimes_A \pi_0 A \otimes_A \pi_0 A)$. It follows that the two maps were
equal to begin with, which proves that \eqref{idemaa} is an isomorphism. 
\end{proof} 

\begin{lemma} \label{idemnilp}
Let $A$ be a discrete $\e{\infty}$-ring and $J \subset A$ a square-zero ideal.
Then, given any $\e{\infty}$-$A$-algebra $A'$, the natural map $A' \to A'
\otimes_A A/J$ induces an isomorphism on idempotents. 
\end{lemma} 
\begin{proof} 
This is a consequence of \Cref{sqzerosurj}.
\end{proof}

\begin{proof}[Proof of \Cref{slightlyconn}] 
Let $R_0$ be the $\e{\infty}$-$R$-algebra given by $\pi_0(R)$ and consider
$R_0/I$ as well. Then we have maps $R \to R_0 \to R_0/I$ and we want to show
that, after base-changing to $R'$, the Galois groups are invariant. 
We will do this in a couple of stages following  the
proof of \Cref{invresult}.

First, suppose $I = 0$. 
Using descent along $R \to R_0$, one concludes that 
$\pi_1(\md({R'}))$ is the coequalizer of the two maps $\pi_1(\md({ R'
\otimes_R R_0 
\otimes_R R_0})) \rightrightarrows \pi_1(\md({ R' \otimes_R R_0}))$. 
We wish to claim that the two maps are equal. 
Now the multiplication map 
$R_0 \otimes_R R_0 \to R_0$ satisfies the
conditions of \Cref{adamssurj}, so  one
concludes that the map $\pi_1( \md({R' \otimes_R R_0})) \to \pi_1( \md({R'
\otimes_R R_0 \otimes_R R_0}))
$ is a surjection, which coequalizes the
two maps considered above. 
Therefore, the two maps are equal. 

Next, we need to allow $I \neq 0$. By composition $R \to \tau_{\leq 0} R \to
R_0/I$, we may assume that $R$ itself is discrete. We may also assume that $I$
is square-zero.  In this case, the map $R \to
R_0/I$ satisfies descent and is universally connected by \Cref{idemnilp}. 
Therefore, we can apply the same argument as above, to write $\pi_1( \md({R'}))$
as the coequalizer of the two maps $\pi_1( \md({R' \otimes_{R_0} R_0/I
\otimes_{R_0} R_0/I })) \rightrightarrows \pi_1( \md({R' \otimes_{R_0} R_0/I}))$.
Moreover, these two maps are the same using 
the surjection $\pi_1( \md({R' \otimes_{R_0} R_0/I} ) )\twoheadrightarrow 
\pi_1( \md({R' \otimes_{R_0} R_0/I
\otimes_{R_0} R_0/I }))$ given to us by \Cref{adamssurj}
as above. 
\end{proof}

\subsection{Coconnective rational $\e{\infty}$-algebras}

Let $k$ be a field of characteristic zero, and let $A$ be an
$\e{\infty}$-$k$-algebra such that:
\begin{enumerate}
\item $\pi_i A = 0$ for $i > 0$.  
\item The map $k \to \pi_0 A$ is an isomorphism. 
\end{enumerate}

Following \cite{DAGVIII}, we will call such $\e{\infty}$-$k$-algebras
\textbf{coconnective;} these are the $\e{\infty}$-rings which enter, for
instance, in rational homotopy theory.  In the following, we will prove: 

\begin{theorem} \label{coconnectivecov}
If $A$ is a coconnective $\e{\infty}$-$k$-algebra, then every finite cover of
$A$ is \'etale. In particular, 
\[ \pi_{ 1} \md(A) \simeq  \mathrm{Gal}(\overline{k}/k).  \]
\end{theorem} 

\begin{proof}
We will prove \Cref{coconnectivecov}
using tools from \cite{DAGVIII}. 
Namely, it is a consequence of \cite[Proposition 4.3.13]{DAGVIII} that every
coconnective $\e{\infty}$-$k$-algebra $A$ can be obtained 
as a totalization of a cosimplicial $\e{\infty}$-$k$-algebra $A^\bullet$
where $A^i$, for each $i \geq 0$,  is in the form $k \oplus V[-1]$ where $V$ is
a vector space over $k$, and this is considered as a trivial ``square zero'' extension. 
In rational homotopy theory, this assertion is dual to the statement that a
connected space can be built as a geometric realization of copies of
wedges of $S^1$. 

Now we know from \Cref{trivsqzero} that the Galois groupoid is invariant under
trivial square-zero extensions, 
so it follows that $\pi_{ 1} \md(A^i) \simeq 
\mathrm{Gal}(\overline{k}/k)$, with the finite covers arising only from
the \'etale extensions (or equivalently, finite \'etale extensions of $k$
itself). 
It follows easily from this that the finite covers in the $\infty$-category
$\mathrm{Tot} \mod( A^\bullet)$ are in natural equivalence with the finite
\'etale extensions of $k$, and 
this completes the proof, since the $\infty$-category of perfect $A$-modules
embeds fully faithfully into this totalization. 
\end{proof}

Note that the strategy of this proof is to give an \emph{upper bound} for the
Galois theory of the $\e{\infty}$-ring $A$ by writing it as an inverse limit of 
square-zero $\e{\infty}$-rings. One might, conversely, hope to use Galois
groups to prove that $\e{\infty}$-rings \emph{cannot} be built as
inverse limits of certain simpler ones. 
For example, in characteristic $p$, the example of cochain algebras shows that
the analog of \Cref{coconnectivecov} is false; in particular, one cannot write
a given coconnective $\e{\infty}$-ring in characteristic $p$ as a totalization
of square-zero extensions.

\section{Stable module $\infty$-categories}

Let $G$ be a finite group and let $k$ be a perfect field of characteristic $p>0$, where
$p$ divides the order of $G$. 
The theory of $G$-representations in $k$-vector spaces is significantly more
complicated than it would be in characteristic zero because the group ring
$k[G]$ is not semisimple: for example, the group $G$ has $k$-valued cohomology.
If one wishes to focus primarily on, for example, the cohomological information
specific to characteristic $p$, then projective $k[G]$-modules are
essentially irrelevant
and, factoring them out, one has the theory of \emph{stable module
categories} reviewed earlier in \Cref{stmodcat}. 
One obtains a compactly generated, symmetric monoidal stable $\infty$-category
$\stm_G(k)$ obtained as the $\mathrm{Ind}$-completion of the Verdier quotient of $\mathrm{Fun}(BG,
\md^\omega(k))$ by the thick $\otimes$-ideal of perfect $k[G]$-module spectra. 

Our goal in this section is to describe the Galois group of a stable module
$\infty$-category for a finite group. 
Since any element in the stable module $\infty$-category can be viewed as an ordinary
linear
representation of $G$ (for compact objects, finite-dimensional representations)
modulo a certain equivalence relation, these results ultimately come down to
concrete statements about the tensor structure on linear representations of $G$
modulo projectives. 

Our basic result (\Cref{galelem}) is that the Galois theory of a stable module
category for an \emph{elementary abelian} $p$-group is entirely algebraic. We
will use this, together with the Quillen stratification theory, to obtain a
formula for the Galois group of a general stable module $\infty$-category, and calculate
this in special cases. 

\subsection{The case of $\mathbb{Z}/p$}

Our first goal is to determine the Galois group of $\stm_V(k)$ when $V$ is
\emph{elementary abelian}, i.e. of the form $(\mathbb{Z}/p)^n$. 
In this case, recall (\Cref{keller}) that $\stm_V(k)$ is symmetric monoidally equivalent to the $\infty$-category of
modules over the Tate construction $k^{tV}$.
We will start by considering the case $V = \mathbb{Z}/p$. 

\begin{proposition} 
\label{galoistaterankone}
Let $k$ be a field of characteristic $p > 0$. 
The Galois theory of  the Tate construction $k^{t \mathbb{Z}/p}$  is algebraic. 
\end{proposition}

\begin{proof}
Without loss of generality, we can assume $k$ perfect. 
In the case $p = 2$, $k^{t \mathbb{Z}/2}$ has homotopy groups given by 
\[ k^{t \mathbb{Z}/2} \simeq k[t^{\pm 1}],  \]
where $|t| = -1$. A (simpler) version of \Cref{fieldreg} shows that any Galois extension
of $k^{t \mathbb{Z}/2}$ is \'etale, since $\pi_0$ satisfies a perfect K\"unneth
isomorphism for $k^{t \mathbb{Z}/2}$-modules and every module over $k^{t
\mathbb{Z}/2}$ is algebraically flat. 
It follows that if $k^{t \mathbb{Z}/2} \to R$ is $G$-Galois, for $G$ a finite
group, then $\pi_0 R$ is a finite $G$-Galois extension of $k$. 

The case of an odd prime is slightly more subtle. 
In this case, we have
\[ k^{t \mathbb{Z}/p} \simeq k[t^{\pm 1}] \otimes_k E(u), \quad |t| = - 2, |u| =
-1,  \]
so that we get a tensor product of a Laurent series ring and an exterior
algebra. 
Since the homotopy ring is no longer regular, we will have to show that any $G$-Galois
extension of $k^{t \mathbb{Z}/p}$ is flat at the level of homotopy groups.
We can do this by comparing with the Tate construction $W(k)^{t \mathbb{Z}/p}$,
where $W(k)$ is the ring of Witt vectors on $k$ and $\mathbb{Z}/p$ acts
trivially on $W(k)$. The $\e{\infty}$-ring $W(k)^{t \mathbb{Z}/p}$ has homotopy
groups given by 
\[ \pi_* W(k)^{t \mathbb{Z}/p} \simeq k[t^{\pm 1}], \quad |t| = 2,  \]
and the $\e{\infty}$-ring that we are interested in is given by 
\[ k^{t \mathbb{Z}/p} \simeq W(k)^{t \mathbb{Z}/p} \otimes_{W(k)}  k. \]
Now \Cref{fieldreg} tells us that the Galois theory of $W(k)^{t \mathbb{Z}/p}$
is algebraic, and the invariance result \Cref{invresult} enables us to conclude
the same for $k^{t \mathbb{Z}/p}$. 
\end{proof}

\subsection{Tate spectra for elementary abelian subgroups}

Let $k$ be a field of characteristic $p$. We know that $k^{t \mathbb{Z}/p}$ has
homotopy groups given by a tensor product of an exterior and Laurent
series algebra
on generators 
in degrees $-1, -2$, respectively. For an elementary abelian $p$-group of
higher rank, the picture is somewhat more complicated: the homotopy ring
behaves irregularly (with entirely square-zero material in positive
homotopy groups), but the Tate construction is still built up from a diagram
of $\e{\infty}$-rings whose homotopy rings come from tensor products of
polynomial (or Laurent series) rings and exterior algebras. This diagram
roughly lives over $\mathbb{P}_k^{n-1}$ where $n$ is the rank  of the given
elementary abelian $p$-group, and the stable module $\infty$-category
$\stm_{(\mathbb{Z}/p)^n}(k)$ can be described as quasi-coherent sheaves on a
derived version of projective space (\Cref{affinestmod}). 
In this subsection, we will review this picture, which will be useful when we
describe the Galois groups in the next section. 

We consider the case of $p > 2$, and leave the minor modifications for $p = 2$ to the
reader. 
Fix an elementary abelian $p$-group $V = (\mathbb{Z}/p)^n$, and let $V_k = V
\otimes_{\mathbb{F}_p} k$. Consider first the
homotopy fixed points $k^{hV}$, whose homotopy ring is given by 
\[ \pi_*(k^{hV}) \simeq E( V_k^{\vee}) \otimes \mathrm{Sym}^*(V_k^{\vee}),  \]
where the exterior copy of $V_k^{\vee}$ is concentrated in degree $-1$, and the
polynomial copy of $V_k^{\vee}$ is concentrated in degree $-2$. For each nonzero homogeneous 
polynomial $f \in \mathrm{Sym}^*(V_k^{\vee})$, we can form the localization 
$k^{hV}[f^{-1}]$, whose degree zero part \emph{modulo nilpotents} is given by
the localization
$\mathrm{Sym}^*(V_k^{\vee})_{(f)}$ (i.e., the degree zero part of the
localization $\mathrm{Sym}^*(V_k^{\vee})[f^{-1}]$). There is also a small nilpotent
part that comes from the evenly graded portion of the exterior algebra. 
In particular, we find, using natural maps between localizations: 
\begin{enumerate}
\item For every Zariski open \emph{affine} subset $U \subset \mathbb{P}(
V_k^{\vee})$, we obtain
a (canonically associated) $\e{\infty}$-ring $\otop(U)$ by localizing $k^{h V}$ at an
appropriate homogeneous form. 
Precisely, $U$ is given as the complement to the zero locus of a homogeneous
form $f \in \mathrm{Sym}^*(V_k^{\vee})$, and we invert $f$ in $k^{hV}$: $\otop(U) =
k^{hV}[f^{-1}]$. 
\item For every inclusion $U \subset U'$ of Zariski open affines, we obtain a map 
of $\e{\infty}$-algebras (under $k^{hV}$) $\otop(U') \to \otop(U)$. 
These maps are canonical; $\otop(U'), \otop(U)$ are localizations of
$k^{hV}$ and $\otop(U)$ has more elements inverted. 
\item For each $U \subset \mathbb{P}(V_k^{\vee})$,  the $\e{\infty}$-ring
$\otop(U)$ has a unit in degree two. The
ring
$\pi_0(\otop(U))$ is canonically an algebra over the (algebraic) ring of
functions $\mathcal{O}_{\mathrm{alg}}(U)$ on $U \subset
\mathbb{P}(V_k^{\vee})$, and is a tensor product of 
$\mathcal{O}_{\mathrm{alg}}(U)$ with the even components of an exterior algebra over $k$. 
\item We have 
natural isomorphisms of sheaves of graded $\mathcal{O}_{\mathrm{alg}}$-modules
\[ 
\pi_*( \otop) \simeq E(V_k^{\vee}) \otimes_{\mathcal{O}_{\mathrm{alg}}} \bigoplus_{r \in \mathbb{Z}}
\mathcal{O}(r)
,  \]
where $\mathcal{O}(1)$ is the usual hyperplane bundle on
$\mathbb{P}(V_k^{\vee})$ and $\mathcal{O}(r) \simeq \mathcal{O}(1)^{\otimes
r}$ is concentrated in degree $-2r$. The generators of the exterior algebra
$E(V_k^{\vee})$ are in degree $-1$.  
\end{enumerate}

It follows that the homotopy groups $\pi_*(\otop(U))$ for $ U \subset
\mathbb{P}(V_k^{\vee})$ fit together into \emph{quasi-coherent sheaves} on the
site of affine Zariski opens $U \subset \mathbb{P}(V_k^{\vee})$ and
inclusions between them. In
particular, we can view the association $U \mapsto \otop(U)$ as defining
a \emph{sheaf} of $\e{\infty}$-ring spectra (under $k$, or even under $k^{hV}$)
over the Zariski site of $\mathbb{P}(V_k^{\vee})$, whose sections over an affine
open $U \subset \mathbb{P}(V_k^{\vee})$ are given by $\otop(U)$.

We will now describe our basic comparison result. 
Since $\otop$ is a sheaf of $\e{\infty}$-algebras under $k^{hV}$, we obtain a
symmetric monoidal, colimit-preserving
functor
\[ \md( k^{hV}) \to \qcoh( \otop),  \]
into the $\infty$-category $\qcoh(\otop)$ of \emph{quasi-coherent
$\otop$-modules}, defined as the homotopy limit $$\qcoh( \otop) = \varprojlim_{U \subset
\mathbb{P}(V_k^{\vee})} \md( \otop(U)),$$
where the homotopy limit is taken over all open affine subsets of
$\mathbb{P}(V_k^{\vee})$. 
Restricting to $\md^\omega(k^{hV}) \simeq \mathrm{Fun}(BV, \md^\omega(k))$, we
obtain a symmetric monoidal exact functor
\[  \mathrm{Fun}(BV, \md^\omega(k)) \to  \qcoh( \otop).\]
We observe that the standard representation of $V$, as an object of the former,
is sent to zero in $\qcoh( \otop)$. In fact, the standard representation of $V$
corresponds to a $k^{hV}$-module with only one nonvanishing homotopy group, and
it therefore vanishes under the types of \emph{periodic} localization that one takes
in order to form $\otop( U)$ for $U \subset \mathbb{P}(V_k^{\vee})$ an open
affine. 
Using the universal property of the stable module $\infty$-category, we obtain
a factorization
\[ \mathrm{Fun}(BV, \md^\omega(k)) \to \stm_V(k) \to \qcoh(\otop),  \]
where the functor $\stm_V(k) \to \qcoh(\otop)$  is symmetric monoidal and
colimit-preserving. 
\begin{theorem} 
\label{affinestmod}
The functor $\md(k^{tV}) \simeq \stm_V(k) \to \qcoh(\otop)$ is an equivalence of symmetric
monoidal $\infty$-categories. 
\end{theorem} 

\begin{proof} 
We start by observing that, by construction of the Verdier quotient
(\Cref{vq}), the stable module $\infty$-category $\stm_V(k)$
is obtained as a \emph{localization} of $\md(k^{hV}) \simeq \mathrm{Ind}(
\mathrm{Fun}(BV, \md^\omega(k)))$, and in particular $k^{tV}$ is a localization
of the $\e{\infty}$-ring $k^{hV}$. 

By construction, $k^{tV}$ is the
localization of $k^{hV}$ at the map of $k^{hV}$-modules $M \to 0$, where $M$ is
the $k^{hV}$-module corresponding to the standard representation of $V$. 
So, in particular, the localization functor
\[ \md(k^{hV}) \to \md(k^{tV}),  \]
given by tensoring up, has a fully faithful right adjoint which embeds
$\md(k^{tV})$ as the subcategory of all $k^{hV}$-modules $N$ such that
$\hom_{\md(k^{hV})}(M, N)$ is contractible. 
If we write $e_1, \dots, e_n \in \pi_{-2}( k^{hV})$ for polynomial generators
of $k^{hV}$, then $k^{hV}/(e_1, \dots, e_n) \in \md^\omega(k^{hV})$ generates the same 
thick subcategory as $M$, as we observed in the discussion immediately
preceding \Cref{quasi}. 
So, the $k^{tV}$-modules are precisely the $k^{hV}$-modules $N$ such that 
\[ N/(e_1, \dots, e_n) N \simeq 0 \in \md(k^{hV}),  \]
using self-duality of 
$k^{hV}/(e_1, \dots, e_n)$.

Now, we have a morphism of $\e{\infty}$-rings 
\begin{equation} \label{projspace} k^{hV} \to \Gamma( \mathbb{P}(V_k^{\vee}), \otop),  \end{equation}
and our first task is to show that this morphism induces an equivalence 
$k^{tV} \to \Gamma( \mathbb{P}(V_k^{\vee}), \otop)$. Observe first that, after
inverting any of $e_1, \dots, e_n \in \pi_{-2}(k^{hV})$, 
\eqref{projspace} becomes an equivalence since we already know what
$\otop$ looks like on the basic open affines; we also know that taking global
sections over $\mathbb{P}(V_k^{\vee})$ is a finite homotopy limit and thus
commutes with arbitrary homotopy colimits. 
However, we also know that
\( k^{hV}/(e_1, \dots, e_n)  \) maps to the zero $\otop$-module since, on
every basic open affine of $\mathbb{P}(V_k^{\vee})$, one of the $\left\{e_i\right\}$ is always invertible. 
Thus we get a map $k^{tV} \to \Gamma( \mathbb{P}(V_k^{\vee}), \otop)$ of
$k^{hV}$-modules with the
dual properties: 
\begin{enumerate}
\item Both modules smash to zero  with $k^{hV}/(e_1, \dots, e_n)$. 
\item The map induces an equivalence after inverting each $e_i$, $1 \leq i \leq
n$. 
\end{enumerate}
By a formal argument, 
it now follows that $k^{tV} \to\Gamma( \mathbb{P}(V_k^{\vee}), \otop)$ is an
equivalence to begin with. In fact, we show that, for each $i$, the map
\begin{equation}  \label{indstepp}k^{tV}/(e_1, \dots, e_i) \to\Gamma( \mathbb{P}(V_k^{\vee}), \otop) /(e_1,
\dots, e_i)\end{equation}
is an equivalence by \emph{descending} induction on $i$. For $i = n$, both
sides are contractible. If we are given that 
\eqref{indstepp} is an equivalence, then the map
$k^{tV}/(e_1, \dots, e_{i-1}) \to\Gamma( \mathbb{P}(V_k^{\vee}), \otop) /(e_1,
\dots, e_{i-1})$
has the property that it becomes an equivalence after either inverting $e_i$
(by the second property above) or by smashing with $k^{hV}/(e_i)$ (by the
inductive hypothesis); it thus has to be an equivalence in turn. This completes
the inductive step and the proof that $k^{tV} \simeq \Gamma(
\mathbb{P}(V_k^{\vee}), \otop)$. 

All in all, we have shown that the functor
\[ \md(k^{tV}) \simeq \stm_V(k) \to \qcoh(\otop)  \]
is \emph{fully faithful.} 
To complete the proof of 
\Cref{affinestmod}, we need to show that the global sections functor is
conservative on $\qcoh (\otop)$. However, if $\mathcal{F} \in \qcoh( \otop)$
has the property that $\Gamma(\mathbb{P}(V_k^{\vee}), \mathcal{F})$ is
contractible, then the same holds for $\mathcal{F}[e_i^{-1}]$. By analyzing the
descent spectral sequence, it follows that the global sections of
$\mathcal{F}[e_i^{-1}]$ are precisely the sections of $\mathcal{F}$ over the
$i$th
basic open affine chart of $\mathbb{P}(V_k^{\vee})$. 
Thus, if $\Gamma( \mathbb{P}(V_k^{\vee}), \mathcal{F})$ is contractible, then
$\mathcal{F}$ has contractible sections over each of the basic open affines,
and is thus contractible to begin with. (This argument is essentially the
ampleness of $\mathcal{O}(1)$.)
\end{proof}

\subsection{$G$-Galois extensions for topological groups}
Our next goal is to calculate the Galois group for $k^{t V}$ for any elementary
abelian $p$-group $V$. In the case of rank one, we had a trick for 
approaching the Galois group. Although $k^{tV}$ was not even
periodic, there was a good integral model (namely,
$W(k)^{t V}$) which was related to $k^{tV}$ by reducing mod $p$, so that we
could 
use an invariance property to reduce to the (much easier) $\e{\infty}$-ring
$W(k)^{t V}$. 

When the rank of $V$ is greater than one, both these tricks break down. There
is no longer a comparable integral model of an $\e{\infty}$-ring such as
$k^{h \mathbb{Z}/p} \otimes k^{t \mathbb{Z}/p}$, as far as we know. Our strategy is based instead on a comparison with the Tate spectra for \emph{tori},
which are much more accessible. To interpolate between the Tate spectra for
tori and the Tate spectra for elementary abelian $p$-groups, we will need a bit
of the theory of Galois extensions for topological groups, which was considered
in \cite{rognes}. We will describe the associated theory of descent in this
section. We refer to \cite{toruspic} for further applications of these ideas
to the Picard group and the classification of localizing subcategories of the
stable module category (recovering older results), as well as a discussion of
how this formulation of $G$-Galois extensions relates to that of Rognes
\cite{rognes} (who uses a definition similar to \Cref{defgalr}).

\begin{definition} 
\label{topgalois}
Fix a topological group $G$ which has the homotopy type of a finite CW complex
(e.g., a compact Lie group). 
Let $R$ be an $\e{\infty}$-ring and let $R'$ be an $\e{\infty}$-$R$-algebra
with an action of $G$  (in the $\infty$-category of
$\e{\infty}$-$R$-algebras). 

We will say that $R'$ is a faithful
\textbf{$G$-Galois extension} of $R$ if there exists a descendable $\e{\infty}$-$R$-algebra $R''$
such that we have an equivalence
of $\e{\infty}$-$R''$-algebras 
\[   R' \otimes_R R'' \simeq C^*(G; R''),   \]
which is compatible with the $G$-action. 
\end{definition} 

Note that the cochain $\e{\infty}$-ring $C^*(G; R'')$ is the ``coinduced'' $G$-action on an $R''$-module. 
It follows in particular that the natural map $R \to R'^{h G}$ is an
equivalence, and is so universally; for any $\widetilde{R} \in \clg_{R/}$, the
natural map
$\widetilde{R} \to (R' \otimes_{R} \widetilde{R})^{h G}$ is an equivalence. 
Moreover, $R'$ is perfect as an $R$-module, since this can be checked locally
(after base-change to $R''$) and $G$ has the homotopy type of a finite CW
complex. It follows from general properties of descendable morphisms that
faithful $G$-Galois extensions are preserved under base-change.

We will need the following version of classical Galois
descent, which has been independently considered in various forms by
several authors, for instance
\cite{hess, GL, meier, banerjee}. 

\begin{theorem} 
\label{galdescthm}
Let $G$ be a topological group of the homotopy type of a finite CW complex, and
let $R \to R'$ be a faithful $G$-Galois extension of $\e{\infty}$-rings. 
The natural functor
\begin{equation} \label{descf} \md(R) \to \md(R')^{h G},  \end{equation}
is an equivalence of $\infty$-categories. 
\end{theorem} 
The ``natural functor'' comes from the expression $R \simeq R'^{hG}$; the
$G$-action on $R'$ induces one on the symmetric monoidal $\infty$-category
$\md(R')$. In particular, we get a \emph{fully faithful} embedding $\md^\omega(R) \to
\md(R')^{hG}$ for free. 
\begin{proof} 
Suppose first that $R' \simeq C^*(G; R)$ with the $G$-action coming from 
the translation action of $G$ on itself. Then, we have a fully faithful,
colimit-preserving embedding
\[ \md( R') \subset \loc_G( \md(R)),  \]
as we saw in \Cref{subseclocsys}. The $G$-action here on $\loc_G(\md(R))$ comes
from the translation action again. Taking homotopy fixed points, we get 
\begin{equation} \label{locsys1} \md(R')^{hG} \subset \loc_{G_{hG}}( \md(R))
\simeq \loc_{\ast}( \md(R)) \simeq \md(R),  \end{equation}
because the construction $X \mapsto \loc_X( \md(R))$ sends homotopy colimits in
$X$ to homotopy limits of stable $\infty$-categories. The natural functor
$\md(R) \to \md(R')^{hG}$ now composes all the way over in 
\eqref{locsys1} to the identity, so that it must have been an equivalence to
begin with since all the maps in \eqref{locsys1} are fully faithful. 

Now suppose $R \to R'$ is a general $G$-Galois extension, so that there
exists a descendable $\e{\infty}$-$R$-algebra $T$ such that $R \to R'$
becomes a trivial Galois extension after base-change along $R \to T$. 
The functor \eqref{descf} is a functor of $R$-linear $\infty$-categories so,
to show that it is an equivalence, it suffices to show that \eqref{descf}
induces an equivalence after applying the construction $\otimes_{\md(R)}
\md(T)$: that is, after considering $T$-module objects in each
$\infty$-category (cf. \Cref{2descformal}). In other words, to show that \eqref{descf} is an
equivalence, it suffices to tensor up and show that
\[ \md(T) \to \left( \md(R')\right)^{hG} \otimes_{\md(R)} \md(T) 
\simeq \left( \md(R') \otimes_{\md(R)} \md(T)\right)^{hG} \simeq \md( C^*(G;
T))^{hG}
,\]
is an equivalence of $\infty$-categories, which we just proved. 
\end{proof}

It follows in particular that whenever we have a $G$-Galois extension in the above
sense, for $G$ a \emph{topological} group then we can relate the fundamental
groups of $R $ and $R'$. 
In fact, we have, in view of \Cref{galdescthm},
\[ \clgf(R) \simeq \clgf(R')^{hG}.  \]
Using the Galois correspondence, it follows that there is a $G$-action on the
object $\pi_{\leq 1}\md(R') \in \pro( \fgp)$, and the homotopy \emph{quotient}
in $\pro( \fgp)$ by this $G$-action
is precisely the fundamental groupoid of $\md(R)$, i.e., 
\[ \pi_{\leq 1} \md(R) \simeq \left( \pi_{\leq 1 } \md(R')\right)_{hG} \in
\pro( \fgp).  \]

We now describe homotopy orbits in $\pro( \fgp)$ in the case that will be of
interest. 
Let $X \in \pro( \fgp)$ be a \emph{connected} profinite groupoid and consider an
action of a \emph{connected} topological group $G$ on
$X$. 

\begin{proposition} 
To give an action of $G$ on $X \in \pro( \fgp)^{\geq 0}$ is equivalent to giving a
homomorphism of groups $\pi_1(G) \to \pi_1(X)$ whose image is contained in the
center of $\pi_1(X)$. In other words, the 2-category $\mathrm{Fun}(BG,
\pro(\fgp)^{\geq 0})$ can be described as follows: 
\begin{enumerate}
\item Objects are profinite groups $F$ together with maps $\phi\colon  \pi_1(G) \to F$
whose image is contained in the center of $F$. 
\item 1-morphisms between pairs $(F, \phi)$ and $(F', \phi')$ are continuous
homomorphisms $\psi\colon  F \to F'$ such that the diagram
\[ \xymatrix{
\pi_1(G) \ar[d]^{\phi} \ar[rd]^{ \phi'} \\
F \ar[r]^{\psi} & F'
},\]
commutes. 
\item 2-morphisms are given by conjugacies between homomorphisms. 
\end{enumerate}
In particular, the forgetful functor $\fun( BG, \pro( \fgp)^{\geq 0}) \to \pro(
\fgp)^{\geq 0}$ induces fully faithful maps on the hom-spaces. 
\end{proposition} 

\begin{proof} 
In order to give an action of $X \in \pro( \fgp)^{\geq 0}$, we need to
construct a map of $\e{1}$-spaces $G \to \mathrm{Aut}_{\pro( \fgp)^{\geq
0}}(X)$, where 
$\mathrm{Aut}_{\pro( \fgp)^{\geq
0}}(X)$ is the automorphism $\e{1}$-algebra of $X$. Since, however, $G$ is
connected, it is equivalent to specifying a map of $\e{1}$-algebras (or loop
spaces) into 
$\tau_{\geq 1}\mathrm{Aut}_{\pro( \fgp)^{\geq
0}}(X)$. However, we know from \Cref{mappingspaces} that 
$\tau_{\geq 1}\mathrm{Aut}_{\pro( \fgp)^{\geq
0}}(X)$ is precisely a $K(Z(\pi_1(X)), 1)$, so the space of $\e{1}$-maps as
above is simply the \emph{set} of homomorphisms $\pi_1(G) \to Z( \pi_1(X))$. 

Finally, we need to understand the mapping spaces in $\fun( BG, \pro(
\fgp)^{\geq 0})$. Consider two connected profinite groupoids $X, Y$ with
$G$-actions. The space of maps $X \to Y$ in $\mathrm{Fun}(BG, \pro( \fgp))$ is
equivalent to the homotopy fixed points $\hom_{\pro( \fgp)}( X, Y)^{hG}$,
where $\hom_{\pro( \fgp)}(X, Y)$ is a groupoid as discussed earlier. 
In
general, given any groupoid $\mathscr{G}$ with an action of $G$, the functor
$\mathscr{G}^{hG} \to \mathscr{G}$ is fully faithful. 
The action of $G$ means that every element in $\pi_1(G)$ determines a natural
transformation from the identity to itself on $\mathscr{G}$, and the homotopy
fixed points pick out the full subcategory of $\mathscr{G}$ spanned by elements
on which that natural transformation is the identity (for any $\gamma \in
\pi_1(G)$). 

In the case of $\hom_{\pro( \fgp)}(X, Y)$, the objects are
continuous homomorphisms $\psi\colon  \pi_1 X \to \pi_1 Y$, and the morphisms between
objects are conjugacies. For $\gamma \in \pi_1(G)$, we obtain elements
$\gamma_x \in \pi_1(X)$ and $\gamma_y \in \pi_1(Y)$ (in view of the $G$-action
on $X, Y$), and the action of $\gamma$ on $\hom_{\pro(\fgp)}(X, Y)$ at the
homomorphism $\psi$ is given by the element $\psi(\gamma_x)
\psi(\gamma_y)^{-1}$, which determines a self-conjugacy from $\psi$ to itself. 
To say that this self-conjugacy is the identity for any $\gamma$, i.e., that
the map is $G$-equivariant (which here is a \emph{condition} instead of extra data), is precisely
the second description of the 1-morphisms. 
\end{proof} 

\begin{remark} 
The above argument would have worked in any $(2, 1)$-category where we could
write down the $\pi_1$ of the automorphism $\e{1}$-algebra easily. 
\end{remark}

In particular, if $G$ acts trivially on $Y \in \pro( \fgp)^{\geq 0}$, 
then to give a map $X \to Y$ is equivalent to giving a map in $\pro( \fgp)$
which annihilates the image of $\pi_1(G) \to \pi_1(X)$. 
It follows that the \emph{homotopy quotients} $X_{hG}$ in $\pro( \fgp)$ can be
described by 
taking the quotient of $\pi_1 X$ by the closure of the image of $\pi_1(G)$:
this is the universal profinite groupoid with a trivial $G$-action to which
$X$ maps. 

Putting all of this together, we find: 
\begin{corollary} 
\label{basicexact}
Let $G$ be a connected topological group of the homotopy type of a finite CW
complex, and let $R \to R'$ be a faithful $G$-Galois extension. Then we have an exact sequence of profinite groups
\begin{equation} \label{topgalexact} \widehat{\pi_1 G} \to \pi_1 \md(R') \to \pi_1 \md(R)  \to 1. \end{equation}
\end{corollary} 

\begin{remark} 
Throughout this section, we shall be somewhat cavalier about the
use of basepoints, since we will be working with connected profinite groupoids. 
\end{remark}

\subsection{The general elementary abelian case}

Let $V$ be an elementary abelian $p$-group and let $k$ be a field of
characteristic $p$. 
In this section, we will prove our main result that the Galois theory of
$k^{tV}$ is algebraic. In order to do this, we will use the presentation in
\Cref{affinestmod} of $\md(k^{tV})$ via quasi-coherent sheaves on a ``derived'' version of
$\mathbb{P}(V_k^{\vee})$. 
Any $G$-Galois extension of $k^{tV}$ clearly gives a $G$-Galois extension  of
$\otop(U)$ for any $U \subset \mathbb{P}(V_k^{\vee})$ by base-change. Conversely,
the affineness result \Cref{affinestmod} implies that to give a 
$G$-Galois extension of $k^{tV}$ is equivalent to giving $G$-Galois extensions
of $\otop(U)$ for $U \subset \mathbb{P}(V_k^{\vee})$ affine together with the
requisite compatibilities. This would be doable if $\otop(U)$ was even
periodic with regular $\pi_0$, although the exterior generators present an
obstacle. Nonetheless, by a careful comparison with the analog for
\emph{tori}, we will prove: 

\begin{theorem} 
\label{galelem}
Let $V$ be an elementary abelian $p$-group. 
If $k$ is a field of characteristic $p$, all finite
coverings of $k^{t V}$ are \'etale, so $\pi_1 ( \md(k^{tV})) \simeq
\mathrm{Gal}(k^{\mathrm{sep}}/k)$. 
\end{theorem} 
\begin{proof}
Since projective space is (geometrically) simply connected, it suffices to show that the Galois theory of 
\[  k^{t \mathbb{Z}/p} \otimes_k k^{h (\mathbb{Z}/p)^{n}} \simeq k^{t
\mathbb{Z}/p} \otimes_k C^*(B (\mathbb{Z}/p)^n; k), \]
for $n > 0$, is algebraic, and thus given by the (algebraic) \'etale
fundamental group of the
corresponding affine open cell in $\mathbb{P}(V_k^{\vee})$. 
These $\e{\infty}$-rings are the $\otop(U)$ for $U \subset
\mathbb{P}(V_k^{\vee})$ the basic open affines of projective space. 
It will follow
that a faithful Galois extension of $k^{tV}$ is locally algebraically \'etale over
$\mathbb{P}(V_k^{\vee})$. 

For this, we will use the fibration sequence
\[ S^1 \to  B \mathbb{Z}/p \to B S^1,  \]
induced by the inclusion $\mathbb{Z}/p \subset S^1$ with quotient $S^1$. This is a principal
$S^1$-bundle and we find in particular an $S^1$-action on $C^*(B \mathbb{Z}/p;
k)$ such that
\begin{equation} C^*(B S^1; k) \simeq C^*(B \mathbb{Z}/p; k)^{h S^1}
\end{equation}
In fact, the map $C^*(BS^1; k) \to C^*(B \mathbb{Z}/p; k)$ is a faithful $S^1$-Galois
extension (in the sense of \Cref{topgalois}): by the Eilenberg-Moore spectral sequence, and the fiber square
\[ \xymatrix{
B \mathbb{Z}/p \times S^1 \ar[d] \ar[r] &  B \mathbb{Z}/p \ar[d] \\
B \mathbb{Z}/p \ar[r] &  B S^1
},\]
expressing the earlier claim that $B \mathbb{Z}/p \to BS^1$ is an $S^1$-torsor,
it follows that 
$$C^*(B \mathbb{Z}/p;k ) \otimes_{C^*(B S^1; k)} C^*(B(\mathbb{Z}/p);k ) 
\simeq C^*(S^1; k) \otimes_k C^*(B\mathbb{Z}/p; k),
$$
with the ``coinduced'' $S^1$-action on the right. 
Moreover, $C^*(BS^1;k ) \to C^*(B\mathbb{Z}/p; k)$ is descendable: in fact, a look
at homotopy groups shows that the latter is a wedge of the former and its shift.  

Let $\mathbb{T}^n \simeq (S^1)^n$ be the $n$-torus, which contains
$(\mathbb{Z}/p)^n$ as a subgroup. 
Similarly, we find that there is a $\mathbb{T}^n$-action on
$C^*(B(\mathbb{Z}/p)^n; k)$ in the $\infty$-category of $C^*( B \mathbb{T}^n;
k)$-algebras which exhibits 
$C^*(B(\mathbb{Z}/p)^n; k)$ as a faithful $\mathbb{T}^n$-Galois extension of
$C^*(B(\mathbb{Z}/p)^n; k)$. 
We can now apply a bit of descent theory. 
Fix any $C^*( B \mathbb{T}^n; k)$-algebra $R$, and 
let $R' \simeq R \otimes_{C^*(B\mathbb{T}^n; k)} C^*(B (\mathbb{Z}/p)^n; k)$.
Since $R'$ is a faithful $\mathbb{T}^n$-Galois extension of $R$, we have a
(natural) exact
sequence
given by \Cref{basicexact}: 
\begin{equation} \label{hugeexact} \widehat{\mathbb{Z}}^n \to \pi_1( \md( R'))
\to \pi_1 (\md(R)) \to 1.  \end{equation}

Finally, we may attack the problem of determining the Galois theory of $k^{t
\mathbb{Z}/p} \otimes_k k^{h(\mathbb{Z}/p)^n} 
$ where $n > 0$. 
We have
\[ \pi_* C^*( B (\mathbb{Z}/p)^{n+1}; k) \simeq 
k[e_0, e_1, \dots, e_n] \otimes E(\epsilon_0, \dots, \epsilon_n), \quad |e_i |
= -2, \ |\epsilon_i|
= -1.
\]
Our goal is to determine the Galois theory of the localization $k^{t
\mathbb{Z}/p} \otimes_k k^{h (\mathbb{Z}/p)^n} \simeq C^*(
B(\mathbb{Z}/p)^{n+1}; k)[e_0^{-1}]$. 
Now, we also have
$$\pi_* C^*( B\mathbb{T}^{n+1}; k) \simeq k[e_0, \dots, e_n], \quad |e_i|  =
-2,$$ 
and the map $C^*(B \mathbb{T}^{n+1}; k) \to C^*( B( \mathbb{Z}/p)^{n+1}; k)$
sends the $\left\{e_i\right\}$ to the $\left\{e_i\right\}$. This map is a
faithful $\mathbb{T}^{n+1} $-Galois extension. 
As we did for $C^*(B (\mathbb{Z}/p)^{n+1}; k)$, consider the localization 
$C^*(B \mathbb{T}^{n+1}; k) [e_0^{-1}]$, whose homotopy groups are given by 
\begin{equation} \pi_* C^*(B \mathbb{T}^{n+1}; k)[e_{0}^{-1}] \simeq k[e_0^{\pm 1},
f_1, \dots,
f_n], \quad |f_i| = 0,  \end{equation}
where for $i \geq 1$, $f_i = e_i/e_0$. In particular, the 
Galois theory of 
$C^*(B \mathbb{T}^{n+1}; k)[e_{0}^{-1}]$ is algebraic thanks to \Cref{etalegalois}, and by
\eqref{hugeexact}, we have an exact sequence
\begin{equation} 
\widehat{\mathbb{Z}}^{n+1} \to \pi_1  \md( C^*(B ( \mathbb{Z}/p)^{n+1};
k)[e_0^{-1}]) \to \pi_1 \md( C^*(B \mathbb{T}^{n+1};
k)[e_0^{-1}]) \to 1
\end{equation} 
Our argument will be that the first map is necessarily \emph{zero}, which will
show that the Galois theory of 
$C^*(
B(\mathbb{Z}/p)^{n+1}; k)[e_0^{-1}]$ is algebraic as desired. 
In order to do this, we will use a naturality argument. 

We can form the completion 
$$A = \widehat{ C^*(B \mathbb{T}^{n+1}; k) [e_0^{- 1}] }_{(f_1, \dots, f_n)},$$
at the ideal $(f_1, \dots, f_n)$,
whose homotopy groups now become the tensor product of 
the Laurent series ring $k[e_0^{\pm 1}]$ together with a \emph{power series}
ring $k[[f_1, \dots, f_n]]$. 
We will prove:

\begin{lemma} \label{analglem}
The Galois theory of $A' \stackrel{\mathrm{def}}{=}A \otimes_{C^*(B
\mathbb{T}^{n+1}; k)} C^*(
B(\mathbb{Z}/p)^{n+1}; k)$ is entirely algebraic (and, in particular, that of $A$). 
\end{lemma} 

\begin{proof} 
The $\e{\infty}$-ring $A' = A \otimes_{C^*(B \mathbb{T}^{n+1}; k)} C^*(
B(\mathbb{Z}/p)^{n+1}; k)$, which by definition is the $\e{\infty}$-ring obtained
from $C^*( B ( \mathbb{Z}/p)^{n+1}; k)$ obtained by inverting the generator $e_0$
and completing with respect to the ideal $(f_1, \dots, f_n)$, admits another
description: it is the homotopy fixed points $(k^{t \mathbb{Z}/p})^{h (\mathbb{Z}/p)^{n}}$ where
$(\mathbb{Z}/p)^{n}$ acts trivially.\footnote{In general, the formation of
homotopy fixed points do not commute with localization from $k^{h
\mathbb{Z}/p}$ to $k^{t \mathbb{Z}/p}$: the failure is precisely measured by
the need to take the completion.} 
Since we
have computed the Galois theory of $k ^{t \mathbb{Z}/p}$ and found it to be
algebraic in 
\Cref{galoistaterankone}, this, together with \Cref{weakinv}, implies the claim. 
\end{proof} 

Finally, consider the diagram
obtained from the faithful $\mathbb{T}^{n+1}$-Galois 
extensions
$
C^*( B\mathbb{T}^{n+1}; k)[e_0^{-1}]
\to 
C^*( B(\mathbb{Z}/p)^{n+1}; k)[e_0^{-1}]$ and $A \to A'$,
\[ \xymatrix{
\widehat{\mathbb{Z}}^{n+1} \ar[d]   \ar[r] & \ar[d]   \pi_1( \md( A'))\ar[r] &   \ar[d]  \pi_1 (\md( A)) \ar[r] &  1 \\
\widehat{\mathbb{Z}}^{n+1} \ar[r] &  \pi_1 \md( C^*(B
(\mathbb{Z}/p)^{n+1}; k)[e_0^{-1}]) \ar[r] &  \pi_1\md( C^*(B
\mathbb{T}^{n+1};
k)[e_0^{-1}]) \ar[r] &  1 }.   \]

In the top row, in view of \Cref{analglem}, the map out of
$\widehat{\mathbb{Z}}^{n+1}$
must be zero. It follows that the same must hold in the bottom row. In other
words,
the Galois theory of $C^*(B (\mathbb{Z}/p)^{n+1}; k)[e_0^{-1}]$ is equivalent to
the (algebraic) Galois theory of $C^*( B \mathbb{T}^{n+1}; k)[e_0^{-1}]$. 
As we saw at the beginning, this is precisely the step we needed to see that
the Galois theory of the Tate construction $k^{tV}$ is ``locally'' algebraic
over $\mathbb{P}(V_k^{\vee})$, and this  completes the proof of \Cref{galelem}. 
\end{proof} 

\begin{remark} 
This argument leaves open a natural question: is the Galois theory of a
general localization $C^*(B ( \mathbb{Z}/p)^{n+1}; k)[f^{-1}]$ algebraic? 
\end{remark} 

\subsection{General finite groups}

Let $G$ be any finite group. In this section, we will put together the
various pieces  (in particular, \Cref{galelem} and Quillen stratification
theory) 
to give a description of the Galois group of the
stable module $\infty$-category $\stm_G(k)$ over a field $k$ of characteristic $p
> 0$. 

For each subgroup $H \subset G$, recall the commutative algebra object
$A_H = \prod_{G/H} k \in \clg(\md_G(k))$. $A_H$ has the property that 
\( \md_{\md_G(k)}(A_H) \simeq \md_H(k),  \)
and the adjunction $\md_G(k) \rightleftarrows \md_{\md_G(k)}(A_H)$ whose left
adjoint tensors with $A_H$ can be identified with \emph{restriction} to the
subgroup $H$. We will need an analog of this at the level of stable module
categories. We refer to \cite[sec. 5.3]{MNNequiv} for a discussion of these types
of equivalences and for a proof of a general result including this in the
$\infty$-categorical setting. 
\begin{proposition}[Balmer \cite{balmerstack}]
\label{moduleres}
Let $\mathscr{A}_H \in \clg( \stm_G(k))$ be the image of $A_H$ in the stable
module $\infty$-category. Then we can identify $\md_{\mathscr{A}_H}( \stm_G(k)) \simeq
\stm_H(k)$ and we can identify the adjunction $\stm_G(k) \rightleftarrows
\md_{\mathscr{A}_H}( \stm_G(k))$ with the restriction-coinduction adjunction 
$\stm_G(k) \rightleftarrows \stm_H(k)$. 
\end{proposition} 
\newcommand{\mow}{\md^{\omega}}
\newcommand{\res}{\mathrm{Res}}

\Cref{moduleres} suggests that we can perform a type of descent in stable
module $\infty$-categories by restricting
to appropriate subgroups. In particular, we can hope to reduce the calculation
of certain invariants in $\stm_G(k)$ to those of $\stm_H(k)$ where $H \subset
G$ are certain subgroups, by performing descent along commutative algebra
objects of the form $\mathscr{A}_H$. We shall carry this out for the Galois
group.

Let $G$ be any finite group, and let $\mathcal{A}$ be a collection of subgroups
of $G$ such that any elementary abelian $p$-subgroup of $G$ is contained in a
conjugate of an element of $\mathcal{A}$. 
For each $H \in \mathcal{A}$, we consider the 
object $\prod_{G/H} k \in \clg(\md_G(k))$. 
\begin{proposition} 
The commutative algebra object $A = \prod_{H \in \mathcal{A}} \left(\prod_{G/H}
k\right) \in \clg( \md_G(k))$ admits descent. 
\end{proposition} 

\begin{proof} 
 In order to
prove this, by \Cref{BC}, it suffices to prove that the above commutative algebra
admits descent after restriction from $G$ to each elementary abelian
$p$-subgroup. However, when we restrict from $G$ to each elementary abelian
$p$-subgroup, the above commutative algebra object contains a copy of the unit
object as a direct factor (as commutative algebras), so that it
clearly admits descent. 
\end{proof} 
\newcommand{\s}[1]{\mathscr{#1}}

In particular, it follows that the image $\mathscr{A} \in \clg(\stm_G(k))$ of the above commutative algebra
object $A = \prod_{H \in \mathcal{A}} \left(\prod_{G/H} k\right) \in \md_G(k)$
in the stable module $\infty$-category also 
admits descent. 
It follows that we have an equivalence
of symmetric monoidal $\infty$-categories
\begin{equation} \label{descdecompstm} \stm_G(k) \simeq 
\mathrm{Tot}\left( 
\md_{\stm_G(k)}(\mathscr{A}) \rightrightarrows \md_{\stm_G(k)}(\mathscr{A}
\otimes \mathscr{A}) \triplearrows \dots \right)
.\end{equation}

There is a classical cofinality argument that enables us to rewrite this
inverse limit in a different fashion. 
Recall:
\begin{definition} 
The \textbf{orbit category} $\mathcal{O}(G)$ is the category of all finite
$G$-sets of the form $G/H$ for $H \subset G$ a subgroup.
\end{definition} 

We have a functor
\[ \mathcal{O}(G) \to \shot, \quad G/H \mapsto \stm_H(k) = \md_{\stm_G(k)}\left( 
\prod_{G/H} k\right). \]
Note that given any finite $G$-set $S$, we can form a commutative algebra
object in $\stm_G(k)$ given by $\prod_S k  = k^S$. This construction takes
coproducts of $G$-sets to products. 

\newcommand{\sF}{\mathcal{F}}
Suppose $\mathcal{A}$ is a collection of subgroups of $G$ which is closed under
finite intersections and conjugation by elements of $G$. 
We will use the following notation:
\newcommand{\oag}{\mathcal{O}_{\mathcal{A}}(G)}
\newcommand{\oagp}{\mathcal{O}'_{\mathcal{A}}(G)}
\begin{definition} 
We let $\oag\subset \mathcal{O}(G)$ be the full
subcategory spanned by the $G$-sets $G/H$ for $H \in \mathcal{A}$. 
We let $\oagp \subset \oag$ be the full subcategory including only the
$\left\{G/H\right\}$ for $H \in \mathcal{A}$ and $H \neq 1$. 
\end{definition} 

Using standard cofinality arguments (cf. \cite[sec. 6.5]{MNNequiv}), we
obtain from the descent statement \eqref{descdecompstm}: 

\begin{corollary} 
Let $\mathcal{A}$ be a collection of subgroups of $G$. 
Suppose that $\mathcal{A}$ is closed under conjugation and finite intersections. 
Suppose every elementary abelian $p$-subgroup of $G$ is contained in a subgroup
belonging to $\mathcal{A}$. Then we have a decomposition
\begin{equation} 
\label{stmdec}
\stm_G(k) \simeq \varprojlim_{G/H \in \mathcal{O}_{\mathcal{A}}(G)^{op}}
\stm_H(k).
\end{equation} 
\end{corollary} 

These types of descent statements at the level of homotopy
categories have been developed in \cite{balmerstack}. 
We also have an analogous (but easier) decomposition
$$
\md_G(k) \simeq \varprojlim_{G/H \in \mathcal{O}_{\mathcal{A}}(G)^{op}}
\md_H(k).
$$

Using \Cref{galelem}, we get: 
\begin{theorem} 
\label{galgpstmodcat}
Let $\mathcal{A}$ be the collection of elementary abelian $p$-subgroups of $G$. 
If $k$ is a separably closed field of characteristic $p$, then the Galois group of $\stm_G(k)$ is the profinite completion of
the fundamental group of 
the nerve of the category $\mathcal{O}'_{\mathcal{A}}(G)$.
\end{theorem} 
\begin{proof} 
The decomposition \eqref{stmdec} implies that there is a decomposition
\[ \pi_{\leq 1} (\stm_G(k) ) = \varinjlim_{G/H \in \oag} \pi_{\leq 1} (
\stm_H(k)).  \]
Now by \Cref{galelem}, when $H $ is nontrivial we have $\pi_{\leq 1} (
\stm_H(k)) = \ast$. When $H = 1$, then $\stm_H(k) = 0$ so that the Galois
groupoid is empty. It follows that the functor $\oag \to \pro( \fgp)$,
$G/H \mapsto \pi_{\leq 1} (
\stm_H(k))$ is the left Kan extension of the constant functor $\ast$ on $\oagp
\subset \oag$. This implies the result. 
\end{proof} 

Unfortunately, we do not know in general a more  explicit description of the
Galois group. We will give a couple of simple examples below. 

\begin{theorem} 
\begin{enumerate}
\item 
Let $G$ be a finite group whose center contains an order $p$ element (e.g., a
$p$-group). Then the Galois group of $\stm_G(k)$ is the
quotient of $G$ by the normal subgroup generated by the order $p$ elements: the
functor
\( \md_G(k) \to \stm_G(k),  \)
induces an isomorphism on fundamental groups. 
\item Suppose $G$ is a finite group such that the intersection of any three
$p$-Sylow subgroups of $G$ is nontrivial. Then 
$\md_G(k) \to \stm_G(k)$ induces an isomorphism on fundamental groups. 
\end{enumerate}
\end{theorem} 
\begin{proof} 
Consider the first case. 
 Choose an order $p$
subgroup $C$ contained in the center of $G$, and consider the collection
$\mathcal{A}$ of all nontrivial elementary abelian $p$-subgroups of $G$ which contain $C$.
Note that $\mathcal{A}$ does not contain the trivial subgroup. 
Then 
we get decompositions 
$\stm_G(k) \simeq \varprojlim_{G/H \in \mathcal{O}_{\mathcal{A}}(G)^{op}}
\stm_H(k)$ and similarly for  $\md_G(k)$. 
In both cases, the Galois groupoid of each term in the inverse limit is 
is a point. It follows that 
\[ \pi_1( \stm_G(k)) \simeq \pi_1 ( \md_G(k)) \simeq \pi_1 N( \oag),  \]
and since we have already computed the Galois group of $\md_G(k)$ (\Cref{Grepgal}),
we are done. 

For the second case, let $G$ be a finite group such that the intersection of
any three $p$-Sylows in $G$ is nontrivial. Here we will argue slightly
differently. We fix a $p$-Sylow $P \subset
G$ and consider the commutative algebra
object $B = \prod_{G/P} k \in \clg( \md_G(k))$ and its image $\mathscr{B} \in
\clg( \stm_G(k))$. 
We observe that $B, B \otimes B, B \otimes B \otimes B$ have the same
fundamental groupoids as $\mathscr{B}, \mathscr{B} \otimes \mathscr{B},
\mathscr{B }\otimes \mathscr{B} \otimes \mathscr{B}$, respectively: in fact,
this follows from the previous item (that the Galois groups for $\md_H(k)$ and
$\stm_H(k)$ where $H$ is a \emph{nontrivial} $p$-group are isomorphic), since
the hypotheses imply that the
$G$-set $G/P \times G/P \times G/P$ has no free component to it.
Therefore, by
descent theory, the Galois groups of $\md_G(k)$ and $\stm_G(k)$ must be
isomorphic; note that the Galois group only depends on the 3-truncation of the descent
diagram. 

\end{proof} 

On the other hand, there are cases in which there are finite covers in the
stable module $\infty$-category that do not come from the representation category. 

\begin{corollary} 
Let $k$ be a separably closed field of characteristic $p$. 
Let $G$ be a finite group such that the maximal elementary abelian $p$-subgroup
of $G$ has rank one (i.e., there is no embedding $\mathbb{Z}/p \times
\mathbb{Z}/p \subset G$) and any two such are conjugate. In this case, the Galois group of $\stm_G(k)$ is the
Weyl group of a subgroup $\mathbb{Z}/p \subset G$. 
\end{corollary} 
\begin{proof} 
This is an immediate consequence of \Cref{galgpstmodcat}. 
\end{proof}

For example, we find that the Galois group of the stable module
$\infty$-category of $\Sigma_p$ is precisely a $( \mathbb{Z}/p)^{\times}$,
which is the Weyl group of $\mathbb{Z}/p \subset \Sigma_p$. 
We can see this very explicitly. The Tate construction $k^{t \Sigma_p}$ has
homotopy groups given by 
\[ \pi_* ( k^{t \Sigma_p}) \simeq E(\alpha_{2p-1}) \otimes P(\beta_{2p-2}^{\pm
1}),  \]
whereas we have $k^{t \mathbb{Z}/p} \simeq E(\alpha_{-1}) \otimes P( \beta_2^{\pm
1})$. The extension $k^{t \Sigma_p} \to k^{t \mathbb{Z}/p}$ is Galois, and is
obtained roughly by adjoining a $(p-1)$st root of the invertible element
$\beta_{2p-2}$.

\section{Chromatic homotopy theory}

In this section, we begin exploring the Galois group in chromatic stable
homotopy theory; this was the original motivating example for this project. 
In particular, we consider Galois groups over certain $E_n$-local
$\e{\infty}$-rings such as $\TMF$ and $L_n S^0$, and over the $\infty$-category
$L_{K(n)} \sp$ of $K(n)$-local spectra. 

\subsection{Affineness and $\TMF$}
\label{sec:dstack}
Consider the $\e{\infty}$-ring $\TMF$ of (periodic) topological modular forms.
We refer to \cite{TMF} for a detailed treatment.  
Our goal in this section is to describe its Galois theory. 
The homotopy groups of $\TMF$ are very far from regular; there is considerable
torsion and nilpotence in $\pi_*(\TMF)$ at the primes $2$ and $3$, coming from
the stable stems. This presents a significant difficulty in the computation of arithmetic invariants of $\TMF$
and $\md( \TMF)$. 

Nonetheless, $\TMF$ itself is built up as an inverse limit of much simpler
(at least, simpler at the level of homotopy groups) $\e{\infty}$-ring spectra. Recall the construction of
Goerss-Hopkins-Miller-Lurie, which builds $\TMF$ as the global sections of a
sheaf of $\e{\infty}$-ring spectra on the \'etale site of the moduli stack of
elliptic curves $\mell$. 
Given a commutative ring $R$, and an elliptic curve $C \to \spec R$ such that
the classifying map $\spec R \to \mell$ is \'etale, the construction assigns an
$\e{\infty}$-ring $\otop( \spec R)$ with the basic properties: 
\begin{enumerate}
\item  $\otop( \spec R)$ is even periodic. 
\item We have a canonical identification $\pi_0 \otop( \spec R) \simeq R $ and
a canonical identification of the formal group of $\otop( \spec R)$ and the
formal completion $\widehat{C}$. 
\end{enumerate}

The construction makes the assignment $(\spec R \to \mell) \mapsto \otop( \spec R)$
into a \emph{functor} from the affine \'etale site of $\mell$ to the
$\infty$-category of $\e{\infty}$-rings, 
and one defines
\begin{equation} \label{invlimTMF}\TMF = \Gamma( \mell, \otop) \stackrel{\mathrm{def}}{=} \varprojlim_{\spec
R \to \mell} \otop( \spec R).\end{equation}

The moduli stack of elliptic curves is \emph{regular}: any \'etale map $\spec R \to \mell$
has the property that $R$ is a regular, two-dimensional domain. 
The Galois theory of each $\otop( \spec R)$ is thus purely algebraic in view of
\Cref{etalegalois}. 
It follows that from the 
expression \eqref{invlimTMF} that we have a fully faithful embedding
\begin{equation} \label{easyqc}\md^{\omega}( \TMF) \subset \varprojlim_{\spec R \to \mell} \md^{\omega}( \otop(
\spec R)),  \end{equation}
which proves that an \emph{upper bound} for the Galois group of $\TMF$ is
given by the Galois group  of the moduli stack of elliptic curves. 
It is a folklore result that the moduli stack of elliptic curves, over
$\mathbb{Z}$, is simply connected; see for instance \cite{conrad}. Therefore, one has: 

\begin{theorem} \label{TMFsep}
$\TMF$ is separably closed, i.e., has trivial Galois group. 
\end{theorem} 

Using more sophisticated arguments, one can 
calculate the Galois groups not only of $\TMF$, but also of various
localizations (where the algebraic stack is no longer simply connected). This
proceeds by a strengthening of \eqref{easyqc}. 

\begin{definition} 
The $\infty$-category $\qcoh( \otop)$ of \textbf{quasi-coherent
$\otop$-modules} is the inverse limit $\varprojlim_{\spec R \to \mell} \md(
\otop( \spec R))$. 
\end{definition} 

As usual, we have an adjunction
\[ \md( \TMF) \rightleftarrows \qcoh( \otop),  \]
since $\TMF$ is the $\e{\infty}$-ring of endomorphisms of the unit in $\qcoh(
\otop)$. At least away from the prime 2 (this restriction is removed in
\cite{MM}), it is a result of Meier, proved in \cite{meier}, that the adjunction
is an equivalence: $\TMF$-modules are equivalent to quasi-coherent
$\otop$-modules. In particular, the unit object in $\qcoh( \otop)$ is compact,
which would not have been obvious a priori. 
It follows that we can make  a stronger version of the argument in
\Cref{TMFsep}. We will do this below in more generality. 

In \cite{MM}, L. Meier and the author 
formulated a more general context for ``affineness'' results such as this. 
We review the results. Let $M_{FG}$ be the moduli stack of formal groups. Let $X$ be a Deligne-Mumford stack and let $X \to
M_{FG}$ be a flat map. It follows that for every \'etale map $\spec R \to X$,
the composite $\spec R \to X \to M_{FG}$ is flat and there is a canonically
associated even periodic, \emph{Landweber-exact}  multiplicative homology theory associated to
it. An \emph{even periodic refinement} of this data is a lift of the diagram of
homology theories to $\e{\infty}$-rings. In other words, it is a sheaf $\otop$
of even periodic $\e{\infty}$-rings on the affine \'etale site of $X$ with
formal groups given by the map $X \to M_{FG}$. This enables in particular the
construction of an $\e{\infty}$-ring $\Gamma(X, \otop)$ of \emph{global
sections}, obtained as a homotopy limit in a similar manner as
\eqref{invlimTMF}, and a stable homotopy theory $\qcoh( \otop)$ of
quasi-coherent modules. 

Now, one has: 
\begin{theorem}[{\cite[Theorem 4.1]{MM}}] \label{MMres}
Suppose $X \to M_{FG}$ is a flat, quasi-affine map and let the sheaf $\otop$
of $\e{\infty}$-rings on the \'etale site of $X$ define an even periodic
refinement of $X$. Then the natural adjunction
\[ \md( \Gamma( X, \otop)) \rightleftarrows \qcoh ( \otop),  \]
is an equivalence of $\infty$-categories. 
\end{theorem} 

In particular, in \cite[Theorem 5.6]{MM}, L. Meier and the author showed that, given $X \to
M_{FG}$ \emph{quasi-affine}, then one source of Galois extensions of $\Gamma(
X, \otop)$ was the Galois theory of the \emph{algebraic} stack. 
If $X$ is regular, we can give the following refinement. 

\begin{theorem} 
\label{galstack}
Let $X$ be a regular Deligne-Mumford stack. 
Let $X \to M_{FG}$ be a flat, quasi-affine map  and fix an even periodic sheaf
$\otop$ as above. Then we have a canonical identification
\[ \pi_1 ( \md( \Gamma(X, \otop))) \simeq \pi_1^{\mathrm{et}} X. \]
\end{theorem} 
\begin{proof} 
This is now a quick corollary of the machinery developed so far. By 
\Cref{MMres}, we can identify modules over $\Gamma(X, \otop)$ with
quasi-coherent sheaves of $\otop$-modules. In particular, we can equivalently
compute the Galois group, which is necessarily the same as the \emph{weak}
Galois group, of $\qcoh(\otop)$. Using
\[ \qcoh(\otop) = \varprojlim_{\spec R \to X} \md( \otop( \spec R)),  \]
where the inverse limit ranges over all \'etale maps $\spec R \to X$, 
we find that the weak Galois groupoid of $\qcoh( \otop)$ is the colimit of the weak Galois
groupoids of the various $\otop( \spec R)$. Since we know that these are
algebraic (\Cref{etalegalois}), we conclude that we arrive precisely at the
colimit of Galois groupoids that computes the Galois groupoid of $X$. 
\end{proof} 

In addition to the case of $\TMF$, we find:
\begin{corollary} 
\begin{enumerate}
\item 
The Galois group of $\Tmf_{(p)}$ (for any prime $p$) is equal to the \'etale
fundamental group of $\mathbb{Z}_{(p)}$. 
\item 
The Galois group of $KO$ is $\mathbb{Z}/2$: the map $KO \to KU$ exhibits $KU$
as the Galois closure of $KO$. 
\end{enumerate}
\end{corollary} 
Here $\Tmf$ is the non-connective, non-periodic flavor of topological modular
forms associated to the compactified moduli stack of elliptic curves. 
\begin{proof} 
The first claim follows because the compactified moduli stack of elliptic
curves is geometrically simply connected;  this follows via the
expression as a weighted projective stack $\mathbb{P}(4, 6)$ when $6$ is
inverted. The second
assertion follows from \Cref{etalegalois}, which shows that $KU$ is simply
connected, since $\spec \mathbb{Z}$ is. 
\end{proof} 

\subsection{$K(n)$-local homotopy theory}

Let $K(n)$ be a Morava $K$-theory at height $n$. The $\infty$-category
$L_{K(n)} \sp$ of $K(n)$-local spectra, which plays a central role in modern
chromatic homotopy theory, has been studied extensively in the
monograph \cite{HoveyS}. 
$L_{K(n)} \sp$ is a basic  example of a  stable homotopy theory where the unit object
is \emph{not} compact, although $L_{K(n)} \sp$ is compactly generated (by the
localization of a finite type $n$ complex, for instance). We describe the
Galois theory of $L_{K(n)} \sp$ here, 
following ideas of \cite{DH, BR2, rognes}, and many other authors. 

According to the ``chromatic'' picture, phenomena in stable homotopy theory are
approximated by the geometry of the moduli stack $M_{FG}$ of formal groups.
When localized at a prime $p$, there is a basic open substack $M_{FG}^{ \leq
n}$ of $M_{FG}$
parametrizing formal groups whose \emph{height} (after specialization to any
field of characteristic $p$) is $\leq n$. There is a closed substack $M_{FG}^{n} \subset
M_{FG}^{\leq n}$ parametrizing formal groups of height \emph{exactly} $n$ over
$\mathbb{F}_p$-algebras. The operation of $K(n)$-localization corresponds
roughly to formally completing along this closed substack (after first
restricting to the open substack $M_{FG}^{\leq n}$, which is
$E_n$-localization). In particular, the Galois theory of $L_{K(n)} \sp$ should
be related to that of this closed substack. 

It turns out that $M_{FG}^{n}$ has an extremely special geometry. 
The substack $M_{FG}^{n}$ is essentially the ``classifying stack'' of a
large profinite group (with a slight Galois twist) known as the \emph{Morava
stabilizer group.}

\renewcommand{\G}{\mathbb{G}}
\newcommand{\Ge}{\mathbb{G}^{\mathrm{ext}}}
\begin{definition} 
Let $k = \overline{\mathbb{F}_p}$ and consider a height $n$ formal group
$\mathfrak{X}$ over $k$. We define the \textbf{$n$th Morava stabilizer group}
$\G_n$ to be the automorphism group of $\mathfrak{X}$ (in the category of
formal groups). 
\end{definition} 

Any two height $n$ formal groups over $k$ are isomorphic, so it does not matter
which one we use. 

\begin{definition} 
We define the \textbf{$n$th extended Morava stabilizer group} $\Ge_n$
to be the group of pairs $(\sigma, \phi)$ where $\sigma \in \mathrm{Aut}(
\overline{\mathbb{F}_p}/\mathbb{F}_p)$ and $\phi\colon  \mathfrak{X} \to \sigma^*
\mathfrak{X}$ is an isomorphism of formal groups. 
\end{definition} 

In fact, $\mathfrak{X}$ can be defined over the prime field $\mathbb{F}_p$
itself, so that $\sigma^* \mathfrak{X}$ is canonically identified with
$\mathfrak{X}$, 
and in this case, every automorphism of $\mathfrak{X}$ is defined over
$\mathbb{F}_{p^n}$. This gives $\G_n$ a natural profinite structure (by looking
explicitly at coefficients of power series), and $\Ge_n \simeq \G_n \rtimes
\mathrm{Gal}(\overline{\mathbb{F}_p}/\mathbb{F}_p)$. 

The picture is that the stack $M_{FG}^{n}$ is the classifying stack of the
group \emph{scheme} of automorphisms of a height $n$ formal group over
$\mathbb{F}_p$. This itself is a pro-\'etale
group scheme which becomes constant after extension of scalars to
$\mathbb{F}_{p^n}$. This picture is justified by the result that any two $n$
formal group are \'etale locally isomorphic, and the scheme of automorphisms
is in fact as claimed. 

This picture has been reproduced closely in chromatic homotopy theory. 
Some of the most important objects in $L_{K(n)}\sp$ are the \emph{Morava
$E$-theories} $E_n$. 
Let $\kappa$ be a perfect field of characteristic $p$ and let $\mathfrak{X}$ be
a formal group of height $n$ over $\kappa$, defining a map $\spec \kappa \to
M_{FG}^{n}$. The ``formal completion'' of $M_{FG}$ along this map can be
described by \emph{Lubin-Tate theory}; in other words, the universal
deformation $\mathfrak{X}_{\mathrm{univ}}$ of the formal group $\mathfrak{X}$ lives over the ring
$W(\kappa)[[v_1, \dots, v_{n-1}]]$ for $W(\kappa)$ the ring of Witt vectors on
$\kappa$. 
The association $(\kappa, \mathfrak{X}) \mapsto (W(\kappa)[[v_1, \dots, v_{n-1}]],
\mathfrak{X}_{\mathrm{univ}})$ 
defines a functor from pairs $(\kappa, \mathfrak{X})$
 to pairs of complete local rings and formal groups over them.  

The result of Goerss-Hopkins-Miller \cite{goersshopkins, rezkHM} is that the above functor can be lifted to topology. 
Each pair
$(W(\kappa)[[v_1, \dots, v_{n-1}]],
\mathfrak{X}_{\mathrm{univ}})$ can be realized by a homotopy commutative ring
spectrum $E_n = E_n( \kappa; \mathfrak{X})$ in view of the Landweber exact
functor theorem. However, in fact one can construct a functor (essentially
uniquely)
\[ (\kappa, \mathfrak{X}) \mapsto E_n(\kappa; \mathfrak{X})  \]
to the $\infty$-category of $\e{\infty}$-rings, 
lifting this diagram of formal groups: for each $(\kappa, \mathfrak{X})$,
$E_n(\kappa; \mathfrak{X})$ is even periodic with formal group identified with
the universal deformation $\mathfrak{X}_{\mathrm{univ}}$ over $W(\kappa)[[v_1,
\dots, v_{n-1}]]$. 

We formally now state a definition that we have used before. 
\begin{definition} 
 Any $E_n(\kappa; \mathfrak{X})$ will be referred to as a \textbf{Morava
 $E$-theory} and will be sometimes simply written as $E_n$. 
\end{definition}

Since $M_{FG}^{n}$ is the classifying stack of a pro-\'etale 
group scheme, we should expect, if we take $\kappa = \overline{\mathbb{F}_p}$,
an appropriate action of the extended Morava stabilizer group on $E_n(\kappa;
\mathfrak{X})$. An action of the group $\Ge_n$ is given to us 
on $E_n(\kappa; \mathfrak{X})$ by the Goerss-Hopkins-Miller theorem. 
However, we should expect a ``continuous'' action of $\Ge_n$ on $E_n( \kappa;
\mathfrak{X})$ on $\md( E_n(\kappa; \mathfrak{X}))$
whose homotopy fixed points are $L_{K(n)} \sp$.

Although this does not seem to have been fully made precise, 
given an open subgroup $U \subset \Ge_n$, Devinatz-Hopkins \cite{DH} construct
homotopy fixed points $E_n(\kappa; \mathfrak{X})^{hU}$ which have the desired
properties (for example, if $U \subset \Ge_n$, one obtains $L_{K(n)} S^0$). 
It was observed in \cite{rognes} that for $U \subset \Ge_n$ open normal, the maps
\[ L_{K(n)} S^0 \to E_n(\kappa; \mathfrak{X})^{hU}  \]
are $\Ge_n/U$-Galois in $L_{K(n)} \sp$; they become \'etale after base-change
to $E_n(\kappa; \mathfrak{X})$. 
The main result 
of this section is that this gives precisely the Galois group of $K(n)$-local
homotopy theory. 

\begin{theorem} \label{knlocalgal}
The Galois group of $L_{K(n)} \sp$ (which is also the weak Galois group) is the extended Morava stabilizer group
$\Ge_n$. 
\end{theorem} 

Away from the prime $2$, this result is essentially due to
Baker-Richter \cite{BR2}. 
We will give a direct proof using descent theory. 
Let $E_n$ be a Morava $E$-theory. 
Using descent for linear $\infty$-categories along $L_n S^0 \to
E_n$ (\Cref{proconstdescentC} and \Cref{HR}), we find:  

\begin{proposition} 
$E_n \in \clg( L_{K(n)} \sp)$ satisfies descent. 
In particular, we have an equivalence
\[ L_{K(n)} \sp \simeq \mathrm{Tot}\left( L_{K(n)} \md(E_n) \rightrightarrows
L_{K(n)} \md(L_{K(n)}(E_n \otimes E_n)) \triplearrows \dots \right).    \]
\end{proposition} 

\begin{proof} 
This follows directly from the fact that since the cobar construction $L_n S^0
\to E_n$ defines a constant pro-object in $\sp$ (with limit $L_nS^0$), it defines a constant
pro-object (with limit $L_{K(n)}S^0$) in $L_{K(n)}\sp$ after $K(n)$-localizing everywhere. 
\end{proof}

Therefore, we need 
to understand the Galois groups of stable homotopy theories such as
$L_{K(n)}\md(E_n)$. We did most of the work in \Cref{etalegalois}, although the
extra localization adds a small twist that we should check first. 

Let $A$ be an even periodic $\e{\infty}$-ring with $\pi_0 A$  a complete
regular local ring with maximal ideal $\mathfrak{m} = (x_1, \dots, x_n)$,
where $x_1 ,\dots, x_n$ is a system of parameters for $\mathfrak{m}$. Let
$\kappa(A) = A/(x_1,\dots, x_n)$ be 
the topological ``residue field'' of $A$, as considered earlier.

\begin{proposition} 
\label{etalegall}
Given a $\kappa(A)$-local $A$-module $M$, the following are equivalent: 
\begin{enumerate}
\item $M$ is dualizable in $L_{\kappa(A)}  \md(A)$. 
\item $M$ is a perfect $A$-module. 
\end{enumerate}
\end{proposition} 
\begin{proof} 
Only the claim that the first assertion implies the second needs to be shown.
If $M$ is dualizable in $L_{\kappa(A)} \md(A)$, then it follows that, since
the \emph{homology theory} $\kappa(A)_*$ is a monoidal functor, 
$\kappa(A)_*(M)$ must be dualizable in the category of graded
$\kappa(A)_*$-modules. In particular, $\kappa(A)_0(M)$ and $\kappa(A)_1(M)$ are
finite-dimensional vector spaces. From this, it follows that $\pi_*(M)$ itself
must be a finitely generated $\pi_*(A)$-module, since $\pi_*(M)$ is
(algebraically) complete. 
For example, given any $i$, we show that 
the $\pi_0(A)$-module
\( \pi_0 (M/(x_1, \dots, x_i) M) \) is finitely generated by descending
induction on $i$; when $i = 0$ it is the assertion we want. When $i = n$, the
finite generation follows from our earlier remarks. If we know finite
generation at $i$, then we use the cofiber sequence
\[ M/(x_1, \dots, x_{i-1}) \stackrel{x_i}{\to} M/(x_1, \dots, x_{i-1}) \to
M/(x_1, \dots, x_i),  \]
to find that
\[ \pi_0(M/(x_1, \dots, x_{i-1})) \otimes_{\pi_0(A)} \pi_0(A)/(x_i) \subset
\pi_0( M/(x_1, \dots, x_i)),  \]
is therefore finitely generated. However, by the $x_i$-adic completeness of
$\pi_0 (M/(x_1, \dots, x_{i-1}))$, this implies that $\pi_0(M/(x_1, \dots,
x_{i-1}))$ is finitely generated. 

Finally, since $\pi_*(A)$ has finite global dimension, this is
enough to imply that $M$ is perfect as an $A$-module. 
\end{proof}

\begin{proof}[Proof of \Cref{knlocalgal}]
We thus get an equivalence
\[ \clgw( L_{K(n)} \sp)
\simeq \mathrm{Tot}\left( \clgw( L_{K(n)} \md(E_n)) \rightrightarrows \clgw(
L_{K(n)} \md(E_n \otimes E_n)) \triplearrows \dots \right).
\]
However, we have shown, as a consequence of \Cref{etalegall} and \Cref{etalegalois}, that 
$\clgw( L_{K(n)} \md(E_n))$ is actually equivalent to the full subcategory spanned
by the \emph{finite \'etale} commutative algebra objects. Since finite \'etale
algebra objects are preserved under base change, we can replace the above
totalization via
\[ \clgw( L_{K(n)} \sp)
\simeq \mathrm{Tot}\left( \clgw_{\mathrm{alg}}( L_{K(n)} \md(E_n))
\rightrightarrows \clgw_{\mathrm{alg}}(
L_{K(n)} \md(E_n \otimes E_n)) \triplearrows \dots \right),
\]
where the subscript $\mathrm{alg}$ means that we are only looking at the
classical finite covers, i.e., the category is equivalent to the category of
finite \'etale covers of $ \pi_0$. 
In other words, we obtain a cosimplicial commutative ring, and we need to take
the geometric realization of the \'etale fundamental groupoids to obtain the
fundamental group of $L_{K(n)} \sp$. 

Observe that each commutative ring $\pi_0 L_{K(n)}(E_n^{\otimes m})$ is
complete with respect to the ideal $(p, v_1, \dots, v_{n-1})$, in view of the
$K(n)$-localization. 
The algebraic fundamental group is thus invariant under quotienting by this
ideal. After we do this, we obtain precisely a presentation for the moduli
stack $M_{FG}^{n}$, so the Galois group of $L_{K(n)} \sp$ is that of this
stack. As we observed earlier, this is precisely the extended Morava stabilizer group. 
\end{proof}

\subsection{Purity}
We next describe a ``purity'' phenomenon in the Galois groups of
$\e{\infty}$-rings in chromatic homotopy theory: they appear to depend only on
their $L_1$-localization. We conjecture below that this is true in general, and
verify it in a few special (but important) cases. 

We return to the setup of \Cref{sec:dstack}.
Let $R$ be an $\e{\infty}$-ring that 
arises as the global sections of the structure sheaf (``functions'') on a
derived stack $(\mathfrak{X}, \otop)$ which is a refinement of a flat
map $X \to M_{FG}$. 
Suppose further that $(\mathfrak{X}, \otop)$ is \emph{0-affine}, i.e., the
natural functor $\md(\Gamma( \mathfrak{X}, \otop)) \to \qcoh( \mathfrak{X})$
is an equivalence, and that $X$ is
\emph{regular}. 

In this case, we have: 
\begin{theorem}[$KU$-purity] \label{E1purity}
The map $R \to L_{KU} R$ induces an isomorphism on Galois groups. \end{theorem} 

In order to prove this result, we recall the \emph{Zariski-Nagata purity theorem}, for
which a useful reference is 
Expos\'e X of \cite{SGA2}. 
\begin{theorem}[Zariski-Nagata] \label{ZN} Let $X$ be a 
regular noetherian scheme and $U \subset X$ an open subset such that $X
\setminus U$ has codimension $\geq 2$ in $X$. Then the restriction functor
establishes an equivalence of categories between finite \'etale covers of $X$
and finite \'etale covers of $U$. 
\end{theorem} 

If $X$ is instead a regular Deligne-Mumford stack, and $U \subset X$ is an open
substack whose complement has codimension $\geq 2$ (a condition that makes
sense \'etale locally, and hence for $X$), then it follows from the above
and descent theory that finite \'etale
covers of $X$ and $U$ are still equivalent.
\begin{proof}[Proof of  \Cref{E1purity}] 
First we work localized at a prime $p$, so that $L_{KU} \simeq L_1$. 
In this case, the result is a now a direct consequence of various results in the preceding sections
together with \Cref{ZN}. 

Choose a derived stack $(\mathfrak{X}, \otop)$ whose global sections give $R$;
suppose $\mathfrak{X}$ is an even periodic refinement of an ordinary
Deligne-Mumford stack $X$,
with a flat, affine map $X \to M_{FG}$. 
Then $L_1 R$ can be recovered as the $\e{\infty}$-ring of functions on the open
substack of $(\mathfrak{X}, \otop)$ corresponding to the open substack $U$ of $X$
complementary to closed substack cut out by the ideal $(p, v_1)$. The derived
version of $U$ is also 0-affine, as observed in \cite[Proposition 3.27]{MM}.

Now, in view of \Cref{galstack}, the Galois group of $L_1 R$ is that of the open
substack $U$, and the Galois group of $R$ is that of $X$. However, the
Zariski-Nagata theorem implies that the inclusion $U \subset X$ induces an
isomorphism on \'etale fundamental groups. Indeed, the complement of $U
\subset X$ has codimension $\geq 2$  as $(p, v_1)$ is a regular sequence on $X$
by flatness and thus cuts out a codimension two substack of $X$. 

To prove this integrally, we need to piece together the different primes
involved. 
Given any $\e{\infty}$-ring $A$, it follows from 
descent theory that there is a sheaf $\mathrm{Gal}_G$ of (ordinary) categories on 
the Zariski site of $\spec \pi_0 A$, such that on a basic open affine $U_f
= \spec
\pi_0 A[f^{-1}] \subset \spec \pi_0 A$, 
$\mathrm{Gal}_G(U_f)$ is the groupoid of $G$-Galois extensions of the
localization $A[f^{-1}]$. Thus we can prove:

\begin{lemma} 
\label{galloc}
Fix a finite group $G$. 
Let $R \to R'$ be a morphism of $\e{\infty}$-rings with the following properties: 
\begin{enumerate}
\item  $R \to R'$ induces an equivalence of categories $\mathrm{Gal}_G(R_{(p)}) \to
\mathrm{Gal}_G(R'_{(p)})$ for each $p$. 
\item $R_{\mathbb{Q}}\to R'_{\mathbb{Q}}$ 
induces an equivalence of categories $\mathrm{Gal}_G(R_{\mathbb{Q}}) \to
\mathrm{Gal}_G(R'_{\mathbb{Q}})$. 
\end{enumerate}
Then the natural functor $\mathrm{Gal}_G(R) \to \mathrm{Gal}_G(R')$ is an
equivalence of categories. 
\end{lemma} 
\begin{proof} 
By the above, there is a sheaf $\mathrm{Gal}(G; R)$ (resp. $\mathrm{Gal}(G;
R')$) of categories on $\spec \mathbb{Z}$, whose value
over an open affine $\spec \mathbb{Z}[N^{-1}]$ is 
the category of $G$-Galois extensions of $R[N^{-1}]$ (resp. of $R'[N^{-1}]$). 
These are  the pushforwards of the sheaves $\mathrm{Gal}_G$ on $\spec \pi_0 R,
\spec \pi_0 R'$
discussed above. 
Now
\Cref{galcolim}, together with the hypotheses of the lemma, imply that
the map of sheaves $\mathrm{Gal}(G; R) \to \mathrm{Gal}(G'; R)$ induces an
\emph{equivalence} of categories on each stalk over every point of $\spec
\mathbb{Z}$. It follows that the map induces an equivalence upon taking global
sections, which is the conclusion we desired. 
\end{proof}

This lemma let us conclude the proof of  \Cref{E1purity}. Namely, the map $R \to L_K R$
satisfies the two hypotheses of the lemma above, since in fact $R_{\mathbb{Q}}
\simeq (L_K R)_{\mathbb{Q}}$, and we have already checked the $p$-local case
above. 
\end{proof}

Using similar techniques, we can prove a purity result for the Galois groups of
the $E_n$-local spheres. 

\begin{theorem} 
The Galois theory of $L_n S^0$ is algebraic  and is given by that of $\mathbb{Z}_{(p)}$. 
\end{theorem} 
\begin{proof} 
We can prove this using descent along the map $L_n S^0 \to E_n$. 
Since this map admits descent, we find that 
\[ \clgf( L_n S^0) \simeq \mathrm{Tot}\left(
\clgf( E_n) \rightrightarrows \clgf(E_n \otimes E_n) \triplearrows \dots
\right). 
\]
Now, $E_n \otimes E_n$ does not have a regular noetherian $\pi_0$. However, $\clgf( E_n)$
is simply the ordinary category of finite \'etale covers of $\pi_0 E_n$, in
view of \Cref{etalegalois}. Therefore, we can replace the above totalization by
the analogous totalization where we only consider the \emph{algebraic} finite
covers at each stage (since the two are the same at the first stage). 
In particular, since the cosimplicial (discrete) commutative ring
\[ \pi_0 (E_n) \rightrightarrows \pi_0 (E_n \otimes E_n) \triplearrows \dots ,  \]
is a presentation for the algebraic stack $M_{FG}^{\leq n}$ of formal groups
(over $\mathbb{Z}_{(p)}$-algebras) of
height $\leq n$, we find that the Galois theory of $L_nS^0$ is the Galois
theory of this stack. The next lemma thus completes the proof. 
\end{proof} 

\begin{lemma} 
For $n \geq 1$,
the maps of stacks $M_{FG}^{\leq n} \to M_{FG} \to \spec \mathbb{Z}_{(p)}$ induce
isomorphisms on
fundamental groups. 
\end{lemma} 
\begin{proof} 
The moduli stack of formal groups $M_{FG}$ has a presentation in terms of the
map $\spec L \to M_{FG}$, where $L$ is the \emph{Lazard ring} (localized at
$p$). $L$ is  a polynomial ring on a countable number of generators over
$\mathbb{Z}_{(p)}$. Similarly, $\spec L \times_{M_{FG}} \spec L$ is a
polynomial ring on a countable number of generators over $\spec
\mathbb{Z}_{(p)}$. 
The \'etale fundamental group of $\mathbb{Z}_{(p)}[x_1, \dots, x_n]$ is that of
$\mathbb{Z}_{(p)}$,\footnote{Let $R$ be a regular, torsion-free ring. Then we
claim that the fundamental group of $\spec R[x]$ and $\spec R$ are isomorphic
under the natural map. 
In fact, this is evident (e.g., by topology) if $R$ is a field of
characteristic zero. Now, if $K$ is the fraction field of $R$, then  to give an \'etale cover of
$\spec R[x]$ is equivalent to giving an \'etale cover of $\spec K[x]$ (i.e.,
an \'etale $K[x]$-algebra $R_K'$) such that
the normalization of $R[x]$ in $R_K'$ is \'etale over $R[x]$, since \'etale
extensions preserve normality. 
The \'etale $K[x]$-algebra $R_K'$ is necessarily of the form $L[x]$ where $L$
is a finite separable extension of $K$ (if it is connected, at least). 
In order for this normalization to be \'etale over $R[x]$, the normalization of
$R$ in $L$ must be \'etale over $R$.} and by taking filtered colimits, the same follows for a
polynomial ring over $\mathbb{Z}_{(p)}$ over a countable number of variables. 
Thus, the \'etale fundamental group $M_{FG}$ is that of $\spec
\mathbb{Z}_{(p)}$. The last assertion follows because, again, the deletion of
the closed subscheme cut out by $(p, v_1)$ does not affect the \'etale
fundamental group in view of the Zariski-Nagata theorem (applied to the
infinite-dimensional rings by the filtered colimit argument). 
\end{proof} 

The above results suggest the following purity conjecture. 
\begin{conj}
\label{purityconj} Let $R$ be \emph{any} $L_n$-local $\e{\infty}$-ring. 
The map $R \to L_1 R$ induces an isomorphism on fundamental groups. \end{conj}

\Cref{purityconj} is supported by the observation that, although not every
$L_n$-local $\e{\infty}$-ring has a regular $\pi_0$ (or anywhere close),
$L_n$-local $\e{\infty}$-rings seem to built from such at least to some extent.
For example, the free $K(1)$-local $\e{\infty}$-ring on a generator is known to
have an infinite-dimensional regular $\pi_0$. 

\begin{remark} 
Conjecture~\ref{purityconj} cannot be valid for general $L_n S^0$-linear stable
homotopy theories: it is specific to $\e{\infty}$-rings. For example, it fails
for $L_{K(n)}  \sp$. 
\end{remark} 

\bibliographystyle{alpha} 
\bibliography{thick}

\begin{thebibliography}{{Mat}15b}

\bibitem[AG14]{AntieauGepner}
Benjamin Antieau and David Gepner.
\newblock Brauer groups and \'etale cohomology in derived algebraic geometry.
\newblock {\em Geom. Topol.}, 18(2):1149--1244, 2014.

\bibitem[Ang08]{angeltveit}
Vigleik Angeltveit.
\newblock Topological {H}ochschild homology and cohomology of {$A\sb \infty$}
  ring spectra.
\newblock {\em Geom. Topol.}, 12(2):987--1032, 2008.

\bibitem[Bal]{balmersep}
Paul Balmer.
\newblock Separable extensions in tt-geometry and generalized {Q}uillen
  stratification.
\newblock {\em Ann. Sci. {\'E}cole Norm. Sup.}
\newblock To appear.

\bibitem[Bal10]{Balmer}
Paul Balmer.
\newblock Tensor triangular geometry.
\newblock In {\em Proceedings of the {I}nternational {C}ongress of
  {M}athematicians. {V}olume {II}}, pages 85--112, New Delhi, 2010. Hindustan
  Book Agency.

\bibitem[Bal15]{balmerstack}
Paul Balmer.
\newblock Stacks of group representations.
\newblock {\em J. Eur. Math. Soc. (JEMS)}, 17(1):189--228, 2015.

\bibitem[Ban13]{banerjee}
Romie Banerjee.
\newblock Galois descent for real spectra.
\newblock {\em arXiv preprint arXiv:1305.4360}, 2013.

\bibitem[BDS15]{BDAS}
Paul Balmer, Ivo Dell'Ambrogio, and Beren Sanders.
\newblock Restriction to finite-index subgroups as \'etale extensions in
  topology, {KK}--theory and geometry.
\newblock {\em Algebr. Geom. Topol.}, 15(5):3025--3047, 2015.

\bibitem[BGT13]{BGT}
Andrew~J. Blumberg, David Gepner, and Gon{\c{c}}alo Tabuada.
\newblock A universal characterization of higher algebraic {$K$}-theory.
\newblock {\em Geom. Topol.}, 17(2):733--838, 2013.

\bibitem[BHH15]{BHH}
Ilan Barnea, Yonatan Harpaz, and Geoffroy Horel.
\newblock Pro-categories in homotopy theory.
\newblock {\em arXiv preprint arXiv:1507.01564}, 2015.

\bibitem[BIK11]{BIK}
David~J. Benson, Srikanth~B. Iyengar, and Henning Krause.
\newblock Module categories for finite group algebras.
\newblock In {\em Representations of algebras and related topics}, EMS Ser.
  Congr. Rep., pages 55--83. Eur. Math. Soc., Z\"urich, 2011.

\bibitem[BJ01]{borceux}
Francis Borceux and George Janelidze.
\newblock {\em Galois theories}, volume~72 of {\em Cambridge Studies in
  Advanced Mathematics}.
\newblock Cambridge University Press, Cambridge, 2001.

\bibitem[BR70]{BRB}
Jean B{\'e}nabou and Jacques Roubaud.
\newblock Monades et descente.
\newblock {\em C. R. Acad. Sci. Paris S\'er. A-B}, 270:A96--A98, 1970.

\bibitem[BR07]{BRrealiz}
Andrew Baker and Birgit Richter.
\newblock Realizability of algebraic {G}alois extensions by strictly
  commutative ring spectra.
\newblock {\em Trans. Amer. Math. Soc.}, 359(2):827--857 (electronic), 2007.

\bibitem[BR08]{BR2}
Andrew Baker and Birgit Richter.
\newblock Galois extensions of {L}ubin-{T}ate spectra.
\newblock {\em Homology, Homotopy Appl.}, 10(3):27--43, 2008.

\bibitem[BS15a]{BSpro}
Bhargav Bhatt and Peter Scholze.
\newblock The pro-\'etale topology for schemes.
\newblock {\em Ast\'erisque}, (369):99--201, 2015.

\bibitem[BS15b]{BhSch}
Bhargav Bhatt and Peter Scholze.
\newblock Projectivity of the {W}itt vector affine {G}rassmannian.
\newblock {\em arXiv preprint arXiv:1507.06490}, 2015.

\bibitem[BZFN10]{BFN}
David Ben-Zvi, John Francis, and David Nadler.
\newblock Integral transforms and {D}rinfeld centers in derived algebraic
  geometry.
\newblock {\em J. Amer. Math. Soc.}, 23(4):909--966, 2010.

\bibitem[Car00]{carlson}
Jon~F. Carlson.
\newblock Cohomology and induction from elementary abelian subgroups.
\newblock {\em Q. J. Math.}, 51(2):169--181, 2000.

\bibitem[CJF13]{cjf}
Alex Chirvasitu and Theo Johnson-Freyd.
\newblock The fundamental pro-groupoid of an affine 2-scheme.
\newblock {\em Appl. Categ. Structures}, 21(5):469--522, 2013.

\bibitem[DFHH14]{TMF}
Christopher~L. Douglas, John Francis, Andr{\'e}~G. Henriques, and Michael~A.
  Hill, editors.
\newblock {\em Topological modular forms}, volume 201 of {\em Mathematical
  Surveys and Monographs}.
\newblock American Mathematical Society, Providence, RI, 2014.

\bibitem[DH04]{DH}
Ethan~S. Devinatz and Michael~J. Hopkins.
\newblock Homotopy fixed point spectra for closed subgroups of the {M}orava
  stabilizer groups.
\newblock {\em Topology}, 43(1):1--47, 2004.

\bibitem[DHS88]{DHS}
Ethan~S. Devinatz, Michael~J. Hopkins, and Jeffrey~H. Smith.
\newblock Nilpotence and stable homotopy theory. {I}.
\newblock {\em Ann. of Math. (2)}, 128(2):207--241, 1988.

\bibitem[EKMM97]{EKMM}
A.~D. Elmendorf, I.~Kriz, M.~A. Mandell, and J.~P. May.
\newblock {\em Rings, modules, and algebras in stable homotopy theory},
  volume~47 of {\em Mathematical Surveys and Monographs}.
\newblock American Mathematical Society, Providence, RI, 1997.
\newblock With an appendix by M. Cole.

\bibitem[FP05]{FP}
Eric~M. Friedlander and Julia Pevtsova.
\newblock Representation-theoretic support spaces for finite group schemes.
\newblock {\em Amer. J. Math.}, 127(2):379--420, 2005.

\bibitem[Gai12]{gaitsnotes}
Dennis Gaitsgory.
\newblock Notes on geometric {L}anglands: generalities on {DG}-categories.
\newblock 2012.
\newblock Available at \url{http://math.harvard.edu/~gaitsgde}.

\bibitem[Gai13]{gaits}
Dennis Gaitsgory.
\newblock Sheaves of categories and the notion of 1-affineness.
\newblock 2013.
\newblock Available at \url{http://arxiv.org/abs/1306.4304}.

\bibitem[GH04]{goersshopkins}
P.~G. Goerss and M.~J. Hopkins.
\newblock Moduli spaces of commutative ring spectra.
\newblock In {\em Structured ring spectra}, volume 315 of {\em London Math.
  Soc. Lecture Note Ser.}, pages 151--200. Cambridge Univ. Press, Cambridge,
  2004.

\bibitem[GL]{GL}
David Gepner and Tyler Lawson.
\newblock Brauer groups and {G}alois cohomology of commutative
  $\mathbb{S}$-algebras.
\newblock To appear.

\bibitem[GR03]{GR}
Ofer Gabber and Lorenzo Ramero.
\newblock {\em Almost ring theory}, volume 1800 of {\em Lecture Notes in
  Mathematics}.
\newblock Springer-Verlag, Berlin, 2003.

\bibitem[Gro03]{sga1}
Alexander Grothendieck.
\newblock {\em Rev\^etements \'etales et groupe fondamental ({SGA} 1)}.
\newblock Documents Math\'ematiques (Paris) [Mathematical Documents (Paris)],
  3. Soci\'et\'e Math\'ematique de France, Paris, 2003.
\newblock S{\'e}minaire de g{\'e}om{\'e}trie alg{\'e}brique du Bois Marie
  1960--61. [Algebraic Geometry Seminar of Bois Marie 1960-61], Directed by A.
  Grothendieck, With two papers by M. Raynaud, Updated and annotated reprint of
  the 1971 original [Lecture Notes in Math., 224, Springer, Berlin; MR0354651
  (50 \#7129)].

\bibitem[Gro05]{SGA2}
Alexander Grothendieck.
\newblock {\em Cohomologie locale des faisceaux coh\'erents et th\'eor\`emes de
  {L}efschetz locaux et globaux ({SGA} 2)}.
\newblock Documents Math\'ematiques (Paris) [Mathematical Documents (Paris)],
  4. Soci\'et\'e Math\'ematique de France, Paris, 2005.
\newblock S{\'e}minaire de G{\'e}om{\'e}trie Alg{\'e}brique du Bois Marie,
  1962, Augment{\'e} d'un expos{\'e} de Mich{\`e}le Raynaud. [With an
  expos{\'e} by Mich{\`e}le Raynaud], With a preface and edited by Yves Laszlo,
  Revised reprint of the 1968 French original.

\bibitem[Hes09]{hess}
Kathryn Hess.
\newblock Homotopic {H}opf-{G}alois extensions: foundations and examples.
\newblock In {\em New topological contexts for {G}alois theory and algebraic
  geometry ({BIRS} 2008)}, volume~16 of {\em Geom. Topol. Monogr.}, pages
  79--132. Geom. Topol. Publ., Coventry, 2009.

\bibitem[HL13]{ambidexterity}
Michael Hopkins and Jacob Lurie.
\newblock Ambidexterity in {$K(n)$}-local stable homotopy theory.
\newblock 2013.
\newblock Available at \url{http://math.harvard.edu/~lurie/}.

\bibitem[HMS15]{HMS}
Drew Heard, Akhil Mathew, and Vesna Stojanoska.
\newblock Picard groups of higher real {$K$}-theory spectra at height {$p-1$}.
\newblock {\em arXiv preprint arXiv:1511.08064}, 2015.

\bibitem[Hop14]{kone}
Michael~J. Hopkins.
\newblock {$K(1)$}-local {$E\sb \infty$}-ring spectra.
\newblock In {\em Topological modular forms}, volume 201 of {\em Math. Surveys
  Monogr.}, pages 287--302. Amer. Math. Soc., Providence, RI, 2014.

\bibitem[HPS97]{axiomatic}
Mark Hovey, John~H. Palmieri, and Neil~P. Strickland.
\newblock Axiomatic stable homotopy theory.
\newblock {\em Mem. Amer. Math. Soc.}, 128(610):x+114, 1997.

\bibitem[HPS99]{HPS}
M.~J. Hopkins, J.~H. Palmieri, and J.~H. Smith.
\newblock Vanishing lines in generalized {A}dams spectral sequences are
  generic.
\newblock {\em Geom. Topol.}, 3:155--165 (electronic), 1999.

\bibitem[HR]{HallRydh}
Jack Hall and David Rydh.
\newblock Algebraic groups and compact generation of their derived categories
  of representations.
\newblock {\em Indiana Univ. Math. J.}
\newblock To appear.

\bibitem[HS98]{HS}
Michael~J. Hopkins and Jeffrey~H. Smith.
\newblock Nilpotence and stable homotopy theory. {II}.
\newblock {\em Ann. of Math. (2)}, 148(1):1--49, 1998.

\bibitem[HS99]{HoveyS}
Mark Hovey and Neil~P. Strickland.
\newblock Morava {$K$}-theories and localisation.
\newblock {\em Mem. Amer. Math. Soc.}, 139(666):viii+100, 1999.

\bibitem[(ht]{conrad}
User22479 (http://mathoverflow.net/users/22479/user22479).
\newblock Fundamental group of the moduli stack of elliptic curves.
\newblock MathOverflow.
\newblock URL:http://mathoverflow.net/q/105062 (version: 2012-08-24).

\bibitem[Joh86]{Johnstone}
Peter~T. Johnstone.
\newblock {\em Stone spaces}, volume~3 of {\em Cambridge Studies in Advanced
  Mathematics}.
\newblock Cambridge University Press, Cambridge, 1986.
\newblock Reprint of the 1982 edition.

\bibitem[Kel94]{keller}
Bernhard Keller.
\newblock Deriving {DG} categories.
\newblock {\em Ann. Sci. \'Ecole Norm. Sup. (4)}, 27(1):63--102, 1994.

\bibitem[Lur09]{HTT}
Jacob Lurie.
\newblock {\em Higher topos theory}, volume 170 of {\em Annals of Mathematics
  Studies}.
\newblock Princeton University Press, Princeton, NJ, 2009.

\bibitem[Lur11a]{DAGIX}
Jacob Lurie.
\newblock {DAG IX}: Closed immersions.
\newblock 2011.
\newblock Available at \url{http://math.harvard.edu/~lurie}.

\bibitem[Lur11b]{DAGss}
Jacob Lurie.
\newblock {DAG VII}: Spectral schemes.
\newblock 2011.
\newblock Available at \url{http://math.harvard.edu/~lurie}.

\bibitem[Lur11c]{DAGQC}
Jacob Lurie.
\newblock {DAG VIII}: Quasi-coherent sheaves and {T}annaka duality theorems.
\newblock 2011.
\newblock Available at \url{http://math.harvard.edu/~lurie}.

\bibitem[Lur11d]{DAGdesc}
Jacob Lurie.
\newblock {DAG XI}: Descent theorems.
\newblock 2011.
\newblock Available at \url{http://math.harvard.edu/~lurie}.

\bibitem[Lur11e]{DAGrat}
Jacob Lurie.
\newblock {DAG XIII}: Rational and {$p$}-adic homotopy theory.
\newblock 2011.
\newblock Available at \url{http://math.harvard.edu/~lurie}.

\bibitem[Lur11f]{DAGVIII}
Jacob Lurie.
\newblock Derived algebraic geometry viii: Quasi-coherent sheaves and {T}annaka
  duality theorems.
\newblock 2011.
\newblock Available at \url{http://math.harvard.edu/~lurie/}.

\bibitem[Lur14]{higheralg}
Jacob Lurie.
\newblock {\em Higher algebra}.
\newblock 2014.
\newblock Available at \url{http://math.harvard.edu/~lurie/higheralgebra.pdf}.

\bibitem[Mag14]{Magid}
Andy~R. Magid.
\newblock {\em The separable {G}alois theory of commutative rings}.
\newblock Pure and Applied Mathematics (Boca Raton). CRC Press, Boca Raton, FL,
  second edition, 2014.

\bibitem[Mat15a]{thick}
Akhil Mathew.
\newblock A thick subcategory theorem for modules over certain ring spectra.
\newblock {\em Geom. Topol.}, 19(4):2359--2392, 2015.

\bibitem[{Mat}15b]{toruspic}
Akhil {Mathew}.
\newblock {Torus actions on stable module categories, Picard groups, and
  localizing subcategories}.
\newblock {\em arXiv preprint arXiv:1512.01716}, 2015.

\bibitem[May01]{mayidem}
J.~P. May.
\newblock Idempotents and {L}andweber exactness in brave new algebra.
\newblock {\em Homology Homotopy Appl.}, 3(2):355--359, 2001.
\newblock Equivariant stable homotopy theory and related areas (Stanford, CA,
  2000).

\bibitem[Mei12]{meier}
Lennart Meier.
\newblock {\em United elliptic homology}.
\newblock PhD thesis, University of Bonn, 2012.

\bibitem[Mil92]{miller}
Haynes Miller.
\newblock Finite localizations.
\newblock {\em Bol. Soc. Mat. Mexicana (2)}, 37(1-2):383--389, 1992.
\newblock Papers in honor of Jos{\'e} Adem (Spanish).

\bibitem[MM15]{MM}
Akhil Mathew and Lennart Meier.
\newblock Affineness and chromatic homotopy theory.
\newblock {\em J. Topol.}, 8(2):476--528, 2015.

\bibitem[MNN15a]{MNNequiv2}
Akhil Mathew, Niko Naumann, and Justin Noel.
\newblock Derived induction and restriction theory.
\newblock {\em arXiv preprint arXiv:1507.06867}, 2015.

\bibitem[MNN15b]{MNNequiv}
Akhil Mathew, Niko Naumann, and Justin Noel.
\newblock Nilpotence and descent in equivariant stable homotopy theory.
\newblock {\em arXiv preprint arXiv:1507.06869}, 2015.

\bibitem[MNN15c]{MNN}
Akhil Mathew, Niko Naumann, and Justin Noel.
\newblock On a nilpotence conjecture of {J.P.} {M}ay.
\newblock {\em J. Topol.}, 8(4):917--932, 2015.

\bibitem[MR]{RM}
Fernando Muro and Oriol Ravent{\'o}s.
\newblock Transfinite {A}dams representability.
\newblock {\em Advances in Mathematics}.
\newblock To appear.

\bibitem[MS]{MS}
Akhil Mathew and Vesna Stojanoska.
\newblock The {P}icard group of topological modular forms via descent theory.
\newblock {\em Geom. Topol.}
\newblock To appear.

\bibitem[Pal97]{palmieri}
John~H. Palmieri.
\newblock A note on the cohomology of finite-dimensional cocommutative {H}opf
  algebras.
\newblock {\em J. Algebra}, 188(1):203--215, 1997.

\bibitem[Pal99]{palmieriquillen}
John~H. Palmieri.
\newblock Quillen stratification for the {S}teenrod algebra.
\newblock {\em Ann. of Math. (2)}, 149(2):421--449, 1999.

\bibitem[Pau15]{pauwels}
Bregje Pauwels.
\newblock {\em Quasi-{G}alois theory in tensor-triangulated categories}.
\newblock PhD thesis, University of California at Los Angeles, 2015.

\bibitem[Qui71]{equiv}
Daniel Quillen.
\newblock The spectrum of an equivariant cohomology ring. {I}, {II}.
\newblock {\em Ann. of Math. (2)}, 94:549--572; ibid. (2) 94 (1971), 573--602,
  1971.

\bibitem[Rav92]{ravenelorange}
Douglas~C. Ravenel.
\newblock {\em Nilpotence and periodicity in stable homotopy theory}, volume
  128 of {\em Annals of Mathematics Studies}.
\newblock Princeton University Press, Princeton, NJ, 1992.
\newblock Appendix C by Jeff Smith.

\bibitem[Rez98]{rezkHM}
Charles Rezk.
\newblock Notes on the {H}opkins-{M}iller theorem.
\newblock In {\em Homotopy theory via algebraic geometry and group
  representations ({E}vanston, {IL}, 1997)}, volume 220 of {\em Contemp.
  Math.}, pages 313--366. Amer. Math. Soc., Providence, RI, 1998.

\bibitem[Rob89]{robinsonobstruct}
Alan Robinson.
\newblock Obstruction theory and the strict associativity of {M}orava
  {$K$}-theories.
\newblock In {\em Advances in homotopy theory ({C}ortona, 1988)}, volume 139 of
  {\em London Math. Soc. Lecture Note Ser.}, pages 143--152. Cambridge Univ.
  Press, Cambridge, 1989.

\bibitem[Rog]{Rognes2}
John Rognes.
\newblock A {G}alois extension that is not faithful.
\newblock Available at \url{http://folk.uio.no/rognes/papers/unfaithful.pdf}.

\bibitem[Rog08]{rognes}
John Rognes.
\newblock Galois extensions of structured ring spectra. {S}tably dualizable
  groups.
\newblock {\em Mem. Amer. Math. Soc.}, 192(898):viii+137, 2008.

\bibitem[RST93]{RST}
Jan Reiterman, Manuela Sobral, and Walter Tholen.
\newblock Composites of effective descent maps.
\newblock {\em Cahiers Topologie G\'eom. Diff\'erentielle Cat\'eg.},
  34(3):193--207, 1993.

\bibitem[RV16]{RV}
Emily Riehl and Dominic Verity.
\newblock Homotopy coherent adjunctions and the formal theory of monads.
\newblock {\em Adv. Math.}, 286:802--888, 2016.

\bibitem[SS03]{schwedeshipley}
Stefan Schwede and Brooke Shipley.
\newblock Stable model categories are categories of modules.
\newblock {\em Topology}, 42(1):103--153, 2003.

\bibitem[ST92]{ST}
M.~Sobral and W.~Tholen.
\newblock Effective descent morphisms and effective equivalence relations.
\newblock In {\em Category theory 1991 ({M}ontreal, {PQ}, 1991)}, volume~13 of
  {\em CMS Conf. Proc.}, pages 421--433. Amer. Math. Soc., Providence, RI,
  1992.

\bibitem[Wil81]{wilkerson}
Clarence Wilkerson.
\newblock The cohomology algebras of finite-dimensional {H}opf algebras.
\newblock {\em Trans. Amer. Math. Soc.}, 264(1):137--150, 1981.

\end{thebibliography}

\end{document}